%% file: CAMSrevision-arxiv.tex
\documentclass[12pt]{article}

\usepackage{amssymb}
\usepackage{amsthm}
\usepackage{amsmath}
\usepackage{latexsym}
\usepackage{graphics}
\usepackage{xspace}
\usepackage{psfrag}
\usepackage{epsfig}
\usepackage{pst-all}

\RequirePackage[sf,bf,small,raggedright]{titlesec}
\usepackage{svg}
\usepackage[english]{babel}
\usepackage{csquotes}

\usepackage[a4paper,top=2cm,bottom=2cm,left=1cm,right=1cm,marginparwidth=1.75cm]{geometry}

\usepackage{amsmath}
\usepackage{amsthm}
\usepackage{amssymb}
\usepackage{hyperref}
\usepackage{mathtools}
\usepackage{natbib}
\usepackage{textcomp}
\usepackage{amssymb}
\usepackage{bbm}
\usepackage{fancyhdr}
\usepackage{amsmath}
\usepackage{amsbsy,amsthm}
\usepackage{amscd}
\usepackage{latexsym}
\usepackage{graphicx}
\usepackage{mathtools}
\DeclarePairedDelimiter\ceil{\lceil}{\rceil}
\DeclarePairedDelimiter\floor{\lfloor}{\rfloor}

\usepackage{pdfsync}
\usepackage{blkarray}
\usepackage{multirow}
\hypersetup{pdfauthor={Yatin Dandi, David Gamarnik, Lenka Zdeborová},pdftitle={Maximally-stable Local Optima in Random Graphs and Spin Glasses: Phase Transitions and Universality},%
            colorlinks, linktocpage=true, pdfstartpage=3, pdfstartview=FitV,%
    breaklinks=true, pdfpagemode=UseNone, pageanchor=true, pdfpagemode=UseOutlines,%
    plainpages=false, bookmarksnumbered, bookmarksopen=true, bookmarksopenlevel=1,%
    hypertexnames=true, pdfhighlight=/O,%
    urlcolor=orange, linkcolor=blue, citecolor=black, %pagecolor=RoyalBlue,%
        }

\usepackage{csquotes}

\newcommand{\R}{\mathbb{R}}

\newcommand{\sgn}{{\rm sgn}}
\newcommand{\N}{\mathbb{N}}
\newcommand{\Z}{\mathbb{Z}}
\newcommand{\E}{\mathbb{E}}

\newcommand{\ER}{Erdős-R\'{e}nyi~}

\newcommand{\pr}{\mathbb{P}}
\newcommand{\G}{\mathbb{G}}

\newcommand{\ignore}[1]{\relax}

\newcommand{\defeq}{\mathrel{\mathop:}=}

\newtheorem{lemma}{Lemma}
\newtheorem{theorem}{Theorem}
\newtheorem{definition}{Definition}
\newtheorem{corollary}{Corollary}
\newtheorem{proposition}{Proposition}

\newtheorem{assumption}{Assumption}

\numberwithin{equation}{section}

\newcommand{\argmin}{\mathop{\mathrm{argmin}}}

\DeclarePairedDelimiterX{\rbr}[1]{(}{)}{#1} % round bracket
\DeclarePairedDelimiterX{\sbr}[1]{[}{]}{#1}
\usepackage{authblk}
\usepackage{graphicx}
\usepackage{graphbox}
\usepackage{booktabs}
\usepackage{algorithm2e}
\usepackage{caption}
\usepackage{subcaption}

\newcommand{\dg}[1]{\textcolor{red}{\emph{ #1}}}
\newcommand{\Furedi}{F{\"u}redi}

\newcommand{\yd}[1]{\textcolor{blue}{\emph{ #1}}}

\input{macros}

\def\qed{{\hfill \vrule width 1.5mm height 1.5mm \medskip}}

 \title{Maximally-stable Local Optima in Random Graphs and Spin Glasses: Phase Transitions and Universality}
\author[1]{Yatin Dandi}

\author[2]{David Gamarnik}

\author[1]{Lenka Zdeborová}

\affil[1]{{\small Statistical Physics of Computation Laboratory, École Polytechnique Fédérale de Lausanne (EPFL). }}

\affil[2]{{\small Operations Research Center and Sloan School of Management, MIT, Cambridge, MA 02139}}

\date{}
\begin{document}

\maketitle

\begin{abstract}
We consider $h$-stable local optima of Ising spin glass models. These are defined as spin configurations such that for nearly all of the spins, flipping their values results in increasing energy by at least a given amount $h$, up to a normalization.  Spins satisfying this condition are referred to as $h$-stable spins for that configuration.
Similarly, we consider a very related notion of $h$-friendly partitions of a  graph. These are defined as bi-partitionings such that for most nodes,  the normalized number of neighbors within the node's partition exceed the normalized number of neighbors outside the partition by a certain amount $h$, with the nodes satisfying this condition being termed $h$-friendly. 
For spin glasses as well as sparse and dense random graphs, while restricting to bisections i.e. partitions of equal sizes, we prove the existence of a phase transition for the normalized energy
level $h$ around a universal value $h^\star$. For $h$ below the phase transition value $h^\star$, bisections exist where the number spins (nodes) which are not $h$-stable (not $h$-friendly) is sub-linear. 
Above the phase transition level $h^\star$ the smallest number of spins that are not $h$-stable  (not $h$-friendly) is linear. This confirms a conjecture from~\cite{behrens2022dis}. Our results also allow the characterization of possible energy values of stable local optima for varying $h$. In particular, for $h=0$, this rigorously proves seminal results in statistical physics regarding the so-called metastable states, such as in the work of \cite{bray1981metastable}. 

Our results extend a recent proof of the so-called Friendly Partition Conjecture in~\cite{ferber2022friendly} from the case $h=0$ to the case when $h$ takes general values. Furthermore, our work offers novel contributions to the rigorous analysis of the local optima in spin glasses, going beyond the results on the number of local optima in \cite{skcount} obtained by the first moment (annealed) method.

Our proofs are obtained by analyzing the model on sparse random graphs and adopting Lindeberg's type universality method to lift the results from sparse to dense graphs and spin systems. 

\end{abstract}

% \noindent
% {\bf Keywords:}

\tableofcontents

\section{Introduction}\label{sec:intro}
In combinatorics and theoretical computer science, the properties of partitions of vertices of random graphs have been analyzed in the context of well-known NP-hard problems of maximum cuts and minimum bisections \citep{gamarnikMaxcutSparseRandom2018,dembo2017extremal, alaouiLocalAlgorithmsMaximum2021}. In statistical physics, such partitions arise naturally due to the association of the vertices with spin values $\in \{-1,1\}$.
This leads to the equivalence between notions in different disciplines such as maximum cuts and minimum bisections of a weighted graph corresponding to the minimum energy configurations in Ising models, and spin glasses.

In this work, we consider the property of a partition which is called $h$-assortative/disassortative as recently studied in \citep{behrens2022dis}.
These are partitions of vertex sets into two parts such that for each vertex the 
difference between the number of neighbors of this vertex within the  part containing the vertex 
and the number of neighbors of the vertex
in the opposite part exceeds/does-not-exceed a given threshold $h$, upon some natural 
normalization.
This property generalizes related notions of friendly and unfriendly partitions \citep{ferber2022friendly,shelah1990graphs,aharoni1990unfriendly}, satisfactory/co-satisfactory graph partitions \citep{gerber2000algorithmic,bazganExistenceDeterminationSatisfactory2003,bazgan2007efficient,bazganSatisfactoryGraphPartition2010}, or internal partitions \citep{banInternalPartitionsRegular2013,linialAsymptoticallyAlmostEvery2017}.
All of these papers focus on the case $h=0$. Similar notions of assortative/disassortative partitions have been considered in presence of additional  minimum-degree constraints \citep{stiebitzDecomposingGraphsDegree1996,bazgan2007efficient,maDecomposingC4freeGraphs2019,liuConjectureSchweserStiebitz2021}, as judicious partitions \citep{bollobas2002problems},  as alliances in game theory \citep{kristiansenAlliancesGraphs2004},  as cohesive subsets \citep{morrisContagion2000} in the theory of 
contagion, as local minimum
bisections, local maximum cuts in combinatorial optimization \citep{angelLocalMaxcutSmoothed2017}, and as d-cuts in generalizations of the matching cut problem \citep{gomesFindingCutsBounded2021}.

To introduce our set up, we consider weighted graphs, 
denoted by $G=(V,W)$, where $V$ is the set of vertices with $n=\abs{V}$ 
nodes, and $W$ is a symmetric matrix in $\R^{n\times n}$ with entries $(w_{ij})_{i,j \in [n]^2}$ representing 
the edge weights for each possible pair of nodes. Concretely, we associate 
each pair of vertices $(i,j)\in [n]\times [n]$ with a weighted edge $w_{ij} \in \R$. Here
 $w_{ij}=0$ is naturally interpreted as the absence of an edge connecting $i$ and $j$. 
We denote by $d_i$ the degree of the node $i$, which is the cardinality of the set $\{j: w_{i,j}\ne 0\}$.

We define partitions as assignment of $\pm 1$ spins to each vertex. Namely, every  $\sigma\in \{+1,-1\}^n$
or $\sigma:[n]\rightarrow \{\pm 1\}$ encodes a partition of the vertex set $[n]$ into two subsets $V_+$
and $V_-$, corresponding to nodes mapped to $+1$ and $-1$, respectively. 
We say that the partition $\sigma$ is a bisection if $\abs{V_+}= \abs{V_-}$ when $n=\abs{V}$ is even, 
and $\abs{V_+}- \abs{V_-}=\pm 1$ for odd $n$. 

Each partition $\sigma$ is associated with a measure of single-spin flip stability for vertices defined as follows:
\begin{align}
    s_\sigma(v,W) &= \frac{1}{\sqrt{d}}(\sum_{u\in [n],\sigma_u=\sigma_v} w_{vu} - \sum_{u\in [n],\sigma_u\neq\sigma_v)} w_{vu})\label{eq:friendli}\\
    &={1\over \sqrt{d}} \sigma_v \sum_{u\in [n]} w_{vu}\sigma_u.
\label{eq:stab_e}
\end{align}
Here $d$ is a parameter of choice to 
be specified, but
usually related to the average
degree of the graph or the number of 
nodes. This
scaling will become justified later.
For ease the of notation, we shall occasionally suppress the dependence of $s_\sigma(v,W)$ on $W$ and simply 
write $s_\sigma(v)$. 
When the entries of $W$ take values $0,1$,
namely, when $W$ encodes a graph, $s_\sigma(v)$
is proportional to the difference 
between the number of neighbors of $v$
in the same part as $v$ vs the number
of neighbors in the opposite part.
Following \cite{ferber2022friendly}, we refer to this difference as ``friendliness" of the $v_{th}$ vertex. 
In statistical physics jargon this 
corresponds to the so-called ferromagnetic
interaction. Increasing the 
sum $\sum s_\sigma(v)$ (the number
of $v$ with non-negative $s_\sigma(v)$)
corresponds to finding partition 
with larger ``level'' of friendliness 
(respectively, larger number of friendly
nodes). Conversely, when
the entries are $-1$ and $0$
$s_\sigma(v)$ encodes the level of 
unfriendliness, and increasing the 
sum (the number) corresponds to decreasing
the level of friendliness (decreasing
the number of friendly nodes). 
In statistical physics jargon this
is the anti-ferromagnetic interaction.
Equivalently, this corresponds to the stability requirement $s_\sigma(v)\le -h$
when $W$ is converted to the adjacency matrix.

The quantities $s_\sigma(v)$ allow us to characterize  configurations satisfying some prescribed minimum stability/friendliness 
requirement. Given a parameter $h$,
we say that a vertex $v$ in a given configuration/partition represented by $\sigma$ is $h$-stable if
\begin{equation}\label{def:h-stable}
    s_\sigma(v) \geq h.
\end{equation}

\ignore{
A configuration $\sigma$ is called $h$-stable if each vertex is $h$-stable, i.e if 
\begin{equation}
  s_\sigma(v) \geq h , \ \forall v \in [n].
\end{equation}
}
The term ``stability" is motivated by statistical physics considerations as follows.
Associating each partition $\sigma$  with a Hamiltonian (energy)
\begin{equation}\label{eq:hamilton}
H(\sigma) \triangleq -\frac{1}{\sqrt{d}}\sum_{1\leq i<j\leq n}\sigma_i\,\sigma_j\,w_{ij},
\end{equation}
we can think of  $s_\sigma(v)$ as quantifying the  Hamiltonian (energy) change obtained  
upon switching the spin value of one of the nodes to an opposite value. Specifically,
letting 
$\sigma^{(v)}$ denote the configuration obtained from $\sigma$ by switching the value of the vertex $v$,
we observe from (\ref{eq:hamilton}), that
\begin{align*}
 H(\sigma^{(v)})-H(\sigma)= 2s_\sigma(v). 
\end{align*}
Thus, the $h$-stability of the partition $\sigma$ means every spin flip results in energy ``increase'' by 
at least $h$ (in quotes since we allow negative values of $h$ as well). When $h=0$, this simply means $\sigma$ is a locally minimum configuration with respect
to the objective function $H(\cdot)$.
$s_\sigma(v, W)$ can also be interpreted as 
the normalized product of the spin value $\sigma_v$ of $v$ with its local field $\sum_{i=1}^n w_{vi}\sigma_i$. 

Locally minimum configurations, namely configurations $\sigma$ satisfying $s_\sigma(v) \geq 0,  \forall v$   are also
called ``metastable states" of the Hamiltonian in the statistical physics language. 
Such states have been extensively studied in physics (e.g.: \citep{bray1981metastable,tanaka1980analytic,de1980white} for spin glass models).

\begin{figure}
    \centering
    \includegraphics[width=0.7\textwidth]{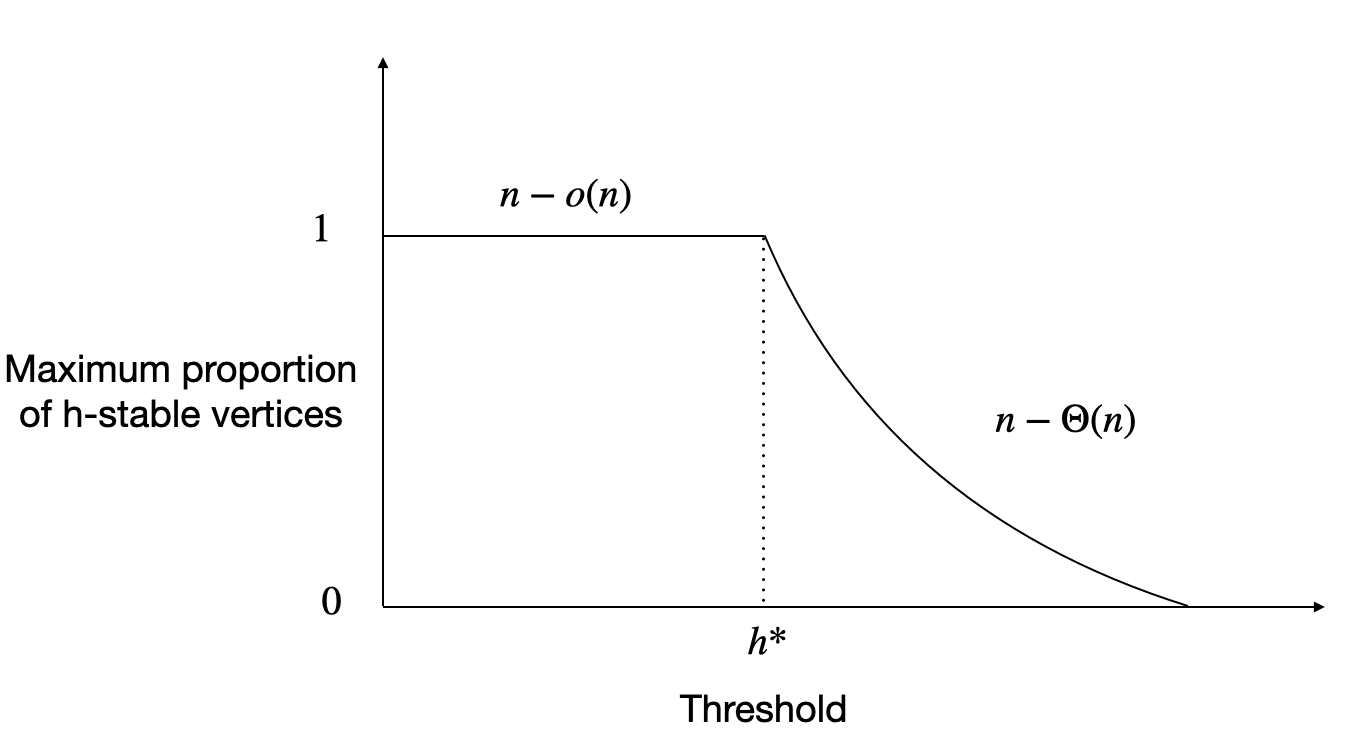}
    \caption{Illustration of the phase transition in the maximum number of $h$-stable vertices across all configurations.}
    \label{fig:phase_trans_ill_h}
\end{figure}

\begin{figure}
    \centering
    \includegraphics[width=\textwidth]{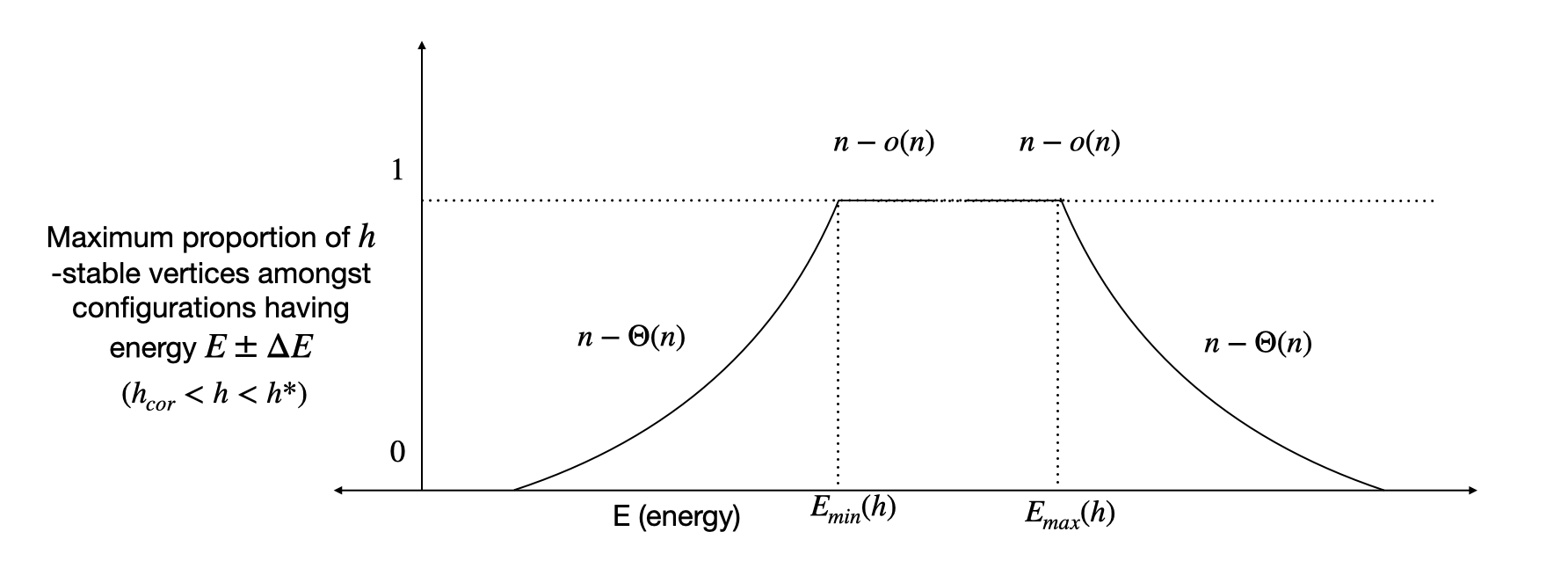}
    \caption{Illustration of the phase transitions in the maximum number of $h$-stable vertices amongst configurations satisfying certain energy constraints}
    \label{fig:phase_trans_ill_E}
\end{figure}

\paragraph{Informal summary of the 
main results.}
We now give an informal summary of our main findings. For clarity, we restrict to the case $h>0$, although the results extend to all $h$ bounded below by a chosen negative constant. 
\begin{enumerate}
    \item \emph{The maximum number of $h$-stable vertices in dense and sparse
    random graphs}. 
    Let $\G(n,1/2)$ denote an \ER random
    graph on $n$ nodes obtained by turning each pair $(i,j), 1\le i<j\le n$
    into an edge with probability $1/2$
    independently across all pairs. 
    Similarly, let
     $\G(n,d/n)$ denote a sparse random
     \ER graph
    on $n$ nodes obtained by turning
    each pair $(i,j), 1\le i<j\le n$
    into an edge with probability $d/n$,
    independently across all ${n\choose 2}$ pairs, where the parameter
    $d$ is fixed. While our results apply to general weighted graphs, for the sake of concreteness, we first explain them for the case of friendly bisections. For both $\G(n,1/2), \G(n,d/n)$, these correspond to setting $W$ as the adjacency matrix of the graph in the definition of $s_\sigma(v,W)$ in (\ref{eq:friendli}): 
    $w_{i,j}= 1$ if $(i,j)$ is
    an edge and $=0$ otherwise. 

For the case of dense \ER graphs $\G(n,1/2)$ we prove the following
phase transition around a threshold
$h^\star$, which we numerically estimate
to be $h^\star \approx 0.3513$: 
for $h<h^\star$ there exists bisections
$\sigma$ such that the number of 
$h$-stable ($h$-friendly) nodes is $n-o(n)$ with 
high probability (w.h.p.) as $n\to\infty$.
Conversely, when $h>h^\star$ 
every   bisection
results in at least $\Theta(n)$ nodes
which are not $h$-stable, w.h.p. as 
$n\to\infty$. For the special
case $h=0$ this result recovers
the one of \citep{ferber2022friendly}. 

For sparse random graphs our
results are similar, though the 
scaling necessarily switches from $n$
to $d$. Specifically, for the same
value of $h^\star$ we prove
that when $h<h^\star$ there exists
a bisection such that
the number of $h$-stable nodes 
is $n-o_d(1)n$ w.h.p. Here 
$o_d(1)$ denotes a function vanishing
in $d$. Alternatively, we can think of 
the setting where $d=d(n)$ is an arbitrarily
slowly growing function, in which 
case the number of $h$-stable nodes
is $n-o(n)$. Conversely, when
$h>h^\star$, for every bisection
there exist at least $\Theta_d(1)n$ many
$h$-non-stable nodes, where $\Theta_d(1)$ is a non-negative function
of $d$ bounded away from $0$ as $d \rightarrow \infty$. Since for any fixed $d$, Erdős-Rényi graphs contain $\Theta(n)$ vertices having no neighbors w.h.p as $n \rightarrow \infty$, the linear term $o_d(1)n$  
in our bounds on the maximum number of $h$-stable partitions are unavoidable.

Our results establish analogous transitions for unfriendly bisections (when $W$ is set as the negative of the adjacency matrix) and for spin-glass models (when $w_{ij} \sim \mathcal{N}(0,1)$). These results are illustrated in 
Figure~\ref{fig:phase_trans_ill_h} and proven in Theorems \ref{thm:sparse_thres}, \ref{thm:uni_sparse}, and \ref{thm:uni_dense}.

    \item 
    \emph{The interdependence between the energies of configurations and the maximum number of $h$-stable vertices:}

    We obtain results regarding the  energy/Hamiltonian ($H(\sigma))$ values of approximately $h$-stable bisections. Specifically, we
     prove for $\G(n,1/2)$ that for every fixed $h \in (0,h^\star)$, there exists $E_{max}(h)<0$ such that the maximum number of $h$-stable vertices amongst bisections 
     $\sigma$ with  energy (Hamiltonian) 
  $\approx nE$ is
     $n-o(n)$ when $E<E_{max}(h)$,
     and is $n-\Theta(n)$ when
     $E>E_{max}(h)$, w.h.p. as $n\to\infty$. 
     Interestingly, at $h=0$, 
     $E_{max}(h)$  is strictly negative, the 
 numerical value of which we found
 to be $E_{max}(0)\approx -0.2857$. This confirms rigorously 
 the computations from the seminal work of \cite{bray1981metastable} on metastable states in spin glass models which 
 was obtained using physics arguments.
     
    The strict negativity of $E_{max}(0)$ implies   that every friendly bisection results in a significant 
     imbalance between the in-degrees 
     and out-degrees: the difference between the sum of in-degrees and out-degrees must be order $\Theta(n^{3\over 2})$. Contrast
     this with the fact (which is easy
     to establish)
     that for any  bisection fixed a priori (without looking at the graph)  this imbalance is order $\Theta(n)$ w.h.p. It is also easy to show
     using a straightforward union
     bound that the maximum imbalance is  $\Theta(n^{3\over 2})$ as well, w.h.p.

Finally and conversely, we obtain analogous results for the minimum values of the energy, which fall into two distinct categories depending on the range of 
$h$. Specifically, we prove the existence of $h_{\rm cor} \approx 0.2860 <h^\star$, such that
     that for every $h\in (h_{\rm cor},h^\star)$  there exists 
     $E_{\rm min}(h)$ with the property 
     that the maximum number of $h$-stable vertices amongst configurations having energy  $\approx nE$ is $n-o(n)$ for $E>E_{min}(h)$ and is $n-\Theta(n)$ for $E < E_{min}(h)$. This is a non-trivial result, in the following
sense. Notice that the ground state, namely
the state with minimum energy is $0$-stable as any spin flip can only increase energy. When the energy $E_{min}(h)$ is strictly greater than the ground-state energy, our
result implies that  no configuration with energy close to one of ground state is $h$-stable, in the sense that all flips result in $h$-increase of 
energy. 

We establish a similar, but only partial
results in the case $h<h_{\rm cor}$. 
Specifically, for each such $h$, we show that $\exists$ a value of energy $E_{\rm cor}(h)$ with  $E_{\min}(h)< E_{\rm cor}(h) < E_{\max}(h)$ 
such that
there exist bisections with $n-o(n)$
$h$-stable vertices and energy $\approx E$ for any $E \in (E_{\rm cor}(h),E_{\max}(h)$ and 
when energy $\approx nE$, while when
$E<E_{\min}(h)$ the largest number
of stable vertices is $n-\Theta(n)$.
The mismatch between $E_{\min}(h)$
and $E_{\rm cor}(h)$ is due to the failure
of the second moment method.
% \yd{Introduce $E_{\rm cor}(h)$ in a clearer manner.}

The value,
$E_{\rm cor}(h)$ is the energy where the expected number of pairs of $h$-stable configurations having non-zero overlap dominates the number of orthogonal pairs of $h$-stable configurations. This corroborates the non-rigorous analysis of \cite{bray1981metastable}, who reported the energy value $\approx -0.672$ as the onset of the correlation of local minima being in the Sherrington-Kirkpatrick model. This matches the value $E_{\rm cor}(0)$ obtained through our analysis.
Illustrations of the above transitions can be found in Figure \ref{fig:phase_trans_ill_E} for $h \in (h_{\rm cor}, h^\star)$ and Figure \ref{fig:ill_ecor} for $h < h_{\rm cor}$.

\end{enumerate}

The proof outline for our results
is as follows. We first establish
our results for sparse graphs
$\G(n,d/n)$ and then we translate
them to dense graphs $\G(n,1/2)$
using  powerful Lindeberg's interpolation/universality
method (more on this below). 
This proof approach is somewhat
unconventional since typically
one uses the Lindeberg's method to reduce
a non-Gaussian model to a Gaussian
or Gaussian ``look alike'' (of which
$G(n,1/2)$ is an example). In our 
case we take the opposite route and 
establish results for  Gaussian-like
setting by lifting from non-Gaussian
setting (sparse \ER graphs). The
advantage of working with sparse
graphs is the availability of the 
configuration model of randomness, which
is rather lost in the dense random graph
setting. 

The
corresponding results for sparse graphs
are established using the second
moment method.

Our proofs are modulo the validity of certain assumptions regarding low-dimensional objectives. These assumptions are specified in Section \ref{sec:main_results}. The numerical verification of all the assumptions occurs through the use of the mpmath library \citep{mpmath} in Python,
and the validity of the numerical procedures is justified in Appendix \ref{sec:num_W}. 
% \dg{revisit later}
\begin{figure}
    \centering \includegraphics[width=\textwidth]{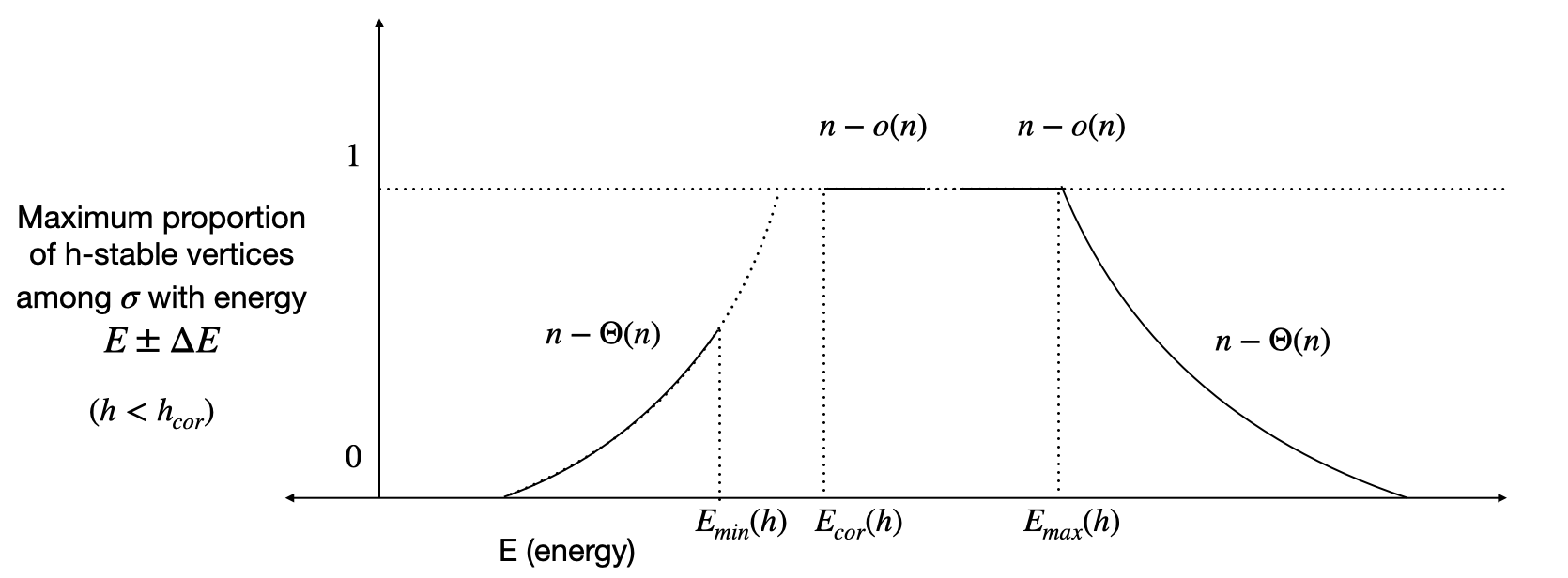}
    \caption{Illustration of the maximum number of $h$-stable vertices around $E_{\rm cor}(h)$ and $E_{min}(h)$ for $h<h_{\rm cor}$.}
    \label{fig:ill_ecor}
\end{figure}
% The value $h_{\rm cor}$ arises due to a transition in the properties of the second moment and is related is related to the typical pairs of $h$-stable optima being ``correlated" i.e. having non-zero overlap. We explain this further in Section \ref{sec:max_energy} and relate the transition to observations in \cite{bray1981metastable}.

\paragraph{Related work.}
Next we discuss some related papers
preceding our work.
In a recent work that directly inspired the present paper, \citep{behrens2022dis} utilized the non-rigorous cavity method from statistical physics to characterize the space of $h$-stable partitions for sparse random regular graphs with ferromagnetic as well as anti-ferromagnetic interactions. 
They conjectured the existence of a threshold $h^\star$ such that for $h > h^\star$ no $h$-stable partitions exist with high probability while for $h< h^\star$, such partitions do exist with high probability as $n \rightarrow \infty$. They obtained the following numerical value: $h^\star \approx 0.3513 + o_d(1)$. Note that our result validates
this estimation for  
sparse \ER graphs. It is not surprising
that the answers for \ER and random
regular graphs are asymptotically
similar as $d\to\infty$ as random \ER
graphs become progressively ``more 
regular'' in this regime.

A number of settings closely related to the case $h=0$ have been studied in the literature. In particular, \cite{ferber2022friendly} proved that in \ER graphs drawn from $\G(n,1/2)$, there exist bisections with $n-o(n)$ friendly vertices with high probability as $n \rightarrow \infty$. Their work constituted a positive resolution of the so-called Friendly Partition Conjecture by~\Furedi~\citep{FurediPersonal}.
The problem was also included in the list of Green's 100 open problems~\citep{Green} as problem 91. As already mentioned 
earlier, their result is covered as a special case of our main result, corresponding to the case $h=0$.
In concurrent work a stronger 
version of the \Furedi's conjecture
was established in \cite{minzer2023perfectly} (along with analogous results for general
$h$) who proved the existence
w.h.p. of bisections such that \emph{all}
nodes (and not just $n-o(n)$ nodes) are $h$-stable, using the second
moment method. Whether this still
holds for a very related
Sherrington-Kirkpatrick model (the case
when entries of $W$ are i.i.d. standard
normal random variables), remains open.

A related question is about the existence of efficient algorithms for finding such h-stable local optima. In this direction, the recent work by \cite{huang2025strong} showed that for the Sherrington-Kirkpatrick model, for any $h>0$, recovering such optima is hard for low-degree polynomial methods. In contrast, for $h=0$, an efficient algorithm for finding friendly bisections was already known as part of the constructive proof in \cite{ferber2022friendly}.

When $h=0$, $h$-stable partitions for anti-ferromagnetic interactions or equivalently, un-friendly partitions exactly correspond to the local maxima for the max-cut problem. The number of such local maxima was considered \citep{gamarnikMaxcutSparseRandom2018} for sparse Erdős-Rényi graphs with a fixed and linear in $n$ number of edges, with the aim of computing tight bounds on the max-cut of sparse random graphs. Earlier work by \citep{coppersmith2004random} had computed such bounds by considering the expected number of partitions having a given cut size. To obtain a tighter upper bound on the max-cut size, \citep{gamarnikMaxcutSparseRandom2018} restricted the count of partitions to ones satisfying local optimality w.r.t the max-cut problem. Analogously, to obtain a lower bound, they considered the second moment of the number of partitions having a given cut size.
The relation between local maxima for max-cut and $h$-stable partitions allows us to utilize the techniques in \citep{gamarnikMaxcutSparseRandom2018} for the computation of first and second moments for arbitrary $h$.
Thus parts of our work will borrow techniques  heavily from~\citep{gamarnikMaxcutSparseRandom2018}.

In a related work \citep{skcount} the authors computed the first moment of  the number of local minima up to leading exponential terms for the Sherrington-Kirkpatrick (SK) mode, again corresponding to the case $h=0$. In physics
jargon, this corresponds to the annealed version of the model.
Local optima for mean-field disordered systems have also been studied through suitably designed simulations. For instance, 
\cite{song2020local} designed an algorithm that discovers local minima for ferromagnetic systems having near-zero magnetization. This allowed them to probe the number as well as typical-energy ranges of such local minima.

The threshold $h^\star$ which characterizes a sharp phase transition in the optimization problem of maximizing the number of $h$-stable vertices,
remarkably is the same 
 for both friendly (assortative) and unfriendly (disassortative) cases, 
 and coincides with one for the Sherrington-Kirkpatrick model (namely $W$ with  standard normal entries).
Such an equivalence is analogous to the relation between the values of the max-cut and min-bisection problems, first conjectured in \citep{zdeborova2010conjecture}
and subsequently proven in \citep{dembo2017extremal} for \ER graphs with large degrees.

%Our problem is related to the literature on the characterization of stable states in the Ising and spherical perceptron models. 
%The pioneering work of \citep{gardner1987maximum} analyzed stable attractors of the following dynamics:
%\begin{equation}
%   \sigma_i(t + 1) = \operatorname{Sign}(\sum_{i=1}^n w_{vi}\sigma_i(t)). 
%\end{equation}
%\citep{gardner1987maximum} further imposed a threshold on the single-flip stability of the attractors, similar to definition \ref{def:h-stable}. Given a set of storage patterns $\sigma$, one may then optimize for weights $w_{ij}$ such that the patterns satisfy the imposed stability.

%Indeed, several interesting phase transitions for the perceptron model have been identified and proven in recent works \citep{aubin2019storage}. However, unlike the case of the perceptron model, 

The $h$-stability constraint in \eqref{def:h-stable} can be viewed as a case of random constraint satisfaction problems. However, unlike related problems in this class such as the binary perceptron \citep{gardner1987maximum,aubin2019storage}, the $h$-stability constraint involves correlations between the different constraints since the weight $w_{ij}$ appears in the constraints for both the $i_{th}$ as well as the $j_{th}$ vertices. Under such correlations, the rigorous evaluation of even the first moment is non-trivial and as mentioned earlier, 
was only recently completed for the case of the Sherrington-Kirkpatrick model with Gaussian weights in \citep{skcount}.
%Another difference is that unlike the perceptron model, in our case, the number of constraints is fixed to be the number of variables i.e $n$. Therefore, we analyze the phase transition in the imposed gap $h$ (Definition \ref{def:h-stable}) instead of the storage capacity. For the perceptron model, a phase transition in the storage capacity and gap are equivalent \citep{aubin2019storage}.

The value of the threshold $h$ is furthermore related to the robustness to perturbation of metastable states in statistical physics. The constraint for robustness can be incorporated into the dynamics as in the case of Hopfield networks. Hopfield networks aim to model associative memory \citep{hopfieldNeuralNetworksPhysical1982} through the convergence of configurations to stable planted attractors.
For such networks, the maximum value of the threshold on the local field such that metastable states exist was considered in \citep{treves1988metastable}. When the ratio of the number of stored patterns to the number of spins approaches infinity, they obtained the maximum value of the threshold as $h^\star\approx0.3513$, which matches the threshold obtained by us for spin glass and random graphs.

\textbf{Concurrent Work:} As mentioned earlier, in a concurrent work \citep{minzer2023perfectly}, Minzer, Sah, and Sawhney prove the existence of \emph{exactly} friendly bisections $\G(n,1/2)$ with high probability, i.e bisections with all $n$ vertices being friendly. They also prove a phase transition in the existence and absence of bisections in $\G(n,1/2)$ with all $n$ vertices having friendliness $h\sqrt{d}$ around a value of $h^\star$ matching the value of the threshold in our results. Their proof relies on a fine-grained application of the second moment method, combining the analysis of \cite{gamarnikMaxcutSparseRandom2018} with switching and enumeration techniques, along with the use of isoperimetry in graphs. Based on our universality results, we conjecture that the existence of bisections with all $n$ $h$-stable vertices for $h<h^\star$ should also hold for other related models such as the Sherrington-Kirkpatrick (SK) model. Furthermore, we believe that our results related to the energy of $h$-stable configurations (Theorems \ref{thm:sparse_max_E_thres}, \ref{thm:sparse_max_E_thres_uni} and \ref{thm:dense_max_E_thres}) should also apply for such exactly $h$-stable configurations.

\paragraph{Organization of the paper.}

The rest of the paper is organized as follows: In Section \ref{sec:main_results}, we present our main
results formally. 
Section \ref{sec:Theorem1} is devoted
to the proof of the main result
regarding the phase transition around
$h^\star$ for the special case of sparse anti-ferromagnetic interactions (Theorem \ref{thm:sparse_thres}). In particular, we derive $h^\star$
 as a root of the non-linear equation specified in \eqref{def:w_firstmom}. Our proof relies on the usage of the second-moment method in sparse \ER graphs combined with concentration arguments. For the computation of the first and second moments up to sub-leading exponential terms, we borrow the notation and several techniques from \citep{gamarnikMaxcutSparseRandom2018}, who computed the same when $h=0$ and all the vertices are required to be locally $h$-stable. Our result includes the computation of the first moment for the number of partitions with at least a constant fraction of the vertices satisfying $h$-stability. This requires generalizations of the large deviation results in \citep{gamarnikMaxcutSparseRandom2018}.
For the second moment,
the expression is obtained
directly through appropriate modifications of the proof in  \citep{gamarnikMaxcutSparseRandom2018}. However, for the sake of completeness, we provide the full proof, based on techniques developed in our first-moment analysis.

% . \dg{we need to 

We subsequently utilize the first and second moment computations to prove the absence of bisections having $n-o_d(1)n$ $h$-stable vertices with high probability for $h>h^\star$ and their existence for $h<h^\star$.
The high-probability existence for $h<h^\star$ is not a direct consequence of the Paley–Zygmund inequality (a standard method
of proving existence of structures using second moment method) and requires proving the concentration of certain auxiliary functions along with a technique of perturbing the threshold.

In Section \ref{sec:max_energy}, we utilize similar techniques to prove the phase transitions in the maximum number of $h$-stable vertices amongst bisections satisfying certain energy constraints.
In Section \ref{sec:sparse_uni}, we prove the universality of the threshold for sparse graphs through an application of the Lindeberg's method to carefully chosen functions of the edge weights for different regimes of $h$. In Section \ref{sec:dense_uni}, we connect this result to dense \ER graphs. In Section \ref{sec:universality_e_max}, we prove analogous universality results for the transitions obtained upon the imposition of energy constraints.
Finally, we conclude in Section \ref{sec:conc} with some open directions.

\paragraph{Notations.} 
% \yd{Do we say that $O,o$ holds uniformly across the remaining variables. Eg: $O_d(f(d))$ denotes a term bounded $C f(d)$ for large enough $d$ where $C$ denotes a constant independent of $n$}.
We use $o_n, O_n, \Theta_n, o_d, O_d, \Theta_d$ to denote standard asymptotic bounds w.r.t the variable in the subscript. 
For every
integer $m$, $[m]$ denotes the set of
integers $1,2,\ldots,m$. $\stackrel{d}{=}$
denotes equality in distribution.
$\R$ and $\Z$ denote the set of real values and
the set of integer values respectively. 
$\R_+$ and $\Z_+$ denote the non-negative parts of these
two sets. We denote by $\Phi$ the  Cumulative
Distribution Function of a standard Normal random
variable. That is 
$\Phi(t)=\int_{-\infty}^t(2\pi)^{-{1\over 2}}
\exp(-t^2/2)dt$. $H(x)=-x\log x-(1-x)\log(1-x)$ 
denotes the binary entropy function.
\textbf{Remark:} Whenever both asymptotics w.r.t parameters $d,n$ are involved, we will operate under the sequential limit $\lim_{d \rightarrow \infty} \lim_{n \rightarrow \infty}$. Therefore, for brevity, we will denote terms of the form $o_n(1)f(d)$ for some function $f:\mathbb{R}\rightarrow \mathbb{R}$ as $o_n(1)$.

% The applicability of the second moment technique relies on the 
% the expected number of uncorrelated $h$-stable configurations being exponentially larger than the expected number of correlated ones.
% For $h < h_{\rm cor}$, however, there exists another energy threshold $E_{\rm cor}$ where in expectation, the number of pairs of uncorrelated $h$-stable configurations is exponentially less than the number of correlated ones. This prevents the applicability of our technique towards characterizing the $h$-stable vertices for configurations having energy less than $E_{corr}$. 
% numerical value of $E_{\rm cor}$ at $h=0$, $E_{\rm cor}(0)\approx -0.6725$, again matches 

\section{Main Results}\label{sec:main_results}

In this section, we summmarize and present our main theorems. Our first set of main results concerns the existence and identification of a threshold $h^\star$, demarcating the phase transition illustrated in Figure \ref{fig:phase_trans_ill_h}, for different choices of weighted random graphs defined on $n$ nodes and indexed by a parameter $d$.

\subsection{Existence and Absence of $h$-stable bisections}

We first define such a threshold for sparse graphs indexed by the corresponding degree parameter $d$ below:
\begin{definition}\label{def:thres_def}
We say that a value $\tilde{h}>0$ is a maximal stability threshold for a family of random symmetric $n\times n$ matrices $W^{(d,n)}$ if as $n\rightarrow \infty$, the following hold:
\begin{enumerate}
    \item For any $h>\tilde{h}$, there exists an $0<\epsilon(h)<1$ and  $d(h)$, such that for  $d > d(h)$, with high probability as $n\rightarrow \infty$, all bisections of $[n]$ have at least  $\epsilon n$ vertices that are not $h$-stable.
    \item For any $h<\tilde{h}$, and any $\epsilon>0$, there exists a $d(h,\epsilon)$, such that for all $d > d(h,\epsilon)$, with high probability as $n\rightarrow \infty$, there exists a bisection of $[n]$ with at most $\epsilon n$ vertices violating $h$-stability. 
\end{enumerate}
\end{definition}
In both cases the probability is with 
respect to the randomness of $W^{(d,n)}.$
The  definition above formalizes the notion that when $h>\tilde{h}$ every bisection has at least $\Theta_d(1)n$ vertices violating $h$-stability, while when $h<\tilde{h}$, there exist bisections with  $(1-o_d(1))n$ $h$-stable vertices.

\ignore{
Note that $h$-stable partitions trivially exist for ferromagnetic interactions if all vertices belong to one of the groups. This is analogous to the distinction between max-cut and min-bisection, where the problem of minimizing the cut size is non-trivial only upon imposing limitations on the imbalance between the partitions.
In such cases, we restrict ourselves to the set of bisections, i.e configurations having magnetization $0$. We refer to an $h^\star$ satisfying the above conditions restricted to the set of bisections as a maximal bisection threshold for $W^{(d,n)}$.
}

To define the threshold value $h^\star$ and state our results, we introduce certain low-dimensional functions related to the first and second moments of the number of $h$-stable bisections. In particular, these functions will be used to define the threshold $h^\star$ as the root of a non-linear equation. Define: 

\begin{equation}\label{eq:def_w}
\begin{split}
w: & \R^2 \rightarrow \R,\\
w(E,h)&=-E^2+\inf_{\theta\in \mathbb{R}}\left(2\theta^2+\log(1+\erf(\theta-(2E+h)/\sqrt{2})) \right),
\end{split}
\end{equation}
where $\erf:\R \rightarrow \R$ denotes the Gaussian error function, defined as $\erf(z) = \frac{2}{\sqrt{\pi}} \int_0^z e^{-t^2} \mathrm{~d} t$ (where the function value is negative when $z$ negative).

We show in Appendix \ref{app:properties} that $w(E,h)$ as defined above satisfies the following property:

\begin{proposition}\label{prop:first_mom_conv}
$w(\cdot,h):\mathbb{R} \rightarrow \mathbb{R}$, is continuously differentiable, strictly concave and admits a unique maximizer strictly less than $-h/2$. Furthermore,
\begin{equation}
     \sup_{E \in \R} w(E,h)=\sup_{E \in \R}\left(-E^2+\log(1-\erf(E+h/\sqrt{2})) \right).
\end{equation}
\end{proposition}

We denote the unique maximizer defined by Proposition \ref{prop:first_mom_conv} as $E^\star(h)$:
\begin{equation}\label{eq:Estar_def}
    E^\star(h) \coloneqq \operatorname{argmax}_E w(E,h).
\end{equation}

We further denote by $w(h)$, the function over $h$ obtained by maximizing $w(E,h)$ w.r.t. the parameter $E$. By Proposition \ref{prop:first_mom_conv}, $w(h)$ is given by:  

\begin{equation}\label{def:w_firstmom}
   w(h) \coloneqq \sup_{E \in \R}\left(-E^2+\log(1-\erf(E+h/\sqrt{2})) \right) = -(E^\star(h))^2+\log(1-\erf(E^\star(h)+h/\sqrt{2}).
\end{equation}
$w(h)$ is plotted in Figure \ref{fig:first_mom} and  satisfies the following properties (proof in Appendix \ref{app:properties}):
\begin{proposition}\label{prop:w_mono}
   $w(h)$ is strictly decreasing in $h$ and possesses a unique root, $h^\star \in \mathbb{R}$. Furthermore, $h^\star > 0$.
\end{proposition}

Numerically, we estimate the value of the above root  as $h^\star\approx 0.3513$ which matches the phase transition value prediction in  \citep{behrens2022dis}. 

% The above proposition follows trivially by noting that $\erf(\cdot)$ is strictly increasing while

Similarly, we introduce another series of functions which our results will associate with the asymptotic second moment at a given overlap between configurations.

Let $\mathcal{Q}:\mathbb{R}^3 \rightarrow \mathbb{R}$ and $F:\mathbb{R} \times (-1,1) \times \R^3 \rightarrow \mathbb{R}$ be defined as:
\begin{equation}\label{eq:Q_def}
   \mathcal{Q}(\theta,a,b) \coloneqq \mathbb{E}_{Y_1,Y_2 \sim \mathcal{N}(0,I_2)}[\exp(\theta  Y_1)\mathbf{1}_{[Y_1+b \geq a\abs{Y_2}]}],
\end{equation}
and:
\begin{align}
 F(E,\omega,h,t,\theta_1, \theta_2)   
& \coloneqq 2 \log 2 +H(\frac{1+\omega}{2}) \notag\\
&-4{t^2\over (1+\omega)^2}-4{(E+t)^2\over (1-\omega)^2}
+(1+\omega)/2
\log \mathcal{Q}\left(\theta_1,\sqrt{1-\omega\over 1+\omega},
{4\sqrt{2}t\over (1+\omega)^{3/2}}-\frac{2h}{\sqrt{(1+\omega)}}\right) \notag\\
&+(1-\omega)/2
\log \mathcal{Q}\left(\theta_2,\sqrt{1+\omega\over 1-\omega},
{4\sqrt{2}(-E-t)\over (1-\omega)^{3/2}}-\frac{2h}{\sqrt{(1-\omega)}}\right)\label{def:F2ndmom}.
\end{align}

We then define:
\begin{align}
W: \ &\R \times (-1,1) \times \R \rightarrow \R, \nonumber \\ 
W(E,\omega,h) &= 
    \sup_{t \in \mathbb{R}} \inf_{\theta_1, \theta_2 \in \mathbb{R}^2} F(E,\omega,h,t,\theta_1, \theta_2).\label{eq:Wxbeta}
\end{align} 

In Appendix \ref{app:properties}, we show that $W(\cdot)$ defined above satisfies the following regularity properties:

\begin{proposition}\label{prop:twice_diff}
 $W(E,\omega,h)$ is twice continuously-differentiable w.r.t $\omega,E,h \in (-1,1) \times \R \times \R$ with $\omega=0$ being a stationary point for any $E,h \in \R^2$ i.e $\frac{dW}{d\omega}\vert_{\omega=0}=0$. 
\end{proposition}

Our first set of results rely on the following assumption about the behavior of $W(E,\omega,h)$ at $h=h^\star$ and $E=E^\star(h)$:

\begin{assumption}\label{ass:unique_max}
$W(E,\omega,h)$ defined by (\ref{eq:Wxbeta}) at $h=h^\star,E=E^\star(h^\star)$ is uniquely maximized w.r.t $\omega$ in $(-1,1)$ at $\omega=0$ i.e.:
    \begin{equation}
\operatorname{argmax}_{\omega \in (-1,1)} W(E^\star(h^\star),\omega,h^\star) = \{0\}.
    \end{equation}
\end{assumption}

We discuss the verification of the above assumption and other related properties of $W(E, \omega, h)$ in Appendix \ref{sec:num_W}.

Under Assumption \ref{ass:unique_max}, our first result establishes the existence of a threshold in the sense of Definition \ref{def:thres_def} for anti-ferromagnetic sparse \ER graphs. Its proof is found in Section \ref{sec:Theorem1}.

\begin{theorem}\label{thm:sparse_thres} (Largest number of unfriendly nodes in $\G(n,d/n)$)
Let $h^\star \in \mathbb{R}$ denote the unique root of $w(h)$. (with the existence and uniqueness established by Proposition \ref{prop:w_mono}).
Set $W^{(d,n)}$ as the  negative of the 
adjacency matrix of $\G(n,d/n)$. Namely,
$w_{ij}=-1$ with probability $d/n$ and
$=0$ otherwise for all pairs $1\le i<j\le n$ independently. Then, under Assumption \ref{ass:unique_max}, 
 $h^\star$ is a maximal stability threshold 
 of the sequence $W^{(d,n)}$ in 
 the sense of Definition~\ref{def:thres_def}. 
\end{theorem}

\ignore{
\yd{
(Threshold for sparse anti-ferromagnets)
Let $h^\star$ denote the unique root of the function
$w(h)$.
Let $W^{(d,n)}$ be the weight matrices corresponding to sparse Erdős-Rényi graphs $\G(n,p=d/n)$ with average degree $d$, defined as $w_{ij}= w_{ji}\in \{0,-1\}$ $p(w_{ij}=-1) = \frac{d}{n}$, with $w_{ij}$ being independent across edges. Then,  $h^\star$ is a maximal stability threshold for $W^{(d,n)}$ in the sense of Definition \ref{def:thres_def}. 
}
}
In the context of this theorem,
the parameter $d$ specifying $s_\sigma(v,W)$ and used in Definition \ref{def:thres_def} is the same 
as the parameter $d$ in the definition
of $\G(n,d/n)$.

The theorem above concerns the existence
of bisections reaching $h$-\emph{unfriendliness} as opposed
to $h$-friendliness for all nodes.
While the same result with the same
threshold holds for the friendliness
version of the problem as well, as we establish later in Corollary \ref{cor:friendly_sparse}, 
we begin with this result since
it is a direct generalization
of the main result in~\citep{gamarnikMaxcutSparseRandom2018}
from the case $h=0$ to the general 
case. In fact, most other results in this paper
will be established by reducing
to this result, namely 
Theorem~\ref{thm:sparse_thres}

\ignore{
Our proof  relies on first and second-moment computation up to sub-leading exponential terms. These rely on existing results in \citep{gamarnikMaxcutSparseRandom2018}, who computed these terms for the case of $h=0$, i.e for local optima of the max-cut problem.
By numerically solving the resulting non-linear systems of equations, we confirm that the first and second moments match up to sub-leading exponential terms for $h \le h^\star$. This, however, only yields an exponentially decaying lower bound on the probability of the existence of almost $h$-stable partitions through the Paley–Zygmund inequality. To boost the probability of existence, we utilize results on concentration along with a technique of perturbing the stability threshold, as discussed in Section \ref{sec:high_prob}.
}

Next,  we extend Theorem~\ref{thm:sparse_thres} to the
case of weighted sparse random graphs. 
This will be achieved by reducing
to the setting of Theorem~\ref{thm:sparse_thres} using the 
Lindeberg's interpolation method.

\begin{theorem}\label{thm:uni_sparse} (Largest number of stable nodes in weighted
sparse random graphs.)
Suppose $W^{(d,n)} \in \R^{n\times n}$  is a  sequence
of symmetric random matrices  satisfying 
the following properties
\begin{enumerate}
    \item $w^{(d,n)}_{ij}, 1\le i<j\le n$ are i.i.d. with mean $\mu=\mu_{d,n}$.
    \item $\E\left[\abs{w^{(d,n)}_{ij}-\mu}^2\right]=\frac{d}{n}(1-\frac{d}{n})$, $1\le i<j\le n$,
     \item $\E\left[\abs{w^{(d,n)}_{ij}-\mu}^3\right]=dO_n(\frac{1}{n})$,
     $1\le i<j\le n.$
\end{enumerate}
Under the above assumptions, the value of $h^\star$ from Theorem \ref{thm:sparse_thres} is a maximal stability bisection threshold for $W^{(d,n)}$.

%Furthermore, if $\mu=0$, the absence of near $h$-stable configurations holds for all partitions.
\ignore{In particular, the above result holds for sparse Erdős-Rényi graphs with independent edges, average degree $d$, and anti-ferromagnetic, ferromagnetic, and spin glass interactions.}
\end{theorem}
As before, the parameter in the above theorem is assumed to be the same $d$ as in Definition \ref{def:thres_def}.
Considering the special case of Theorem \ref{thm:uni_sparse} when
$w_{ij}=1$ with probability $d/n$
and $=0$ otherwise, we obtain 
the phase transition result analogous to Theorem \ref{thm:sparse_max_E_thres} for 
friendly bisections as well. This is formalized below:

\begin{corollary}\label{cor:friendly_sparse}(Largest number of friendly nodes in $\G(n,d/n)$)
 Let $W^{(d,n)}$ be the 
adjacency matrix of $\G(n,d/n)$. Then, 
 $h^\star$ is a maximal stability threshold 
 of the sequence $W^{(d,n)}$ in 
 the sense of Definition~\ref{def:thres_def}.
\end{corollary}

We then extend the above universality result to dense graphs, including the Sherrington-Kirkpatrick model which
corresponds to the case when $w_{ij}$
is distributed as standard normal
random variable, namely
$w_{ij}\stackrel{d}{=}\mathcal{N}(0,1)$. For such graphs, the parameter $d$ in Definition \ref{def:thres_def} is replaced by $n$, and we obtain the following result:
\begin{theorem}\label{thm:uni_dense} (Stability threshold for dense graphs)
Let $W^{(n)} \in \R^{n\times n}$ be a sequence of symmetric random matrices  
with $W^{(n)}=(w^{(n)}_{ij})_{i,j \in [n]}$ satisfying:
\begin{enumerate}
    \item $w^{(n)}_{ij}, 1\le i<j\le n$ are independent and identically distributed with mean $\mu=\mu_n$.
    \item $\E\left[\abs{w^{(n)}_{ij}-\mu}^2\right]=1, 1\le i<j\le n$.
     \item $\E\left[\abs{w^{(n)}_{ij}-\mu}^3\right]=O_n(1), 1\le i<j\le n$.
\end{enumerate}
Then for $h^\star$ from Theorem \ref{thm:sparse_thres}, the following holds: 
\begin{enumerate}
    \item For any $h>h^\star$, there exists an $\epsilon(h)$ such that w.h.p. as $n\rightarrow \infty$, all bisections of $[n]$ have at least $\epsilon n$ vertices violating $h$-stability.
    \item For any $h<h^\star$, and any $\epsilon>0$, with high probability as $n\rightarrow \infty$, there exists a bisection of $[n]$ with at most $\epsilon n$ vertices violating $h$-stability.
\end{enumerate}
In particular, the above result holds for dense graphs $\G(n,1/2)$ and
the Sherrington-Kirkpatrick model, 
namely the case of Gaussian $w_{ij}\stackrel{d}{=} \mathcal{N}(0,1)$,
independently for $1\le i<j\le n.$
\end{theorem}

\begin{remark}\label{rem:furedi}
For the case of dense \ER graphs
we have $\mu=1/2$ and the variance
of each entry is $1/4$, ostensibly
not matching the requirement of 
having variance equal unity. However,
we  note that by the definition of $s_\sigma(v,W)$ in \eqref{def:h-stable}, for any $k \in \mathbb{R}, W \in \mathbb{R}^{n \times n}$, we have that $s_\sigma(v,kW)=ks_\sigma(v,W)$.
Therefore, the above result can be generalized to $W$ such that $\E[\abs{w_{ij}-\mu}^2]=k^2$ and $\E[\abs{w_{ij}-\mu}^3]=O_n(1)$ for any $k \in \mathbb{R}$, simply by scaling $h$ by $k$. Thus the result above
indeed applies to dense \ER graphs,
and considering the special case
$h=0$ we recover the main result,
 Theorem 1.1 in \cite{ferber2022friendly}, namely the positive
 resolution of the Furedi's conjecture \citep{FurediPersonal}.

We did not take advantage of the 
Definition~\ref{thm:sparse_thres} 
in stating Theorem~\ref{thm:uni_dense},
since in the dense case,
since in this case the complexity
of double asymptotic way of describing
the phase transition is gone.
\end{remark}

\subsection{Energy of $h$-stable bisections}

Next, we turn to our results
regarding the existence of $h$-stable
bisections with a fixed energy values, illustrated in Figures \ref{fig:phase_trans_ill_E} and \ref{fig:ill_ecor}. For notational convenience, we introduce the following normalized energy function (Recall (\ref{eq:hamilton})):
\begin{equation}\label{eq:def_energy}
    E(\sigma)=\frac{1}{n}H(\sigma)= -\frac{1}{n\sqrt{d}}\sum_{1\leq i<j\leq n}\sigma_i\,\sigma_j\,w_{ij}.
\end{equation}

To describe the phase transition associated to the range of energy values, we introduce the following definition, analogous
to Definition~\ref{thm:sparse_thres}
pertaining to sparse random graphs.

\begin{definition}\label{def:max_e}
Given  $h \in \R$, an interval $(E_1(h), E_2(h))$ with $-\infty \leq E_1(h) < E_2(h) \leq \infty$, and a sequence of random matrices $W^{(d,n)}$,
we say that:
\begin{enumerate}
    \item $h$-stable bisections
are asymptotically absent over $(E_1(h),E_2(h))$ if  there exists an $0<\epsilon(h)<1$ 
    and a $d(h)$, such that for every
    $d > d(h)$, w.h.p. as $n\rightarrow \infty$, all bisections $\sigma$ of $[n]$ with $E(\sigma) \in (E_1(h),E_2(h))$ have at least $\epsilon n$ vertices violating $h$-stability. 
    \item  $h$-stable bisections
are asymptotically present throughout  $(E_1(h),E_2(h))$ if for every $\epsilon>0$, there exists  $d(h,\epsilon)$, such that for all $d > d(h,\epsilon)$, w.h.p. as $n\rightarrow \infty$, there exists a bisection $\sigma$ of $[n]$ with at most $\epsilon n$ vertices violating $h$-stability and $E(\sigma) \in (E_1(h),E_2(h))$.
\end{enumerate}
\end{definition}

% \dg{NOTE: I have stated everything in terms
% of bisections, so we can delete
% this comment}
Definition \ref{def:max_e}
will allow us to state our results on existence of $h$-stable bisections within intervals of energy values, arbitrarily small but independent of $n$. These results rely on certain assumptions in addition to Assumption \ref{ass:unique_max} on the behavior of $W(\cdot)$, which we state next. We will subsequently utilize these assumptions to define $h_{\rm cor}, E_{\rm cor}(h)$ and utilize these definitions in our main results. Throughout, $h^\star$ denotes the unique root of $w(h)$ defined in \eqref{def:w_firstmom}, which we recall is strictly positive by Proposition \ref{prop:w_mono}. By Proposition \ref{prop:first_mom_conv}, for any $h < h^\star$, $w(h)=\sup_{E \in \mathbb{R}} w(E,h)$ is strictly positive and therefore, $w(E,h)$ possesses exactly two roots. In what follows, we denote the smaller of these two roots as $E_{\rm min}(h)$ and the larger root as 
$E_{\rm max}(h)$. Recall $W(E,\omega,h)$ defined by \eqref{eq:Wxbeta}.
We denote by $\operatorname{argmax}_{\omega \in (-1,1)} W(E,\omega,h)$, the following set:
\begin{equation}
\operatorname{argmax}_{\omega \in (-1,1)} W(E,\omega,h) = \{\omega^\star \in (-1,1): W(E,\omega^\star,h) = \sup_{\omega \in (-1,1)} W(E,\omega,h)\}.
\end{equation}

Our set of assumptions for the subsequent results specify this set of maximizers at certain values of $E,h$. These in-turn demarcate the validity of the second moment method as $E,h$ are varied. The numerical verification of these assumptions is detailed in Appendix \ref{sec:num_W}.

\begin{assumption}\label{ass:0_unique}
At $h=0$, $\exists E \in (E_{min}(0),E_{max}(0))$, such that $W(E,\omega,h)$ is not maximized at $\omega=0$ i.e:
 \begin{equation}
 W(E,0,0) < \sup_{\omega \in (-1,1)} W(E,\omega,0),     
 \end{equation}
for some $E \in (E_{min}(0),E_{max}(0))$.
\end{assumption}

\begin{assumption}\label{ass:e_max_h}
For all $h \in (-0.1,h^\star]$, $W(E,\omega,h)$ is uniquely maximized at $\omega=0$ for $E=E_{\rm max}(h)$ i.e.
\begin{equation}
\operatorname{argmax}_{\omega \in (-1,1)} W(E_{\rm max}(h),\omega,h) = \{0\},
\end{equation}
for all $h \in (-0.1,h^\star)$.
\end{assumption}

\begin{remark}
The value $-0.1$ above corresponds to the lower limit of $h$ up to which we numerically verify Assumption \ref{ass:e_max_h} to hold. One can check numerically 
the validity of the assumption up to any negative finite bound. Such numerical verification, however, does not establish its validity for an infinite range. For simplicity, we opt to restrict $h$ to $h > -0.1$. 
\end{remark}

Under the above assumptions, $h_{\rm cor},E_{\rm cor}(h)$ are  defined as:

\begin{equation}\label{def:h_cor}
    h_{\rm cor}= \sup(h < h^\star : \exists E \in (E_{min}(h),E_{max}(h)), \operatorname{argmax}_{\omega \in (-1,1)} W(E,\omega,h) \neq \{0\}).
\end{equation} 

\begin{equation}\label{def:E_cor}
    E_{\rm cor}(h) = \sup(E < E_{max}(h) : \operatorname{argmax}_{\omega \in (-1,1)} W(E,\omega,h) \neq \{0\}).
\end{equation}

In words, $h_{\rm cor}$ corresponds to the largest value of $h<h^\star$, where $W(E,\omega,h)$ is not uniquely maximized at $\omega=0$ for some $E \in (E_{min}(h),E_{max}(h))$. Our results will show that this marks the first failure of the second moment method within the feasible energy range $(E_{min}(h),E_{max}(h))$ as $h$ is lowered from $h^\star$. Similarly, $E_{\rm cor}(h)$ denotes the largest value of energy $E \leq E_{\max}(h)$ where $W(E,\omega,h)$ is not uniquely maximized at $\omega=0$. 

Assumption  \ref{ass:0_unique} implies that at $h=0$, $W(E,\omega,h)$ is not maximized at $\omega=0$ for some $E \in (E_{min}(0),E_{max}(0))$ Hence, the set in the definition of $h_{\rm cor}$ given by \eqref{def:h_cor} is non-empty and thus $h_{\rm cor}$ is well-defined. 
However, our results further require that $h_{\rm cor}$ is strictly less than $h^\star$. We show that such a strict inequality is indeed true under Assumption \ref{ass:unique_max}. Furthermore we show that under Assumption \ref{ass:e_max_h}, $E_{\rm cor}(h)$ is well-defined and $E_{\rm cor}(h)$ is strictly below $E_{max}(h)$ for any $h \in (-0.1,h_{\rm cor})$. These results are summarized in the following proposition, whose proof is found in Section \ref{sec:exist_ecor}:

\begin{proposition}\label{prop:hcor}
Under Assumptions \ref{ass:unique_max}, \ref{ass:0_unique}, and \ref{ass:e_max_h}, $h_{\rm cor}$ exists and
satisfies $0 < h_{\rm cor} < h^\star$.  Furthermore, $E_{\rm cor}(h)$ exists for all $h <h^\star$, and satisfies $E_{\rm min}(h) \leq E_{\rm cor}(h)<E_{\rm max}(h)$ for all $h \in (-0.1,h_{\rm cor})$.
\end{proposition}

\begin{remark} In the proof of Proposition \ref{prop:hcor} (Appendix \ref{sec:exist_ecor}), we show that the set under the definition of $E_{\rm cor}(h)$ is non-empty for all $h \in \mathbb{R}$. However, for $h \in (h_{\rm cor},h^\star)$, $E_{\rm cor}(h)$ simply equals $E_{\rm min}(h)$. Therefore, the definition of $E_{\rm cor}(h)$ is non-trivial only for $h<h_{\rm cor}$.
\end{remark}

Having defined $h_{\rm cor}$ and $E_{\rm cor}(h)$, we're now ready to state our main results concerning the range of energies of $h$-stable bisections. Analogous to Theorem \ref{thm:sparse_thres}, our first result applies to sparse \ER graphs with anti-ferromagnetic interactions. Its proof is found in Section \ref{sec:max_energy}.

\begin{theorem}\label{thm:sparse_max_E_thres} (Largest number of unfriendly nodes in $\G(n,d/n)$ at fixed energies.)
Let $W^{(d,n)}$ be the negative
of the adjacency matrix of the graph $\G(n,d/n)$:
\begin{itemize}
    \item $\forall h \in (-\infty,h^\star)$, 
    $h$-stable bisections are asymptotically absent in  $(E_{\rm max}(h), \infty) \cup (-\infty,E_{\rm min}(h))$.
    \item $\forall h \in (h_{\rm cor},h^\star)$,  $h$-stable bisections are asymptotically present throughout any sub-interval of \\ $(E_{\rm min}(h),E_{\rm max}(h))$
    \item $\forall h \in (-0.1,h_{\rm cor})$,  $h$-stable bisections are asymptotically present throughout any sub-interval of \\ $(E_{\rm cor}(h),E_{\rm max}(h))$.
\end{itemize}
\end{theorem}
% \yd{Should we change from Theorem to Result or maybe state it as a theorem after some numerical assumptions?}
% \dg{No, for math paper we should do proofs
% property}

% \dg{Note to self: the part below
% needs to be redone properly later}

The results in Theorem \ref{thm:sparse_max_E_thres} are illustrated in Figures \ref{fig:phase_trans_ill_E}, \ref{fig:ill_ecor} and explained further in Section \ref{sec:max_energy}. We numerically estimate the value of $h_{\rm cor}$ to be  $h_{\rm cor}\approx 0.2860$. 

Analogous to the extension of Theorem \ref{thm:sparse_thres} to Theorem \ref{thm:uni_sparse}, we next extend Theorem \ref{thm:sparse_max_E_thres} to sparse graphs satisfying the assumptions in Theorem \ref{thm:uni_sparse}.

\begin{theorem}\label{thm:sparse_max_E_thres_uni} (Energy thresholds for sparse graphs)
Let $W^{(d,n)} \in \R^{n\times n}$ be a family of random variables satisfying the assumptions in Theorem \ref{thm:uni_sparse}. Let $E_{\rm max}(h),E_{\rm min}(h), E_{\rm cor}(h)$ be as defined in Theorem \ref{thm:sparse_max_E_thres}.  Then:
\begin{itemize}
    \item $\forall h \in (-\infty,h^\star)$, 
    $h$-stable bisections are asymptotically absent in  $(E_{\rm max}(h), \infty)$ and $(-\infty,E_{\rm min}(h))$.
    \item $\forall h \in (h_{\rm cor},h^\star)$,  $h$-stable bisections are asymptotically present throughout any sub-interval of \\ $(E_{\rm min}(h),E_{\rm max}(h))$
    \item $\forall h \in (-0.1,h_{\rm cor})$,  $h$-stable bisections are asymptotically present throughout any sub-interval of \\ $(E_{\rm cor}(h),E_{\rm max}(h))$.
\end{itemize}
\end{theorem}

Similarly, we further extend the above results to dense graphs in the setting of Theorem \ref{thm:uni_dense}.

\begin{theorem}\label{thm:dense_max_E_thres} (Energy thresholds for dense graphs)
Let $W^{(n)} \in \R^{n\times n}$ be a family of random variables satisfying the assumptions in Theorem \ref{thm:uni_dense}. Then, with high-probability as $n \rightarrow \infty$:
\begin{itemize}
    \item $\forall h \in (-\infty,h^\star)$, $\exists \epsilon(h) >0$ such that all bisections in $W^{(n)}$ with $E \in (E_{\rm max}(h), \infty) \cup (-\infty,E_{\rm min}(h))$ have at least $\epsilon n$ vertices violating $h$-stability.
    \item For any $h \in (h_{\rm cor},h^\star)$, $\epsilon > 0$, and any sub-interval of $(E_{\rm min}(h),E_{\rm max}(h))$, $\exists$ bisections  with at-most $\epsilon n$ vertices violating $h$-stability and energy $E(\sigma)$ in the given sub-interval.
    \item For any $h \in (-0.1,h_{\rm cor})$, $\epsilon > 0$, and any sub-interval of $(E_{\rm cor}(h),E_{\rm max}(h))$, $\exists$ bisections  with at-most $\epsilon n$ vertices violating $h$-stability and energy $E(\sigma)$ in the given sub-interval.
\end{itemize} 
\end{theorem}

The proofs of the last two Theorems are provided in Section \ref{sec:universality_e_max}. We numerically estimate the values of $E_{\rm min}(0),E_{\rm max}(0)$ as $E_{min}(0)\approx -0.791$ and $E_{max}(0)\approx -0.2860$ respectively, matching the theoretical physics predictions for the range of energies of local minima in \cite{bray1981metastable}
% \end{theorem}

\section{Proof 
of Theorem~\ref{thm:sparse_thres}.}
\label{sec:Theorem1}
We begin by recalling certain standard
models of sparse \ER random graphs. As mentioned in Section \ref{sec:intro}, we denote by $\G(n,p)$, an \ER random graph where each of the ${n\choose 2}$ ordered pairs of nodes is assigned as an edge, independently with probability $p$, which is possibly dependent on $n$. Similarly,
given $n$ and $m_n$, let $\bar{\G}(n,m_n)$
a random graph on the nodes set $[n]$
obtained by selecting $m_n$ edges uniformly
at random from the set of all 
${n\choose 2}$ unordered pairs of nodes. We further consider a 
configuration model $\tilde\G(n,m)$ similar to $\bar{\G}(n,m)$ where $m$ edges are generated uniformly at random with replacement 
from all $n^2$ ordered pairs of nodes, thus allowing loops and parallel edges. 
The following facts are well known,
see for example~\citep{janson2011random}).

% \dg{note to self: redo this lemma
% properly. Add a statement
% about the configuration models
% if need be added}

% \yd{We need the Lemma for independent edges to apply Lindeberg. For configuration model, we have contiguity but with a factor $e^{d}$ so we need to show that the probability of property not-occuring decays faster than this.}
\begin{lemma}\label{lemma:2-random-graphs}
Let $\mathcal{P}_n$ denote a graph property of the
set of $n$-node graphs, defined as a sequence of subsets of $[\binom{n}{2}]$. Let $d>0$ be fixed and suppose $P_n$ holds for $\bar{\G}(n,m)$ for any 
$m(n) = \frac{d}{2}n +O(\sqrt{n})$
 w.h.p. as $n\rightarrow \infty$. 
 Then $P_n$ also holds for $\G(n,d/n)$ w.h.p.
 as $n\to\infty$.
\end{lemma}

From Corollary $9.6$ in \citep{janson2011random}, we further have that:
\begin{lemma}\label{lem:config}
    For any $d>0$, under the configuration model 
    $\tilde{\G}(n, \frac{d}{2}n)$, \[\mathbb{P}[\tilde{\G}(n, \frac{d}{2}n) \text{\ contains no loops and parallel edges}] \geq c\exp(-C d),\]
for some constants $c, C > 0$. 
As a consequence, if any property $\mathcal{P}_n$, defined as a subset of multi-graphs, holds for $\bar{\G}(n, \frac{d}{2}n)$ with probability greater than $1-o_n(1)$, for large enough $d$, it also holds for  $\tilde{\G}(n, \frac{d}{2}n)$ w.h.p as $n \rightarrow \infty$. 
\end{lemma}

Since our results apply to bisections, it will be convenient to introduce a notation for the set of all bisections in $\{-1,+1\}^n$. We denote this set as $M_0$,  with $M_0$ defined for even $n$ as:
\begin{equation}\label{eq:M_0def}
    M_0 \coloneqq
        \{\sigma:\sum_{i=1}^n \sigma_i = 0\}.
\end{equation}
For odd $n$, we allow a minor violation of the bisection condition and define $M_0$ as $M_0 \coloneqq
        \{\sigma:\sum_{i=1}^n \sigma_i = 1\}$.

Throughout the present section, we refer to $h$-stability as the condition in (\ref{def:h-stable}) when $W$ is set as the negative of the adjacency matrix of the (multi)-graph under consideration.
% {\color{blue} LZ: From here the structure of the proof should be improved.}

\ignore{
Here a property $P_n$ is defined as the event when $G$ lies in a fixed subset of the $\binom{n}{2}$ possible edges.
Similarly, we obtain a reduction from $\G'(n,m=\frac{d}{2}n)$ to $\G'(n,p=d/n)$.
In our case, the property of interest corresponds to the existence of almost $h$-stable partitions with high probability. 
}

\subsection{Proof sketch}
Before we turn to the proof we provide
the reader with a proof sketch and
the proof plan for Theorem~\ref{thm:sparse_thres}, largely
following \cite{gamarnikMaxcutSparseRandom2018}.
\begin{enumerate}
    \item \emph{First moment conditioned on the number of violations}. 
    A natural approach is to compute
    the expected number of bisections
    satisfying the $h$-stability condition
    for every node. Provided this expectation converges to zero, by 
Markov's inequality this 
 would only allow us to prove the absence of $h$-stable partitions in the regime $h>h^\star$ with high probability. This is however insufficient to
    prove the stronger result in Theorem~\ref{thm:sparse_thres} that with high probability for $h>h^\star$, at least $\epsilon(h)n$ vertices violate $h$-stability. Furthermore, the absence of $h$-stable bisections already follows trivially from the fact that for $\G(n,d/n)$ graphs, there exist $\Theta(n)$ vertices with less than $h\sqrt{d}$ neighbors w.h.p. as $n \rightarrow \infty$. Therefore, we instead compute the first moment of the number of bisections \emph{with at least} $\floor{rn}$  vertices satisfying $h$-stability for some chosen $0 < r \leq 1$.  To account for such violations, we define $X(h,r)$ to be the total number of bisections with at least $\floor{rn}$ vertices satisfying $h$-stability. 
   Our first goal is thus a delicate
   approximations for the first moment of $X(h,r)$.
    
    \item \emph{Conditioning on cut size and the configuration model.}
    To obtain the first moment, it turns out to be convenient to switch to the configuration model and condition on the cut-sizes induced the by bisections Such conditioning ``decouples" the $h$-stability constraints across vertices, leading to tractable computations.  Specifically, denote by $X(z,h,r)$ the number of bisections with at 
    least $\floor{rn}$ $h$-stable nodes and with
    the induced cut size of value $\floor{zn}$. Here the $\floor{\cdot}$ operator ensures that the cut-sizes and the number of $h$-stable nodes remain integer valued.
    The cut size is defined as the number of edges with end points in different parts of the partition i.e $\abs{\{(i,j) \in E: \sigma_i \neq \sigma_j\}}$. When $r=1$, we denote $X(z,h,1)$ simply as $X(z,h)$.
    
    The variable $z$ representing the cut-size is related to another variable $E\in \mathbb{R}$, denoting the (normalized) energy of a configuration as per (\ref{eq:hamilton}), through the following equality:
    \begin{equation}\label{eq:cut_size}
           z=d/4-\frac{E}{2}\sqrt{d}.
    \end{equation}
    The factor $\frac{1}{2}$ arises due to the following relationship between the cut-size $z$ and the energy:
    \begin{equation}
\begin{split}
    -\frac{1}{n\sqrt{d}}\sum_{1\leq i<j\leq n}\sigma_i\,\sigma_j\,W_{ij}
    &= \frac{1}{n\sqrt{d}}\sum_{(i,j)\in E} \sigma_i\,\sigma_j\\
    &=\frac{1}{n\sqrt{d}}(d/4n-\frac{E}{2}n\sqrt{d}-(d/4n+\frac{E}{2}n\sqrt{d}))\\
    &=E\sqrt{d}.
\end{split}
\end{equation}
    
    We note that since the cut size only ranges from $0$ to $\frac{n(n-1)}{2}$, the first moment satisfies the following trivial bounds: \begin{equation}\label{eq:max_ent}
   \max_{z}{\mathbb{E}[X(z,h,r)]}  \leq \mathbb{E}[X(h,r)] \leq \frac{n(n-1)}{2} \max_{z}{\mathbb{E}[X(z,h,r)]}.
\end{equation}
    
    \item \emph{First moment entropy density} Exact, non-asymptotic closed-form expressions for the moments
    of $X(z,h,r)$ appear intractable. Instead, following \cite{gamarnikMaxcutSparseRandom2018},
    we obtain asymptotic expressions for the leading terms in the exponents of the first, and second moments, when conditioned on the cut sizes. We define the first-moment entropy density at a given number of nodes $n$, degree $d$, the proportion of the constrained vertices $r$, threshold $h$, and cut size $\floor{zn}$ (the $\floor{\cdot}$ ensures that the cut-size is an integer),
    as:
\begin{equation}\label{def:firstmomh_cut}
\frac{1}{n} \log \mathbb{E}[X(z,h,r)],
\end{equation}
where $z$ is defined as a function of $E$ through
(\ref{eq:cut_size}).
% % For $r=1$, our results further imply the existence of the above limits (i.e the $\limsup$ in the above limits can be replaced by $\lim$). 
% \dg{only for $r=1$? Perhaps delete this comment}

Note that at $r=1$, the above corresponds to the first moment entropy density of the number of $h$-stable bisections  with a prescribed
cut-size. 
% \ignore{
% \begin{equation}\label{def:firstmomh_cut_r1}
%     w(x,h)=\lim_{d\rightarrow \infty}\lim_{n\rightarrow \infty} \frac{1}{n} \log \mathbb{E}[X(z,h,1)].
% \end{equation}
% }

\item \emph{Second Moment}.
For the second moment, instead of considering all possible bisections of arbitrary cut sizes for the two partitions, we restrict to bisections having a given cut size $\floor{zn}$.  
We further set $r=1$, i.e consider only the partitions with all the vertices satisfying $h$-stability. 
 We define the second moment entropy density at the given value of $z \in \mathbb{R}^n$, and threshold $h$ as:
\begin{equation}
   \frac{1}{n} \log [\mathbb{E}[X^2(z,h)]],
\end{equation}
% where the existence of the above limit is an implication of our results.
The second-moment entropy density can be further divided into contributions from pairs of configurations having a fixed overlap. For two bisections $\sigma, \sigma' \in \{-1,+1\}^n$, we quantify the overlap between the two configurations through a parameter $\omega$ measuring the overlap between the two configurations, defined as:
\begin{align*}
    \omega(\sigma_i,\sigma'_i) = (\sum_{i=1}^n\sigma_i\sigma'_i)/n.
\end{align*}
Let $\cE_{opt}(\sigma,h,z)$ denote the event that the bisection $\sigma$ is $h$-stable and has a cut of size $\floor{zn}$.
We have:
\begin{align*}
X^2(z,h)=\sum_{\omega} \sum_{\sigma, \sigma' \in M_0, \omega(\sigma_i,\sigma'_i)=\omega } \mathbf{1}[\cE_{opt}(\sigma,h,z)\cap \cE_{opt}(\sigma',h, z)],
\end{align*}
where the sum is over $\omega \in [-1,-1+1/n,\cdots,1]$. We further denote by $ X^2_\omega(z,h)$, the contribution to the above sum from a fixed $\omega$ i.e:
\begin{equation}\label{eq:X2omeg}
    X^2_\omega(z,h) \coloneqq \sum_{\sigma, \sigma' \in M_0, \omega(\sigma_i,\sigma'_i)=\omega } \mathbf{1}[\cE_{opt}(\sigma,h, \floor{zn})\cap \cE_{opt}(\sigma',h, \floor{zn})],
\end{equation}
where $\mathbf{1}[\cdot]$ denotes the indicator function.
The second-moment entropy density conditioned on the overlap is then defined as:
\begin{equation}\label{eq:sec_mom_def}
    \frac{1}{n}\log[ \Ea{X^2_\omega(z,h)}]\, .
\end{equation}

% where the existence of the above limit is established through our results (Proposition \ref{prop:second_mom_main}).

% Through the techniques mentioned above, we prove that $w_r(h)$ and $W(x,\omega,h)$ can be described through systems of non-linear equations.

 \item \emph{Poisson, Normal Approximations}. To derive analytic forms for the asymptotic first and second moment entropy densities specified by Equations \ref{def:firstmomh_cut} and \ref{eq:sec_mom_def} respectively, we use the so-called Balls-into-Bins Poisson
approximation, which effectively 
decouples the dependency of 
the $h$-stability condition between nodes.
Specifically, for the Poisson approximation
model, the probability that nodes $u$ and
$v$ are stable is the product of their
respective probabilities. The Poisson
parameters thus arising are growing 
with $d$ and thus well approximated
by the bi-variate normal random variables
(bi-variation corresponding to the joint
distribution of in and out degrees). 
The $h$-stability condition is then 
interpreted simply as the condition that the first component of the bi-variate
dominates the second component. The 
large deviations estimations arise
from the necessity of looking at the
stability events corresponding to 
all (for the first moment) or at least 
one (for the second moment) subset
of cardinality $\floor{rn}$.

    \item \emph{Boosting the probability of existence.} Results from the theory of large deviations, however, only allow us to obtain the first and second moments up to the leading exponential terms. An application of the Paley-Zygmund inequality then only yields a lower bound decaying exponentially with $n$. We boost this lower bound through the use of the concentration of a suitably chosen random variable. This is described in Section \ref{sec:high_prob}. 
\end{enumerate}

\subsection{Proof of Theorem \ref{thm:sparse_thres}, case $h>h^\star$.
}
\label{sec:first_moment}

The proof of Theorem \ref{thm:sparse_thres} when $h>h^\star$ is based on the first moment method. To this end, this section, we calculate up to leading exponential terms, the first moment of the number of bisections containing at least $\floor{rn}$ $h$-stable vertices, denoted by $X(h,r)$ as earlier. Through Markov's inequality, this will result in the proof of the theorem for the
case $h>h^\star$.
The evaluation of the first moment is based on the analysis presented in \cite{gamarnikMaxcutSparseRandom2018}, which concentrated on the setup of local maxima in the Max-cut problem, corresponding to the case of $r=1$ and $h=0$. The primary novel contributions of this section are:

\begin{enumerate}
    \item The generalization of the first moment computation in \cite{gamarnikMaxcutSparseRandom2018} to accommodate configurations with $h$-stability constraint imposed on a fraction $r<1$ of the vertices, rather than all vertices: This requires proving a large deviation result for mixture of distributions and controlling the effect of possibly different fraction of vertices satisfying $h$-stability in the two partitions.
    \item Utilizing the first moment $\mathbb{E}[X(h,r)]$ and a continuity argument to show that with high probability at least $\Theta_d(1)n-o(n)$ vertices violate $h$-stability for any $h>h^\star$. 
\end{enumerate}

Define the function  $w(E,h,r)$
as follows:
\begin{equation}\label{eq:wrh}
    w(E,h,r)=H(r)+(1-r)\log 2 -E^2+ \inf_{\theta \in \mathbb{R}} \left(2\theta^2 +(2r-1)\log(1+\erf(\theta-(2E+h)/\sqrt{2})) \right),
\end{equation}
where $H(r)$ denotes the binary entropy function $H(r) = -r\log r-(1-r)\log(1-r)$. For $r=1$, $w(E,h,1)$ reduces to $w(E,h)$ defined by (\ref{eq:def_w}):
\begin{equation}\label{eq:first_mom_bound}
  w(E,h,1)= -E^2+\inf_{\theta} \left(2\theta^2+\log(1+\erf(\theta-(2E+h)/\sqrt{2}))\right) = w(E,h).
\end{equation}

Recall the relation between the cut-size parameter $z$ and the energy $E$ given by (\ref{eq:cut_size}), i.e:
\begin{equation}
     z=d/4-\frac{E}{2}\sqrt{d}.
\end{equation}

Our first result provides a dimension-independent asymptotic bound on the first moment $\E[X(z,h,r)]$ up to leading exponential terms, with the bound being tight for $r=1$:

\begin{proposition}\label{prop:neg_first_mom}
Consider a random multi-graph $\tilde{\G}(n, \frac{d}{2}n)$ generated under the configuration model with the total number of edges fixed to $\frac{d}{2}n$. Let $X(z,h,r)$ denote the total number bisections in $\tilde{\G}(n, \frac{d}{2}n)$ with cut-size $\floor{zn}$ and at least $rn$ vertices satisfying $h$-stability i.e $s_\sigma(v,W) \geq h$ when $W$ is set as the negative of the adjacency matrix of $\tilde{\G}(n, \frac{d}{2}n)$.

Suppose that either of the following conditions hold:
\begin{enumerate}
    \item $1/2< r < 1$ and $E \in \mathbb{R}$ is arbitrary.
     \item $r=1$ and $E \leq -\frac{h}{2}$. 
\end{enumerate}
Then, $X(z,h,r)$ for $z=d/4-\frac{E}{2}\sqrt{d}$, satisfies:
\begin{align*}
\limsup_{d \rightarrow \infty}\limsup_{n \rightarrow \infty} {1\over n}
\log\E[X(z,h,r)]& \leq w(E,h,r),
\end{align*}
where $w(E,h,r)$  defined in \ref{eq:wrh}.
Furthermore, the inequality is tight when $r=1$ and $E \leq -\frac{h}{2}$. Namely,
\begin{equation}
    \lim_{d \rightarrow \infty}\lim_{n \rightarrow \infty} {1\over n}
\log\E[X(z,h,1)] = w(E,h),
\end{equation}
where $w(E,h)$ is  defined in \ref{eq:def_w}. 
\end{proposition}

\textbf{Remark:} The upper bound $-h/2$ in the range of $E$ when $r=1$ is a consequence of (\ref{eq:stab_e}) and (\ref{eq:hamilton}), which imply that:
\begin{equation}
    H(\sigma)= -\sum_{v \in [n]} s_\sigma(v).
\end{equation}
Thus, the constraint $s_\sigma(v) \geq h$ for all $v \in [n]$ implies that $ H(\sigma) \leq -\frac{h}{2}n$. $X(z,h,1)$ is therefore identically $0$ whenever $E > \frac{h}{2}$. The condition $1/2< r \leq 1$ however, is an artifact of our proof and choice of the upper bound $w(E,h,r)$. Such a restriction on $r$ suffices for our main results.

The proof of the above proposition 
closely follows \cite{gamarnikMaxcutSparseRandom2018}, 
and we relegate the full proof to Appendix \ref{app:first_mom}.

% We record here the following additional property of $w(E,h)$ for subsequent usage:
% \begin{proposition}\label{prop:w_convex}
% For any $h \in \mathbb{R}$, $w(E,h)$ is strictly concave in $E$.
% \end{proposition}

% \yd{New proposition since it doesn't directly follow from above due to discretization the restriction to $[-2,2]$
% Next, using (\ref{eq:max_ent})

Proposition \ref{prop:neg_first_mom} characterizes the first moment of $h$-stable bisections having cut-size $\floor{zn}$ for a fixed value of $z \in \mathbb{R}$. As one would expect from (\ref{eq:max_ent}), maximizing over $E$ in Proposition \ref{prop:neg_first_mom} results in the characterization of the first moment of the total number of $h$-stable bisections across all values of cut-sizes.  For subsequent usage, we prove a slightly more general result, where we maximize over $E$ in an interval $(a,b)$ with $-\infty \leq a< b \leq \infty$. The first moment entropy density upon such maximization is established in the proposition below, which follows from Proposition \ref{prop:neg_first_mom} and a straightforward control over the range of $z$ and discretization errors:

\begin{proposition}\label{prop:first_moment_sup}
For any $a,b$ with $-\infty \leq a<b \leq \infty$, define:
\begin{equation}\label{def:xab}
    X_{a,b}(h,r) = \sum_{ z: zn \in \mathbb{N},\frac{d}{4}-\frac{b}{2}\sqrt{d} < z < \frac{d}{4}-\frac{a}{2}\sqrt{d}} X(z,h,r).
\end{equation}
Then, the following holds for any $1/2<r \leq 1$:
\begin{equation}
    \limsup_{d\rightarrow \infty}\limsup_{n\rightarrow \infty} \frac{1}{n} \log \mathbb{E}[X_{a,b}(h,r))] \leq \sup_{E \in (a,b)} w(E,h,r). 
\end{equation}
Furthermore, for $r=1$, the above inequality is tight:
\begin{equation}
     \lim_{d\rightarrow \infty}\lim_{n\rightarrow \infty} \frac{1}{n} \log \mathbb{E}[X_{a,b}(h,1)] = \sup_{E \in (a,b)} w(E,h),
\end{equation}
% Lastly, for every $h>h^\star$
% there exists $1/2<r(h)<1$ such that
% $w_{r(h)}(h)<0$.

where $w(E,h)$ is as defined in (\ref{eq:def_w}).

% and the supremum in $\sup_{E \in \mathbb{R}} w(E,h)$ is achieved at a unique value $E^\star(h)<-\frac{h}{2}$. 
\end{proposition}

The above proposition directly allows us to establish $\Theta(n)$ violations in $h$-stability amongst bisections for all $h>h^\star$. Recall that $w(h)=\sup_{E \in \mathbb{R}} w(E,h)$. Therefore, setting $a=-\infty, b=\infty$ in Proposition \ref{prop:first_moment_sup} yields:

\begin{equation}\label{eq:whr}
     \lim_{d\rightarrow \infty}\lim_{n\rightarrow \infty} \frac{1}{n} \log \mathbb{E}[X(h,1)] = \sup_{E \in \mathbb{R}} w(E,h) = w(h).
\end{equation}

\subsubsection{[Proof of Theorem \ref{thm:sparse_thres} for $h>h^\star$]}\label{sec:cont_neg}

In this section, we establish the absence of $h$-stable bisections for $h>h^\star$ (while allowing for $o(n)$ violations). This constitutes Theorem \ref{thm:sparse_thres} for $h>h^\star$.
The proof utilizes Proposition \ref{prop:first_moment_sup} along with a continuity-based argument. The required continuity properties are based on the following proposition, whose proof is be found in Appendix \ref{app:properties}.

\begin{proposition}\label{prop:first_mom_conv_r}
For all $h \in \mathbb{R}$, $r \in (\frac{1}{2},1]$, $w(\cdot,h,r):\mathbb{R} \rightarrow \mathbb{R}$ is continuously differentiable, strictly concave and admits a unique maximizer.
\end{proposition}

In light of Proposition \ref{prop:first_mom_conv_r}, we define:
\begin{equation}
    w(h,r)=\sup_{E \in \mathbb{R}} w(E,h,r),
\end{equation}
and:
\begin{equation}
    E^\star(h,r) =\operatorname{argmax}_{E \in \mathbb{R}} w(E,h,r),
\end{equation}
where the uniqueness of the maximizer is guaranteed by Proposition \ref{prop:first_mom_conv_r}.

A consequence of 
Proposition \ref{prop:first_mom_conv_r} is the continuity of $E^\star(h,r)$ w.r.t $h,r$ (proof in Appendix  \ref{app:properties}):

\begin{corollary}\label{cor:Econthr}
    $E^\star(h,r)$
is continuous w.r.t $r$ in $(1/2,1]$.
\end{corollary}

The above results imply that for all $h>h^\star$, $w(h,r)<0$ for $r$ sufficiently close to $1$:
\begin{proposition}\label{prop:moment_neg}
For any $h>h^\star$, there exists 
$r(h)<1$ such that:
\begin{equation}
    w(h,r(h)) < 0.
\end{equation}
\end{proposition}
\begin{proof}
Recall that, by definition $w(h,r)=w(E^\star(h),h,r)$. Therefore, by the closure of continuous functions under composition, the continuity of $w(E,h,r)$ (Proposition \ref{prop:first_mom_conv_r}) and that of $E^\star(h)$ (Corollary \ref{cor:Econthr}) imply the continuity of $w(h,r)$ w.r.t $r,h$ for $h<h^\star$. Since, by proposition \ref{prop:w_mono}, $w(h)$ is strictly decreasing for $h>h^\star$ with $w(h) < 0$, we obtain that $w(h,r) < 0$ for a sufficiently small neighborhood around $r=1$.
\end{proof}

With Proposition \ref{prop:moment_neg}, we're now ready to prove the case $h>h^\star$ in Theorem \ref{thm:sparse_thres}.
To see this, note that Proposition \ref{prop:moment_neg} along with
Proposition \ref{prop:first_moment_sup} with $a=-\infty, b=\infty$ imply that $\exists$ a constant $C(h)$ such that for large enough $d,n$:
\begin{equation}
    \mathbb{E}[X(h,r(h))] \leq \exp{(-C(h)n)},
\end{equation}
where $r(h)$ is as defined by Proposition \ref{prop:moment_neg}. 
Therefore, using Markov's inequality, we obtain that for any $h>h^\star$ and large enough $d$, with high probability as $n \rightarrow \infty$, any bisection contains at least $(1-r(h))n$ vertices violating $h$-stability. Since, $r(h)$ is independent of $d,n$, this completes the proof of Theorem \ref{thm:sparse_thres} for the case $h>h^\star$ for the configuration model. Furthermore, since $\Pr[X(h,r) >0 ] \leq \exp(-Cn)$ for some constant $C>0$, Lemmas \ref{lemma:2-random-graphs} and \ref{lem:config} then imply the result for the associated models $\bar{\G}(n, \frac{d}{2}n)$ and $\G(n, \frac{d}{2})$.

% Furthermore, the above optimality condition is equivalent to
% $\theta^*$ being optimal for
% $\log(M(\theta,\tilde h,x')$ with $\tilde h=,x'=$.

\begin{figure}
    \centering
    \includesvg[width=0.7\textwidth]{firstmomthres_8.svg}
    \caption{The first moment entropy density $w(h)$ as a function of the imposed threshold $h$, obtained from (\ref{eq:first_mom_bound}). The threshold $h^\star \approx 0.3513$ marks the point where the entropy becomes negative. Another value of interest is the entropy density $w(0) \approx 0.1992$ for $h=0$ that is well known from \cite{bray1981metastable}}. 
    \label{fig:first_mom}
\end{figure}

\qed

\begin{remark}
We note in Figure \ref{fig:first_mom} that at $h=0$, the numerical value of the first moment entropy density $w(0) \approx 0.1992$ matches the one obtained in \cite{skcount}. Furthermore, it matches the value of the entropy density i.e $\frac{1}{n}\E[\log{X}]$ reported in \cite{bray1981metastable}. This indicates similar to the high-temperature regime of the SK model, the number of local optima equals the corresponding first-moment entropy density. Furthermore, we conjecture that the value of the entropy density is universal. We leave the rigorous confirmation of this result to future work.
     
\end{remark}

\subsection{Proof of Theorem \ref{thm:sparse_thres}, case $h<h^\star$}\label{sec:second_mom}

To prove the existence of bisections with nearly all vertices being $h$-stable, we move to the approximation of the second moment entropy density. Unlike the proof for the first moment entropy density in Section \ref{sec:first_moment}, here we do not consider the $h$-stablity of a fraction $r<1$ of the vertices since it suffices to prove the existence of partitions for $r=1$. Therefore, the proof of the second moment does not require substantial modifications to the arguments in \cite{gamarnikMaxcutSparseRandom2018}, who obtained the second moment for local maxima of MAX-CUT i.e the case $h=0$. However, for the sake of completeness, we provide the full proof in Appendix \ref{app:sec_mom}, where we leverage techniques and results from our first-moment analysis.

Analogous to Propositions \ref{prop:neg_first_mom}, \ref{prop:first_moment_sup}, we obtain an asymptotic characterization of the second moment entropy density through the function $W(E,\omega,h)$ defined by (\ref{eq:Wxbeta}) (proof in Appendix \ref{app:sec_mom}):

\begin{proposition}\label{prop:second_mom_main}
Recall the definition of $X^2_\omega(z,h)$ in (\ref{eq:X2omeg}).
For any $h \in \R, \omega \in (-1,1)$ and any $E < -h/2$, with $z=\frac{d}{4}-\frac{E}{2}\sqrt{d}$
\begin{equation}
  \lim_{d \rightarrow \infty} \lim_{n 
  \rightarrow \infty} \frac{1}{n}\log[\mathbb{E}[ X^2_\omega(z,h)]] = W(E,\omega,h),
\end{equation}
where $W$ is as defined in (\ref{eq:Wxbeta}).
\end{proposition}
% \begin{align}
% W(x,\omega,h)=&-2\beta\log\beta-2(1/2-\beta)\log(1/2-\beta) \notag\\
% &-{1\over 2}{t^*(x,\beta)^2\over \beta^2}-{1\over 2}{(x-t^*(x,\beta))^2\over (1/2-\beta)^2}
% +2\beta
% \log P\left(\theta_1^*(x,\beta),\sqrt{1/2-\beta\over \beta},
% {t^*(x,\beta)\over \beta^{3/2}}-\frac{h}{\sqrt{2\beta}}\right) \notag\\
% &+2(1/2-\beta)
% \log P\left(\theta_2^*(x,\beta),\sqrt{\beta\over 1-2\beta},
% {x-t^*(x,\beta)\over (1/2-\beta)^{3/2}}-\frac{h}{\sqrt{1-2\beta}}\right) , \label{eq:Wxbeta}
% \end{align}
% where
% \begin{align*}
%     \beta &= (\omega+1)/4\\
%     P(\theta,a_1,a_2)= &\frac{1}{\pi} \exp(\theta^2/2)\int_{0}^{\infty} \int_{a_1 z_2}^{\infty} \exp(-((z_1-\theta-a_2)^2+z_2^2)/2) dz_1 dz_2,  \label{eq:Pfunction}
% \end{align*}

For $h=h^\star$ and $E=E^*(h^\star)$, i.e. the unique maximizer of $w(E,h^\star)$ w.r.t $E$, $W(E,\omega,h)$ is plotted in Figure \ref{fig:2ndmom_thres}. Maximizing over $\omega \in (-1,1)$ then leads to a description of the total second moment entropy density:
\begin{proposition}\label{prop:second_mom_sup}
For any $h \in \mathbb{R}$ and $E < -h/2$:
    \begin{equation}
         \lim_{d \rightarrow \infty} \lim_{n \rightarrow \infty} \frac{1}{n} \log [\mathbb{E}[X^2(z,h)]] = \sup_{\omega \in (-1,1)} W(E,\omega,h)
    \end{equation}
\end{proposition}

 Under Assumption \ref{ass:unique_max}
 and as illustrated in Figure \ref{fig:2ndmom_thres}, 
   $W(E^*(h),\omega, h^\star)$ is maximized at $\omega=0$. 
   
We further have the following general relation between $W(E,\omega,h)$ and $w(E,h)$ (proof in Appendix \ref{app:properties}):
\begin{proposition}\label{prop:second_mom_max}
For any $E,h \in \mathbb{R}$,
\begin{equation}
    W(E,0,h) = 2 w(E,h).
\end{equation}
\end{proposition}

Proposition \ref{prop:second_mom_sup} and  Assumption \ref{ass:unique_max} then yield the following lower bound on the second moment:
\begin{lemma}\label{lem:paley}
For $h \in \mathbb{R}$, let $z^\star(h)\coloneqq -\frac{d}{4}-\frac{E^\star(h)}{2}\sqrt{d}$.
Then:
    \begin{equation}
        (\E[X(z^*(h^\star),h^\star)])^2 \geq \exp(-o_d(1)n-o(n)) \E[X^2(z^*(h^\star),h^\star)]\, .
    \end{equation}
\end{lemma}
\begin{proof}

By Assumption \ref{ass:unique_max} and Proposition \ref{prop:second_mom_max}, we have that $W(E^*(h^\star),0,h^\star)=2w(E^*(h^\star),h^\star)=0$. Thus, Proposition \ref{prop:first_moment_sup}  and Proposition \ref{prop:second_mom_sup} imply :
    \begin{equation}\label{eq:paley_2ndmom}
        \frac{1}{n} \log \E[X^2(z^*(h^\star),h^\star)] =  \frac{1}{n} 2\log \E[X(z^*(h^\star),h^\star)]+o_n(1).
    \end{equation}
    Exponentiating both sides completes the proof.
\end{proof}

% To avoid numerical computation of the double integrals in the above expressions, we use the following lemma:

% \begin{lemma}
% \label{basis_cal}
% For $a\in \mathbb{R}$,
% \begin{align*}
% \label{dinte1} \int_{0}^{\infty}  \int_{a_1t_2}^{\infty} \exp \left(-((t_1-a_2)^2+t_2^2)/2 \right) dt_1 dt_2&=\int_0^{\infty } \sqrt{\frac{\pi }{2}} e^{-\frac{z_2^2}{2}} \left(\text{erf}\left(\frac{a_2-a_1 z_2}{\sqrt{2}}\right)+1\right) \, dz_2\\
% &= \frac{\pi}{4} + \int_0^{\infty } \sqrt{\frac{\pi }{2}} e^{-\frac{z_2^2}{2}} \left(\text{erf}\left(\frac{a_2-a_1 z_2}{\sqrt{2}}\right)\right) \, dz_2\\
% \end{align*}
% \end{lemma}

\begin{figure}
    \centering    \includesvg[inkscapelatex=false,width=0.8\textwidth]{secondmomhthres_3.svg}
    \caption{The second-moment entropy density at the threshold i.e $h=h^\star$, constrained at the normalized energy $E^*(h^\star)$. The curves are obtained by numerically solving the system of equations defined by (\ref{eq:Wxbeta}).}
    \label{fig:2ndmom_thres}
\end{figure}
% \subsection{Rademacher weights for d-regular graphs}

% The asymptotic expression for the number d-regular graphs is given by:
% \begin{align*}
%     g_{d, n} \sim \frac{(d n) ! e^{-\left(d^{2}-1\right) / 4}}{(d n / 2) ! 2^{d n / 2} \prod d_{i} !}
% \end{align*}

% The number of bipartite graphs with given degree sequences $\mathbf{r}$ and $\mathbf{s}$ is given by

% $$
% b(\mathbf{r}, \mathbf{s}) \sim \frac{M_{1} ! e^{-\alpha}}{\prod_{i=1}^{m} r_{i} ! \prod_{i=1}^{n} s_{i} !}
% $$

% Lastly, the number of graphs with a given degree sequence $\mathbf{d}$ is given by:

% $$
% g(\mathbf{d}) \sim \frac{M_{1} ! e^{-\alpha}}{\left(M_{1} / 2\right) ! 2^{M_{1} / 2} \prod_{i=1}^{n} d_{i} !}
% $$

% Therefore, the asymptotic probability of a given partition being friendly with friendliness $H$ can be expressed as:

% \begin{align*}
%     \frac{1}{\frac{(d n) ! e^{-\left(d^{2}-1\right) / 4}}{(d n / 2) ! 2^{d n / 2} \prod d_{i} !}}\sum_{r_1,r_2,\cdots, s_1,s_2,s_3,\cdots, s_n=H}^d \frac{M_{1} ! e^{-\alpha}}{\left(M_{1} / 2\right) ! 2^{M_{1} / 2} \prod_{i=1}^{n} r_{i} !} \frac{(nd-M_{1})! e^{-\alpha}}{\prod_{i=1}^{m} r_{i} ! \prod_{i=1}^{n} (d-r_{i}) !}\frac{(M_1! e^{-\alpha}}{\left (M_{1}) / 2\right) ! 2^{M_{1} / 2} \prod_{i=1}^{n} d_{i} !}
% \end{align*}
% where the sum is over in-degrees $r_1,r_2,\cdots,r_n$ and $s_1,s_2,\cdots,s_n$ satisfying $\sum_{i=1}^n r_i = \sum_{i=1}^n s_i$
% We now use Stirling's approximation and the series expansion for $\log (1+x)$.

\subsection{Existence of partitions with high probability}\label{sec:high_prob}

Equipped with the asymptotic first and second moment entropy densities (through Propositions \ref{prop:first_moment_sup} and \ref{prop:second_mom_main}), one could hope to apply the Paley-Zygmund inequality to obtain a lower bound on the probability of the existence of $h$-locally optimal configurations. However, this only leads to a sub-exponential lower bound. Concretely, using the Paley-Zygmund inequality and Lemma \ref{lem:paley}, we have:
\begin{equation}\label{eq:lower_bound}
    \pr(X(z^*(h),h^\star)\ge 1)\ge {(\E[X(z^*(h^\star),h)])^2\over (\E[X^2(z^*(h^\star),h^\star)])^2} \ge \exp(-(o_d(1)n+o(n)).
\end{equation}

% However, an Azuma-Hoeffding bound when directly applied to the optimal $h$  results in:

% \begin{align}
% \label{Azuma_MC}
% \mathbb{P}[\lvert h_{d,n}-\mathbb{E}[h_{d,n}]  \rvert \geq \epsilon\sqrt{d}] \leq 2 \exp\left( -\frac{\epsilon^2}{2n} \right), \; \text{for} \;  t>0.
% \end{align}

% The above bound does not vanish as $n\rightarrow \infty$. Indeed, unlike the value of the max-cut, $h_{d,n} = \mathcal{O}(\sqrt{d})$ instead of $dn+\mathcal{O}(\sqrt{d}n)$
A standard technique for boosting such lower-bounds is through the use of concentration inequalities involving sharper tails. For instance, see \cite{frieze1990independence}. This requires as a first step, identifying a suitable concentrated random variable.

We therefore introduce a global function measuring the total deficit in $h$-stability, defined as follows:
\begin{definition}\label{def:deficit}
For each $h$ and $\sigma\in B_n=\{\pm 1\}^n$ let $D(W,h,\sigma)=\sum_i \left(h-{1\over\sqrt{d}}\sum_j w_{ij}\sigma_i\sigma_j\right)^+$,
where $(\cdot)^+$ denotes the linear threshold function defined as $(x)^+=x$ for $x> 0$ and $0$ otherwise. 
We say that $D(W,h,\sigma)$ is the total $h$-deficit associated with partition $\sigma$. Let $D^*(W,h)=\min_{\sigma  \in M_0}D(W,h,\sigma)$. 
\end{definition}

Since $X(z^*(h),h^\star)\ge 1$ if and only if $D^*(W,h)=0$, (\ref{eq:lower_bound}), implies that for every $h<h^\star$
\begin{align}\label{eq:sub-exponential}
\pr\left(D^*(W,h)=0\right)\ge \exp(-(o_d(1)n+o(n)).
\end{align}

To amplify the above weak lower-bound, we start by proving the concentration of $D^*(W,h)$ at an exponential rate. The concentration is a consequence of McDiarmid's inequality, which we state here again for convenience:
\begin{lemma}\label{lem:Mcdiarmid}(McDiarmid's inequality)
Let $f:\cX^k\rightarrow \R$ be a real valued function of $k$ independent random variables taking values in a measurable space $\cX$. Suppose $f$ satisfies the following bounded difference properties for all $x_1,\cdots,x_k \in \cX$, and $1 \leq i \leq k$:
$$
    \sup_{x'_i \in \cX}\abs{f(x_1,\cdots,x'_i,\cdots, x_k)-f(x_1,\cdots,x_i,\cdots, x_k )} \leq c_i.
$$
Then, for any $\epsilon > 0$:
$$
    \pr\left(\abs{f(x_1,\cdots, x_k )-\Ea{f(x_1,\cdots,x_k )}}\geq \epsilon\right) \leq 2\exp{\left(-2\frac{\epsilon^2}{\sum_{i=1}^kc^2_i}\right)}.
$$
\end{lemma}
\begin{lemma}\label{lem:concent}
Let $\tilde{\G}(n,d/n)$, denote a multi-graph sampled under configuration model with a fixed number
of edges $m=dn/2$. Set $W$ as the negative of the adjacency matrix of $\tilde{\G}(n,d/n)$ (anti-ferromagnetic interactions). Then, for any $h<h^\star$ the minimum deficit $D^*(W,h)$,satisfies:
\begin{align*}
\pr\left(|D^*(W,h)-\Ea{D^*(W,h)}| \ge \epsilon n \right) \le 2\exp\left(-\epsilon^2n\right).
\end{align*}
\end{lemma}
\begin{proof}
We represent the generative process under the configuration model with a fixed number of $d/2n$ edges, through $d/2n$ variables $e_1,e_2,\cdots,e_{d/2n}$.
Where $e_i$ denotes the pairs of vertices independently assigned to each of the $d/2n$ edges.
Let $W$ and $W'$ denote the weight matrices for two vertex assignments $e_1,e_2,\cdots,e_i,\cdots,e_{d/2n}$ and $e_1,e_2,\cdots,e'_i,\cdots,e_{d/2n}$ differing only at the $i_{th}$ edges.
Due to the $1$-Lipschitzness of the threshold function $(\cdot)^+$, we have $\abs{D^*(W,h)-D^*(W',h)} \leq \frac{2}{\sqrt{d}}$. Therefore, an application of McDiarmid's inequality (Lemma \ref{lem:Mcdiarmid}) yields:
\begin{equation}
   \pr\left(|D^*(W,h)-\Ea{D^*(W,h)}| \ge \epsilon n \right) \le 2\exp\left(-2\frac{\epsilon^2n^2}{(d/2)n\times 4/d}\right) = 2\exp\left(-\epsilon^2n\right).
\end{equation}
\end{proof}
The above lemma allows us to bound the expectation $\Ea{D^*(W,h)}$:
\begin{corollary}\label{cor:def_1}
Let $W$ be a weight matrix as in Lemma \ref{lem:concent}, then for any $h < h^\star$:
\begin{equation}
    \Ea{D^*(W,d)}=o_d(1)n+o(n)\, .
\end{equation}
\end{corollary}
\begin{proof}
Suppose that, on the contrary, there exists $\epsilon > 0$ such 
that there exist arbitrarily large $d(\epsilon), n(\epsilon)$ satisfying:
\begin{equation}
    \Ea{D^*(W,d,h)}\ge \epsilon n.
\end{equation}

Then Lemma \ref{lem:concent} implies that
 $\mathbb{P}[D^*(W,h)=0]$ is at most $2\exp(-\epsilon^2n)$ for $d=d(\epsilon),n=n(\epsilon)$.
 However, (\ref{eq:sub-exponential}) which implies that for any large enough $d,n$,  $\mathbb{P}[D^*(W,h)=0]\geq \exp(-kn)$ for any constant $k$. By setting $k>\epsilon^2$, we obtain a contradiction.
\end{proof}

Applying Lemma \ref{lem:concent} again, we obtain the following high-probability bound on $D^*(W,h)$:
\begin{corollary}\label{corr:def_2}
For every $\epsilon>0$, and $h<h^\star$,for large enough $d,n$ we have:
\begin{align}\label{eq:sub-linear}
\pr\left( D^*(W,h)\ge \epsilon n\right) 
\le 
2\exp(-\epsilon^2n/4).
\end{align}     
Therefore, $D^*(W,h) =o_d(1)n+o(n)$ with high probability.
\end{corollary}
\begin{proof}
For large enough $d,n$, we have by Corollary \ref{cor:def_1}, that $\Ea{D^*(W,h)} \leq \epsilon n/2$ . Therefore, for large enough $d,n$:
\begin{equation}
    \pr\left( D^*(W,h)\ge \epsilon n\right) \leq \pr\left( | D^*(W,h)-\Ea{D^*(W,h)} |\ge \epsilon n/2\right) \le 
2\exp(-\epsilon^2n/4).
\end{equation}
Where the last inequality follows from Lemma \ref{lem:concent}.
\end{proof}

While the existence of $h$-stable bisections implies a $0$ minimal deficit, a small minimal deficit doesn't imply the existence of nearly $h$-stable partitions. This is because the deficit can be distributed amongst a large number of vertices. Therefore, to relate the above bound on the total deficit to the maximum number of $h$-stable vertices, we introduce a technique of perturbing the stability threshold $h$.

Let $N(W,h,\sigma)=\sum_i {\bf 1}\left({1\over \sqrt{d}}\sum_j W_{ij}\sigma_i\sigma_j-h\ge 0\right)$ denote the count of h-stable vertices. 
Define $N^*(W,h)=\max_{\sigma \in M_0}N(W,h,\sigma)$. We obtain the following result, whose proof illustrates the use of our perturbation technique:
\begin{lemma}\label{lem:high_prob_exist}
For every $h<h^\star$ and $\epsilon>0$,  $\pr(N^*(W,h)\le n-\epsilon n)\le \exp(-\Theta_d(1)n-o(n))$.
\end{lemma}

\begin{proof}
Let $\mathcal{E}_n=\mathcal{E}_n(h)$ denote the event $N^*(W,h)\leq n-\epsilon n$. Fix $W$ such that this event holds.
Fix any $\sigma$ and 
define $S(\sigma)$ as the set of vertices $i\in [n]$
such that 
\begin{align*}
h-{1\over \sqrt{d}}\sum_j w_{ij}\sigma_i\sigma_j\ge 0.
\end{align*}
Then, the event $\mathcal{E}_n(h)$ can be equivalently expressed as $|S(\sigma)|\ge \epsilon n$.
Fix $\tilde h$ with $h<\tilde h<h^\star$. 
We have
\begin{align*}
D(W,\tilde h,\sigma)=\sum_i \left(\tilde h-{1\over \sqrt{d}}\sum_j w_{ij}\sigma_i\sigma_j\right)^+
&= 
\sum_i \left(\tilde h-h+h-{1\over \sqrt{d}}\sum_j w_{ij}\sigma_i\sigma_j\right)^+ \\
&\ge 
\sum_{i\in S(\sigma)} \left(\tilde h-h+h-{1\over \sqrt{d}}\sum_j w_{ij}\sigma_i\sigma_j\right)^+  \\
&=
\sum_{i\in S(\sigma)} \left(\tilde h-h+h-{1\over \sqrt{d}}\sum_j w_{ij}\sigma_i\sigma_j\right) \\
&\ge 
\sum_{i\in S(\sigma)} \left(\tilde h-h\right) \\
&\ge (\tilde h-h)\epsilon n.
\end{align*}
Since $\sigma$ was arbitrary, we conclude $D^*(W,\tilde h)\geq (\tilde h-h)\epsilon n$ which occurs with probability at most  $\exp{(-Cn)}$ for some constant $C$ and large enough $d,n$ due to Corollary \ref{corr:def_2}.
Thus the event $\mathcal{E}_n$ occurs with probability at most $\exp(-\Theta_d(1)(n)+o(n))$.
\end{proof}

Lemmas \ref{lemma:2-random-graphs} and \ref{lem:config} then imply that the event $N^\star(W,h) = n-o(n)$ holds with high-probability for the associated models $\bar{\G}(n, \frac{d}{2}n)$ and $\G(n, \frac{d}{2})$. This establishes Theorem \ref{thm:sparse_thres} for $h<h^\star$.

\section{Proof of Theorem \ref{thm:sparse_max_E_thres}: Energy of $h$-stable configurations}\label{sec:max_energy}

% \dg{note to self: introduce this somewhere here
% For notational convenience, we further define the following normalized energy function:
% \begin{equation}\label{eq:def_energy}
%     E(\sigma)=\frac{1}{n}H(\sigma)= -\frac{1}{n\sqrt{d}}\sum_{1\leq i<j\leq n}\sigma_i\,\sigma_j\,w_{ij} \, . 
% \end{equation}https://arxiv.org/pdf/1611.09940.pdf
% }

Our proof of Theorem \ref{thm:sparse_thres}, relied on restricting the $h$-stable configurations to different values of Energy (or equivalently the cut-size $z$). This restriction led to the decoupling of the constraints across partitions crucial for the remaining arguments.

However, as a by-product we naturally obtained the first and second moment entropy densities of $h$-stable configurations and different energy levels. In this section we will exploit this to establish phase transitions related to the range of energies of $h$-stable optima.

Similar to the previous section, we first demonstrate this for the case anti-ferromagnetic interactions (i.e local max-cut) in sparse \ER graphs. This constitutes Theorem \ref{thm:sparse_max_E_thres}. In Section \ref{sec:universality_e_max}, we extend these results through a universality argument to other distributions over sparse and dense graphs.
 
Recall that, we have from Proposition \ref{prop:neg_first_mom} that  the first moment entropy density $ \frac{1}{n} \log \mathbb{E}[X(z,h)]$ at fixed energy $E$ with $z=\frac{d}{4}-\frac{E}{2}\sqrt{d}$ converges in probability under the limit $n \rightarrow \infty$ and $d \rightarrow \infty$ to a deterministic value $w(E,h)$ given by:
\begin{equation}\label{eq:fin_bound_r1_E1}
w(E,h)=-E^2+\inf_{\theta\in \mathbb{R}}(\theta^2+\log(1+\erf(\theta-(2E+h)/\sqrt{2}))).
\end{equation}
Analogously, Proposition \ref{prop:second_mom_main} establishes $W(E,\omega,h)$ as the limit of the asymptotic  second moment entropy density. 

Our proof of Theorem \ref{thm:sparse_max_E_thres} follows directly through certain properties of $w(E,h), W(E,\omega,h)$ derived from Assumptions \ref{ass:unique_max}, \ref{ass:0_unique}, and \ref{ass:e_max_h}, combined with the introduction of associated deficit functions and boosting of probability analogous to Section \ref{sec:high_prob}. 
Recall that as a first consequence, Assumptions \ref{ass:unique_max}, \ref{ass:0_unique}, and \ref{ass:e_max_h} imply that $E_{\rm cor}(h), h_{\rm cor}$ exist and satisfy $h_{\rm cor} < h^\star$ and $E_{\rm cor}(h)<E_{\rm max}(h)$. (Proposition \ref{prop:hcor})

% Therefore the set $h: E_{min}(h) < E_{\rm cor}(h) < E_{max}(h)$ for all $h \in [h,h^\star]$

\subsection{Illustration and numerical search}

Before proceeding further, we provide numerical illustrations of $h_{\rm cor}, E_{\rm cor}(h)$ to elucidate the corresponding phase transitions more clearly. We remark that numerical search procedure described here is only utilized to estimate the values of $h_{\rm cor}, E_{\rm cor}(h)$ while our results simply require the existence of $h_{\rm cor}$ which relies solely on Assumptions \ref{ass:unique_max}, \ref{ass:0_unique}, and \ref{ass:e_max_h}. Thus, this section does not contribute any mathematical results.

Recall that $E_{\rm cor}$ denotes the largest energy below $E_{\rm max}$ where overlap $\omega=0$  ceases to remain the global maximizer. Numerically, we observe that this happens through $\omega=0$ turning from a local maxima to a local minima. This corresponds to the point where the second derivative of $W(E,\omega,h)$ at $\omega=0$ vanishes. We illustrate the existence of $E_{\rm cor}(h)$ for $h=0$ in Figure \ref{fig:cor0}. Observe that as $E$ decreases from $E_{\rm max}(0)$, there exists a value of the energy $E_{\rm cor}(0)>E_{\rm min}$ where the overlap $\omega=0$ stops being a maxima of $W(E,\omega,h)$. Numerically, we observe that $E_{\rm cor}(0)\approx -0.6721$. The value $E_{\rm cor}(0)$ matches the value described in   \cite{bray1981metastable} as the onset of the ``correlation of local minima" for the Sherrington-Kirkpatrick Model. (Note however, that our result relies solely on the existence of $E_{\rm cor}(h)$ and not its precise value.

% In statistical physics, such a transition usually precedes or corresponds to the onset of replica symmetry breaking.

Next, to estimate $h_{\rm cor}$, recall that as implied by Assumption \ref{ass:0_unique}, at
$h=0$, we have that $E_{\rm min}<E_{\rm cor}<E_{\rm max}$.
Increasing $h$ from $0$, we observe that $E_{\rm cor}$ and $E_{\rm min}$ intersect at an intermediate value $< h^\star$. We treat this value as our estimate of $h_{\rm cor}$ defined as in \eqref{def:h_cor}, whose existence is implied by Proposition \ref{prop:hcor}. Numerically, we obtain that $h_{\rm cor} \approx 0.2860$. As illustrated in Figure \ref{fig:trans_e_h_cor}, at $h\approx h_{\rm cor}$, the second moment entropy density $W(E,\omega,h)$ at $E=E_{\rm cor}(h)$ is maximized at $\omega=0$ and vanishes.
For $h < h_{\rm cor}$, the quantity $E_{\rm min}(h)$ is deemed unphysical, given that it may not represent the minimum energy of nearly $h$-stable configurations. For instance, when $h=0$, the minimum value of the energy has been proven to be $E \approx -0.7632$ \citep{dembo2017extremal}, while $E_{min}(0) \approx-0.7907$

\begin{figure}
    \centering
\includesvg[width=0.6\textwidth]{secondmomhthreshcorbefore_fixed.svg}
\includesvg[width=0.6\textwidth]{secondmomhthreshcor_fixed.svg}
\includesvg[width=0.6\textwidth]{secondmomhthreshcorafter_fixed.svg}
    \caption{Illustration of the transition in the second moment entropy density at (from top to bottom), $h=h_{\rm cor}$ and $E=E_{\rm cor}(h)+0.01, E_{\rm cor}(h), E_{\rm cor}(h)-0.01$ respectively. }
    \label{fig:trans_e_h_cor}
\end{figure}

\begin{figure}
    \centering
    \includesvg[width=0.8\textwidth]{firstmomfull0_updated.svg}
    \caption{First moment entropy density at $h=0$ as a function of $E$. The curves are obtained by numerically solving (\ref{eq:fin_bound_r1_E1}). 
    }
    \label{fig:first_mom_h0}
\end{figure}
\begin{figure}
    \centering
\includesvg[width=0.6\textwidth]{secondmomhthresh0before_fixed.svg}
\includesvg[width=0.6\textwidth]{secondmomhthresh0_fixed.svg}
\includesvg[width=0.6\textwidth]{secondmomhthresh0after_fixed.svg}
    \caption{Illustration of the transition in the second moment entropy density at $h=0$ and (from top to bottom), $E_{\rm cor}(0)+0.01, E=E_{\rm cor}(0), E_{\rm cor}(0)-0.01$}
    \label{fig:cor0}
\end{figure}

\subsection{Summary of the transitions in $E,h$}

The values $h_{\rm cor},E_{\rm cor}, E_{max},E_{min}$ demarcate the transitions in the behavior of the first and second moments of $h$-stable bisections as one varies $E, h$. Below, we sketch how these transitions imply Theorem \ref{thm:sparse_max_E_thres}:
\begin{itemize}
\item For any $h < h^\star$, the first moment entropy density $w(E,h)$ is negative for $E>E_{max}(h)$ or $E<E_{min}(h)$. The first-moment method as in Section \ref{sec:first_mom_fin}, then implies the absence of bisections with energies outside the range $(E_{min}(h),E_{max}(h))$ and $o(n)$ violations of $h$-stability.

\item By the definition of $h_{\rm cor}$, for all $h_{\rm cor} \leq h \leq  h^\star$, the second-moment entropy-density is maximized at $\omega=0$ throughout the range $(E_{min}(h),E_{max}(h))$. Therefore, by applying the argument in Section \ref{sec:second_mom}, we obtain that bisections with energies $\approx E$ and $o(n)$ violations of $h$-stability exist throughout the range $E \in (E_{min}(h),E_{max}(h))$.
 \item At $h<h_{\rm cor}$
there exists  $E_{\rm cor}(h) \in \R$ with $E_{\rm min}(h) < E_{\rm cor}(h)< E_{\rm max}(h)$ such that $W(E,\omega,h)$  is maximized at $\omega=0$ for  $E \in (E_{\rm cor}(h),E_{max}(h))$. The second moment method then implies the existence of bisections with energies $\approx E$ and $o(n)$ violations of $h$-stability in the range $(E_{\rm cor}(h),E_{max}(h))$.

\item For some $E<E_{\rm cor}(h)$, $W(E,\omega,h)$ is maximized at some non-zero overlap $\omega \neq 0$, making our results inconclusive towards existence of $h$-stable bisections with energies in the interval $(E_{min}(h),E_{\rm cor}(h))$
\end{itemize}
\begin{figure}
    \centering
    \includesvg[width=0.8\textwidth]{Evaluethres_updated.svg}
    \caption{The three values of energies $E_{max},E_{\rm cor},E_{min}$ as a function of the threshold $h$. $E_{max},E_{min}$ are obtained as the roots of the (\ref{eq:fin_bound_r1_E1}). $E_{\rm cor}$ is obtained numerically as the smallest $E$ such that $W(E, \omega,h)$ defined through Proposition \ref{prop:second_mom_main} is not locally maximized at $\omega = 0$. It corresponds to the definition of $E_{\rm cor}$ (\ref{def:E_cor}) only for $h <h_{\rm cor}$} \label{fig:max_energy}
\end{figure}

These points are illustrated in Figure \ref{fig:max_energy} for $h \in (-0.1,h^\star)$. While our results do not require or establish that $E_{\rm cor}(h) > E_{min}(h)$ for all values of $h$ in the range $(-0.1,h_{\rm cor})$, we numerically find this to be the case.

\subsection{Proof of Theorem \ref{thm:sparse_max_E_thres}}

% The existence of $h_{\rm cor}$ established in Section \ref{sec:exist_ecor} along with Assumption
% \ref{ass:unique_max} allows us to generalize the above arguments to any $h < h^\star$.

In this section, we establish Theorem \ref{thm:sparse_max_E_thres} through Propositions \ref{prop:neg_first_mom}-\ref{prop:second_mom_sup}.

\subsection{Maximal energies of $h$-stable configurations}

To establish the absence of nearly $h$-stable bisections with energy greater than $E_{\rm max}$, we introduce the following random variable:
\begin{equation}
   X_{\geq E}(h,r)=\sum_{z: zn \in \mathbb{N}, zn\leq dn-\frac{E}{2}n\sqrt{d}}X(z,h,r).
\end{equation}
The above term equals the number of configurations having at least $rn$ $h$-stable vertices and normalized energy greater than or equal to $E$. 

Similarly, we define the corresponding first moment entropy density restricted to energy values at least $E$:
\begin{equation}
\begin{split}
   \frac{1}{n} \log \mathbb{E}[X_{\geq E}(h,r)] \, .\\
\end{split}
\end{equation}

By setting $a=E,b=\infty$ in  Proposition \ref{prop:first_moment_sup}, the above first-moment entropy density can be asymptotically bounded by maximizing $w(\tilde{E},h,r)$ over the range $(\tilde{E} > E)$.
Define:
\begin{equation}
    w_{\geq E}(h,r) \coloneqq \sup_{\tilde{E} \geq E} w(\tilde{E},h,r).
\end{equation}

Then, Proposition \ref{prop:first_moment_sup} implies the following result:

\begin{proposition}\label{prop:first_mom_max}
    For any $h \in \mathbb{R}$:
    \begin{equation}
        \limsup_{d \rightarrow \infty}\limsup_{n \rightarrow \infty}\frac{1}{n} \log \mathbb{E}[X_{\geq E}(h,r)] \leq  w_{\geq E}(h,r),
    \end{equation}
    with the inequality being tight at $r=1$, in which case:
    \begin{equation}
        \lim_{d \rightarrow \infty}\lim_{n \rightarrow \infty}\frac{1}{n} \log \mathbb{E}[X_{\geq E}(h,1)] =  w_{\geq E}(h,1)
    \end{equation}
\end{proposition}

With the above asymptotic bound, a continuity argument similar to Section \ref{sec:cont_neg} results in a vanishing first moment even upon allowing for $\Theta(n)$ violations:
\begin{proposition}
\label{prop:first_mom_e}
Suppose that  $h \in (-0.1,h^\star)$, and $E > E_{\rm max}(h)$. Then $\exists r^*(E,h)$ such that for all $r> r^*(E,h)$, for sufficiently large $d$, we have:
    \begin{equation}
     w_{\geq E}(h,r) < 0.
\end{equation}
\end{proposition}
\begin{proof}
The strict concavity of $w(E,h)$ w.r.t $E$ (Proposition \ref{prop:first_mom_conv_r}) implies that $w(E,h,r)$ is strictly decreasing in $E$ for $E > E^\star(h,r)$. Subsequently, the continuity of $ E^\star(h,r)$ w.r.t $r$ (Corollary \ref{cor:Econthr}) implies that 
$\exists \tilde{r} < 1$ such that $E^\star(h,r)<E_{\rm max}(h)$  and $r \in (\tilde{r},1]$. Therefore, for any $E > E_{\rm max}(h)$, and  $r \in (\tilde{r},1]$, we have:
\begin{equation}
    w_{\geq E}(h,r)= w(E,h,r).
\end{equation}
Since $w(E,h)<0$ for $E>E_{\max}(h)$ and $w(E,h,r)$ is continuous in $r$ (Proposition \ref{prop:first_mom_conv_r}), we obtain that $w(E,h,r)<0$ for any $E>E_{\max}(h)$ and large enough $r < 1$.
\end{proof}

The above result combined with Proposition \ref{prop:first_mom_max} then implies:
\[
\limsup_{d \rightarrow \infty}\limsup_{n \rightarrow \infty}\frac{1}{n} \log \mathbb{E}[X_{\geq E}(h,r)] < 0.
\]

Hence applying Markov's inequality then proves Theorem \ref{thm:sparse_max_E_thres} for $E>E_{\rm max}(h)$. An identical argument but with $w_{\geq E}$ replaced by $w_{\leq E}(h,r) \coloneqq \sup_{\tilde{E} \leq E} w(\tilde{E},h,r)$ then yields Theorem \ref{thm:sparse_max_E_thres} for $E<E_{\rm min}(h)$.

To obtain the existence of approximately $h$-stable bisections for energies in the range $(E_{\rm min}(h),E_{\rm max}(h))$, for $h \in (h_{\rm cor}, h^\star)$ and in the range $(E_{\rm cor}(h),E_{\rm max}(h))$ for $h \in (-0.1, h^\star)$, we proceed with obtaining the second moment. Note that similar to Section \ref{sec:high_prob}, to prove existence, it suffices to consider a fixed value of cut-size.

\begin{lemma}\label{lem:paley_second_mom}
    Suppose that either $h \in (-0.1, h_{\rm cor})$, and  $E \in (E_{\rm min},E_{\rm max}(h))$ or  $h \in (h_{\rm cor}, h_{\rm max})$, and  $E \in (E_{\rm cor}(h),E_{\rm max}(h))$. Then, with $z=\frac{d}{4}-\frac{E}{2}\sqrt{d}$:
    \begin{equation}
        \Pr[X(z,h) \geq 1] \geq \exp(-o_d(1)n)\, .
    \end{equation}
\end{lemma}
\begin{proof}
    By the definition of $h_{\rm cor}, E_{\rm cor}(h)$ we have that for all $h \in (h_{\rm cor},h^\star)$,   and $E \in (E_{\rm min}(h),E_{\rm max}(h))$ or for $h < h_{\rm cor}$ and $E \in (E_{\rm cor}(h),E_{\rm max}(h))$, $W(E,\omega,h)$ maximized at $\omega=0$.     
    In either case, we have by Proposition \ref{prop:second_mom_max} $W(E,0,h)=2w(E,h)$. Proposition \ref{prop:first_moment_sup}  and Proposition \ref{prop:second_mom_sup} then imply:    \begin{equation}\label{eq:paley_2ndmom_{const}}
         \frac{1}{n} \log \E[X^2(z,h)] =  \frac{1}{n} 2\log \E[X(z,h)]+o_d(1)n+o(n).
    \end{equation}
Exponentiating both sides and applying the Paley-Zygmund inequality completes the proof.
\end{proof}

As in Section \ref{sec:high_prob}, the subsequent step
is to boost the above probability through the use of concentration of a suitable random variable. To this end, we introduce the maximal energy deficit function:
\begin{definition}\label{def:const_deficit}
For each $h$, $\sigma\in B_N=\{\pm 1\}^N$, and energy $E$, define the maximal energy deficit function $D_{ E_1,E_2}$ as follows:
$$D_{E_1,E_2}(W,h,\sigma)=\sum_i \left(h-{1\over\sqrt{d}}\sum_j w_{ij}\sigma_i\sigma_j\right)^++(H(\sigma)-nE_1)^+ +(nE_2-H(\sigma))^+. $$
\end{definition}
Similarly, we define $D^\star_{E_1,E_2}(W,h)$ as $\min_{\sigma \in M_0}D^*_{E_1,E_2}(W,h,\sigma)$. We note that $D^\star_{E_1,E_2}(W,h)= o(n)$ implies that there exists a bisection satisfying $\sum_{i=1}^n \left(h-{1\over\sqrt{d}}\sum_j w_{ij}\sigma_i\sigma_j\right)^+ = o(n)$ and $(H(\sigma)-nE_1)^+=o(n)$, and $(nE_2-H(\sigma))^+=o(n)$.

Lemma \ref{lem:paley_second_mom} then implies that for all $E \in (E_{\rm cor},E_{\rm max})$ for $h \in (-0.1, h^\star)$, and for all $E \in (E_{\rm min},E_{\rm max})$ for $h \in (h_{\rm cor}, h^\star)$, we have:
\begin{align}\label{eq:sub-exponential_max_e}
\pr\left(D^\star_{E_1,E_2}(W,h)=0\right)\ge \exp(-o_d(1)n-o(n)).
\end{align}

As in Lemma \ref{lem:concent}, using McDiarmid's inequality, we have the following concentration result:
\begin{lemma}\label{lem:concent_const}
Let $W$ be a weight matrix of a graph sampled from the anti-ferromagnetic configuration model with average degree $d$. Then, for any $E_1 < E_2,h \in \mathbb{R}$, the minimum maximal energy deficit $D_{E_1, E_2}^*(W,h)$, satisfies:
\begin{align*}
\pr\left(|D_{E_1,E_2}^*(W,h)-\Ea{D_{E_1,E_2}^*(W,h)}| \ge \epsilon n \right) \le 2\exp\left(-\epsilon^2n\right).
\end{align*}
\end{lemma}

Next, analogous to Corollaries \ref{cor:def_1} and \ref{corr:def_2}, we obtain:
\begin{corollary}\label{cor:def_const_1}
Let $W$ be a weight matrix as in Lemma \ref{lem:concent}, then for any $h \in  (-0.1,h^\star),E=E_{\rm max}(h)$:
\begin{equation}
    \Ea{D_{E_1,E_2}^*(W,h)}=o_d(1)n+o(n)\, .
\end{equation}
\end{corollary}
\begin{proof}
Suppose that, on the contrary, there exist $\epsilon$ and arbitrarily large $d(\epsilon), n(\epsilon)$  such that at $d=d(\epsilon), n=n(\epsilon)$, we have:
\begin{equation}
\Ea{D_{E_1,E_2}^*(W,h)}\ge \epsilon n.    
\end{equation}
 Lemma \ref{lem:concent_const} then implies that
 $\mathbb{P}[D_{E_1,E_2}^*(W,h)=0]$ is at most $\exp(-\epsilon^2n)$, contradicting (\ref{eq:sub-exponential_max_e}). 
\end{proof}
The above Corollary and Lemma \ref{lem:concent_const} then imply the following result:

\begin{corollary}\label{corr:def_const}
Let $W$ be a weight matrix as in Lemma \ref{lem:concent}, then for any $h \in  (h_{\rm cor},h^\star)$ and $(E_1,E_2) \subset (E_{\rm min}(h),E_{\rm max}(h))$,
 $D^\star_{E_1,E_2}(W,h) = o_d(1)n+o(n)$ with high probability as $n \rightarrow \infty$. Similarly, for any $h \in  (-0.1,h_{\rm cor})$ and $(E_1,E_2) \subset (E_{\rm cor}(h),E_{\rm max}(h))$,
 $D^\star_{E_1,E_2}(W,h) = o_d(1)n+o(n)$ with high probability as $n \rightarrow \infty$
\end{corollary}

Now, let $N_{E_1,E_2}^*(W,h)$ denote the maximum number of vertices satisfying $h$-stability amongst bisections with energy in $(E_1,E_2)$.
Using a similar technique as Lemma \ref{lem:high_prob_exist}, but with the perturbation applied to both $E$ and $h$, we will translate the bound on the deficit $D^\star_{E_1,E_2}$ in Corollary \ref{corr:def_const}, to show that $N^\star_{E_1,E_2}(h) = n-o(n)$ in the feasible range of $E,h$. To this end, we require the following Lemma, whose proof can be found in Appendix \ref{sec:exist_ecor}:

\begin{lemma}\label{lem:contecor}
$E_{\rm min}(h),E_{\rm max}(h)$ are continuous for all $h \in (-\infty,h^\star)$ and $E_{\rm cor}(h)$ is continuous for all $h \in (-0.1,h_{\rm cor})$.
\end{lemma}

We're now equipped to bound the desired high-probability lower-bound for $N^\star_{E_1,E_2}(h)$.

\begin{lemma}\label{lem:high_prob_max}
Suppose that $h \in (h_{\rm cor},h^\star)$. Then, for every $(E_1,E_2) \subset (E_{\rm min}(h),E_{\rm max}(h))$ and $\epsilon>0$,  $\pr(N_{E_1,E_2}^*(W,h)\le n-\epsilon n)\le \exp(-\Theta_d(1)n+o(n))$. Similarly, for any $h \in (-0.1,h_{\rm cor})$, and $(E_1,E_2) \subset (E_{\rm cor}(h),E_{\rm max}(h))$ and $\epsilon>0$,  $\pr(N_{E_1,E_2}^*(W,h)\le n-\epsilon n)\le \exp(-\Theta_d(1)n+o(n))$.
\end{lemma}
\begin{proof}
Let $\mathcal{E}_n=\mathcal{E}_n(h,E_1,E_2)$ be the event $N^*_{E_1,E_2}(W,h)\leq n-\epsilon n$. Fix $W$ such that this event holds.
Fix any $\sigma$ and let $S(\sigma)$ be the set of $i\in [n]$
such that 
\begin{align*}
h-{1\over \sqrt{d}}\sum_j w_{ij}\sigma_i\sigma_j\ge 0.
\end{align*}
Suppose $\mathcal{E}_n$ occurs. Then, $|S(\sigma)|\ge \epsilon n$.
Lemma \ref{lem:contecor} implies that when $h \in (h_{\rm cor},h^\star)$,  for arbitrarily small $\delta>0$, 
$\exists \tilde h \in (h,h^\star)$ such that 
$E_{\rm min}(h) < E_{\rm min}(\tilde{h}) < E_{\rm min}(h)+\delta$ and $E_{\rm max}(h)-\delta < E_{\rm max}(\tilde{h}) < E_{\rm max}(h)$. As a consequence, $\exists \tilde h \in (h,h^\star)$ and $\tilde E_1, \tilde E_2$ such that $(\tilde E_1, \tilde E_2) \subset (E_1, E_2)$ and $(\tilde E_1, \tilde E_2) \subset (E_{\rm min}(\tilde{h}),E_{\rm max}(\tilde{h}))$.  Similarly, when $h \in (-0.1, h_{\rm cor})$, $\exists \tilde h \in (-0.1,h_{\rm cor})$ and 
 $\tilde E_1, \tilde E_2$ such that $(\tilde E_1, \tilde E_2) \subset (E_1, E_2)$ and $(\tilde E_1, \tilde E_2) \subset (E_{\rm cor}(\tilde{h}),E_{\rm max}(\tilde{h}))$. 
In either case, Corollary \ref{corr:def_2} implies that $D_{\tilde E_1, \tilde E_2}(W,\tilde h,\sigma) = o(n)$ w.h.p as $n \rightarrow \infty$. 

Then, for any $\sigma$, either:
\begin{enumerate}
    \item $E(\sigma)<E_1$. Therefore $\tilde E_1n-E(\sigma)n \geq (\tilde E_1-E_1)n$.
    \item $E(\sigma)>E_2$. Therefore $E(\sigma)n-\tilde E_2n\geq (E_2-\tilde E_2)n$.
    \item $N(W,h,\sigma) \geq n-\epsilon n$.
\end{enumerate}
In either case, we obtain:
\begin{align*}
D_{\tilde E_1, \tilde E_2}(W,\tilde h,\sigma)&=\sum_i \left(\tilde h-{1\over \sqrt{d}}\sum_j w_{ij}\sigma_i\sigma_j\right)^+ +(n\tilde E_2-H(\sigma))^+ + +(H(\sigma)-n \tilde E_1)^+
\\&= 
\sum_i \left(\tilde h-h+h-{1\over \sqrt{d}}\sum_j w_{ij}\sigma_i\sigma_j\right)^+ +(n\tilde E-nE+nE-H(\sigma))^+\\
&\ge \operatorname{\rm min}((\tilde h-h)\epsilon n,(\tilde E_1-E_1),(E_2-\tilde E_2)n).
\end{align*}
Pick $\epsilon >0$ such that $\operatorname{\rm min}((\tilde h-h)\epsilon n,(\tilde E_1-E_1),(E_2-\tilde E_2)n) \geq \epsilon 'n$.
Since $\sigma$ was arbitrary, we conclude $D_{\tilde E_1, \tilde E_2}^*(W,\tilde h)\geq \epsilon'n$ which occurs with probability at most  $\exp{(-Cn)}$ for some constant $C$ and large enough $d,n$ due to Corollary \ref{corr:def_2}.
Thus the event $\mathcal{E}_n$ occurs with probability at most $\exp(-\Theta_d(1)(n)-o(n))$.
\end{proof}

The above result proves the high-probability existence part of Theorem \ref{thm:sparse_max_E_thres}.

% Now we prove concentration. $F^*(X,h)$ is a function of the $n^2$ i.i.d. standard normals. Consider $F^*(x,h)$ as a function of the vector $x\in \R^{n^2}$,
% by Gaussian concentration inequality we have
% \begin{align*}
% \pr\left( | F^*(X,h) - \E{F^*(X,h)}|\ge t\right)\le \exp\left(-{t^2\over 4 \sup_x\|\nabla F^*(x,h)\|_2^2}\right).
% \end{align*}
% We have
% \begin{align*}
% \|\nabla F^*(x,h)\|_2^2=\sum_{1\le i,j\le n} \left({\partial F^*(x,h) \over \partial x_{ij}}\right)^2
% \end{align*}
% We have that this partial derivative is at most $2/\sqrt{n}$ and therefore $\|\nabla F^*(x,h)\|_2^2\le n^2(4/n)=n/4$.
% So we obtain a bound $\exp(-t^2/(16 n))$. 

\section{Universality for Sparse graphs: Proof of Theorem \ref{thm:uni_sparse}}\label{sec:sparse_uni}

In this section, we establish the universality of the phase transition in Definition \ref{def:thres_def} and the associated threshold $h^\star$ for families of random variables satisfying the assumptions of Theorem \ref{thm:uni_sparse}. In particular, this includes sparse Erdős-Rényi graphs having ferromagnetic, anti-ferromagnetic or spin glass interactions.
Since we do not possess analytic expressions or proofs of the existence of a threshold for arbitrary distributions over graphs,
directly relating the thresholds for different distributions seems challenging. 
Instead, we prove that for a given $h < h^\star$, a partition with at-most $o_d(1)n$ vertices violating $h$-stability exists with high probability for all distributions. Analogously, for $h>h^\star$, we prove that every partition has $\Theta_d(1)n$ vertices violating $h$-stability. 

Our proof of universality is based on the Lindeberg's method.
The Lindeberg's method is a technique derived from Lindeberg's proof of the  Central Limit Theorem \citep{lindeberg1922neue}, which was extended to a general invariance principle for functions of weakly-dependent random variables in \citep{chatterjee2006generalization}.

% \yd{Mention similarities to the }

We first recall the central idea: Suppose we wish to prove that $\Ea{f(a_1,a_2,\cdots,a_n)}$ approximates $\Ea{f(b_1,b_2,\cdots,b_n)}$ for a thrice differentiable function $f$, and two sequences of independent variables, $(a_1,a_2,\cdots,a_n)$ and $(b_1,b_2,\cdots,b_n)$ with identical means and variance.
Lindeberg's method involves iteratively swapping $a_i$ to $b_i$ and utilizing Taylor's theorem and the matching of the first two moments of $a_i$ and $b_i$.

Therefore, to apply the Lindeberg's method, the first step is to identify a sufficiently smooth function to interpolate. For us, this role will be played by the Deficit function in Definition \ref{def:deficit}.

\subsection{Universality for near $h$-stable configurations ($h<h^\star$)}

For a given choice of weights $W$ and a configuration $\sigma$, following Definition \ref{def:deficit}, we denote by $D(W,h,\sigma)$, the total deficit of the $h$-stability vertices satisfying $h$-stability.
We have:
\begin{align*}
D(W,h,\sigma)=\sum_{1\le i\le n}\left(h-{1\over \sqrt{d}}\sum_j w_{ij}\sigma_i\sigma_j\right)^+.
\end{align*}
We define $D^*(W,h) = \min_{\sigma \in M_0} D(W,h,\sigma)$.  We further denote by $N(W,h,\sigma)$ the  number of $h$-stable vertices in $\sigma$ and $N^*(W,h)=\max_{\sigma \in M_0} N(W,h,\sigma)$ the maximum number of $h$-stable vertices in any partition.

Analogous to Lemma \ref{lem:high_prob_exist} , through a perturbation argument, we next show that a small value of $D^*(W,h)$ translates to the existence of bisections with $n(1-o_d(1))$ $h$-stable vertices.

\begin{lemma}\label{lem:def_success}
Suppose $\tilde h \in \R$ satisfies $D^*(W,\tilde h) = o_d(1)n + o(n)$ with high probability as $n \rightarrow \infty$. Then for any $h<\tilde h$, we have $N^*(W,h) =  n(1-o_d(1))$ with high probability as $n \rightarrow \infty$.
\end{lemma}
\begin{proof}
We utilize the argument used in the proof of Section \ref{sec:high_prob}. Assume that $N^*(W,h,\sigma) \leq (1-\epsilon)n$ for some $\epsilon > 0$. Then for any configuration, $\sigma$, at least $\epsilon n$ vertices violate $ h$-stability. Since $\tilde h>h$, each such vertex contributes $\tilde h-h$ to the total $\tilde h$-deficit for $\sigma$ i.e $D(W,\tilde h,\sigma) \geq \epsilon(h-h^\star)n$. Since $\sigma$ is arbitrary, we have $D^*(W,\tilde h,\sigma) \geq \epsilon(\tilde h - h)n$.  Therefore, $N^*(W,h,\sigma) \leq (1-\epsilon)n \implies D^*(W,h) \geq \epsilon(\tilde h -h)n$. By assumption, we have $\lim_{n \rightarrow \infty}\pr[D^*(W,h) \geq \epsilon(\tilde h-h)n] = 0$. Thus, we obtain that for any $\epsilon > 0$, $\lim_{n \rightarrow \infty}\pr[N^*(W,h,\sigma) \leq (1-\epsilon)n] = 0$. 
\end{proof}

We now show that the value of the threshold exhibits a universality property.
Recall that Theorem \ref{thm:sparse_max_E_thres} established that $h^\star$ defined as the unique root of \eqref{def:w_firstmom} is a maximal stability threshold when $W_{d,n} \in \mathbb{R}^{n \times n}$ are set to be the negative of the adjacency matrix for $\mathbb{G}(n,d/n)$.

Using a modification of Lindeberg's argument, we obtain the following result:

\begin{proposition}\label{prop:sparse_deficit_success}
Let $W_{d,n}$ be a family of random weighted graphs satisfying the assumptions in Theorem \ref{thm:uni_sparse}.
Let $h^\star$ be the threshold defined in Theorem \ref{thm:uni_sparse}.
Then, for every $\epsilon$ and $h<h^\star$, there exists a degree $d(\epsilon,h)$ such that for all $d\geq d(\epsilon,h)$, w.h.p as $n \rightarrow \infty$, we have:
\begin{equation}
    D^*(W_{d,n},h) \leq \epsilon n.
\end{equation}
\end{proposition}
Equivalently, $D^*(W_{d,n},h) = o_d(1)n+o(n)$.
Now,  for any $h < h^\star$, pick $h<\tilde{h}<h^\star$. Using the above proposition, we have $D^*(W_{d,n},\tilde h)= o_d(1)n+o(n)$ with high probability. Therefore, using Lemma \ref{lem:def_success}, we further obtain the following Corollary:
\begin{corollary}\label{cor:sparse_deficit_success}
Let $W$ be as in Proposition \ref{prop:sparse_deficit_success}. Then for any $h < h^\star$, with high probability,
\begin{equation}
N^*(W,h) \geq  n(1-o_d(1))-o(n).
\end{equation}
Where the above equation denotes convergence in probability under the sequentiallimit $n \rightarrow \infty$ and $d \rightarrow \infty$.
\end{corollary}
In particular, the above result covers the cases of sparse graphs with ferromagnetic or  anti-ferromagnetic interactions, or equivalently the case of friendly or unfriendly partitions in sparse random graphs.

\begin{proof}
Let $W$ denote an arbitrary random weight matrix with i.i.d entries $w_{ij}$ with mean $\mu$. Let $W'$ denote the matrix with entries $W'_{ij}=w_{ij}-\mu$. We observe that for any $\sigma \in M_0$, $N(W,h,\sigma)=N(W',h,\sigma)$.
Therefore, without loss of generality, we restrict to weight distributions having $0$ mean.

We prove Proposition \ref{prop:sparse_deficit_success} by establishing the universality of $D^*(W,h)$ for families of distributions satisfying the given assumptions as $d\rightarrow \infty$ and $n\rightarrow \infty$.
This is expressed through the following result:
\begin{lemma}\label{lem:sparse_uni}
Let $A, B$ be arbitrary random weight matrices with i.i.d random entries $a_{ij},b_{ij}$ satisfying the assumptions in Theorem \ref{thm:uni_sparse} for parameters $d,n$ with means $0$. Then:
\begin{equation}
    \abs{D^*(A,h)-D^*(B,    h)} = o_d(1)n+o(n),
\end{equation}
 with high probability as $n \rightarrow \infty$.
\end{lemma}

We first explain how the above result implies proposition \ref{prop:sparse_deficit_success}. Let  $W^{(d,n)}$ be an arbitrary family of weighted random graphs on $n$-nodes satisfying the assumptions in Theorem \ref{thm:uni_sparse}. Note that sparse Erdős-Rényi graphs with anti-ferromagnetic interactions sampled from $\G(n,p=d/n)$ correspond to one such family of random graphs. 
Let $W=W^{(d,n)}$ and $W'$ be a weighted-graph sampled from $\G'(n,p=d/n)$ with anti-ferromagnetic interactions.
Corollary \ref{corr:def_2} implies that for any $h<h^\star$,  $D^*(W',h) = o_d(1)n+o(n)$ with high probability as $n \rightarrow \infty$. Therefore, applying Lemma \ref{lem:sparse_uni} yields $D^*(W,h) = o_d(1)n+o(n)$.
\end{proof}

\subsubsection{Proof of Lemma \ref{lem:sparse_uni}}

We now apply the Lindeberg's method to the function $D^*(W,h)$. However, $D^*(W,h)$ is not-differentiable and involves a minimization over the set of configurations $\sigma$. We, therefore, introduce a series of smooth approximations to $D^*(W,h)$. A convenient approximation technique leading to simplified derivatives is through the introduction of a Hamiltonian \citep{sen2018optimization}.

Fix any function $g:\R\to \R$. Let 
\begin{equation}\label{eq:ham}
\cH_d(W,g,\sigma)=\sum_{1\le i\le n}g\left(h-{1\over \sqrt{d}}\sum_j w_{ij}\sigma_i\sigma_j\right),
\end{equation} for $w=(w_{ij}, 1\le i,j\le n), \sigma\in B_n$. 
%Let
% \begin{align*}
% h^\star(W,g)=\min_\sigma H_d(W,g,\sigma).
% \end{align*}
% The associated energy with inverse temperature $\beta$ is defined as
% \begin{align*}
% F(W,g)={1\over \beta}\log Z(W,g),
% \end{align*}
% where $Z(W,g)$ is the partition function
% \begin{align*}
% \sum_\sigma  \exp\left(-\beta H_d(W,g,\sigma)\right).
% \end{align*}
% We have the following standard bounds:
% \begin{align*}
% -h^\star(W,g) \le F(W,g) \le -h^\star(W,g)+{\log 2\over \beta} n.
% \end{align*}

We now define the ground state and energy conditioned on the magnetization being $0$, and the associated partition function:
\begin{align*}
\cH_d^*(W,g)&=\min_{\sigma \in M_0} \cH_d(W,g,\sigma),\\
Z(W,g)&=\sum_{\sigma \in M_0}  \exp\left(-\rho \cH_d(W,g,\sigma)\right),\\
F(W,g)&={1\over \rho}\log Z(W,g)\, ,
\end{align*}
where we introduced a parameter $\rho \in \R^+$, commonly referred to as the inverse temperature.
We have:
\begin{equation}\label{eq:approx_3}
-\cH_{d}^*(W,g) \le F(W,g) \le -\cH_{d}^*(W,g)+{\log 2\over \rho} n.
\end{equation}
For any observable  $f:B_n\to \R$ we denote by $\langle f\rangle$ the associated Gibbs average
\begin{align*}
\langle f\rangle =\sum_{\sigma \in M_0} f(\sigma){\exp(-\rho \cH_d(W,g,\sigma)) \over Z}.
\end{align*}

Fix now $\epsilon_1>0$ and let $g_{\epsilon_1}$ be an infinitely differentiable function $g_{\epsilon_1}:\R \to \R$ with uniformly bounded first, second, third derivatives such that $g_{\epsilon_1}$ that is uniformly $\epsilon_1$ close to $()^+$, i.e:
\begin{align}
\sup_{t\in \R} |g_{\epsilon_1}(t)-(t)^+| \le \epsilon_1. \label{eq:g-vs-plus}
\end{align}
Therefore, $g$ acts a smooth approximation of threshold function $(\cdot)^+$. For example, we may choose $g$ to be the following soft-plus function:
\begin{equation}\label{eq:smooth_approx}
    g_{\epsilon_1}(t) = \gamma\ln{(1+\exp{(t/\gamma)})},
\end{equation}
where $\gamma =\frac{\epsilon_1}{\ln{2}}$. Through a straightforward calculation, one may check that $g_{\epsilon_1}(0)=\epsilon_1$ and $g_{\epsilon_1}(t)-(t)_+$ is maximized at $t=0$. Furthermore, $g_{\epsilon_1}'(t), g_{\epsilon_1}''(t), g_{\epsilon_1}'''(t)$ are uniformly bounded by constants depending on $\epsilon_1$.
As per our previous discussion, our goal is to show that $\abs{\Ea{\cH_d^*(A,(\cdot)^+)}-\Ea{\cH_d^*(B,(\cdot)^+)}}$ is sufficiently small when $h<h^\star$. We have
\begin{equation}\label{eq:approx_1}
|\cH_d(W,(\cdot)^+,\sigma)-\cH_d(W,g_{\epsilon_1},\sigma)| \le \epsilon_1 n.
\end{equation}
for all $\sigma$ and thus we will focus on the case $g=g_{\epsilon_1}$. 
(\ref{eq:g-vs-plus})  along with the definition of $\cH_d^*$ further imply that:
\begin{align*}
    \cH_d^*(W,(\cdot)^+) &\leq \cH_d^*(W,g_{\epsilon_1}) +  \epsilon_1 n,\\
    \cH_d^*(W,g_{\epsilon_1})  &\leq \cH_d^*(W,(\cdot)^+) +  \epsilon_1 n\, .\\
\end{align*}
Therefore:
\begin{equation}\label{eq:approx_2}
\abs{\cH_d^*(W,(\cdot)^+)-\cH_d^*(W,g_{\epsilon_1})}\leq \epsilon_1 n\, . \end{equation}

With the above definitions, we now apply the Lindeberg's argument to the two random weighted graphs $A$ and $B$. Recall that by assumption, $A=A^T$ and $B=B^T$. Both $A$ and $B$ therefore can be represented as sets of $m=n(n+1)/2$ i.i.d entries $a_{ij}$ and $b_{ij}$ respectively with $i\leq j$.
We fix an order on the entries in the set $\{(i,j):i\leq j\}$ arbitrarily by iterating from $i=1,\cdots,n$ and $j=1,\cdots,i$
and switch each element from $A$ to $B$ in this order. 
Let $(k,r)$ be the $t_{th}$ element of the sequence.
Define the matrix $J^{(t)}$ with entries $j^{(t)}_{ij}=A_{ij}$ when $(i,j)\leq (k,r)$,
$j^{(t)}_{ij}=B_{ij}$ when $(i,j)>(k,r)$. Let $J^{(t)}_{/{kr}}=J^{(t-1)}_{/{kr}}$  denote the matrix with entries at positions $(k,r)$ and $(r,k)$ set to $0$ and all the remaining entries being identical to $J^{(t)}$.
Note that $J_0 = A$ and $J_m=B$.
Lindeberg's argument relies on the following telescoping sum:
\begin{equation}\label{eq:teles}
   \E[F(B,g_{\epsilon_1})]-\E[A,g_{\epsilon_1})]=\sum_{t=0}^m \E[F(J^{(t)},g_{\epsilon_1})]-\E[F(J^{(t-1)},g_{\epsilon_1})]\, .
\end{equation}

We denote by $F^{(l)}$ the partial $l$-th derivative of $F$ with respect to $j_{kr}$. Then, using the third order Taylor expansion over the variable $j^{(t)}_{kr}$ at a fixed value of the matrix $J^{(t)}_{/{kr}}$, along with the independence of the edge weights,  we obtain.
\begin{small}
\begin{equation}\label{eq:taylor_s}
\begin{split}
\abs{\E[F(J^{(t)},g_{\epsilon_1})|J^{(t)}_{/kr}]-F(J^{(t)}_{/{kr}},g_{\epsilon_1})-F^{(1)}(J^{(t)}_{/{kr}},g_{\epsilon_1}) \E[b_{kr}]-{1\over 2}F^{(2)}(J^{(t)}_{/{kr}},g_{\epsilon_1}) \E[(b_{kr})^2]} \leq {1\over 3!} \norm{F^{(3)}}_\infty\E[\abs{b_{kr}}^3],\\
\abs{\E[F(J^{(t-1)},g_{\epsilon_1})|J^{(t)}_{/{kr}}]-F(J^{(t)}_{/{kr}},g_{\epsilon_1})-F^{(1)}(J^{(t)}_{/{kr}},g_{\epsilon_1}) \E[a_{kr}]-{1\over 2}F^{(2)}(J^{(t)}_{/{kr}},g_{\epsilon_1})\E[(a_{kr})^2]}\leq {1\over 3!} \norm{F^{(3)}}_\infty\E[\abs{a_{kr}^3}],\\
\end{split}
\end{equation}
\end{small}
Recall that $a_{kr}$ and $b_{kr}$ have the same first and second moment. Therefore, an application of the triangle inequality yields:
\begin{equation}\label{eq:lind_bound}
\abs{\E[F(J^{(t)},g_{\epsilon_1})|J^{(t)}_{/kr}]-\E[F(J^{(t-1)},g_{\epsilon_1})|J^{(t)}_{/{kr}}]} \leq {1\over 3!} \norm{F^{(3)}}_\infty\E[\abs{b_{kr}}^3]+ {1\over 3!} \norm{F^{(3)}}_\infty\E[\abs{a_{kr}^3}]\, .\\
\end{equation}
The expected cost of a swap from $a_{kr}$ to $b_{kr}$ is therefore bounded by the 3-d derivative and the absolute centered $3_{rd}$ moment of $b_{ij}$ and $a_{ij}$.
By considering the expectation of $\E[F(J^{(t)},g_{\epsilon_1})|J^{(t)}_{/kr}]-\E[F(J^{(t-1)},g_{\epsilon_1})|J^{(t)}_{/{kr}}]$ over $J^{(t)}_{/{kr}}$, we obtain:
\begin{equation}
   \abs{\E[F(J^{(t)},g_{\epsilon_1})]-\E[F(J^{(t-1)},g_{\epsilon_1})]} \leq {1\over 3!} \norm{F^{(3)}}_\infty\E[\abs{b_{kr}}^3]+ {1\over 3!} \norm{F^{(3)}}_\infty\E[\abs{a_{kr}^3}]. 
\end{equation}
Substituting in (\ref{eq:teles}) yields:
\begin{equation}\label{eq:final_lind_bound}
   \abs{\E[F(J^{(t)},g_{\epsilon_1})]-\E[F(J^{(t-1)},g_{\epsilon_1})]} \leq \frac{n(n+1)}{2}\rbr*{{1\over 3!} \norm{F^{(3)}}_\infty\E[\abs{b_{kr}}^3]+ {1\over 3!} \norm{F^{(3)}}_\infty\E[\abs{a_{kr}^3}]}\, .
\end{equation}

% I claim that the expected value of the 3-d derivative term is order $1/n^{3\over 2}$. Assuming this is the case the entire switch from $x$ to $y$
% has expected cost order $\sqrt{n}$ since there are order $n^2$ variables. Since the ground states $h^\star$ are order $n$, this is good enough. 

To bound  $\norm{F^{(3)}}_\infty$, we now compute the derivatives of $F$ for an arbitrary symmetric weight matrix $J$ with entries $j_{ij}$ for $i,j \in [n]$. We recall that the derivatives in (\ref{eq:taylor_s}) are w.r.t the variable $j_{kr}=j_{rk}$.
Note that
\begin{align*}
{d\over d j_{kr}}\sum_{l} g_{\epsilon_1}(h-{1\over \sqrt{d}}\sum_m j_{lm} \sigma_l\sigma_m)
=-\dot g_{\epsilon_1}(h-{1\over \sqrt{d}}\sum_{m}j_{km}\sigma_k\sigma_m){1\over \sqrt{d}}\sigma_k\sigma_r-\dot g_{\epsilon_1}(h-{1\over \sqrt{d}}\sum_{m}j_{rm}\sigma_r\sigma_l){1\over \sqrt{d}}\sigma_k\sigma_r.
\end{align*}
We thus have:
\begin{align*}
F^{(1)}(J,g_{\epsilon_1}) &=  {1\over\rho}Z^{-1} \sum_\sigma \rho \dot g_{\epsilon_1}(h-{1\over \sqrt{d}}\sum_m j_{km}\sigma_k\sigma_m){1\over \sqrt{d}}\sigma_k\sigma_r
\exp(-\rho H_d(g_{\epsilon_1},J^{(t)},\sigma)) \\
\\&+{1\over\rho}Z^{-1} \sum_\sigma \rho \dot g_{\epsilon_1}(h-{1\over \sqrt{d}}\sum_l j_{rl}\sigma_r\sigma_l){1\over \sqrt{d}}\sigma_k\sigma_r
\exp(-\rho H_d(g_{\epsilon_1},J^{(t)},\sigma))\\
&=
\underbrace{{1\over \sqrt{d}}\langle \dot g_{\epsilon_1}(h-{1\over \sqrt{d}}\sum_m j_{km}\sigma_k\sigma_m)\sigma_k\sigma_r\rangle}_{T_1} + \underbrace{{1\over \sqrt{d}}\langle \dot g_{\epsilon_1}(h-{1\over \sqrt{d}}\sum_l j_{rl}\sigma_r\sigma_l)\sigma_k\sigma_r\rangle}_{T_2}. 
\end{align*}
Since $\langle \cdot \rangle$ is Gibbs average it is at most the max term. We have, by assumption, that $\sup_ t \dot g_{\epsilon_1}$ 
is bounded by a constant (independent of $n$). As $\sigma_i=\pm 1$ we obtain a bound order $1/\sqrt{d}$. 
Consider, the derivative of the first term:
\begin{align*}
{d\over d j_{kr}}T_1 &=  -\underbrace{\frac{1}{d} \langle \ddot g_{\epsilon_1}(h-{1\over \sqrt{d}}\sum_m j_{km}\sigma_k\sigma_m)\rangle}_{T_3}\\
&+\underbrace{\frac{\rho}{d}\langle(\dot g_{\epsilon_1}(h-{1\over \sqrt{d}}\sum_m j_{km}\sigma_k\sigma_m))^2\rangle}_{T_4} +\underbrace{\frac{\rho}{d}\langle(\dot g_{\epsilon_1}(h-{1\over \sqrt{d}}\sum_m j_{km}\sigma_k\sigma_m))(\dot g_{\epsilon_1}(h-{1\over \sqrt{d}}\sum_l j_{rl}\sigma_r\sigma_l))\rangle}_{T_5}\\
&- \underbrace{\frac{\rho}{d}(\langle(\dot g_{\epsilon_1}(h-{1\over \sqrt{d}}\sum_m j_{km}\sigma_k\sigma_m))\rangle)^2}_{T_6}- \underbrace{\frac{\rho}{d}(\langle(\dot g_{\epsilon_1}(h-{1\over \sqrt{d}}\sum_m j_{km}\sigma_k\sigma_m))\rangle)(\langle(\dot g_{\epsilon_1}(h-{1\over \sqrt{d}}\sum_m j_{km}\sigma_k\sigma_m))\rangle)}_{T_7}
\\&= \cO(\frac{1}{d})\, .
\end{align*}

Using similar computations, we show that the derivatives of each of the terms in ${d\over d j_{kr}}T_3,{d\over d j_{kr}}T_4,\cdots,{d\over d j_{kr}}T_7$ are $\cO(\frac{1}{d^{\frac{3}{2}}})$. For instance, we have:
\begin{align*}
   &{d\over d j_{kr}}T_3 = 
   -\frac{1}{d^{3/2}} \langle \dddot g_{\epsilon_1}(h-{1\over \sqrt{d}}\sum_m j_{km}\sigma_k\sigma_m)\rangle + \frac{\rho}{d^{3/2}} \langle (\ddot g_{\epsilon_1}(h-{1\over \sqrt{d}}\sum_l j_{km}\sigma_k\sigma_m)\dot g_{\epsilon_1}(h-{1\over \sqrt{d}}\sum_m j_{km}\sigma_k\sigma_m)\sigma_k\sigma_r)\rangle\\
   &+\frac{\rho}{d^{3/2}} \langle (\ddot g_{\epsilon_1}(h-{1\over \sqrt{d}}\sum_l j_{km}\sigma_k\sigma_m)\dot g_{\epsilon_1}(h-{1\over \sqrt{d}}\sum_l j_{rl}\sigma_r\sigma_l)\sigma_k\sigma_r)\rangle\\&-\frac{\rho}{d^{3/2}} \langle (\ddot g_{\epsilon_1}(h-{1\over \sqrt{d}}\sum_m j_{km}\sigma_k\sigma_m)\rangle\langle (\dot g_{\epsilon_1}(h-{1\over \sqrt{d}}\sum_m j_{km}\sigma_k\sigma_m)\sigma_k\sigma_j)\rangle\\
   &-\frac{\rho}{d^{3/2}} \langle (\ddot g_{\epsilon_1}(h-{1\over \sqrt{d}}\sum_m j_{km}\sigma_k\sigma_m)\rangle\langle (\dot g_{\epsilon_1}(h-{1\over \sqrt{d}}\sum_l j_{rl}\sigma_r\sigma_l)\sigma_k\sigma_r)\rangle\, .\\
\end{align*}
Similarly, the derivatives of $T_2$ are obtained by replacing $k$ and $m$ with $r$ and $l$ respectively.
We thus obtain:
\begin{equation}\label{eq:2derbound}
    \norm{F^{(2)}}_\infty \leq \frac{C_{2,\epsilon,\rho}}{d},
\end{equation}
for some constant $C_{2,\epsilon,\rho}$.

Now, from repeated applications of the chain and product rule of differentiation, we see that the derivatives of each of the terms $T_4, T_5, T_6, T_7$ will be $O(d^{\frac{3}{2}})$. We thus have:
\begin{equation}\label{eq:thirder}
     \norm{F^{(3)}}_\infty \leq  \frac{C_{3,\epsilon,\rho}}{d^\frac{3}{2}},
\end{equation}
for some constant $C_{3,\epsilon,\rho}$.

% When we take the 2nd and 3d derivatives, those will be bounded by $1/n$ and $1/n^{3\over 2}$ 

We recall that by assumption over $B$ and the definition of $A$, we have, for large enough $n$:
\begin{align*}
    \E[\abs{a_{kr}}^3]\leq C_a(\frac{d}{n}),\\
    \E[\abs{b_{kr}}^3]\leq C_b(\frac{d}{n})\, .
\end{align*}
Substituting the above bounds in (\ref{eq:optimal}), conditioning on each entry in turn, and summing over, we obtain:
\begin{align*}
\abs{\E[F(A,g_{\epsilon_1})]-\E[F(B,g_{\epsilon_1})]} &\leq \frac{n(n+1)}{2}\frac{2C_{3,\epsilon,\rho}}{d^\frac{3}{2}}(C_a(\frac{d}{n})+C_b(\frac{d}{n}))\\
   &\leq C_{F,\epsilon_1,\rho}\frac{n}{\sqrt{d}}.
\end{align*}
for some constant $C_{F,\epsilon_1,\rho}$ dependent on $\epsilon_1,\rho$. Now, let $\epsilon_2$ be arbitrary. For $d\geq \frac{C_{F,\epsilon_1,\rho}^2}{\epsilon_2^2}$, we have:
\begin{align*}
   \abs{\E[F(A,g_{\epsilon_1})]-\E[F(B,g_{\epsilon_1})]} \leq \epsilon_2n.
\end{align*}

We note that using triangle inequality:
\begin{align*}
     \abs{\E[\cH_d^*(A,(\cdot)^+)]-\E[\cH_d^*(B,(\cdot)^+)]} &\leq \abs{\E[\cH_d^*(A,(\cdot)^+)]-\E[\cH_d^*(A,g_{\epsilon_1})]} + \abs{\E[\cH_d^*(B,(\cdot)^+)]-\E[\cH_d^*(B,g_{\epsilon_1})]}\\ &+ \abs{\E[\cH_d^*(A,g_{\epsilon_1})]-\E[F(A,g_{\epsilon_1})]} + \abs{\E[\cH_d^*(B,g_{\epsilon_1})]-\E[F(B,g_{\epsilon_1})]}\\&+\abs{\E[F(A,g_{\epsilon_1})]-\E[F(B,g_{\epsilon_1})]}\\&\leq (2\epsilon_1+ 2\frac{\log 2}{\rho} +\epsilon_2)n\, ,
\end{align*}
where we used Equations \ref{eq:approx_1}, \ref{eq:approx_2}, \ref{eq:approx_3}.

Let $\epsilon >0$ be arbitrary. Then, with $\epsilon_1=\epsilon/6, \epsilon_2=\epsilon/3, \rho =2\log 2/\epsilon$, we obtain that for $d \geq \frac{C_{F,\epsilon_1,\rho}^2}{\epsilon_2^2}$, we have: \begin{equation}
\abs{\E[\cH_d^*(A,(\cdot)^+)]-\E[\cH_d^*(B,(\cdot)^+)]} \leq (\epsilon/3+ \epsilon/3 +\epsilon/3)n=\epsilon n.
\end{equation}
Since $\epsilon$ was arbitrary, we equivalently obtain
\begin{equation}\label{eq:ham_universality}
\abs{\E[\cH_d^*(A,(\cdot)^+)]-\E[\cH_d^*(B,(\cdot)^+)]} = o_d(1)n.
\end{equation}

Next, we establish the concentration of $\cH_d^*(A,(\cdot)^+)$ and $\cH_d^*(B,(\cdot)^+)$.
Let $A$ and $A^{(i,j)}$ be weight matrices sampled from $P_A$ 
differing only at the edges $(i,j),(j,i)$. Let $a_{ij},a'_{ij}$, denote the weight of the $(i,j)$ edge in $A$ and $A^{(i,j)}$ respectively.
We note that from the $1$-Lipschitzness of $(\cdot)^+$, and the definition of $h^\star$, we have:
\begin{equation}
   \abs{\cH_d^*(A,(\cdot)^+)-\cH_d^*(A^{(i,j)},(\cdot)^+)} \leq \frac{2}{\sqrt{d}}\abs{a_{ij}-a'_{ij}}.
\end{equation}
Therefore:
\begin{align*}
   \Ea{(\cH_d^*(A,(\cdot)^+)-\cH_d^*(A^{(i,j))^2},(\cdot)^+)} &\leq \frac{2}{\sqrt{d}}\Ea{(a_{ij}-a'_{ij})^2}\\ &= \frac{4}{\sqrt{d}}\frac{d}{n}(1-\frac{d}{n}).
\end{align*}
Therefore, using the Efron-Stein inequality, we obtain:
\begin{equation}\label{eq:eff-stein}
\begin{split}
    \operatorname{Var}((\cH_d^*(A,(\cdot)^+)) &\leq \frac{1}{2}\frac{n(n+1)}{2}\frac{4}{\sqrt{d}}\frac{d}{n}(1-\frac{d}{n})\\
    &\leq Cn\sqrt{d},
\end{split}
\end{equation}
for some constant $C$ and large enough $n$.
Therefore, Chebychev's inequality yields:
\begin{align*}
 \Pr{(\abs{\cH_d^*(A,(\cdot)^+)-\Ea{\cH_d^*(A,(\cdot)^+)}}\geq \epsilon n)} &\leq \frac{Cn\sqrt{d}}{\epsilon^2 n^2}\\
 &\frac{C\sqrt{d}}{\epsilon^2 n} \underset{n \rightarrow \infty}{\rightarrow} 0.
\end{align*}
Therefore, we have $\cH_d^*(A,(\cdot)^+)=\Ea{\cH_d^*(A,(\cdot)^+)}+o(n)$ with high probability as $n \rightarrow \infty$. This completes the proof of Lemma \ref{lem:sparse_uni}.

\subsection{Universality of extensive violation of $h$-stability ($h > h^\star$)}\label{sec:absense_uni}

Note that a large value of $D^\star(W,h)$ does note imply a large number of vertices violating $h$-stability. Therefore, to prove universality in the regime $h>h^\star$, we introduce the following truncated deficit function:
\begin{align*}
T(W,h,\sigma)=\sum_{1\le i\le n}\left(h-{1\over \sqrt{d}}\sum_j w_{ij}\sigma_i\sigma_j\right)_1^+,
\end{align*}
where the function $(\cdot)_1^+$ is defined as:
\begin{equation}\label{eq:trunc_thres}
    (x)_1^+ = \begin{cases}
    0, & x <0\\
    x, & 0 < x < 1\\
    1 & \text{otherwise}
    \end{cases}.
\end{equation}

Similar to the setup for $h < h^\star$, we restrict the Gibbs measure defined by the above Hamiltonian to be supported on the set of bisections. Therefore, we define $N^*(W,h)=\max_{\sigma \in M_0} N(W,h,\sigma)$ and $T^*(W,h)=\min_{\sigma \in M_0} T(W,h,\sigma)$ .

\begin{lemma}\label{lem:def_3}
Suppose $h \in \R$ satisfies $N^*(W,h) = n-\Theta_d(1)n-o(n)$ with high probability as $n \rightarrow \infty$. Then for any $\tilde h>h$, we have $T^*(W,\tilde h) \geq (\Theta_d(1))n+o(n)$ with high probability as $n \rightarrow \infty$.
\end{lemma}
\begin{proof}
$N^*(W,h) = n-\Theta_d(1)n -o(n)$ implies that $\exists \epsilon$ independent of $d$ such that for every partition, large enough $d$, as $n \rightarrow \infty$, at least $\epsilon n$ vertices violate $h$-stability. Since $\tilde h>h$, any vertex violating $h$-stability also violates $\tilde h$-stability.
By the definition of $(\cdot)_1^+$, each vertex contributes at least $(\tilde h-h)^+_1$ to the total truncated deficit at $\tilde h$ i.e $T(W,h,\sigma) \geq \epsilon(h-h^\star)n$ for all $\sigma$. We thus have $T^*(W,h) \geq \epsilon(\tilde h-h)n$ with a high probability for large enough $d$ as $n \rightarrow \infty$.
\end{proof}

Next, we observe that, unlike the original deficit function, a large value of the truncated deficit directly implies the existence of a large number of vertices violating $h$-stability. This is expressed through the following lemma:
\begin{lemma}\label{lem:def_2}
Suppose $h \in \R$ satisfies $T^*(W,h) = \Theta_d(1)n - o(n)$ with high probability as $n \rightarrow \infty$. Then, we have $N^*(W,h)= n-(\Theta_d(1))n-o(n)$ with high probability as $n \rightarrow \infty$.
\end{lemma}

We recall that Theorem \ref{thm:sparse_thres} implies that when $W$ corresponds to an Erdős-Rényi graph with anti-ferromagnetic interactions, then for any $h<h^\star$, $N^*(W,h) = n-\Theta_d(1)n-o(n)$ with high probability as $n \rightarrow \infty$. Therefore,  Lemma \ref{lem:def_3} implies that
$T^*(W,h) = \Theta_d(1)n - o(n)$.
Subsequently, replacing the threshold function $()_+$ by the truncated threshold $()^1_+$ in the proof of Lemma \ref{lem:sparse_uni} results in the following universality result:
\begin{lemma}\label{lem:sparse_deficit_failure}
Let $W$ be a weighted graph satisfying the assumptions in Theorem \ref{thm:uni_sparse}.
Let $h^\star$ be the threshold defined in Theorem \ref{thm:sparse_thres}.
Then, $\forall h > h^\star$, there exists an $\epsilon = \epsilon(h)$, such that for large enough $d$, w.h.p as $n \rightarrow \infty$, we have:
\begin{equation}
  T^*(W,h)\geq \epsilon n.
\end{equation}
\end{lemma}

Combining the above Lemma with Lemma \ref{lem:def_2}, we obtain the following result:
\begin{corollary}\label{cor:sparse_deficit_failure}
Let $W$ and $h$ be as in Lemma \ref{lem:sparse_deficit_failure}
\begin{equation}
N^*(W,h) \leq  n(1-\Theta_d(1))+o(n),
\end{equation}
with high probability.
\end{corollary}
\subsection{Proof of Theorem \ref{thm:uni_sparse}}

We finally note that corollaries \ref{cor:sparse_deficit_success} and \ref{cor:sparse_deficit_failure} together imply Theorem \ref{thm:uni_sparse}.
\section{Proof of Theorem \ref{thm:uni_dense}: Sparse to Dense Reduction}\label{sec:dense_uni}

% For Erdős-Rényi graphs with degree $d$, let the edge weights $w_{ij}$ be i.i.d Rademacher variables.
% $w_{ij}$ then have mean $0$ and variance $\frac{d}{n-1}$.
% For simplicity, we first assume that $\E[w_{ij}]=0$.
% Since $H_d(x,g,\sigma)$ is independent of theƒ mean for $0$ magnetization(bisections), we will subsequently generalize the result to arbitrary means.

Now, consider the case of weight matrices $W$ with i.i.d entries indexed by $n$ satisfying the assumptions in Theorem \ref{thm:uni_dense} i.e. $\E[\abs{w_{ij}-\mu}^2]=1$, $\E[\abs{w_{ij}-\mu}^3]=O_n(1)$ 
This includes weight matrices with i.i.d Gaussian entries  i.e.
$w_{ij}\sim \mathcal{N}(0,1)$.
For simplicity, we include self-interactions i.e $w_{ii} \neq 0$. The proof can be generalized to exclude self-interactions by interpolating between the corresponding sparse model without loops. Let $D(W,h,\sigma)$ denote the total $h$-deficit for the dense graph $W$, i.e:
\begin{align*}
D(W,h,\sigma)=\sum_{1\le i\le n}\left(h-{1\over \sqrt{d}}\sum_j w_{ij}\sigma_i\sigma_j\right)^+,
\end{align*}

Analogous to Section \ref{sec:sparse_uni}, we let $D^*(W,h) = \min_{\sigma \in M_0} D(W,h,\sigma)$, and define the following Hamiltonian for the dense graph:
\begin{align*}
\cH(W,g,\sigma)=\sum_{1\le i\le n}g\left(h-{1\over \sqrt{n}}\sum_j w_{ij}\sigma_i\sigma_j\right)\, .
\end{align*}

We observe that the above dense Hamiltonian $\cH_d$ can now be expressed as the sparse Hamiltonian $H$ for rescaled and shifted variables $p_{ij}=\sqrt{\frac{d}{n}}w_{ij}$.
\begin{equation}\label{eq:rescale_ham}
\cH(W,g,\sigma)= \cH_d(p,g,\sigma)=\sum_{1\le i\le n}g\left(h-{1\over \sqrt{d}}\sum_j p_{ij}\sigma_i\sigma_j\right)\, .
\end{equation}
 
Let $A$ and $J$ be random weight matrices corresponding to edge weights sampled from the sparse anti-ferromagnetic random graphs with loops and the standard normal distribution respectively. Define $P=\sqrt{\frac{d}{n}}J$.
Repeating the Lindeberg's argument while going from the matrix $A$ to the matrix $P$, while replacing each element in turn as in the previous section, we obtain:
\begin{small}
\begin{equation}\label{eq:taylor_d}
\begin{split}
\abs{\E[F(J^{(t)},g_{\epsilon_1})|J^{(t)}_{/kr}]-F(J^{(t)}_{/kr},g_{\epsilon_1})-F^{(1)}(J^{(t)}_{/kr},g_{\epsilon_1}) \E[b_{kr}]-{1\over 2}F^{(2)}(J^{(t)}_{/kr},g_{\epsilon_1}) \E[(p_{kr})^2]} \leq {1\over 3!} \norm{F^{(3)}}_\infty\E[\abs{p_{kr}}^3],\\
\abs{\E[F(J^{(t-1)},g_{\epsilon_1})|J^{(t)}_{/{kr}}]-F(J^{(t)}_{/{kr}},g_{\epsilon_1})-F^{(1)}(J^{(t)}_{/{kr}},g_{\epsilon_1}) \E[a_{kr}]-{1\over 2}F^{(2)}(J^{(t)}_{/{kr}},g_{\epsilon_1}) \E[(a_{kr})^2]}\leq {1\over 3!} \norm{F^{(3)}}_\infty\E[\abs{a_{kr}^3}],
\end{split}
\end{equation}
\end{small}
where $J^{(t)}$ as before denotes the intermediate matrix.
Note that unlike the previous section, the variances $\E[(a_{kr})^2] = \frac{d}{n}$ and $\E[(s_{kr})^2] = \frac{d}{n}(1-\frac{d}{n})$ are not equal. However, (\ref{eq:taylor_d})
and triangle inequality yield:
\begin{align*}
\abs{\E[F(J^{(t)},g_{\epsilon_1})|J^{(t)}_{/kr}]-\E[F(J^{(t-1)},g_{\epsilon_1})|J^{(t)}_{/{kr}}]} &\leq {1\over 2}\norm{F^{(2)}}_\infty (\E[(a_{kr})^2]- \E[(p_{kr})^2])\\&+{1\over 3!} \norm{F^{(3)}}_\infty\E[\abs{p_{kr}}^3]+ {1\over 3!} \norm{F^{(3)}}_\infty\E[\abs{a_{kr}^3}]\\
&\leq {1\over 2}\norm{F^{(2)}}_\infty (\frac{d^2}{n^2})+{1\over 3!} \norm{F^{(3)}}_\infty\E[\abs{p_{kr}}^3]+ {1\over 3!} \norm{F^{(3)}}_\infty\E[\abs{a_{kr}^3}].
\end{align*}

We have the following bounds on the absolute third moment: 
\begin{equation}\label{eq:thirdmom}
\begin{split}
    \E[\abs{a_{kr}}^3]\leq C_b(\frac{d}{n}).\\
    \E[\abs{p_{kr}}^3]\leq C_m(\frac{d}{n})^\frac{3}{2}.
\end{split}
\end{equation}
Utilizing the above bounds, along with the bounds on $\norm{F^{(2)}}_\infty$ and $\norm{F^{(3)}}_\infty$ in Equations \ref{eq:2derbound},\ref{eq:thirder} and taking expectation over $J^{(t)}_{/{kr}}$, we obtain:
\begin{align*}
\abs{\E[F(J^{(t)},g_{\epsilon_1})]-\E[F(J^{(t-1)},g_{\epsilon_1})|]} \leq \frac{C_{2,\epsilon,\rho}d}{n^2}+\frac{C_{3,\epsilon,\rho}C_b}{3!n\sqrt{d}}+\frac{2C_{3,\epsilon,\rho}(C_m)}{3!n^\frac{3}{2}}.
\end{align*}
Finally, summing over the $\frac{n(n+1)}{2}$ variables and using $\frac{n+1}{n} \leq 2 \forall n \in \N$ results in the following bound:
\begin{equation}
   \abs{\E[F(P,g_{\epsilon_1})]-\E[F(A,g_{\epsilon_1})]} \leq C_{2,\epsilon,\rho}d+\frac{2C_{3,\epsilon,\rho}C_bn}{3!\sqrt{d}}+\frac{4C_{3,\epsilon,\rho}(C_m)}{3!\sqrt{n}}.
\end{equation}
 
Now, let $\epsilon_2>0$ be arbitrary. Suppose:
\begin{equation}\label{eq:nd_bounds}
    d \geq (\frac{3!}{2C_{3,\epsilon,\rho}C_b\epsilon_2})^2,n \geq \max{(\frac{C_{2,\epsilon,\rho}d}{\epsilon_2},(\frac{4C_{3,\epsilon,\rho}(C_m)}{3!\epsilon_2})^2)},
\end{equation}
then we have:
\begin{align*}
\abs{\E[F(P,g_{\epsilon_1})]-\E[F(A,g_{\epsilon_1})]}  \leq 
\frac{\epsilon_2}{3}n+\frac{\epsilon_2}{3}n+\frac{\epsilon_2}{3}n = \epsilon_2n.
\end{align*}

Next, similar to the proof of Theorem \ref{thm:uni_sparse},
we use the error bounds in Equations \ref{eq:approx_1}, \ref{eq:approx_2}, \ref{eq:approx_3} to obtain, for large enough $d,n$ satisfying (\ref{eq:nd_bounds}):
\begin{equation}
    \abs{\E[\cH_d^*(A,(\cdot)^+)]-\E[\cH_d^*(P,(\cdot)^+)]} \leq (2\epsilon_1 + 2\frac{\log 2}{\rho}+\epsilon_2)n.
\end{equation}
Then, for any $\epsilon >0$ setting $\epsilon_1=\epsilon/6, \epsilon_2=\epsilon/3, \rho =2\log 2/\epsilon$ and $d,d$ satisfying (\ref{eq:nd_bounds}), results in:
\begin{equation}
     \abs{\E[\cH_d^*(A,(\cdot)^+)]-\E[\cH_d^*(P,(\cdot)^+)]} \leq \epsilon n.
\end{equation}
Therefore, 
\begin{equation}
\abs{\E[\cH_d^*(A,(\cdot)^+)]-\E[\cH_d^*(P,(\cdot)^+)]} = o_d(1)n +o(n)\, .
\end{equation}
Next, using Efron-Stein inequality as in (\ref{eq:eff-stein}), we have $\abs{\cH_d^*(A,(\cdot)^+)-\Ea{\cH_d^*(A,(\cdot)^+)}}=o(n)$  and $\abs{\cH_d^*(P,(\cdot)^+)-\Ea{\cH_d^*(P,(\cdot)^+)}}=o(n)$
with high probability as $n \rightarrow \infty$. Therefore, we obtain, with high probability as $n \rightarrow \infty$
\begin{equation}
    \abs{\cH_d^*(A,(\cdot)^+)-\cH_d^*(P,(\cdot)^+)} = o_d(1)n+o(n).
\end{equation}

Since $\cH_d^*(P,(\cdot)^+)=D^*(J,h)$ we obtain the following result:
\begin{lemma}\label{lem:dense_uni}
Let $A,J$ be arbitrary weight matrices with i.i.d random entries $a_{ij},j_{ij}$ satisfying assumptions in Theorem \ref{thm:uni_sparse}, \ref{thm:uni_dense}  for parameters $d,n$ and $n$ respectively. 
Then:
\begin{equation}
    \abs{D^*(A,h)-D^*(J,h)} = o_d(1)n+o(n),
\end{equation}
 with high probability as $n \rightarrow \infty$.
\end{lemma}

Since Theorem \ref{thm:uni_sparse} implies that $D^*(A,h) = o_d(1)n+o(n)$ for $h<h^\star$ with high probability as $n \rightarrow \infty$, considering the limit $d \rightarrow \infty$, we obtain:
\begin{lemma}\label{lem:dense_deficit_success}
Let $J$ be a weighted graph on $n$ nodes satisfying the assumptions in Theorem \ref{thm:uni_dense}.
Let $h^\star$ be the threshold defined in Theorem \ref{thm:sparse_thres}, then for any $h<h^\star$, with high probability as $n \rightarrow \infty$:
\begin{equation}
    D^*(J,h)=o(n).
\end{equation}
\end{lemma}

We note that Lemma \ref{lem:def_success} also applies to the deficit $D^*(J,h)$ and the maximum number of $h$-stable vertices $N^*(J,h)$ for dense graphs. Therefore, we obtain the following Corollary: \begin{corollary}\label{cor:dense_deficit_success}
Let $J,h^\star$ be as in Proposition \ref{lem:dense_deficit_success}. Then for any $h<h^\star$
\begin{equation}
N^*(J,h) \geq  o(n).
\end{equation}
\end{corollary}
Similarly, using the truncated deficit function $(\cdot)_1^+$ instead of $(\cdot)^+$ defined in (\ref{eq:trunc_thres}), we obtain the following results:
\begin{lemma}\label{lem:dense_deficit_failure}
Let $J,h^\star$ be as in Proposition \ref{lem:dense_deficit_success} then for any $h>h^\star$
\begin{equation}
    T^*(J,h)=\Theta(n).
\end{equation}
\end{lemma}
\begin{corollary}\label{cor:dense_deficit_failure}
Let $J,h^\star$ be as in Proposition \ref{lem:dense_deficit_success}. Then for any $h>h^\star$:
\begin{equation}
    N^*(J,h) \leq  n(1-\Theta(1))+o(n).
\end{equation}
\end{corollary}
\subsection{Proof of Theorem \ref{thm:uni_dense}}

Corollaries \ref{cor:dense_deficit_success} and \ref{cor:dense_deficit_failure} together imply Theorem \ref{thm:uni_dense}.

\subsection{Extension of the result to all configurations}

In this paper, we restrict our attention to theorems concerning bisections, as this approach enables us to concurrently establish the universality result for distributions with arbitrary means of the edges $w_{ij}$. For specific cases where the mean is $0$, such as the Sherrington-Kirkpatrick (SK) model, it is feasible to extend our results to encompass all configurations. This extension can be done leveraging the first moment results in \cite{skcount}, who derived an expression for the first moment entropy density for the SK model when considering all partitions, matching the form of \eqref{def:w_firstmom} derived by us for sparse graphs.

\section{Proof of Theorems  \ref{thm:sparse_max_E_thres_uni} and \ref{thm:dense_max_E_thres}: Universality of the Maximal/Minimal Energy of Local Optima}\label{sec:universality_e_max}

Using similar techniques as the previous section, we now prove the universality of the maximal energy thresholds at fixed values of $h$.
As in the proof of Lemma \ref{lem:high_prob_max} and Lemma \ref{lem:def_success}, we have the following relation between $N_{E_1,E_2}^*(W,h)$ i.e. the maximum number of $h$-stable vertices amongst configurations having energy in the interval $(E_1,E_2)$, and the maximal energy deficit $D_{E_1,E_2}$, defined in Definition \ref{def:const_deficit}:
\begin{lemma}\label{lem:max_E_def_success}
Suppose $\tilde h, \tilde{E}_1, \tilde{E}_2 \in \R$ satisfy $D^*_{\tilde E_1, \tilde{E}_2}(W,\tilde h) = o_d(1)n + o(n)$ with high probability as $n \rightarrow \infty$. Then for any $(E_1, E_2)\subset (\tilde{E}_1,\tilde{E}_2)$ and $h<\tilde h$, we have $N_{E_1,E_2}^*(W, h) =  n(1-o_d(1))-o(n)$ with high probability as $n \rightarrow \infty$.
\end{lemma}
\begin{proof}
    Assume that $N_{E_1,E_2}^*(W,h) \leq (1-\epsilon)n$ for some $\epsilon > 0$ and large enough $d,n$. Then, for any configuration $\sigma$, either $E(\sigma) < \tilde E_1$, $E(\sigma) > \tilde E_2$, or at least $\epsilon n$ vertices violate $\tilde h$ stability. In either case, we obtain:
\begin{align*}
    D_{\tilde E_1, \tilde E_2}(W,\tilde h,\sigma) &\geq \sum_{1\le i\le n}\left(h-{1\over \sqrt{d}}\sum_j w_{ij}\sigma_i\sigma_j\right)^+  + \left(n E_1+{1\over \sqrt{d}}\sum_{i,j \leq n} w_{ij}\sigma_i\sigma_j\right)^+ + \left({1\over \sqrt{d}}\sum_{i,j \leq n} w_{ij}\sigma_i\sigma_j-nE_2\right)^+\\
    &= \sum_i \left( h-\tilde h+\tilde h-{1\over \sqrt{d}}\sum_j w_{ij}\sigma_i\sigma_j\right)^+ + n\left(E_1-\tilde{E}_1+\tilde{E}_1-H(\sigma)\right)^+ \\&+n\left(H(\sigma)-E_2+\tilde{E}_2-E_2\right)^+\\
     &\geq \max\rbr*{\epsilon n\left( h-\tilde h\right)^+,n(E_1-\tilde{E}_1)^+,n(\tilde E_2-E_2)^+}.
\end{align*}
This contradicts $D^*_{E_1, E_2}(W,h) = o_d(1)n + o(n)$, proving that $N_{\tilde E_1 \tilde E_2}^*(W,\tilde h) =  n(1-o_d(1))-o(n)$.
\end{proof}

Let $g$ be a smooth uniform approximation of the threshold function $()_+$ as in (\ref{eq:smooth_approx}). Analogous to \ref{eq:hamilton}, we define the following Hamiltonian corresponding to the maximal energy deficit $D_{\geq E}$ function defined in \eqref{def:deficit}: 
\begin{equation}\label{eq:ham_const}
\begin{split}
\cH_{E_1, E_2,d}(W,g,\sigma)&=\sum_{1\le i\le n}g\left(h-{1\over \sqrt{d}}\sum_j w_{ij}\sigma_i\sigma_j\right) + g\left(nE_1+{1\over \sqrt{d}}\sum_{i,j \leq n} w_{ij}\sigma_i\sigma_j\right)\\&+g\left(-{1\over \sqrt{d}}\sum_{i,j \leq n} w_{ij}\sigma_i\sigma_j-nE_2\right),
\end{split}
\end{equation} 
and the associated partition function:
\begin{align*}
Z_{E_1,E_2}(W,g)&=\sum_{\sigma \in M_0}  \exp\left(-\rho \cH_{ E_1,E_2,d}(W,g,\sigma)\right)\\
F_{E_1,E_2}(W,g)&={1\over \rho}\log Z_{E_1,E_2,d}(W,g).
\end{align*}

The rest of the proof follows that of Theorem \ref{thm:uni_sparse}. Let $A, B$ be two random weight matrices satisfying the assumptions in Theorem \ref{thm:uni_sparse}. We again apply the Lindeberg's method to the function $F_{E_1,E_2}$. The derivatives of $F_{E_1,E_2}$ now involve additional terms due to the addition of the term $g({1\over \sqrt{d}}\sum_{i,j \leq n} w_{ij}-nE)$ to the Hamiltonian. However, we recover the same bounds on the derivatives as Equations \ref{eq:2derbound}, \ref{eq:thirder}. For instance, with $J$ as an arbitrary symmetric matrix, and $g_{\epsilon_1}$ as defined in (\ref{eq:smooth_approx}), we have:
\begin{align*}
F_{E_1,E_2}^{(1)}(J,g_{\epsilon_1}) 
&=
{1\over \sqrt{d}}\langle \dot g_{\epsilon_1}(h-{1\over \sqrt{d}}\sum_m j_{km}\sigma_k\sigma_m)\sigma_k\sigma_r\rangle + {1\over \sqrt{d}}\langle \dot g_{\epsilon_1}(h-{1\over \sqrt{d}}\sum_l j_{rl}\sigma_r\sigma_l)\sigma_k\sigma_r\rangle\\
&+ {1\over \sqrt{d}}\langle \dot g_{\epsilon_1}(nE_1+{1\over \sqrt{d}}\sum_{i,j \leq n} w_{ij}) \rangle + {1\over \sqrt{d}}\langle \dot g_{\epsilon_1}(-{1\over \sqrt{d}}\sum_{i,j \leq n} w_{ij}-n E_2) \rangle.
\end{align*}

Proceeding similarly, we obtain that:
\begin{equation}\label{eq:thirder_const}
     \norm{F_{E_1,E_2}^{(3)}}_\infty \leq  \frac{C_{3,\epsilon,E,\rho}}{d^\frac{3}{2}}.
\end{equation}
Now, applying equations \ref{eq:lind_bound}
,\ref{eq:final_lind_bound}, and summing over all $w_{i,j}$, we obtain:
\begin{align*}
\abs{\E[F_{E_1,E_2}(A,g_{\epsilon_1})]-\E[F_{E_1,E_2}(B,g_{\epsilon_1})]}
   &\leq C_{F,E,\epsilon_1,\rho}\frac{n}{\sqrt{d}},
\end{align*}
Where $C_{3,\epsilon_1,E}
,C_{F,E,\epsilon_1,\rho}$ denote constants dependent on $E,\epsilon_1,\rho$.
Subsequently, following the proof of (\ref{eq:ham_universality}), we obtain:
\begin{lemma}\label{lem:uni_sparse_max_E_def}
Let $A, B$ be arbitrary random weight matrices with i.i.d random entries satisfying the assumptions in Theorem \ref{thm:uni_sparse}. For any $h,E \in \R$, we have,
\begin{equation}
    \abs{D_{E_1,E_2}^*(A,h)-D_{E_1,E_2}^*(B,h)} = o_d(1)n+o(n),
\end{equation}
 with high probability as $n \rightarrow \infty$.
\end{lemma}

\begin{proposition}\label{prop:sparse_deficit_success_max_e}
Let $W_{d,n}$ be a family of random weighted graphs satisfying the assumptions in \ref{thm:uni_sparse}.
Let $h^\star$ be the threshold defined in Theorem \ref{thm:sparse_thres}.
Then, for every $\epsilon$ and $E_1,E_2$ such that either $i)$ $h\in (h_{\rm cor}, h^\star)$, and  $(E_1,E_2) \subset (E_{\rm min}(h),E_{\rm max}(h))$ or $ii)$ $h\in (-0.1, h_{\rm cor})$, and  $(E_1,E_2) \subset (E_{\rm cor}(h),E_{\rm max}(h))$, we have:
\begin{equation} 
N_{E_1,E_2}^*(W,h) \geq  n(1-o_d(1))-o(n).
\end{equation}
\end{proposition}

\begin{proof}
Let $W'$ be a random weighted graph from the sparse anti-ferromagnetic model with parameters $d,n$.
Let $\tilde{h},\tilde{E}_1,\tilde{E}_2$ be as in Lemma \ref{lem:high_prob_max}.
Corollary \ref{cor:def_const_1} then implies that $D^*_{\tilde E_1,\tilde E_2}(W',\tilde h)=o_d(1)n + o(n)$. Applying Lemma \ref{lem:uni_sparse_max_E_def} then yields $D^*_{\tilde E_1,\tilde E_2}(W,\tilde h)=o_d(1)n + o(n)$. Subsequently, we apply Lemma \ref{lem:max_E_def_success} to obtain $N_{E_1,E_2}^*(W,h) =  n(1-o_d(1))-o(n)$
\end{proof}

Next, we use the truncated $()_1^+$ defined in (\ref{eq:trunc_thres}), to define the truncated maximal energy deficit function:
\begin{equation}\label{eq:ham_const_trunc}
\begin{split}
T_{E_1,E_2}(W,h,\sigma)&=\sum_{1\le i\le n}\left(h-{1\over \sqrt{d}}\sum_j w_{ij}\sigma_i\sigma_j\right)^+_1 + n\left(E_1+{1\over n\sqrt{d}}\sum_{i,j \leq n} w_{ij}\sigma_i\sigma_j\right)^+_1\\&+ n\left(-{1\over n\sqrt{d}}\sum_{i,j \leq n} w_{ij}\sigma_i\sigma_j-E_2\right)^+_1,
\end{split}
\end{equation}  
and similarly we define $T^\star_{E_1,E_2}(W,h) = \min_{\sigma \in M_0} T_{E_1,E_2}(W,h,\sigma)$

The following results relate the truncated maximal energy deficit cut to the maximum number of $h$-stable vertices amongst configurations having sufficient energy:
\begin{lemma}\label{lem:def_3_const}
Suppose $E,h \in \R$ satisfy $N_{E_1,E_2}^*(W,h) = n-\Theta_d(1)n$ with high probability as $n \rightarrow \infty$. Then for any $(\tilde E_1,\tilde E_2) \subset (E_1, E_2)$ and $\tilde h<h$, we have $T_{\tilde E_1, \tilde E_2}^*(W,\tilde h) \geq (\Theta_d(1))n+o(n)$ with high probability as $n \rightarrow \infty$.
\end{lemma}
\begin{proof}
$N_{E_1,E_2}^*(W,h)=n-\Theta_d(1)n$ implies that for any configuration $\sigma$, for large enough $d$, $\exists \epsilon$ such that, we have either:
\begin{enumerate}
    \item $E(\sigma)<E_1$.
    \item $E(\sigma)>E_2$.
    \item $N_{E_1,E_2}(W,h,\sigma)\leq n-\epsilon n$.
\end{enumerate}
In either case, we obtain:
\begin{align*}
    T_{\tilde E_1,\tilde E_2}(W,\tilde h,\sigma) &\geq \sum_{1\le i\le n}\left(h-{1\over \sqrt{d}}\sum_j w_{ij}\sigma_i\sigma_j\right)^+_1 + \left(n \tilde E_1+{1\over \sqrt{d}}\sum_{i,j \leq n} w_{ij}\sigma_i\sigma_j\right)^+_1\\&+\left(-{1\over \sqrt{d}}\sum_{i,j \leq n} w_{ij}\sigma_i\sigma_j-n \tilde E_2\right)^+_1\\
    &= \sum_i \left(\tilde h-h+h-{1\over \sqrt{d}}\sum_j w_{ij}\sigma_i\sigma_j\right)_1^+ +\left(n \tilde E_1-n E_1+nE_1-H(\sigma)\right)_1^+ \\&+\left(H(\sigma)-nE_2 + nE_2-n\tilde{E}_2\right)_1^+\\
     &\geq \max\rbr*{\epsilon n\left(\tilde h-h\right)_1^+,n(\tilde E_1-E_1)^+_1, n(E_2-\tilde E_2)^+_1}.
\end{align*}
\end{proof}
\begin{lemma}\label{lem:def_2_const}
Suppose $E,h \in \R$ satisfy $T_{E_1,E_2}^*(W,h) = \Theta_d(1)n - o(n)$ with high probability as $n \rightarrow \infty$. Then, we have $N_{E_1,E_2}^*(W,h)= n-(\Theta_d(1))n-o(n)$ with high probability as $n \rightarrow \infty$.
\end{lemma}
\begin{proof}
$T_{E_1,E_2}^*(W,h) = \Theta_d(1)n - o(n)$ implies that there exists an $\epsilon > 0$ such that for large enough $d$, any configuration $\sigma$ satisfies  $T_{E_1,E_2}(W,h,\sigma) \geq \epsilon n$. Since $()^+_1$ is bounded by $1$, we obtain that any configuration with $(nE_1+{1\over n\sqrt{d}}\sum_{i,j \leq n} w_{ij}\sigma_i\sigma_j)^+_1>0$ and $(nE_1+{1\over n\sqrt{d}}\sum_{i,j \leq n} w_{ij}\sigma_i\sigma_j)^+_1=0$ and $(-{1\over n\sqrt{d}}\sum_{i,j \leq n} w_{ij}\sigma_i\sigma_j-nE_2)^+_1>0$  must satisfy $\left(h-{1\over \sqrt{d}}\sum_j w_{ij}\sigma_i\sigma_j\right)>0$ for at least $\epsilon n$ vertices.
\end{proof}
\begin{proposition}\label{prop:sparse_deficit_fail_max_e}
Let $W_{d,n}$ be a family of random weighted graphs satisfying the assumptions in \ref{thm:sparse_max_E_thres}.
Let $h^\star$ be the threshold defined in Theorem \ref{thm:sparse_thres}.
Then, for every $\epsilon>0$ and $h < h^\star$, we have,
\begin{equation} 
N_{-\infty,E_{\rm min}(h)}^*(W,h) \leq  n(1-\Theta_d(1))+o(n),
\end{equation}
and
\begin{equation} 
N_{E_{\rm max}(h),\infty}^*(W,h) \leq  n(1-\Theta_d(1))+o(n).
\end{equation}
\end{proposition}

Propositions \ref{prop:sparse_deficit_success_max_e},\ref{prop:sparse_deficit_fail_max_e} along with Lemma \ref{lem:max_E_def_success} imply Theorem \ref{thm:sparse_max_E_thres_uni}.

Similarly, for the proof of Theorem \ref{thm:dense_max_E_thres}, we consider the constrained Hamiltonian for the dense graph: 
\begin{equation}\label{eq:const_ham}
\begin{split}
\cH_{E_1,E_2}(W,g,\sigma)&=\sum_{1\le i\le n}g\left(h-{1\over \sqrt{n}}\sum_j w_{ij}\sigma_i\sigma_j\right) + g\left(nE_1+{1\over \sqrt{n}}\sum_{i,j \leq n} w_{ij}\sigma_i\sigma_j\right)\\
&+ g\left(-{1\over \sqrt{n}}\sum_{i,j \leq n} w_{ij}\sigma_i\sigma_j-nE_2\right).
\end{split}
\end{equation}

Next, analogous to (\ref{eq:rescale_ham}), we introduce the rescaled and shifted variables $p_{ij}=\sqrt{\frac{d}{n}}w_{ij}$ and observe that the above Hamiltonian $\cH_E$ reduces to the constrained Hamiltonian $\cH_{E,d}$ for sparse graphs:
\begin{equation}\label{eq:rescale_ham_const}
\begin{split}
\cH_{E_1,E_2}(w,g,\sigma)= \cH_{E_1,E_2,d}(p,g,\sigma)&=\sum_{1\le i\le n}g\left(h-{1\over \sqrt{d}}\sum_j p_{ij}\sigma_i\sigma_j\right) + g\left(nE_1+{1\over \sqrt{d}}\sum_{i,j \leq n} p_{ij}\sigma_i\sigma_j\right)\\+ g\left(-{1\over \sqrt{d}}\sum_{i,j \leq n} p_{ij}\sigma_i\sigma_j-nE_2\right).
\end{split}
\end{equation}

Therefore, analogous to Lemma \ref{lem:dense_deficit_success}, we utilize the bound in (\ref{eq:thirder_const}) to obtain:
\begin{lemma}\label{lem:dense_deficit_success_const}
Let $J$ be a weighted graph on $n$ nodes satisfying the assumptions in Theorem \ref{thm:uni_dense}.
Suppose that either $i)$ $h\in (h_{\rm cor}, h^\star)$, and  $(E_1,E_2) \subset (E_{\rm min}(h),E_{\rm max}(h))$ or $ii)$ $h\in (-0.1, h_{\rm cor})$, and  $(E_1,E_2) \subset (E_{\rm cor}(h),E_{\rm max}(h))$, then with high probability as $n \rightarrow \infty$:
\begin{equation}
    D^*_{E_1,E_2}(J,h)=o(n).
\end{equation}
\end{lemma}

The rest of the proof of Theorem \ref{thm:dense_max_E_thres}, follows that of Theorem \ref{thm:sparse_max_E_thres}, namely we utilize Lemma \ref{lem:max_E_def_success} to obtain that $N^*_{E_1,E_2}(W,h)=n-o(n)$.  Similarly, we apply the Lindeberg's argument to the truncated deficit function along with Lemmas \ref{lem:def_3_const}, \ref{lem:def_2_const} to obtain the statement of Theorem \ref{thm:dense_max_E_thres}. 

. 

\section{Conclusions and Open Problems}\label{sec:conc}

In this work, we analyzed several interesting phenomena related to the single-spin-flip-stability in random graphs and spin glasses.
There are several promising directions for future work:
\begin{enumerate}
    \item Geometry of solutions: Given the existence of near $h$-stable solutions, it is natural to wonder about the geometry of solutions. 
    In Figure \ref{fig:2ndmom_thres}, we observe that the second-moment entropy density at the fixed energy $E^\star(h)$ as a function of the overlap demonstrates a steep drop below $0$ near overlaps $-1$ and $1$.
 Using the Markov inequality, this implies that for large enough $d$, with high probability as $n\rightarrow \infty$, all pairs of $h$-stable configurations do not have overlaps in a certain range.
 Such a separation property, also known as the ``Overlap Gap Property", has been linked to algorithmic hardness in a recent line of work \citep{gamarnik2021overlap}. By the above argument, the  negativity of second moment entropy density $W(E,\omega,h)$ for an interval of $\omega$ thus implies the ``Overlap Gap Property" for the given value of $h$.
 However, for small enough $h>0$, we observe that the second-moment entropy density stays above $0$ for all overlaps $\omega \in (-1,1)$.
 We illustrate this in Fig.~\ref{fig:thres_h_005} for $h=0.05$. Thus, our results only imply ``Overlap Gap Property" for sufficiently large $h$. Going beyond our results, \cite{huang2025strong} recently established the same for all $h>0$, using it to prove the hardness for finding such states through low-degree polynomial methods.

 \begin{figure}
     \centering
     \includesvg[width=0.6\textwidth]{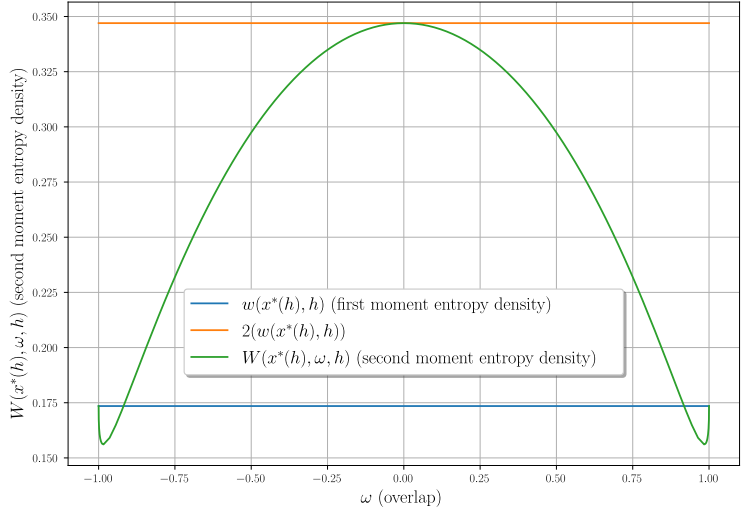}
     \caption{The second moment entropy density at $E=E^*(h)$ for $h=0.05$}
     \label{fig:thres_h_005}
 \end{figure}

Such a result for all $h>0$ is also known for sparse regular graphs, wherein \citep{behrens2022dis} used the small-set expansion to prove that with high probability, any two $h$-stable partitions have a hamming distance at least $C(d)n$ for some constant $C$ possibly dependent on $d$. 
 
 % Unlike the case of regular graphs, however, Erdős-Rényi graphs contain $\Theta(n)$ vertices having degree less than $h\sqrt{d}$ and thus violating $h$-stability. Therefore, we must restrict ourselves to solutions having $n-o(n)$ $h$-stable vertices. In any such configuration, switching a vertex violating $h$-stability results in another configuration having $n-o(n)$ $h$-stable vertices preventing such configurations from being ``frozen". 
% We leave to future work the relation between the frozen property and the geometry of approximate $h$-stable optima in Erdős-Rényi graphs, while accounting for the violations. It could also be useful to investigate the universality properties of such gaps.

\item Universality and the existence of fully $h$-stable partitions. Unlike the case of sparse Erdős-Rényi graphs, the SK model might contain configurations with all vertices being $h$-stable. We leave to future work extending our results for the existence of configurations containing $n-o(n)$ $h$-stable vertices to configurations with all $n$ vertices being $h$-stable. 
A concurrent work by Minzer, Sah, Sawhney \citep{minzer2023perfectly} recently proved such a result for dense \ER graphs from $\G(n,1/2)$. 
A promising direction would be to extend the result to the SK model with Gaussian disorder. We believe that our proof technique could also allow establishing the universality of the phase transition for configurations with all spins being $h$-stable. 

\item Extension to $p$-spin Ising models and hypergraphs: 
We believe that our proof techniques can also be generalized to analyze the single-flip stability in models involving interactions between $p\geq 3$ spins such as $p$-spin models and hypergraphs.

\item Extensions of universality and the limiting the entropy density: As mentioned in sections \ref{sec:intro} and \ref{sec:universality_e_max}, our results corroborate the predictions in \cite{bray1981metastable} for $E_{min}(0)$ and $E_{max}(0)$. However, further inspection of the numerical values reveals that even the values of the first moment entropy density and $E_{\rm cor}(0)$ obtained through Theorem \ref{thm:sparse_max_E_thres} for sparse graphs with anti-ferromagnetic interactions match the ones reported in \cite{bray1981metastable}. This leads us to conjecture that a large number of properties of the $h$-stable configurations obey universality, including the first and second moment entropy densities. Furthermore, the first moment entropy density obtained by us for specific values of energy matches the quenched entropy density in \cite{bray1981metastable}, defined as the expectation of logarithm of the number of local optima at the given energy level. We therefore conjecture that the first moment entropy density asymptotically matches the quenched one. More precisely, let $X_{E_1,E_2}(h)$ denote the number of $h$-stable configurations with normalized energy in range $(E_1,E_2)$ for some $E_1<E_2$. We conjecture that:
\begin{equation}
    \lim_{n \rightarrow \infty} \frac{1}{n}\Ea{\log (1+X_{E_1,E_2}(h))}  =  \lim_{n \rightarrow \infty} \frac{1}{n}\log (1+\Ea{X_{E_1,E_2}(h)}),
\end{equation}
for $E_1,E_2,h$ such that the RHS is positive. Here we use $\log (1+X_{E_1,E_2}(h))$ instead of $\log (X_{E_1,E_2}(h))$ ensure that the term remains defined for $X_{E_1,E_2}(h)=0$.
\end{enumerate}

\section{Acknowledgments}
David Gamarnik acknowledges funding from NSF grant DMS-2015517. We thank Freya Behrens for discussions and pointers to existing literature.

% \yd{Fix capitalization in references}
\bibliographystyle{plainnat}
\bibliography{citations}

\appendix
\section{First Moment Entropy Density}\label{app:first_mom}

\subsection{Setup}\label{sec:setup}

Let $\tilde{\G}(n,d/n)$, denote a multi-graph sampled under the configuration model with a fixed number
of edges $m=dn/2$ (as described in Section \ref{sec:Theorem1}). Define $A^{(d,n)}$ to be
the adjacency matrix of $\tilde{\G}(n,d/n)$ with entries 
$a_{ij}$ for $i,j \in [n]$ being equal to the number of edges between vertices $i$ and $j$. Following the setup of Theorem \ref{thm:sparse_thres}, we further define the edge weights $W^{(d,n)}$ as the negative of the adjacency matrix i.e $W^{(d,n)}=-A^{(d,n)}$. The single
flip stability \eqref{def:h-stable} value of a node $v$ 
with respect to bisection 
$\sigma:[n]\to \{\pm 1\}$ can then be expressed as:
\begin{align}
    s_\sigma(v,W) &= -{1\over \sqrt{d}} \sigma_v \sum_{u\in [n]} w_{vu}\sigma_u,
    \notag\\
    &=\frac{1}{\sqrt{d}}\left(\sum_{u\in [n],\sigma_u\ne \sigma_v} a_{vu} - \sum_{u\in [n],\sigma_u=\sigma_v} a_{vu}\right). \label{eq:friendli-1}
\end{align}
while $h$-stability of a vertex $v\in [n]$ 
is defined as before, namely,
$s_\sigma(v)\ge h$. In words, the stability $s_\sigma(v,W)$ equals the differences between the number of cross-partition neighbors and in-partition neighbors of $v$.

% Lemma~\ref{lemma:2-random-graphs} will allow us to translate
% this result to a similar result for $\G(n;dn/2)$. 

% \dg{finish later}

For every $z \in \mathbb{R}$, let $X(z,h,r)$ 
denote
the (random) number of bisections
$\sigma$ of $[n]$ in $\G(n;dn/2)$ with at least $rn$
$h$-stable nodes, 
such that the cut
value associated with $\sigma$ is $\floor{zn}$.
Namely the cardinality of the set
of pairs $(i,j), 1\le i<j\le n$ 
with $\sigma(i)\ne \sigma(j), w_{i,j}=1$
is $\floor{zn}$. 
 Proposition~\ref{prop:neg_first_mom} claims  the bound of the form $\E[X(z,h,r)]\le \exp(w(E,h,r)n+o(n))$, 
where $w(E,h,r)$ is defined by (\ref{eq:wrh})

To compute $\mathbb{E}[X(z,h,r)]$  we further divide the number of bisections
according to the number of edges within each side of the bisections. 
For any bisection 
$\sigma:[n]\to \{\pm 1\}$,
let $V_1=\{i\in [n]: \sigma(i)=1\},
V_2=\{i\in [n]: \sigma(i)=-1\}$.
Let $E_{ij}$ denote the number of edges 
having ends in sets $V_i, V_j$. This 
partitions the set of edges of $\G(n;dn/2)$
into  three subsets of edges with cardinalities
$E_{11},E_{12},E_{22}$. 

When $r=1$, by adding
(\ref{eq:friendli-1}) over $v \in V_1$, we obtain the following relation between the subset sizes $E_{11}, E_{12}, E_{22}$ and the $h$-stabilities:
\begin{equation}\label{eq:E_1_const}
E_{12}-2E_{11} = \sum_{v \in V_1} s_\sigma(v, W) \geq h\frac{n}{2}\sqrt{d},
\end{equation}
and
\begin{equation}\label{eq:E_2_const}
    E_{12}-2E_{22} = \sum_{v \in V_2} s_\sigma(v, W) \geq h\frac{n}{2}\sqrt{d}.
\end{equation}
Since $E_{11}+ E_{12}+E_{22}=\frac{d}{2}n$ and $E_{12}=\floor{zn}$,
we obtain the bound:
\begin{equation}\label{eq:E_const}
    4\floor{zn}-d n \geq h\sqrt{d}.
\end{equation}

For any $h \in \mathbb{R}, r \in [0,1]$ and $z,z_1,z_2 \in \mathbb{R}$, let $X(z,z_1,z_2, h,r)$
denote the number
of bisections satisfying $E_{12}=\floor{zn},E_{11}=z_1n,E_{22}=z_2n$ and $s_\sigma(v) \geq h$ for at least $\floor{rn}$ vertices. Trivially we have
 $X(z,z_1,z_2, h,r)\le X(z,h,r)$, and $X(z,h,r)$ can be decomposed as:
\begin{equation}\label{eq:X-conditioned-z}
   \mathbb{E}[X(z,h,r)] = \sum_{z_1,z_2} \mathbb{E}[X(z,z_1,z_2, h,r)],
\end{equation}
where the sum is over $z_1,z_2$ satisfying $(z_1+z_2)=(d/2)n-\floor{zn}$. 

Our first task to bound the range of values for $z,z_1,z_2$ for large $n$ and $d$. For any constant $C>0$ and $z \in \mathbb{R}$, define the following set:
\begin{equation}\label{eq:def_Dz}
D(C,z)=\{z_1,z_2: z_1n, z_2n \in \Z_+, |2z_1n-dn/4|, |2z_2n-dn/4| \le C\sqrt{d}n, (\floor{zn}+(z_1+z_2)n=d/2n\}.
\end{equation}

\begin{lemma}\label{lemma:range-z}
For every  $C\ge 2$, $d\ge 1$, and $z \in \mathbb{R}$, for large enough $n$, the sum in (\ref{eq:X-conditioned-z})
restricted to pairs $(z_1,z_2)$ outside $D(C,z)$ is bounded by $\exp(-2n)$.
\end{lemma}

\begin{proof}
The constant $2$ above is somewhat arbitrary, but suffices for our purposes. Fix any bisection $\sigma$ and 
consider the associated random variables
$E_{11}, E_{22}$. The generative process under the configuration model can be represented as independently assigning the $i_{th}$ half-edge for $i \in [dn]$ uniformly to one of the $n^2$ pairs of vertices (including self-edges). Therefore, under the configuration model, the total number of half-edges $2E_{11}, 2E_{22}$ are Binomial random variables with $nd$ trials with success probability $1/4$.(assuming $n$ is even for convenience). 
Applying Hoeffding inequality to $E_{11}, E_{22}$, we have
\begin{align*}
\pr\left( 2 E_{jj} -\E[2 E_{jj}] |\ge t\right)\le 2\exp\left(-{2t^2\over nd}\right),
\end{align*}
for $j=1,2$. Using $t=C\sqrt{d}n$, we obtain a bound $2\exp\left(-2C^2n\right)$. The number of bisections is trivially at most  $2^n$. Relaxing the $h$-stability 
requirement, we obtain 
\begin{align*}
\sum_{(z_1,z_2)\notin D(C,z)} \mathbb{E}[X(z,z_1,z_2, h,r)] 
&\le \sum_{(z_1,z_2)\notin D(C,z)}\sum_\sigma  \pr(E_{12}=zn,E_{11}=z_1n,E_{22}=z_2n) \\
  & \le 6(dn/2)^3 \times 2^n\exp\left(-C^2n\right).
\end{align*}
Since $\exp(-C^2)2 < \exp(-C)$ for $C \geq 2$, the LHS is bounded by $\exp(-Cn)$ for large enough $n$.

Threfore, by our choice of $C\geq 2$ the claim is established.
\end{proof}
From this point on we let $C$ denote a fixed constant $\geq 2$ to obtain:

\begin{equation}\label{eq:X-conditioned-z-restricted}
   \mathbb{E}[X(z,h,r)] = \sum_{(z_1,z_2)\in D(C,z)} \mathbb{E}[X(z,z_1,z_2, h,r)]+\exp(-Cn),
\end{equation}
where, as we recall, 
\begin{align}
D(C,z)=\{z_1,z_2: z_1n, z_2n \in \Z_+, |z_1n-dn/8|, |z_2n-dn/8| \le 2\sqrt{d}n, z+z_1+z_2=d/2 \}.
\label{eq:D2}
\end{align}

Next, we fix a bisection $\sigma$ with parts $V_1,V_2$ 
and consider the probability distribution over graphs under the configuration model, conditioned on some fixed values of $z_1,z_2,z$
with respect to bisection $\sigma$.
This configuration model can be described as being obtained through two independent assignments of a fixed number of half-edges to vertices within each partition, corresponding to in-partition and cross-partition edges. Consider the partition $V_1$. First,
$2z_1$ half-edges are assigned independently with replacement to the vertices in $V_1$
and subsequently paired with each other.
Similarly,
$2z_2$ half-edges are assigned independently with replacement to the vertices in $V_1$
and subsequently paired with each other.
The second assignment matches $\floor{zn}$
cross-partition half-edges independently with vertices in $V_1$.
Similarly, the corresponding assignment of the remaining half-edges to vertices in $V_2$ is performed independently.

 Let
$\cE(\sigma, z,z_1,z_2)$ denote the event $E_{12}=\floor{zn},E_{11}=z_1,E_{22}=z_2$
that bisection $\sigma$ has $z_1n$ and $z_2n$ edges respectively, inside 
each  partition $V_1$ and $V_2$ of 
the bisection $\sigma$, 
 and $zn$ cross edges. By symmetry, for every fixed values of $z,z_1,z_2$, $\pr[ \mathcal{E}(\sigma,z,z_1,z_2)]$
is identical for all bisections $\sigma$.
Let also  $\cE_{opt}(\sigma,h,r)$ denote the event that the bisection $\sigma$ satisfies $h$-stability for at least
one subset of $\floor{rn}$ vertices. 

We have:
\begin{equation}\label{eq:firs_mom-0}
    \E[X(z,z_1,z_2,h,r)] = \sum_{\sigma} \pr[ \mathcal{E}(\sigma,z,z_1,z_2)]\pr[\mathcal{E}_{opt}(\sigma,h,r)|\mathcal{E}(\sigma,z,z_1,z_2)]. 
\end{equation}

Next we bound the term
 $\pr[\cE_{opt}(\sigma,h,r)|\cE(\sigma,z,z_1,z_2)]$. 
 For this goal we perform 
 the computation associated with the 
 probability that a \emph{fixed} subset
 of $rn$ nodes satisfies $h$-stability.
 Given a subset $V_r\subset [n]$ 
 with cardinality $rn$, let
   $\cE_{opt}(\sigma, h,V_r)$ denote the event that all nodes in $V_r$ are 
   $h$-stable with respect to $\sigma$.
Let $r_1$n denote the cardinality of the intersection
of $V_r$ with the set of nodes $i$ with $\sigma(i)=1$.
Namely, it is the number of nodes from $V_r$ in the $+1$
part of the partition $\sigma$.
Similarly define $r_2n$.
We note that, by symmetry, $\pr[\cE_{opt}(\sigma,h,V_r)|E(\sigma,z,z_1,z_2)]$ is the same 
for any fixed $\sigma$ and any fixed values of $r_1$ and $r_2$. 
For any $r_1,r_2 \in \mathbb{R}$, denote by $\cE_{opt}(\sigma,h,r_1,r_2)$ 
the event $\cE_{opt}(\sigma,h,V_r)$ when $V_r$ consists
of the first $\floor{r_1n}$ nodes of $+1$ part of $\sigma$ and
first $\floor{r_2n}$ nodes of the $-1$ part of $\sigma$. 
Then, from the union bound, we obtain the following (when $n$ is even): 
\begin{equation}\label{eq:union_bound}
    \pr[
    \cE_{opt}(\sigma,h,r)|\mathcal{E}(\sigma,z,z_1,z_2)] \leq \sum_{r_1+r_2=r}\binom{{n\over 2}}{r_1n} \binom{{n\over 2}}{r_2n}  
    \pr[\cE_{opt}(\sigma,h,r_1,r_2)|\cE(\sigma,z,z_1,z_2)].
\end{equation}    
The corresponding expression when $n$ is odd is similar and omitted. 
Combining with (\ref{eq:X-conditioned-z-restricted}) and (\ref{eq:firs_mom-0}) we obtain:
\begin{align}
    \mathbb{E}[X(z,h,r)] 
    &\leq 2^n\sum_{(z_1,z_2)\in D(C,z)}  \pr(\cE(\sigma,z,z_1,z_2)) 
     \sum_{r_1+r_2=r}\binom{{n\over 2}}{r_1n} \binom{{n\over 2}}{r_2n}  
    \pr[\cE_{opt}(\sigma,h,r_1,r_2)|\cE(\sigma,z,z_1,z_2)] 
    \label{eq:firs_mom-1}\\
    &+ \exp(-Cn) \notag. 
\end{align}

Given a bisection $\sigma$
with parts $V_1,V_2$,
  let $O_1$ denote the event 
 that  the first $r_1n$ vertices in $V_1$ are $h$-stable. Similarly, let $O_2$ denote the event of that the first $r_2n$ vertices in $V_2$ are $h$-stable.
An important observation we use 
is that conditioned on the value 
of $E_{12}$, namely the number of edges
cut by $\sigma$, the events $O_1$ and $O_2$
are independent. This is because the
the event $O_1$ depends on the location
of the ``left'' half edges of $E_{12}$
and the event $O_2$ depends on the location
of the ``right'' half edges of $E_{12}$
and these are independent.
Therefore, we have:
\begin{equation}\label{eq:opt_prob}
\pr[\mathcal{E}_{opt}(\sigma,h,r_1,r_2)|\mathcal{E}(\sigma,z,z_1,z_2)]=\mathbb{P}[O_1]\mathbb{P}[O_2].
\end{equation}

Our next goal is to 
 compute $\mathbb{P}[O_1]$.
The event leading to $O_1$ 
can be described through the following generative model.
Fix non-negative integers 
$\mu_1,\mu_2,N \in \mathbb{N}$.
Consider the process
of assigning $\mu_1$ balls (half-edges)
into  $N$ bins independently. 
Denote the number of balls in bin
$i$ by $E_i$. Repeat the experiment
independently using $\mu_2$ balls 
with the same $N$ bins and denote the 
number of balls in bin $i$ in the second
round by $F_i$.
Given $\zeta\in [0,1]$ let
\begin{equation}
\label{eq:Kmu_mubar}
K(N,\mu_1,\mu_2,h,\zeta)\triangleq \mathbb{P}[ E_i \geq F_i+h\sqrt{d}, 1\leq i\leq \zeta N].
\end{equation}
Letting $E_i$ and $F_i$ denote
the out-degree and in-degree of the node
$i$ associated with bisection $V_1,V_2$,
we obtain
\begin{equation}\label{eq:p_optimal}
\pr[\mathcal{E}_{opt}(\sigma,h,r_1,r_2)|\mathcal{E}(\sigma,z,z_1,z_2,)]=K(n/2, zn, 2z_1 n,h,r_1) 
K(n/2 , zn, 2z_2 n,h,r_2). \notag
\end{equation}
Combining with (\ref{eq:firs_mom-1}), we obtain:
\begin{align}
    &\mathbb{E}[X(z,h,r)]  \notag\\
    &\le 
    2^n\sum_{(z_1,z_2)\in D(C,z)}  
    \pr(\cE(\sigma,z,z_1,z_2)) 
     \sum_{r_1+r_2=r}\binom{{n\over 2}}{r_1n} \binom{{n\over 2}}{r_2n}  
     K(n/2, zn, 2z_1 n,h,r_1) K(n/2 , zn, 2z_2 n,h,r_2) \label{eq:firs_mom}\\
     &+\exp(-Cn). \notag
\end{align}

We additionally note that for $r=1$, the union bound in (\ref{eq:union_bound}) is tight, leading to a corresponding lower-bound:

\begin{align}
    &\mathbb{E}[X(z,h,r)]  \notag\\
    &\ge 
    2^n\sum_{(z_1,z_2)\in D(C,z)}  
    \pr(\cE(\sigma,z,z_1,z_2)) 
     \sum_{r_1+r_2=r}\binom{{n\over 2}}{r_1n} \binom{{n\over 2}}{r_2n}  
     K(n/2, zn, 2z_1 n,h,r_1) K(n/2 , zn, 2z_2 n,h,r_2) \label{eq:firs_mom_lb}\\
     &+\exp(-Cn). \notag 
\end{align}

\subsection{Poissonization}
Next we turn to  the Poissonization technique.

\begin{lemma}
{\cite{durrett2019probability} (Exercise 3.6.13), \cite{coja2013upper}(Corollary 2.4)}
\label{lem:poissonization}
Consider the balls-into-bins model
above where $\mu$ balls are thrown into
$N$ bins. Let $E_i$ denote 
the number of balls assigned to the 
bin $i\in [N]$.  
Let $(B_i)_{i \in [N]}$ be a sequence of independent Poisson random variables with the mean $\mu/N$. 
Then for any sequence of non-negative integers $(t_i)_{i \in [N]}$, summing up to $\mu$, the  
joint distribution of $t_i$ 
satisfies
\begin{align}
\label{Poi_approx1}
\mathbb{P}[ E_i=t_i, 1\leq i\leq N] 
&=\mathbb{P}[ B_i=t_i, 1\leq i\leq N| 
\sum_{i=1}^{N} B_i=\mu]
=\Theta_{\mu}(\sqrt{\mu}) \mathbb{P}[ 
B_i=t_i, 1\leq i\leq N],
\end{align}
where the constant in $\Theta_\mu$ is 
universal.
\end{lemma}

The above lemma expresses the   fact that conditioned on the total number of balls, the joint probability of obtaining a particular distribution of the number of 
balls in each of the bins is 
approximately factorized.

Next, we introduce the event that the first $\zeta N$ bins simultaneously satisfy 
$h$-stability defined as follows:
\begin{align*}
S(\mu_1,\mu_2,h,\zeta)
=\{((t_i, s_i))_{1\leq i \leq N} \in 
(\mathbb{Z}_{\geq 0})^{2N}: t_i \geq 
s_i+h\sqrt{d}, 1\leq i \leq \zeta N; 
\sum_{i=1}^{N}t_i=\mu_1; 
\sum_{i=1}^{N}s_i=\mu_2\}.
\end{align*}
We have
\begin{align*}
 \mathbb{P}[E_i \geq F_i+h\sqrt{d},  1\leq 
 i\leq \zeta N]=\sum_{S(\mu_1,\mu_2,h,\zeta)} 
 \mathbb{P}[E_i=t_i, 1\leq i\leq 
 N]\pr[F_i=s_i, 1\leq i\leq N]\, .
\end{align*}
Using Lemma \ref{lem:poissonization}, the 
terms $\mathbb{P}[E_i=t_i, 1\leq i\leq N]$ 
and $\mathbb{P}[F_i=s_i, 1\leq i\leq N]$
can be approximated through sequences of 
independent Poisson random variables. 
Namely, we claim 
\begin{lemma}
\label{lemma_twoPo}
Let $(B_i)_{i \in [N]},(C_i)_{i \in [N]}$ be sequences of independent Poisson variables with the means $\mu_1/N$,$\mu_2/N$. Then, for any $h \in R,\zeta \in [0,1]$ following holds:
\begin{small}
\begin{align}
&K(N,\mu_1,\mu_2,h,\zeta) \nonumber\\
&=\Theta_{\mu_1}(\sqrt{\mu_1}) \Theta_{\mu_2}(\sqrt{\mu_2})
\mathbb{P}\left[\sum_{i=1}^{N} B_i=\mu_1, \sum_{i=1}^{N} C_i=\mu_2 \Bigm\vert B_i \geq C_i+h\sqrt{d}, 1\leq i\leq \zeta N \right]
(\mathbb{P}[B_1 \geq C_1 + h\sqrt{d}])^{\zeta N} 
\notag,
\end{align}
\end{small}
where $B_i,C_i$ denote independent Poisson variables with means $\frac{\mu_1}{N}$ and $\frac{\mu_2}{N}$ respectively. 
\end{lemma}

\begin{proof}
We have
\begin{align*}
&K(N,\mu_1,\mu_2,h,\zeta) \\
&=\sum_{S(\mu_1,\mu_2,h,\zeta)} 
 \mathbb{P}[E_i=t_i, 1\leq i\leq 
 N]\mathbb{P}[F_i=s_i, 1\leq i\leq N] \\
&=\Theta_{\mu_1}(\sqrt{\mu_1}) \Theta_{\mu_2}(\sqrt{\mu_2})
\sum_{S(\mu_1,\mu_2,h,\zeta)} 
 \mathbb{P}[B_i=t_i, 1\leq i\leq 
 N]\mathbb{P}[C_i=s_i, 1\leq i\leq N] \\
&=
\Theta_{\mu_1}(\sqrt{\mu_1}) \Theta_{\mu_2}(\sqrt{\mu_2})
\sum_{S(\mu_1,\mu_2,h,\zeta)} 
 \mathbb{P}[B_i=t_i, 1\leq i\leq 
 N, C_i=s_i, 1\leq i\leq N] \\
&=
\Theta_{\mu_1}(\sqrt{\mu_1}) \Theta_{\mu_2}(\sqrt{\mu_2})
 \pr[\sum_{1\le i\le N}B_i=\mu_1,\sum_{1\le i\le N}C_i=\mu_2,
 B_i\ge C_i+h\sqrt{d}, 1\le i\le \zeta N] \\
 &=
 \Theta_{\mu_1}(\sqrt{\mu_1}) \Theta_{\mu_2}(\sqrt{\mu_2})
 \pr[\sum_{1\le i\le N}B_i=\mu_1,\sum_{1\le i\le N}C_i=\mu_2|
 B_i\ge C_i+h\sqrt{d}, 1\le i\le \zeta N]\times \\
 &\times \pr[B_i\ge C_i+h\sqrt{d}, 1\le i\le \zeta N] \\
 &=
  \Theta_{\mu_1}(\sqrt{\mu_1}) \Theta_{\mu_2}(\sqrt{\mu_2})
 \pr[\sum_{1\le i\le N}B_i=\mu_1,\sum_{1\le i\le N}C_i=\mu_2|
 B_i\ge C_i+h\sqrt{d}, 1\le i\le \zeta N]\times \\
 &\times \left(\pr[B_1\ge C_1+h\sqrt{d}]\right)^{\zeta N}.
\end{align*}
\end{proof}

\ignore{
We denote the term $\mathbb{P}[B_1 \geq 
C_1 + h\sqrt{d}]$ by $P_1(h)$.
Thus the probability of optimality is 
reduced to computing the large deviation 
rate of a single variable (number of 
edges) under independent constraints on 
each vertex. 
This is achieved through generalizations 
of large deviation results in 
\citet{gamarnikMaxcutSparseRandom2018} to 
the case of $r < 1$ and $h \neq 0$.
}

\subsection{Large Deviations}

Lemma \ref{lemma_twoPo} relates the estimation of $K(N,\mu_1,\mu_2,h,\zeta)$ to large-deviation rates for centered and rescaled 2-dimensional
Poisson distributions considered above. In light of this, our next result provides  
large deviations rate for sums of 
independent variables arising from two 
distributions supported on lattices. 

The following is a known fact from the theory of large deviations, but we provide the proof 
for completeness.
\begin{lemma}\label{LCLTLDP}
Fix 
$b=(b_1,\ldots,b_k),
h=(h_1,\ldots,h_k)
\in 
\R^k, j=1,2.$
Consider a lattice $\mathcal{L}$  defined as:
\begin{align*}
    \mathcal{L} 
    &= 
    \{b_1+z_1h_1,
b_2+z_2h_2,
\cdots,b_k+z_kh_k, z_i 
\in  \mathbb{Z}, 1\le i\le k \}.
\end{align*}
Let $P_1$ and $P_2$ be two  discrete
probability measures supported 
on $\mathcal{L}$ with full-rank (positive definite) 
covariances. 
Let $M_1(\theta)=
\mathbb{E}_{P_1}[e^{\langle \theta,X \rangle}], 
M_2(\theta)=\mathbb{E}_{P_2}[e^{\langle \theta,X 
\rangle}]$ be the associated 
Moment Generating Functions, and let
$\Lambda_1(\theta) \triangleq \log M_1(\theta)$ and  $\Lambda_2(\theta) \triangleq \log M_2(\theta)$.
Suppose that $M_1(\theta) < \infty$ and $M_2(\theta) < \infty$ for all $\theta \in \mathbb{R}^d$.

Fix $\zeta\in [0,1]$.
Let $X_1, X_2, \dots, X_N$ be independent random vectors in $\mathbb{R}^k$, distributed as:
\begin{align*}
    X_i \stackrel{d}{=} \begin{cases}
    P_1, & 1 \leq i \leq \zeta N,\\
    P_2, & \text{otherwise},
    \end{cases}
\end{align*}
Suppose $y\in\R^k$
 is such that for all $N$, $Ny$
is of the form $\sum_{j\in [N]}x_j$ for some $x_j\in\mathcal{L}, j\in [N]$. (Namely $Ny$ is a sum
of $N$ lattice elements).
Suppose furthermore that   
$\theta^*\in \R^k$ satisfies
\begin{align}
    y=\zeta\nabla \Lambda_1(\theta^*)
+(1-\zeta)\nabla\Lambda_2(\theta^*).\label{eq:y}
\end{align}

Then the
following large deviations result holds for $S_N = \sum_{i=1}^N X_i$:
\begin{align}
\lim_{N\rightarrow \infty} \frac{1}{N} \log \mathbb{P}[S_N=y N] 
&=-\langle \theta^{*},y \rangle+\zeta\Lambda_1(\theta^*)+(1-\zeta)\Lambda_2(\theta^*). \label{LLPD1}
\end{align} 
\end{lemma}

\begin{proof}
\ignore{
\dg{Check that not having full support in $\mathcal{L}$
is not an issue} \yd{Full-support isn't an issue yet since we assume $Ny$ is sum over elements of the lattice and have non-generate covariance. Durret only has the result for single-variable case but I can look for another reference.}
}

% \vspace{.1in}OLD STUFF
% The function $g$ is a known  to have a property that of diverging to $\infty$ for any
% sequence $\theta_r\to\infty$~\citep{dembo2009large}.
% \yd{I think this further requires $y$ to be in the support. For instance, the log MGF of the Poisson distribution doesn't diverge to infty for $\theta \rightarrow -\infty$ when $y$ is negative}\dg{You are probably right. Do you have an idea how to fix this part of the proof?}
% \yd{I couldn't find a reference discussing a general case but as long as there is positive probability mass for any quadrant around y it will diverge to infinity since since we just need to cinsider the restriction of the MGF to $\langle x-y, \theta \rangle > 0$. This will always be the case for Poisson random variables}
% It is also known to be differentiable in $\theta$ and strictly convex. 
% Thus the optimal value is achieved by the unique point $\theta^*$ satisfying the first order condition $\nabla g=0$.
% The claim is obtained by changing the order of differentiation and expectation which is also justified since the Moment
% Generating Functions are finite for all $\theta$.\\
% END OLD STUFF

To  prove  (\ref{LLPD1}).
introduce the probability measure $\tilde{P_1}$  through the following exponential tilting
\begin{equation}\label{eq:measure_change_1}
    \frac{d\tilde{P_1}}{dP_1}(z)=e^{\langle \theta^{*}, z \rangle-\Lambda_1(\theta^{*})}, \qquad z\in \mathcal{L}.
\end{equation}
Namely,  each point $z\in \mathcal{L}$ is associated with probability mass 
$e^{\langle \theta^{*}, z \rangle-\Lambda_1(\theta^{*})}P_1(z)$. 
Analogously,  let $\tilde{P_2}$ be the measure defined by:
\begin{equation}\label{eq:measure_change_2}
    \frac{d\tilde{P_2}}{dP_2}(z)=e^{\langle \theta^{*}, z \rangle-\Lambda_2(\theta^{*})}, \qquad z\in \mathcal{L}.
\end{equation}

\ignore{
Since $P_1,P_2$ are discrete, the above tilting corresponds to reweighing the probability
mass of each point in the corresponding lattice. 
Therefore, equations \ref{eq:measure_change_1}, \ref{eq:measure_change_2} define the measures $\tilde{P_1}$, $\tilde{P_2}$ uniquely. 
Furthermore, $\tilde{P_1},\tilde{P_2}$ are probability measures since:
\begin{equation}
\begin{split}
    \int_{\mathbb{R}^d} d \tilde{P_1} &= \int_{\mathbb{R}^d} e^{\langle \theta^{*}, z \rangle-\Lambda_1(\theta^{*})} d P_1=\frac{1}{M_1(\theta^{*})}\int_{\mathbb{R}^d} e^{\langle \theta^{*}, z \rangle}  dP_1=1\\
    \int_{\mathbb{R}^d} d \tilde{P_2} &= \int_{\mathbb{R}^d} e^{\langle \theta^{*}, z \rangle-\Lambda_1(\theta^{*})} d P_1=\frac{1}{M_2(\theta^{*})}\int_{\mathbb{R}^d} e^{\langle \theta^{*}, z \rangle}  dP_2=1
\end{split}
\end{equation}
}

We note that under the measure changes defined by Equations \ref{eq:measure_change_1} and \ref{eq:measure_change_2}, the means of the distributions can be evaluated as follows.
Letting $\tilde X\stackrel{d}{=}P_j$
we have 
\begin{align*}
m_1\triangleq\mathbb{E}_{P_1}[\tilde{X}]=\frac{1}{M_1(\theta^{*})}\int_{\mathbb{R}^d} z e^{\langle \theta^{*}, z \rangle}  d P^{(1)} &= \nabla \Lambda_1(\theta^{*}),\\
m_2\triangleq \mathbb{E}_{P_2}[\tilde{X}]=\frac{1}{M_2(\theta^{*})}\int_{\mathbb{R}^d} z e^{\langle \theta^{*}, z \rangle}  d P^{(2)} &= \nabla \Lambda_2(\theta^{*}).
\end{align*}
In particular, we have 
$\zeta m_1+(1-\zeta)m_2=\zeta\nabla \Lambda_1(\theta^{*})+(1-\zeta)\nabla \Lambda_2(\theta^{*})=y$.

\ignore{
Let $\tilde{X_1}, \tilde{X_2}, \dots, \tilde{X_n}$ be independent random vectors distributed according to the above measure changes:
\begin{align*}
    \tilde{X_i} \sim \begin{cases}
    \tilde{P_1}, & 1 \leq i \leq \zeta n,\\
    \tilde{P_2}, & \text{otherwise}.
    \end{cases}
\end{align*}
}

\ignore{
Since an exponential tilting of a measure doesn't affect its support, the measures $\tilde{\mu}^{(1)}$ and  $\tilde{\mu}^{(2)}$ still possess the biases and periods given by $b_i^{(1)}, h_i^{(1)}$ and $b_i^{(2)}, h_i^{(2)}$ respectively.
Let $\tilde{S}^{(1)}_n = \sum_{i=1}^{rn} \tilde{X_i}$ and $\tilde{S}^{(2)}_n = \sum_{i=rn+1}^n \tilde{X_i}$. Let $\tilde{S}^{(1)}_n = \sum_{i=1}^{n} \tilde{X_i} = \tilde{S}^{(1)}_n + \tilde{S}^{(2)}_n$.
We next define the following three measures:
\begin{align*}
    \tilde{P}_n(\{x\})&=\mathbb{P}[\tilde{S}_n/n=x],\\
    \tilde{P}^{(1)}_n(\{x\})&=\mathbb{P}[(\tilde{S}^{(1)}_n-rnm_1)/\sqrt{rn}=x],\\
    \tilde{P}^{(2)}_n({x})&=\mathbb{P}[(\tilde{S}^{(2)}_n-(1-r)nm_2)/\sqrt{n-rn}=x].
\end{align*}
}

We have:
\begin{align*}
\mathbb{P}[S_N/N=y]=&  
\int_{\sum_{i=1}^N z_i=y N} \prod_{i=1}^{\zeta N}P_1(z_i)\prod_{j=\zeta N+1}^{N}P_2(z_j)\nonumber \\
=& \int_{\sum_{i=1}^N z_i=y N}
e^{-\sum_{i=1}^N \langle \theta^{*}, z_i \rangle+N(\zeta\Lambda_1(\theta^{*})+(1-\zeta)\Lambda_2(\theta^{*}))}  \prod_{i=1}^{\zeta N}\tilde P_1(z_i)
\prod_{j=\zeta N+1}^{N}\tilde{P_2}(z_j) \\
&=
e^{- \langle \theta^{*}, y N\rangle
+N(\zeta\Lambda_1(\theta^{*})+(1-\zeta)\Lambda_2(\theta^{*}))}
\int_{\sum_{i=1}^N z_i=y N}
  \prod_{i=1}^{\zeta N}\tilde P_1(z_i)
\prod_{j=\zeta N+1}^{N}\tilde{P_2}(z_j),
\end{align*}
further implying 
\begin{align*}
{1\over N}\log\mathbb{P}[S_N/N=y]
&=
- \langle \theta^{*}, y \rangle
+\zeta\Lambda_1(\theta^{*})+(1-\zeta)\Lambda_2(\theta^{*}) \\
&+{1\over N}\log 
\int_{\sum_{i=1}^N z_i=y N}
  \prod_{i=1}^{\zeta N}\tilde P_1(z_i)
\prod_{j=\zeta N+1}^{N}\tilde{P_2}(z_j).
\end{align*}

Let $\tilde S_1$ be a sum of $\zeta N$ i.i.d. random variables distributed as $\tilde P_1$,
which we recall has mean value $m_1$.
Define $\tilde S_2$ similarly as sum of $(1-\zeta)N$ i.i.d.
random variables distributed 
as $\tilde P_2$. 
Then
\begin{align*}
{1\over N}\log\mathbb{P}[S_N/N=y]=
- \langle \theta^{*}, y \rangle
+\zeta\Lambda_1(\theta^{*})+(1-\zeta)\Lambda_2(\theta^{*})
+{1\over N}\log 
\sum_{t_1,t_2\in \mathcal{L}: t_1+t_2=y N} 
\pr(\tilde S_1=t_1)\pr(\tilde S_2=t_2).
\end{align*}
The proof of the lemma will be completed upon 
verifying the limit ${1\over N}\log(\cdot)\to 0$
on the right-hand side,
which we do next.

For this purpose,
let $\tilde\Sigma_1,\tilde\Sigma_2$ be the covariance matrices of 
measures $\tilde P_1,\tilde P_2$ respectively. 
Since covariances of measure $P_1,P_2$ are full
rank, the same applies to $\tilde\Sigma_1,\tilde\Sigma_2$.
Indeed, a degenerate covariance of $\tilde\Sigma_j$ 
would imply an almost sure linear relation between
components of $X\stackrel{d}{=}\tilde P_j$ which 
would imply the same for $X\stackrel{d}{=}P_j$ 
since two measures are continuous with respect to each
other.

Let $\tilde{\phi}_1(x),\tilde{\phi}_2(x)$ be the densities 
of normal variables with means $0$ and
covariances $\tilde \Sigma_1,\tilde\Sigma_2$ 
respectively. The non-degeneracy of $\tilde \Sigma_j$
allows us to invoke the local 
Central Limit Theorem (CLT)~\citep{durrett2019probability} and obtain for 
$j=1,2$
\begin{align}
\lim_N
\sup_{z\in \mathcal{L}}
\Big| \frac{(\zeta_j N)^{k\over 2}}{\prod_{i=1}^k h_i} 
\pr\left(\tilde S_j=z\right)-
\tilde\phi_j\left({z-m_j\zeta_j N\over \sqrt{\zeta_j N}}\right) \Big|
=0, \label{eq:LCLT}
\end{align}
where $\zeta_1 = \zeta, \zeta_2 = (1-\zeta)$.

Recall that 
$y=\zeta m_1+(1-\zeta)m_2$.
We find now any sequence
$\bar t_1^N\in \mathcal{L}$
 within any constant ($N$-independent)  distance from 
$\zeta m_1 N$.
Then $(\bar t_1^N-\zeta m_1 N)/\sqrt{N}\to 0$
and therefore 
\begin{align*}
\lim_N\tilde \phi_1\left({\bar t_1^N-\zeta m_1 N\over \sqrt{N}}\right)=\tilde\phi_1(0)>0.
\end{align*}
Applying the local CLT (\ref{eq:LCLT})
\begin{align*}
\lim_N
\Big| (\zeta N)^{k\over 2} 
\pr\left(\tilde S_j=\bar t_1^N\right)-
\tilde\phi_j(0) \Big|
=0.
\end{align*}
The strict positivity  of $\tilde\phi_1(0)$ implies
This implies
\begin{align*}
\lim_N {1\over N}\log\pr\left(\tilde S_j=\bar t_1^N\right)
=0.
\end{align*}
Next we let $\bar t_2^N=y N-\bar t_1^N$, 
which we observe belongs to $\mathcal{L}$.
Also,
\begin{align*}
\bar t_2^N &=yN-\bar t_1^N \\
&= (1-\zeta)m_2 N-\left(\bar t_1^N-\zeta m_1 N\right).
\end{align*}
Then 
\begin{align*}
\lim_N {\bar t_2^N-(1-\zeta) m_2 N
\over \sqrt{N}}=0.
\end{align*}
Thus, for a similar reason we obtain
\begin{align*}
\lim_N {1\over N}\log\pr
\left(\tilde S_j=\bar t_2^N\right)
=0.
\end{align*}
We have 
\begin{align*}
1\ge 
\sum_{t_1,t_2\in \mathcal{L}: t_1+t_2=y N} 
\pr(\tilde S_1=t_1)\pr(\tilde S_2=t_2))
\ge 
\pr
\left(\tilde S_1=\bar t_1^N\right)
\pr
\left(\tilde S_j=\bar t_2^N\right).
\end{align*}
Taking $\log$ of both sides, dividing by $N$
and taking the limit $N\to\infty$
we obtain
\begin{align*}
\lim_{N \rightarrow \infty}
{1\over N}\log \sum_{t_1,t_2\in \mathcal{L}: t_1+t_2=y N} 
\pr(\tilde S_1=t_1)\pr(\tilde S_2=t_2)
=0,
\end{align*}
as claimed.
\end{proof}

We now derive a corollary of the local large deviations
fact above to the case of two dimensional Poisson 
distributions.

\begin{corollary}\label{coro:Poisson-local-LD}
Suppose $Z_1,Z_2$ are two indpendent Poisson random
variables with positive integer parameters $\lambda_1,\lambda_2$
respectively. Let $\lambda=(\lambda_1,\lambda_2)$.
Fix $a\ge 0$ and $\zeta \in (0,1]$ and suppose that $\lambda_1,\lambda_2$ are strictly positive integers satisfying either of the following conditions:
\begin{enumerate}
    \item $\zeta=1$ and $\lambda_1\ge \lambda_2+a+2$.
    \item $\zeta < 1$, $\min(\lambda_i, \lambda_i/a^2) \geq C$ for $i = 1,2$,
\end{enumerate}
for a sufficiently large constant $C > 0$.
.

Let $P_1$ be the probability distribution of 
$\left({Z_1-\lambda_1\over \sqrt{\lambda_1}},
{Z_2-\lambda_2 \over\sqrt{\lambda_2}}\right)$ conditioned on $Z_1\ge Z_2+a$ 
and let $P_2$ be the probability distribution of the same
pair without conditioning. 
Suppose $X_i\stackrel{d}{=} P_1$ i.i.d. 
$1\le i\le \zeta N$,
and $X_i\stackrel{d}{=} P_2$ i.i.d. $\zeta N<i\le N$.
Let $S_N=\sum_{i\in [N]}X_i$ and let $\Lambda_j$ for $ j=1,2$ denote the log-Moment Generating
Function of $P_j$. 

Then, the function $g(\theta)= \zeta \Lambda_1(\theta)+(1-\zeta) \Lambda_2(\theta)$ admits a unique minimizer $\theta^*\in \R^2$:
\begin{align}
    \theta^* \coloneqq \argmin_\theta \left( \zeta \Lambda_1(\theta)+(1-\zeta) \Lambda_2(\theta) \right),
\end{align}
characterized by the following first-order conditions:
\begin{align}
    \zeta\nabla \Lambda_1(\theta^*)
+(1-\zeta)\nabla\Lambda_2(\theta^*)=0.
\label{eq:y-target}
\end{align}
Furthermore,
\begin{align}
\lim_N {1\over N}\log 
\pr\left(S_N=0\right)
&= \inf_{\theta \in \mathbb{R}}\zeta\Lambda_1(\theta)+(1-\zeta)\Lambda_2(\theta)\\
&=\zeta\Lambda_1(\theta^*)+(1-\zeta)\Lambda_2(\theta^*).
\label{eq:lD-rate-Poisson}
\end{align}
\end{corollary}

\begin{proof}
The Poisson distribution
assumption also implies that the Moment Generating
Functions of $P_1,P_2$ are finite for all $\theta$.
The uniqueness of the solution $\theta^*$ of (\ref{eq:y-target}) follows
from the well-known fact of 
strict convexity of the log-moment 
generating function 
$\zeta\Lambda_1(\theta)+
(1-\zeta)\Lambda_2(\theta)$. 
We now turn to the existence. Consider
the variational problem
\begin{align*}
\inf_\theta \zeta\Lambda_1(\theta)+
(1-\zeta)\Lambda_2(\theta).
\end{align*}
We claim that
\begin{align}\label{eq:theta-B}
\liminf_{B\to\infty}
{\inf_{\theta: \|\theta\|\ge B}} \zeta\Lambda_1(\theta)+
(1-\zeta)\Lambda_2(\theta)=\infty.
\end{align}
This implies that this variational
problem is solved by points 
$\theta^*$ satisfying
$\|\theta\|\le B$ for large enough $B$,
and thus by setting the gradient
to zero, completing the proof of
the existence of $\theta^*$ satisfying
(\ref{eq:y-target}).

To prove the claim (\ref{eq:theta-B}), we 
first consider the case $\zeta=1$, where we recall that we further assume $\lambda_1 \geq \lambda_2 + a +2$.

The assumptions $\lambda_1 \geq \lambda_2 + a +2$ and $\lambda_1, \lambda_2 \in \mathbb{Z}$ with  $\lambda_1, \lambda_2 >0 $ imply that $Z_1, Z_2=(\lambda_1+1, \lambda_2+1), (\lambda_1+1, \lambda_2-1), (\lambda_1-1, \lambda_2+1), (\lambda_1-1, \lambda_2-1)$ satisfy $Z_1 \geq Z_2 +a$ and
therefore
the
following points lie in the support of $P_1$:
\begin{equation}\label{eq:support_points}
    v_1 = (\frac{1}{\sqrt{\lambda}_1},\frac{1}{\sqrt{\lambda}_2}), \  v_2 = (\frac{1}{\sqrt{\lambda}_1},-\frac{1}{\sqrt{\lambda}_2}), v_3 = (-\frac{1}{\sqrt{\lambda}_1},\frac{1}{\sqrt{\lambda}_2}),
    v_4 = (-\frac{1}{\sqrt{\lambda}_1},-\frac{1}{\sqrt{\lambda}_2}).
\end{equation}
It is easy to check that the following relation holds:
\begin{equation}
    \cup_{i=1}^4 \{\theta: \langle v_i, \theta \rangle > 0\} = \mathbb{R}^2 \backslash 0.
\end{equation}
Define:
\begin{equation}
   k^\star = \inf_{\theta: \norm{\theta}=1} \max_i (\langle v_i, \theta \rangle).
\end{equation}
Then, by the compactness of the set $\norm{\theta}=1$ and continuity of $\max_i (\langle v_i, \theta \rangle)$, we have $k^\star > 0$. We exploit this to obtain a lower-bound on $\Lambda_1(\theta)$ as follows:
\begin{align*}
\Lambda_1(\theta)&=\log \left(\sum_{x}\exp\left(\langle \theta, x \rangle \right)P_1(X=x)\right) \\
&\ge
\log[\exp(k^\star \norm{\theta})[\min_{i} P_1(X=v_i)]],
\end{align*}
where $X \sim P_1$. Since $k^\star > 0$, we obtain that $\Lambda_1(\theta) \rightarrow \infty$ whenever $\norm{\theta} \rightarrow \infty$.

% % To prove the claim (\ref{eq:theta-B}) we take advantage of the assumption $\zeta<1$.
% We have
% \begin{align*}
% \Lambda_1(\theta)&=\log \left(\sum_{k_1\ge k_2+a}\exp\left(\theta_1 (k_1-\lambda_1)/\sqrt{\lambda_1}
% +\theta_2 (k_2-\lambda_2)/\sqrt{\lambda_2}\right)\pr(Z_1=k_1,Z_2=k_2|Z_1\ge Z_2+a)\right) \\
% &\ge
% \log\pr(Z_1=\lambda_1,Z_2=\lambda_2|Z_1\ge Z_2+a).
% \end{align*}
% We note that the right-hand is finite since $\lambda_1\ge \lambda_2+a$ and does not depend on $\theta_j, j=1,2$.

Next, consider the case $\zeta< 1$, where we further assumed $\min(\lambda_i, \lambda_i^2/a) \geq C$, for $i=1,2$. First, by Jensen's inequality, $\Lambda_1(\theta)$ is lower-bounded as follows:
\begin{equation}\label{eq:jensen}
\begin{split}
    \Lambda_1(\theta) &= \log \Eb{X\sim P_1}{\exp(\langle \theta, X \rangle)}\\
    & \geq \Eb{X\sim P_1}{\langle \theta, X \rangle}\\
    &=\theta_1 \Eb{X\sim P_1}{X_1}+\theta_2 \Eb{X\sim P_1}{X_2}.
\end{split}
\end{equation}

Next, we note that $\Lambda_2(\theta)$ is given as:
\begin{align*}
\Lambda_2(\theta)&=\sum_{j=1,2}\log \E\exp\left(\theta_j(Z_j-\lambda_j)/\sqrt{\lambda_j}\right) \\
&=\sum_{j=1,2}\log\exp\left(\lambda_je^{\theta_j\over \sqrt{\lambda_j}}
-\theta_j\sqrt{\lambda_j}-\lambda_j\right) \\
&=\sum_{j=1,2}\left(\lambda_je^{\theta_j\over \sqrt{\lambda_j}}
-\theta_j\sqrt{\lambda_j}-\lambda_j\right).
\end{align*}

Now, suppose that either $\theta_1 >0$ or $\theta_2 >0$. The exponential term $\lambda_je^{\theta_j\over \sqrt{\lambda_j}}$ combined with the lower-bound in (\ref{eq:jensen}) then ensures that:
\begin{equation}
\zeta\Lambda_1(\theta)+
(1-\zeta)\Lambda_2(\theta) \rightarrow \infty.
\end{equation}
For the case when both $\theta_1<0, \theta_2 <0$, we will leverage the condition $\lambda_i, \lambda_i/a^2 \geq C$ for $i=1,2$.

By Cauchy-Schwartz, we obtain the following bounds on $\Eb{X\sim P_1}{X_1},  \Eb{X\sim P_1}{X_2}$:
\begin{align}
  \abs{\Eb{X\sim P_1}{X_1}} 
  &= \abs{\frac{\Ea{\frac{Z_1-\lambda_1}{\sqrt{\lambda_1}}\mathbf{I}_{Z_1 \geq Z_2 + a}}}{\pr[Z_1 \geq Z_2 + a]}} \nonumber\\
  & \leq \frac{1}{{\pr[Z_1 \geq Z_2 + a]}}{\Ea{(\frac{Z_1-\lambda_1}{\sqrt{\lambda_1}})^2}^{1/2}\Ea{(\mathbf{I}_{Z_1 \geq Z_2 + a})^2}}^{1/2}\nonumber\\  
  & \leq \frac{1}{{\pr[Z_1 \geq Z_2 + a]}^{1/2}}, \label{eq:boundp1}
\end{align}

where in the last line, we used that $\Ea{(\frac{Z_1-\lambda_1}{\sqrt{\lambda_1}})^2}^{1/2}=1$ since $\frac{Z_1-\lambda_1}{\sqrt{\lambda_1}} \sim \mathcal{N}(0,1)$
and
$\Ea{(\mathbf{I}_{Z_1 \geq Z_2 + a})^2} \leq 1$.
From the Poisson-Normal approximation, as  $\lambda_1, \lambda_2 \rightarrow \infty$ the sequence of distributions of $(\frac{Z_1-\lambda_1}{\sqrt{\lambda_1}}, \frac{Z_2-\lambda_2}{\sqrt{\lambda_2}})$ converges to standard normal variables. The condition $a \leq C(\max (\sqrt{\lambda_1}, \sqrt{\lambda_2}))$, further implies that $\pr[Z_1 \geq Z_2 + a]$ remains lower-bounded by a constant.

Therefore, there exist constants $c_1 > 0, K > 0$ such that for $\lambda_i, \lambda_i/a^2 \geq K$, $\pr[Z_1 \geq Z_2 + a] > 
\frac{1}{c_1}$. The upper bound in (\ref{eq:boundp1}) then implies:
\begin{equation}
    \abs{\Eb{X\sim P_1}{X_1}} \leq c_1. 
\end{equation}
Similarly, we obtain:
\begin{equation}
    \abs{\Eb{X\sim P_1}{X_2}} \leq c_2, 
\end{equation}
for some constant $c_2 > 0$.

The above bounds along with (\ref{eq:jensen}) yield:
\begin{equation}\label{eq:lam_1_low_bound}
    \Lambda_1(\theta) \geq -\abs{\theta_1 c_1}-\abs{\theta_2 c_2}.
\end{equation}

Combining with the expression for $\Lambda_2(\theta)$ then results in:
\begin{equation}\label{eq:zeta}
\zeta\Lambda_1(\theta)+
(1-\zeta)\Lambda_2(\theta) \geq \sum_{j=1,2} -\theta_j\sqrt{\lambda_j} +\theta_j c_j.
\end{equation}

Now, setting $C> \max(K, c_1^2, c_2^2)$, we obtain that the condition  $\min(\lambda_i, \lambda_i/a^2) \geq C$ for $i=1,2$ further implies that
 $\sqrt{\lambda_i} > \abs{c}_i$ for $i=1,2$. We  therefore, obtain that the coefficient of $\theta_i$ in (\ref{eq:zeta}) for $i=1,2$ is strictly negative. Hence, $\zeta\Lambda_1(\theta)+
(1-\zeta)\Lambda_2(\theta) \rightarrow \infty$ as $\theta_1, \theta_2 \rightarrow -\infty$. This proves the claim in (\ref{eq:theta-B}).

Given the existence of $\theta^\star$, consider now the $k=2$-dimensional lattice
with $b_j=-\sqrt{\lambda_j}, h_j=1/\sqrt{\lambda_j}, j=1,2$,
so that $(z-\lambda_j)/\sqrt{\lambda_j}$ is the element
of this lattice for each integer $z$. 
Note that  any sequence of integers $z_i, i\in [N]$
satisfying $\sum_{1\le i\le N} z_i=\lambda_j N$ 
(which exists by integrality of $\lambda_j$) also satisfies 
$\sum_{1\le i\le N} (z_i-\lambda_j)/\sqrt{\lambda_j}=0$,
and therefore $0$  can be represented as a sum of some $N$ lattice elements. Thus we are in the setting 
of Lemma~\ref{LCLTLDP}. (\ref{eq:lD-rate-Poisson}) then
follows by the conclusion of the lemma.
\end{proof}

\subsection{Derivation of the Asymptotic Rate Function}
Our next goal is to obtain an asymptotic
form of Corollary~\ref{coro:Poisson-local-LD} when
the means of the two Poisson are growing, and use bi-variate
Gaussian approximations for the large deviations
rate functions. This will be of relevance to our 
analysis of $K(\cdot)$ term in (\ref{eq:Kmu_mubar}).

Let $Y=(Y_1,Y_2)$ be a pair of independent 
standard normal random 
variables. Fix a constant $c$ and consider the
distribution of $Y$ conditioned on the event
$Y_1\ge Y_2+c$. Let $\Lambda_1(\theta)=
\log \E[\exp(\langle 
Y,\theta\rangle)]$ 
denote the log-MGF
of this conditional distribution of $Y$. 
Recall that $\Phi$ stands for the Cumulative
Distribution Function of a standard normal
random variable.

\begin{lemma}\label{lemma:conditional-log-MGF}
The following identity holds
\begin{align}\label{eq:conditional-log-MGF}
\Lambda_1(\theta)=-\log 
\Phi\left(-c/\sqrt{2}\right)
+{1\over 2}\theta_1^2+{1\over 2}\theta_2^2
+\log \Phi\left(-{c+\theta_2-\theta_1\over \sqrt{2}}\right).
\end{align}
\end{lemma}

\begin{proof} 
We have
\begin{align*}
\E[&\exp\left(\theta_1 Y_1+\theta_2 Y_2\right) | Y_1>Y_2+c] \\
&=\pr^{-1}(Y_1>Y_2+c)
{1\over 2\pi}\int_{t_1>t_2+c}
\exp\left(\theta_1 t_1+\theta_2 t_2 -{1\over 2}t_1^2-{1\over 2}t_2^2\right)
dt_1dt_2 \\
&=
\pr^{-1}(Y_1>Y_2+c)
\exp\left({1\over 2}\theta_1^2+{1\over 2}\theta_2^2\right)
{1\over 2\pi}\int_{t_1>t_2+c}
\exp\left( -{1\over 2}(t_1-\theta_1)^2-{1\over 2}(t_2-\theta_2)^2\right)
dt_1dt_2 \\
&=
\pr^{-1}(Y_1>Y_2+c)
\exp\left({1\over 2}\theta_1^2+{1\over 2}\theta_2^2\right)
{1\over 2\pi}\int_{s_1>s_2+c+\theta_2-\theta_1}
\exp\left( -{1\over 2}s_1^2-{1\over 2}s_2^2\right)
dt_1dt_2.
\end{align*}
We simplify $\pr(Y_1>Y_2+c)$ as $\pr(Y_1<-c/\sqrt{2})$.
We recognize the integral as 
\begin{align*}
\pr(Y_1>Y_2+c+\theta_2-\theta_1)=
\pr\left(Y_1<-{c+\theta_2-\theta_1\over \sqrt{2}}\right).
\end{align*}
Taking the log of both sides we obtain the result.

\end{proof}

Corollary \ref{coro:Poisson-local-LD} and Lemma \ref{lemma:conditional-log-MGF} yield the final $N$-independent asymptotics for the large deviation rates appearing in $K(\cdot)$. This is achieved in the subsequent lemma.

\begin{lemma}\label{lemma:Gaussian-of-Pois}
Consider the setting of  
Corollary~\ref{coro:Poisson-local-LD}. 
Fix any $\tau_1>\tau_2, \tau_1,\tau_2\in\R, b\in \R$ such that, either:
\begin{enumerate}
    \item $\zeta=1$, $\tau_1-\tau_2>b$.
    \item $0 \leq \zeta < 1$.
\end{enumerate}
Suppose that 
$\lambda(t)=(\lambda_j(t), j=1,2)$ 
are sequences of rates 
indexed by $t\in\Z_+$
satisfying the following 
\begin{enumerate}

\item $\lambda_j(t)=t+\tau_j\sqrt{t}+o_t(\sqrt{t}), 
j=1,2$, as $t\to\infty$.

\item $\lambda_j(t), j=1,2$ are integers.

\end{enumerate}

Suppose further that $a(t)=b\sqrt{t}+o(\sqrt{t})$.
Let $S_N$ and $a=a(t)$ stand for the same quantities as 
in Corollary~\ref{coro:Poisson-local-LD}.
Then the function $\theta^2
+\zeta\log\Phi\left(-{\tau_2-\tau_1+b-2\theta\over\sqrt{2}}\right)$ admits a unique minimizer $\bar{\theta}$ characterized through the following first-order conditions:
\begin{align}\label{eq:fixed-point-theta}
\bar{\theta}-
{\zeta\over\sqrt{2}}{\dot \Phi\left(-{\tau_2-
\tau_1+b-2\bar{\theta}\over \sqrt{2}}\right)
\over \Phi\left(-{\tau_2-
\tau_1+b-2\bar{\theta}\over \sqrt{2}}\right)}=0,
\end{align}
Furthermore,
\begin{align}
&\lim_{t\to\infty}\lim_{N\to\infty} 
{1\over N}\log 
\pr\left(S_N=0\right) \notag\\
&=
-\zeta\log \Phi\left(-{\tau_2-
\tau_1+b\over\sqrt{2}}\right)+
\inf_{\theta} \left( \theta^2
+\zeta\log\Phi\left(-{\tau_2-\tau_1+b-2\theta\over\sqrt{2}}\right) \right) \notag\\
&=
-\zeta\log \Phi\left(-{\tau_2-
\tau_1+b\over\sqrt{2}}\right)+
 \bar{\theta}^2
+\zeta\log\Phi\left(-{\tau_2-\tau_1+b-2\bar{\theta}\over\sqrt{2}}\right) \label{eq:S_N-limit}\\
\end{align}

\end{lemma}

\begin{proof}
The property 
$\lambda_j(t)\to\infty$ as $t\to\infty$
implies 
that $\left({Z_1-\lambda_1\over \sqrt{\lambda_1}},
{Z_2-\lambda_2 \over\sqrt{\lambda_2}}\right)$ 
converges distributionally to $(Y_j,j=1,2)$ which is a 
pair of independent standard normal random variables.
The condition $Z_1\ge Z_2+a(t)$ is equivalent to
\begin{align*}
{Z_1-\lambda_1(t)\over \sqrt{\lambda_1(t)}}
&\ge
{\sqrt{\lambda_2(t)}\over \sqrt{\lambda_1(t)}}
{Z_2-\lambda_2(t)\over \sqrt{\lambda_2(t)}}
+{\lambda_2(t)-\lambda_1(t)+a(t) \over\sqrt{\lambda_1(t)}} \\
&=
(1+o_t(1)){Z_2-\lambda_2(t)\over \sqrt{\lambda_2(t)}}
+(1+o_t(1))(\tau_2-\tau_1+b). 
\end{align*}
Thus conditional on the events  $Z_1\ge Z_2+a(t)$
the sequence of distributions of 
 $\left({Z_1-\lambda_1\over \sqrt{\lambda_1}},
{Z_2-\lambda_2 \over\sqrt{\lambda_2}}\right)$ 
converges distributionally to $(Y_j,j=1,2)$ conditioned
on $Y_1\ge Y_2+\tau_2-\tau_1+b$.
This implies that the sequence of log-MGFs 
$\Lambda_1(t)$ of these 
pairs converges to the log-MGF $\Lambda_1$
of $(Y_j,j=1,2)$ conditioned
on $Y_1\ge Y_2+\tau_2-\tau_1+b$.
Similarly, the log-MGFs $\Lambda_2(t)$
of the  sequence
$\left({Z_1-\lambda_1\over \sqrt{\lambda_1}},
{Z_2-\lambda_2 \over\sqrt{\lambda_2}}\right)$ 
converges to the log-MGF $\Lambda_2$ of $(Y_1,Y_2)$ 
without conditioning, which is simply
$\Lambda_2(\theta)={1\over 2}(\theta_1^2+\theta_2^2)$.

Let $\theta^*(t)$ be the unique solution 
of (\ref{eq:y-target}) for $\Lambda^{(t)}_1(\theta),\Lambda^{(t)}_2(\theta)$,
associated with conditional and unconditional
pairs Poisson random variables 
$\left({Z_1-\lambda_1(t)\over \sqrt{\lambda_1(t)}},
{Z_2-\lambda_2(t)\over \sqrt{\lambda_2(t)}}\right)$,
where existence and uniqueness is verified by Corollary~\ref{coro:Poisson-local-LD}.

We claim that there exists a unique  solution $\theta^*$
of (\ref{eq:y-target}) for $\Lambda_1,\Lambda_2$,
associated with conditional and unconditional
pairs of standard normals $(Y_1,Y_2)$. As before, the uniqueness follows
from strict convexity. The existence is established
by setting $\theta^*$ as a solution of the variational problem
$\inf_\theta \zeta \Lambda_1(\theta)+(1-\zeta)\Lambda_2(\theta)$. Similarly to the Poisson distribution case earlier, we claim 
that the optimum of this problem is obtained by the first order condition on the gradient thus implying the existence. 
%For this purpose
% it suffices to establish
% $\liminf_{\theta:\|\theta_2\|_2\to\infty}\Lambda_j(\theta)=\infty, j=1,2.$

% For the case $j=1$ this follows immediately
% since $\Lambda_2=(1/2)\|\theta\|_2^2$.

First, consider the case $\zeta=1$.  Set $c=\tau_2-\tau_1+b$, which is negative by assumption, and recall the expression for $\Lambda_1(\theta)$ (\ref{eq:conditional-log-MGF})
from Lemma~\ref{lemma:conditional-log-MGF}.  Fix a large constant $B>0$ and suppose
\begin{align}\label{eq:upper-B}
(c+\theta_2-\theta_1)/\sqrt{2}\ge B.
\end{align}
By
well known expansion
\begin{align*}
\log\Phi\left(-{c+\theta_2-\theta_1\over \sqrt{2}}\right) 
&=
-(1/4)\left(c+\theta_2-\theta_1\right)^2+o_B(1) \\
&=
-(1/4)\left(\theta_1^2+\theta_2^2-2\theta_1\theta_2
+2c(\theta_2-\theta_1)+c^2\right)+o_B(1).
\end{align*}
Applying (\ref{eq:conditional-log-MGF}) we obtain
\begin{equation}\label{eq:mgf_exp}
\begin{split}
\Lambda_1(\theta)&=
-\log 
\Phi\left(-c/\sqrt{2}\right)
+{1\over 4}\theta_1^2+{1\over 4}\theta_2^2
+(1/2)\theta_1\theta_2
-(c/2)(\theta_2-\theta_1)-c^2/4+o_B(1) \\
&=
(1/4)(\theta_1+\theta_2)^2-(c/2)(\theta_2-\theta_1)+
D(c)+o_B(1),
\end{split}
\end{equation}
For some function $D$ which depends on  $c$ alone.
From (\ref{eq:upper-B}) we have 
$\theta_2-\theta_1\ge \sqrt{2}B-c$. Thus
$-c(\theta_2-\theta_1)\to +\infty$ as $B\to+\infty$,
and therefore $\Lambda_1(\theta)\to\infty$ when
$\theta\to\infty$ within the set satisfying (\ref{eq:upper-B}) and $B\to\infty$.

On the other hand, when (\ref{eq:upper-B}) does not hold
we obtain that $\Lambda_1(\theta) \geq \frac{1}{2}\theta_1^2+\frac{1}{2}\theta_2^2+D'(B,c)$
for some function $D'$ which depends on $B$ and $c$ alone.

Similarly, for the case $\zeta<1$, combining $\Lambda_2(\theta)={1\over 2}(\theta_1^2+\theta_2^2)$ with \eqref{eq:mgf_exp}, gives:
\[
\zeta \Lambda_1(\theta)+(1-\zeta)\Lambda_2(\theta) = {(1-\zeta)\over 2}(\theta_1^2+\theta_2^2)+ \zeta\left((1/4)(\theta_1+\theta_2)^2-(c/2)(\theta_2-\theta_1)+
D(c)+o_B(1)\right).
\]

Under the transformation, $\tilde{\theta}_1=\theta_1+\theta_2,\tilde{\theta}_2=\theta_2-\theta_1$, the above can be re-expressed as:
\[
\zeta \Lambda_1(\theta)+(1-\zeta)\Lambda_2(\theta) = {(1-\zeta)\over 2}((\tilde{\theta}_1)^2+ (\tilde{\theta}_2)^2)+\zeta\left((1/4)(\tilde{\theta}_1)^2-(c/2)(\tilde{\theta}_2)+
D(c)+o_B(1)\right).
\]

One readily verifies that for any value of $c$, the RHS $\rightarrow \infty$ as $\tilde{\theta} \rightarrow \infty$. Hence, we obtain that in either case, $\zeta \Lambda_1(\theta)+(1-\zeta)\Lambda_2(\theta)$ as $\theta \rightarrow \infty$. This establishes the existence of a minimizer $\theta^\star$.

Next we  prove that the the asymptotic value 
of the right-hand side of (\ref{eq:lD-rate-Poisson})
for $\theta^*(t),\Lambda_j(t), j=1,2$ converges to the 
right-hand side of (\ref{eq:lD-rate-Poisson})
for $\theta^*,\Lambda_j, j=1,2$
as $t\to\infty$. We will further prove that the latter
evaluates to (\ref{eq:S_N-limit}).
The result will follow then from  
Corollary~\ref{coro:Poisson-local-LD}.

Define $g_t(\theta) =  \zeta\Lambda^{(t)}_1(\theta)+(1-\zeta)\Lambda^{(t)}_2(\theta)$. 
Recall that $g_t(\theta)$ is strictly convex with $\theta^*(t)$ being the unique minimizer of $g_t(\theta)$. 
Analogously, we have that $g(\theta) = \zeta\Lambda_1(\theta)+(1-\zeta)\Lambda_2(\theta)$ is strictly convex, and as mentioned earlier, $g(\theta)$ is furthermore coercive i.e: $g(\theta) \rightarrow \infty$ for any sequence of $\theta$ satisfying $\norm{\theta} \rightarrow \infty$. 

Let $M>0$ be fixed. The coercivity of $g(\theta)$ implies that  $\exists R > 0$ such that  $g(\theta) > M$ for all $\theta \in \R^2$ with $\norm{\theta} > R$.
Denote by $\mathcal{S}_{R'}$ the circle $\mathcal{S}_{R'}= \{\theta \in \R^2: \norm{x} = R'\}$ for some $R'>R$.  Then, since point-wise convergence for convex functions implies uniform convergence on compact sets \citep{Rockafellar+1970}, we have that
$g_t(\theta)$ converges uniformly to $g(\theta)$ on $\mathcal{S}_{R'}$. Subsequently, since $g(\theta) > M $ for $\theta \in \mathcal{S}_{R'}$, we obtain that $\exists t' \in \mathbb{N}$ such that $g_t(\theta) > M$ for $\theta \in \mathcal{S}_{R'}$ for all $t > t'$.

Suppose $\theta \in \R^2$ satisfies $\norm{\theta} \ge R'$. 
Let $\theta' \in \mathcal{S}_{R'}$ be given as $\theta' = \alpha \theta$ where $\alpha = \frac{\norm{\theta'}}{\norm{\theta}} \le 1$.
The convexity of $g_t(\theta)$ then implies:
\begin{equation}\label{eq:mgf_conv}
    \alpha g_t(\theta) + (1-\alpha) g_t(0) \geq g_t(\theta').
\end{equation}
Recall that by the choice of $R'$, we have that 
$\theta' \in \mathcal{S}_{R'} \implies g_t(\theta') > M \ \forall t>t'$. 
Combining  $g_t(\theta') > M$ and $g_t(0)=0$ with (\ref{eq:mgf_conv}), we obtain, for $t > t'$:
\begin{equation}
    g_t(\theta)> M > g_t(0)=0.
\end{equation}
Therefore, for any $t > t'$, $\theta^*(t)$ lies 
in the disc $B_R' = \{\theta \in \R^2: \norm{\theta}\leq R'\}$.

Finally, $\theta^*(t), \theta^* \in B_R'$ for $t>t'$, and the uniform convergence of $g_t$ on 
the compact set  $B_R'$ \citep{Rockafellar+1970}, imply that for any $\epsilon > 0$ and large 
enough $t$: 
\begin{equation}
 g^\star(\theta^\star) - \epsilon \leq g^\star(\theta^\star(t))  -\epsilon \leq  g^t(\theta^\star(t)) \leq g^t(\theta^\star) \leq  g^\star(\theta^\star) + \epsilon, 
\end{equation}
implying that $\lim_{t \to \infty} g^t(\theta^\star(t)) = g(\theta^\star)$.

Now we complete the evaluation 
of (\ref{eq:S_N-limit}).
We need to compute the 
right-hand side of 
(\ref{eq:lD-rate-Poisson}) where we recall that
$\theta^*$ is the unique solution of 
(\ref{eq:y-target}) and $\Lambda_1,\Lambda_2$
are log-MGF for conditional and unconditional
versions of $Y$. 
Applying Lemma~\ref{lemma:conditional-log-MGF} 
we have for $c=\tau_2-\tau_1+b$ 
\begin{align*}
{\partial \Lambda_1 \over \partial \theta_1}
=\theta_1-{1\over\sqrt{2}}{\dot \Phi\left(-{\tau_2-
\tau_1+b+\theta_2-\theta_1\over \sqrt{2}}\right)
\over \Phi\left(-{\tau_2-
\tau_1+b+\theta_2-\theta_1\over \sqrt{2}}\right)}.
\end{align*}
Similarly, 
\begin{align*}
{\partial \Lambda_1 \over \partial \theta_2}
=\theta_2+{1\over\sqrt{2}}{\dot \Phi\left(-{\tau_2-
\tau_1+b+\theta_2-\theta_1\over \sqrt{2}}\right)
\over \Phi\left(-{\tau_2-
\tau_1+b+\theta_2-\theta_1\over \sqrt{2}}\right)}.
\end{align*}
For $\Lambda_2$ we have 
${\partial \Lambda_2\over \partial \theta_j}=
\theta_j, j=1,2.$ Combining, (\ref{eq:y-target})
translates into
\begin{align*}
\theta_1-\zeta{1\over\sqrt{2}}
{\dot \Phi\left(-{\tau_2-
\tau_1+b+\theta_2-\theta_1\over \sqrt{2}}\right)
\over \Phi\left(-{\tau_2-
\tau_1+b+\theta_2-\theta_1\over \sqrt{2}}\right)}=0,
\end{align*}
and,
\begin{align*}
\theta_2+\zeta{1\over\sqrt{2}}{\dot \Phi\left(-{\tau_2-
\tau_1+b+\theta_2-\theta_1\over \sqrt{2}}\right)
\over \Phi\left(-{\tau_2-
\tau_1+b+\theta_2-\theta_1\over \sqrt{2}}\right)}
=0.
\end{align*}
Summing the two identity we obtain $\theta_2=-\theta_1$. Denoting $\theta_1$ by $\bar\theta$, we obtain $\bar\theta$ is the unique solution
of (\ref{eq:fixed-point-theta}), confirming
the first part of the claim. Finally, again using Lemma~\ref{lemma:conditional-log-MGF} and $\Lambda_2(\theta)
=(1/2)(\theta_1^2+\theta_2^2)$ we verify (\ref{eq:S_N-limit}).
\end{proof}

The above result yields the asymptotic limit of a given set of $\zeta N$ vertices being $h$-stable conditioned on the cut-sizes
i.e.
$K(N, \lambda_1(t) N, \lambda_2(t) N, a(t), \zeta)$: 

\begin{proposition}\label{prop:K_asymp}
Let $\mu(d), \mu_1(d)$ denote a sequence of integers satisfying
\begin{equation}
    \mu(d) = (\frac{d}{2}+2\tau\sqrt{d} + o_d(\sqrt{d}))N, \ \mu_1(d) = (\frac{d}{2}+2\tau_1\sqrt{d} + o_d(\sqrt{d}))N.
    \end{equation}
Then:
\begin{equation}\label{eq:Kblim}
   \lim_{d \rightarrow \infty}  \lim_{N \rightarrow \infty} \frac{1}{N} \log K(N, \mu N, \mu_1 N, b, \zeta) = L(\zeta, 2\sqrt{2}(\tau_1-\tau), \sqrt{2}b),
\end{equation}
where:
\begin{equation}\label{eq:def_L}
    L(\zeta,t,b)  \triangleq \inf_\theta \left( \theta^2
+\zeta\log\Phi\left(-{t+b-2\theta\over\sqrt{2}}\right) \right).
\end{equation}
\end{proposition}
\begin{proof}
By Lemma~\ref{lemma_twoPo},  $K(N, \mu N, \mu_1 N, b, \zeta)$ equals:
\begin{align*}
\Theta(N) 
\mathbb{P}\left[\sum_{i=1}^{N} B_i=\mu, \sum_{i=1}^{N} 
C_i=\mu_1 \Bigm\vert B_i \geq C_i+b\sqrt{d}, 1\leq i\leq \zeta N \right]
(\mathbb{P}[B_1 \geq C_1 + b\sqrt{d}])^{\zeta N},
\end{align*}
where $B_i,C_i$ denote i.i.d. Poisson variables with integer means $\mu, \mu_1$ respectively.
The event 
$\sum_{i=1}^{N} B_i=\mu, \sum_{i=1}^{N}  C_i=\mu_1$ is equivalent to the event:
\begin{align*}
&\sum_{i=1}^{N} {B_i- \frac{\mu}{N} \over \sqrt{\frac{\mu}{N}}}=0, \\
&\sum_{i=1}^{N} {C_i-\frac{\mu_1}{N} \over  \sqrt{\frac{\mu_1}{N}}}=0,
\end{align*}
which matches the form described by Lemma \ref{lemma:Gaussian-of-Pois} with $t=\frac{d}{2}$. We therefore obtain:
\begin{align*}
&\lim_{N \rightarrow \infty}\frac{1}{N} \log \mathbb{P}\left[\sum_{i=1}^{N} B_i=\mu, \sum_{i=1}^{N} 
C_i=\mu_1 \Bigm\vert B_i \geq C_i+b\sqrt{d}, 1\leq i\leq \zeta N \right]\\&= 
-\zeta\log \Phi\left(-{2\sqrt{2}(\tau_1-
\tau)+b\sqrt{2}\over\sqrt{2}}\right)+
 \bar{\theta}^2
+\zeta\log\Phi\left(-{2\sqrt{2}(\tau_1-\tau)+b-2\bar{\theta}\over\sqrt{2}}\right).
\end{align*}

Furthermore, by the Poisson-Gaussian approximation:
   $$
       \lim_{d \rightarrow \infty} \log \pr[B_1\ge C_1+b\sqrt{d}] = \log \Phi\left(-{2\sqrt{2}(\tau_1-\tau)+b\sqrt{2}\over\sqrt{2}}\right).
  $$
Therefore, combining Lemmas \ref{lemma:Gaussian-of-Pois} and \ref{lemma_twoPo}, and noting the cancellation of the above term results in (\ref{eq:Kblim}).
\end{proof}

The variational solution to $L(\cdot)$ in (\ref{eq:def_L}) will allow us to 
estimate the probability 
that  fixed subsets of 
sizes $r_1,r_2$
of the two parts of a bisections consist of stable nodes,
as appears in (\ref{eq:p_optimal}).
To turn this into existence property, we need to take
a union over such subsets which, when combined with Proposition \ref{prop:K_asymp} in the log scale
will let to the expression of the form
\begin{equation}\label{eq:def_cK}
\mathcal{K}(\zeta_1,\zeta_2,\tau_{12},\tau_1,\tau_2, \sqrt{2}b) \coloneqq  \frac{1}{2}L(\zeta_1,2\sqrt{2}(\tau_1-\tau_{12}),\sqrt{2}b)
+
\frac{1}{2}L(\zeta_2,2\sqrt{2}(\tau_2-\tau_{12}),\sqrt{2}b).
\end{equation}
 As we justify in the subsequent section, the sum over probabilities of events $O_1,O_2$ in the bound on $\mathbb{E}[X(h,r)]$
will lead to the optimization of $\mathcal{K}(\zeta_1,\zeta_2,\tau_{12},\tau_1,\tau_2, b)$ w.r.t $\tau_{12},\tau_1,\tau_2$ subject to  appropriate constraints. The resulting optimization will be simplified through certain properties of $\mathcal{K}(\cdot)$, which we establish in Appendix \ref{app:properties}

\begin{lemma}\label{lem:L_conv}
For any $\zeta > 0, b \in \mathbb{R}$,  $L(\zeta,t,b)$ defined in (\ref{eq:def_L}) is concave in $t$ and non-increasing in $\zeta$.
 \end{lemma}

The above result allows us to simplify the objective $\mathcal{K}(\cdot)$ defined in (\ref{eq:def_cK}) when $\tau_1+\tau_2$ is constrained to a fixed value. This is achieved in the subsequent Lemma:
\begin{lemma}\label{lemma:Optimum-LD-rates}
For any $1/2 < \zeta < 1$ and $t \in \mathbb{R}$,
the value of the function $\mathcal{K}(\zeta_1,\zeta_2,\tau_{12},\tau_1,\tau_2, b)$
subject to constraints 
$\zeta_1+\zeta_2=2\zeta$ 
and $(\tau_1+\tau_2)=t$ 
is bounded above as follows:
\begin{equation}
\mathcal{K}(\zeta_1,\zeta_2,\tau_{12},\tau_1,\tau_2, b) \leq \mathcal{K}(2\zeta-1, 2\zeta-1,\tau_{12},\frac{t}{2},\frac{t}{2}, b) = 2L(2\zeta-1,2\sqrt{2}(\frac{t}{2}-\tau_{12}),\sqrt{2}b).
\label{eq:2-rates}
\end{equation}
Furthermore, at $\zeta=1$, the bound is tight at $\zeta_1=\zeta_2=\zeta$ and $\tau_1=\tau_2= \frac{t}{2}$.
\end{lemma}

\begin{proof}

We start by noting that $\zeta_1 \leq 1, \zeta_2  \leq 1$ with the constraint $\zeta_1+\zeta_2 = 2 \zeta$ imply that $\min{\zeta_1, \zeta_2} \geq 2\zeta-1$. Recall that by assumption $1/2 < \zeta$ and thus $2\zeta-1>0$. Combining  this with the non-increasing property of $L$ established in Lemma \ref{lem:L_conv}, we obtain the following bound:
\begin{equation}\label{eq:object}
\mathcal{K}(\zeta_1,\zeta_2,\tau_{12},\tau_1,\tau_2, b) \leq L(2\zeta-1,2\sqrt{2}(\tau_1-\tau_{12}),\sqrt{2}b)+L(2\zeta-1,2\sqrt{2}(\tau_2-\tau_{12}),\sqrt{2}b).
\end{equation}

The RHS of the above bound can be further bounded through the concavity of $L$ (Lemma \ref{lem:L_conv}), yielding:
\begin{equation}
L(2\zeta-1,2\sqrt{2}(\tau_1-\tau_{12}),\sqrt{2}b)+L(2\zeta-1,2\sqrt{2}(\tau_2-\tau_{12}),\sqrt{2}b) \leq 2L(2\zeta-1,\sqrt{2}t-2\sqrt{2}\tau_{12},\sqrt{2}b),
\end{equation}
where we used that $\tau_1+\tau_2=t$.

Substituting in (\ref{eq:object}), we obtain:
\begin{equation} \mathcal{K}(\zeta_1,\zeta_2,\tau_{12},\tau_1,\tau_2, b) \leq 2L(2\zeta-1,\sqrt{2}t-2\sqrt{2}\tau_{12},\sqrt{2}b),
\end{equation}
which equals the RHS of (\ref{eq:2-rates}). This completes the proof for $1/2 < \zeta < 1$. For $\zeta=1$, we notice that $2\zeta-1=\zeta$ and the convexity-based inequalities for $L$ and $H$ are tight. Therefore, the above upper bound is tight for $\zeta=1$.
\end{proof}

Before proceeding further, we establish the continuous differentiability of $L(\cdot)$ which will be utilized later (Proof in Appendix \ref{app:properties}):

\begin{lemma}\label{lem:L_lips}
For any $0 \leq \zeta < 1$, and  $L(\zeta, t, b)$ is continuously differentiable in $t,b$ on $\mathbb{R}^2$. While for $\zeta=1$, 
$L(1, t, b)$ is continuously differentiable over the open set $\{(t,b): t+b < 0\}$.
\end{lemma}

\subsection{Putting everything together (Proof of Proposition~\ref{prop:neg_first_mom})}\label{sec:first_mom_fin}
We now return to estimating (\ref{eq:firs_mom}) which we reproduce here for convenience. For any $C \geq 2$:

\begin{equation}\label{eq:firs_mom_rep}
\begin{split}
    &\mathbb{E}[X(z,h,r)]\\ &\le 2^n\sum_{z_1,z_2 \in D(C,z)}  \pr(\cE(\sigma, z_1,z_2)) 
     \sum_{r_1+r_2=2r}\binom{{n\over 2}}{r_1n} \binom{{n\over 2}}{r_2n}  
     K(n/2, \floor{zn}, 2z_1 n,h,r_1) K(n/2 , \floor{zn}, 2z_2 n,h,r_2) \\&+ \exp{(-C n)}.
\end{split},
\end{equation}

where the first sum is over the set 

\begin{align}
D(C,z)=\{(z_1,z_2): |2z_1-dn/4|, |2z_2-dn/4| \le C\sqrt{d}n, \floor{zn}+(z_1+z_2)n=d/2 \}.
\label{eq:D2}
\end{align}
Recall the notation $z=\frac{d}{4}n-\frac{E}{2}\sqrt{d}$ with the cut-size represented as $E_{12} = \floor{zn}$ and the constraint $E\leq -\frac{h}{2}$ by the assumption in Proposition \ref{prop:neg_first_mom}.
For any such value of $E$, Equations \ref{eq:E_1_const} and \ref{eq:E_2_const} further imply that under the event $\mathcal{E}_{\text{opt}}(\sigma,h,1)$,  $E_{12}-2 E_{11} \geq h\sqrt{d}n/2$ and $E_{12}-2 E_{22} \geq h\sqrt{d}n/2$ hold respectively. Therefore, we may further restrict the sum in (\ref{eq:firs_mom_rep}) to $z_1,z_2$ satisfying $\floor{zn} \geq 2z_1n+h\sqrt{d}n/2, \floor{zn} \geq 2z_2n+h\sqrt{d}n/2$. 
To impose such a restriction, we define the admissible set $\mathcal{A}(C,z,r)$ as:
\begin{equation}
    \mathcal{A}(z,r) = \begin{cases}
        D(C,z) \cap \{z_1,z_2: , \floor{zn} \geq \max{(2z_1,2z_2)}+h\}, & \text{if} \ r=1\\
        D(C,z) & \text{otherwise}
    \end{cases}.
\end{equation}

% \yd{I think the above should read:
% \begin{align}
% D(2)=\{(z,z_1,z_2): |z-dn/4|, |2z_1n-dn/4|, |2z_2n-dn/4| \le 2\sqrt{d}n, z+z_1+z_2=d/2 \}.
% \label{eq:D2}
% \end{align}
% (just added $n$ next to $z,z_1,z_2$
% }

To estimate the $K(\cdot)$ terms in (\ref{eq:firs_mom_rep}), we cannot directly apply Proposition \ref{prop:K_asymp} to the above sum since Proposition \ref{prop:K_asymp} applies to a fixed set of values of $\tau_1, \tau_2, \tau$. Hence, to allow us to choose $d,n$ uniformly large for all the choices of $z_1,z_2$ in $D(C,z)$, we introduce a discretization of $[-C,C]$ into a finite number of approximating points. 

Fix $\epsilon >0$ and $C>\max(\abs{h},2)$ and consider
the partition of $[-C,C]$ as 
\begin{equation}\label{eq:C_partition}
\mathcal{C}(\epsilon, C)=    
\{-C,-C+\epsilon ,-C+2\epsilon,
    -C+3\epsilon,\ldots, C\},
\end{equation}
where we assume for convenience that
$C/\epsilon$ is an integer.

Fix any $(z_1,z_2)\in D(C,z)$.
We find $\tau,\tau_1,\tau_2\in\mathcal{C}(\epsilon), b,\zeta_1,\zeta_2\in [0,\epsilon,2\epsilon,\ldots, 1]$, so that 
\begin{equation}\label{eq:tau_construct}
\begin{split}
d/4+\tau\sqrt{d}+\alpha_{\rm l} &\le z < d/4+\tau\sqrt{d}+\epsilon\sqrt{d}-\alpha_{\rm u}, \\
d/4+\tau_1\sqrt{d}+\alpha_{\rm l,1} &\le 2z_1 < d/4+\tau_1\sqrt{d}+\epsilon\sqrt{d}-\alpha_{\rm u,1}, \\
d/4+\tau_2\sqrt{d}+\alpha_{\rm l,2} &\le 2z_2 < d/4+\tau_2\sqrt{d}+\epsilon\sqrt{d}-\alpha_{\rm u,2}, \\
\zeta_1&\le r_1< \zeta_1+\epsilon, \\
\zeta_2&\le r_2< \zeta_2+\epsilon, \\
b &\le h\le b+\epsilon,
\end{split}
\end{equation}
where $\alpha\in [0,1)$ are terms to ensure that the lower and upper bounds in the first three inequalities are integers. 
From the constraint $\floor{zn}+(z_1+z_2)n=d/2$, we obtain 
\begin{align}
|\tau+\frac{\tau_1+\tau_2}{2}|\le 2\epsilon+O(1/\sqrt{d}). \label{eq:tau-constrained}
\end{align}

It is immediate that the 
function (probability) $K(N,\mu_1,\mu_2,h,\zeta)$ is non-decreasing in $\mu_1$ and non-increasing in $\mu_2, \zeta$ and $h$.
Thus 
\begin{align}
K(n/2, \floor{zn}, 2z_j n,h,r_j)\le 
K\left(n/2, \left(d/4+\tau\sqrt{d}+\epsilon\sqrt{d}-\alpha_{\rm u}\right)n, 
\left(d/4+\tau_j\sqrt{d}+\alpha_{\rm l,j}\right) n,b,\zeta_j\right),   \qquad j=1,2. \label{eq:K_in}
\end{align}

Observe that when $r=1$, for any admissible, $z_1,z_2 \in \mathcal{A}(z,1)$ and $\epsilon > 0$ we have that $\tau,\tau_1,\tau_2$ satisfy:
\begin{equation}\label{eq:tau_cond}
    \tau + \epsilon > \max{(\tau_1,\tau_2)}+b.
\end{equation}
This ensures that assumption $1$ in Lemma~\ref{lemma:Gaussian-of-Pois} for $\zeta=1$ applies to the poisson random variables $\lambda_1(d)=\frac{d}{4}+\tau+\sqrt{d}\epsilon-\alpha_u$ and $\lambda_2(d) = \frac{d}{4}+\tau_j\sqrt{d}+\alpha_{1,j}$ for $j=1,2$.
Therefore, applying Proposition \ref{prop:K_asymp}, we obtain that for large enough $n, d$ (independent of $z_j, z$), with high probability as $n \rightarrow \infty$:
\begin{align*}
\frac{1}{n}\log K\left(n/2, \left(d/4+\tau\sqrt{d}+\epsilon\sqrt{d}-\alpha_{\rm u}\right)n, 
\left(d/4+\tau_j\sqrt{d}+\alpha_{\rm l,j}\right) n,b,\zeta_j\right) \le  L(\zeta_j,2\sqrt{2}(\tau_j-(\tau+\epsilon)),\sqrt{2}b)+\epsilon
\end{align*}
Combining with (\ref{eq:K_in}), we obtain (for large enough $d,n$):
\begin{align}\label{eq:KL_ineq}
{1\over n/2}\log K(n/2, zn, 2z_j n,h,r_j)\le L(\zeta_j,2\sqrt{2}(\tau_j-(\tau+\epsilon)),\sqrt{2}b) +\epsilon.
\end{align}

% \yd{In progress: Need to justify why the boundary case $\tau = \tau_1+b$ can be neglected as the lipchitz-constant will blow up here.}

Next, we address the combinatorial terms ${{n\over 2}\choose r_j n}$. Using Stirling's approximation we have:
\begin{align*}
\log {{n\over 2}\choose r_j n}={n\over 2}H(2r_j)+o(n),
\end{align*}
for $j=1,2$.

The concavity of $H(\cdot)$ and $r_1+r_2 = r$, further implies:
\begin{equation}
    {1\over 2}H(2r_1)+{1\over 2}H(2r_2) \leq H(r). 
\end{equation}
Therefore:
\begin{equation}
    \log {{n\over 2}\choose r_1 n} + \log {{n\over 2}\choose r_2 n} \leq n H(r)+o(n).
\end{equation}
% Recall that, $H(x)$ is non-increasing in $x \in (1/2,1]$. Since 

% We therefore obtain the bound
% \begin{align*}
% \log {{n\over 2}\choose r_j n} \le {n} \max(H(2\zeta_j), H(2\zeta_j+2\epsilon))+o(n),
% \end{align*}
% for $j=1,2$

Finally, we combine the above bounds and incorporate the discretization error due to restriction to $\mathcal{C}(\epsilon, C)$. Recall that $z=\frac{d}{4}-\frac{E}{2}\sqrt{d}$.

To account for the constraint given by (\ref{eq:tau-constrained}) in addition to the discretization set $\mathcal{C}(\epsilon, C)$, we introduce the following constraint set:

\begin{equation}\label{eq:const_set}
    \mathcal{\tilde{S}}(E,h) \coloneqq
        \{\tau \in [-\frac{E}{2},-\frac{E}{2}+\epsilon], \tau_1, \tau_2 \in \mathbb{R}^2, \zeta_1, \zeta_2 \in [0, 1]^{\otimes 2}, \abs{\tau+\frac{\tau_1+\tau_2}{2}} \leq 3\epsilon, r \leq \zeta_1+\zeta_2 \leq r+2\epsilon\},
\end{equation}

where we relaxed the constraint in (\ref{eq:tau-constrained}) slightly, by absorbing the extra $\mathcal{O}(\frac{1}{\sqrt{d}})$ into another $\epsilon$.
 
To simultaneously, handle the cases $r=1$ and $r <1$, we denote by $\mathcal{S}(E,h)$ the set  $ \mathcal{\tilde{S}}(E,h) \cap \{\tau+\epsilon > \max{\tau_1,\tau_2}+b\}$ if $r=1$ and $ \mathcal{\tilde{S}}(E,h)$ otherwise.

From Equations \ref{eq:tau_construct} and \ref{eq:tau-constrained}, we note that for any admissable $z_1,z_2 \in \mathcal{A}(z,r)$,
the parameters $\tau, \tau_1,\tau_2, \zeta_1, \zeta_2, b$ defined by Equations \ref{eq:tau_construct} lie in the set $\mathcal{S}(E,h) \cap C(\epsilon)^{\otimes 6}$.
Therefore, for all large enough $d$ and $n$, we obtain by (\ref{eq:KL_ineq}):
\begin{align*}
&{1\over n}\log\left[
\binom{{n\over 2}}{r_1n} \binom{{n\over 2}}{r_2n}  
     K(n/2, zn, 2z_1 n,h,r_1) K(n/2 , zn, 2z_2 n,h,r_2)
     \right] \\
     &\le \sum_{j=1}^2{1\over 2}
     L(\zeta_j,2\sqrt{2}\tau_j-2\sqrt{2}(\tau+\epsilon),\sqrt{2}b)
+H(r) + 2\epsilon \\
&\le \max_{\tau,\tau_1,\tau_2,\zeta_1,\zeta_2 \in \mathcal{S}(E,h) \cap C(\epsilon)^{\otimes 6}}
\sum_{j=1}{1\over 2}
     L(\zeta_j,2\sqrt{2}\tau_j-2\sqrt{2}\tau,\sqrt{2}b)
+H(r)+2\epsilon\\
&\stackrel{(a)}{=} \max_{\tau,\tau_1,\tau_2,\zeta_1,\zeta_2 \in \mathcal{S}(E,h) \cap C(\epsilon)^{\otimes 6}} \mathcal{K}(\zeta_1,\zeta_2, \tau+\epsilon,\tau_1,\tau_2,b)+H(r)+2\epsilon\\
&\stackrel{(b)}{\leq} \sup_{\tau,\tau_1,\tau_2,\zeta_1,\zeta_2 \in \mathcal{S}(E,h)} \mathcal{K}(\zeta_1,\zeta_2, (\tau+\epsilon),\tau_1,\tau_2,b)+H(r)+2\epsilon.
\end{align*}
where in $(a)$, we used the definition of $\mathcal{K}$ in (\ref{eq:def_cK}) and in $(b)$ we relaxed the constraint of the variables belonging to $C(\epsilon)^{\otimes 6}$.
Here the term $o(n)$ was absorbed by an extra $\epsilon$ term added at the end. 

We next introduce a parameter $t=\tau_1+\tau_2$ and translate the optimization over $\tau_1,\tau_2$ to one over $t$ through the following bound:
\begin{equation}
\sup_{\tau,\tau_1,\tau_2,\zeta_1,\zeta_2 \in \mathcal{S}(x,h)} \mathcal{K}(\zeta_1,\zeta_2, \tau+\epsilon,\tau_1,\tau_2,b)\leq \sup_{\tau,\zeta,t} \left[\sup_{\tau_1,\tau_2,\zeta_1,\zeta_2} \mathcal{K}(\zeta_1,\zeta_2, \tau+\epsilon,\tau_1,\tau_2,b)\right],
\end{equation}
where the first supremum is over the set $\{\tau \in [-E/2, -E/2 + \epsilon,\zeta \in [r, r+\epsilon], \abs{\tau+\frac{t}{2}} \leq 3\epsilon\}$ and the second supremum is over the set $\{\tau_1+\tau_2=t, \zeta_1+\zeta_2=\zeta\}$. In the second supremum, we relaxed the constraint $\tau > \max{\tau_1,\tau_2}+b$ on $\tau_1, \tau_2$.
By Lemma~\ref{lemma:Optimum-LD-rates}, the inner supremum is bounded by the objective at $\tau_1=\tau_2=\frac{t}{2}, \zeta_1=\zeta_2=2\zeta-1$, yielding:
 \begin{align}
&{1\over n}\log\left[
\binom{{n\over 2}}{r_1n} \binom{{n\over 2}}{r_2n}  
     K(n/2, zn, 2z_1 n,h,r_1) K(n/2 , zn, 2z_2 n,h,r_2)    \right] \nonumber\\
  &\leq \sup_{\tau \in [-\frac{E}{2}, -\frac{E}{2}] + \epsilon,\zeta \in [r, r+\epsilon], \abs{\tau+\frac{t}{2}} \leq 3\epsilon} \mathcal{K}(2\zeta-1, 2\zeta-1,\tau+\epsilon,\frac{t}{2}, \frac{t}{2},b).\label{eq:gen_K_bound}
\end{align}

To proceed further, we shall treat the edge case $E=-\frac{h}{2}, r = 1$ and the general case $E < -\frac{h}{2}$ or $r < 1$ seperately. Recall that by the statement of Proposition \ref{prop:neg_first_mom}, either $r < 1$ or $r=1$ and $E < -\frac{h}{2}$. In the first case, for small enough $\epsilon$, $[r, r+\epsilon] \in [0,1)$. While in the second case i.e. when $r=1$ and  $E < -\frac{h}{2}$, the constraints $\abs{\tau+\frac{t}{2}} \leq 3\epsilon$ and 
 $\tau \in [-\frac{E}{2}, -\frac{E}{2}+\epsilon]$ imply that
 $\abs{\tau-\frac{t}{2}-E} \leq \abs{2\tau-E} +\abs{\tau+\frac{t}{2}} \leq 5\epsilon$. Therefore, for small enough $\epsilon$,
 $\tau-\frac{t}{2} > h$.
In either case, Lemma \ref{lem:L_lips}  yields a uniform bound on the Lipschitzness of $\mathcal{K}$ over the (compact) range of $\tau, t, r$ resulting in the bound:
\begin{align}
    &{1\over n}\log\left[
\binom{{n\over 2}}{r_1n} \binom{{n\over 2}}{r_2n}  
     K(n/2, zn, 2z_1 n,h,r_1) K(n/2 , zn, 2z_2 n,h,r_2) \right] \nonumber \\ &\le \mathcal{K}(2\zeta-1, 2\zeta-1, -\frac{E}{2},\frac{E}{2}, \frac{E}{2},b)+C_{E}\epsilon.\label{eq:E_uni_bound}
\end{align}
for some constant $C_{E}$ dependent on $E$. By the definition of $\mathcal{K}(\cdot)$ in (\ref{eq:def_cK}), the RHS above equals $2L(2\zeta-1,2\sqrt{2}E,\sqrt{2}b)$.

The edge-case $E= -\frac{h}{2}$ and $r=1$ is controlled through the following proposition, whose proof can be found in Appendix \ref{app:properties}:
\begin{proposition}\label{prop:Linf}
Let $b \in \mathbb{R}$ be fixed. Then:
\begin{equation}
    \lim_{t \uparrow -b} L(1, t, b) = -\infty.
\end{equation}
\end{proposition}

Recall that $b < h \leq b+\epsilon$. Hence when $E=-\frac{h}{2}$, taking $\epsilon \rightarrow 0$ implies that $2E \uparrow -b$. Subsequently, by proposition \ref{prop:Linf}, $\mathcal{K}(2\zeta-1, 2\zeta-1, -\frac{E}{2},\frac{E}{2}, \frac{E}{2},b) = 2L(2\zeta-1,2\sqrt{2}E,\sqrt{2}b)$ diverges to $-\infty$ as $\epsilon \rightarrow 0$.

Hence, we take 
$\epsilon \rightarrow 0$ in (\ref{eq:gen_K_bound}) and $i)$ apply proposition \ref{prop:Linf} when $E=-\frac{h}{2}$, $r=1$ and 
$ii)$ apply (\ref{eq:E_uni_bound}) when either $E>-\frac{h}{2}, r=1$ or $r<1$, to obtain:
\begin{equation}\label{eq:combined_bound}
\begin{split}
    &\limsup_{n \rightarrow \infty}\limsup_{d \rightarrow \infty} {1\over n}\log\left[
\binom{{n\over 2}}{r_1n} \binom{{n\over 2}}{r_2n} 
     K(n/2, zn, 2z_1 n,h,r_1) K(n/2 , zn, 2z_2 n,h,r_2)    \right]\\ &\leq \begin{cases}
         - \infty, & r=1, E=-\frac{h}{2}\\
         \mathcal{K}(2\zeta-1, 2\zeta-1, -\frac{E}{2},\frac{E}{2}, \frac{E}{2},b)+H(r), & \text{otherwise}
     \end{cases}.
\end{split}
\end{equation}

% \yd{Todo: show divergence for $E = -\frac{h}{2}$ by splitting the domain of integration.}

Recalling (\ref{eq:firs_mom_rep}), it remains to incorporate the contribution of the term $\pr[ \mathcal{E}(\sigma,z,z_1,z_2)]$. 
Under the configuration model, it is easy to obtain that:
\begin{equation} \label{eq:prob_cut_size}
    \pr[ \mathcal{E}(\sigma,\floor{zn},z_1n,z_2n)] = \binom{dn}{2z_1n,2z_2n,\floor{zn}, \floor{zn}} (\frac{1}{4})^{2z_1 n}(\frac{1}{4})^{2z_2 n}(\frac{1}{2})^{\floor{zn}}. 
\end{equation}

The above follows directly by recalling that under the configuration model, each edge is assigned independently to one of the $n^2$ pairs of vertices (including self-edges). Thus, $E_{11}, E_{22}, E_{12}$ follow a multi-nomial distribution with probabilities $1/4,1/4, 1/2$ respectively.

Next, we utilize asymptotic approximations and the restriction of $z_1,z_2$ to $D(C,z)$ to simplify the term $\frac{1}{n}\log \pr[ \mathcal{E}(\sigma,\floor{zn},z_1n,z_2n)]$ up to leading order terms.

First, applying Stirling's approximation and $\floor{zn} = zn +o(1)$, we have, for large enough $n$:
\begin{align}\label{eq:binom_stir}
   \frac{1}{n} \log \binom{dn}{2z_1,2z_2,\floor{zn}, \floor{zn}} &= 2(-z_1 \log 2 z_1n - z_2 \log 2 z_2n -2z \log zn) + d \log dn) + \mathcal{O}(\frac{1}{n} \log (d n))\\
   &=  2(-z_1 \log 2 z_1n - z_2 \log 2 z_2n - 2z \log zn) + d \log dn+ o_n(1).
\end{align}

Combining with the factor $(\frac{1}{4})^{2z_1 n}(\frac{1}{4})^{2z_2 n}(\frac{1}{2})^{\floor{zn}}$ in (\ref{eq:prob_cut_size}), we obtain:
\begin{equation}\label{eq:logz}
    \frac{1}{n}\log \pr[ \mathcal{E}(\sigma,\floor{zn},z_1n,z_2n)] = z\log (1/4)+2z_1 \log (1/2) + 2z_2 \log (1/2) -2z\log z n - z_1\log 2z_1 n - z_2\log 2z_2 n+ o_n(1).
\end{equation}

Write $z_1, z_2$ as $2z_1=\frac{d}{4}-\frac{E_1}{2}\sqrt{d}, 2z_2=\frac{d}{4}-\frac{E_2}{2}\sqrt{d}$. The constraint $z_1,z_2 \in D(C,z)$ implies that $E_1, E_2$ are uniformly bounded in $d$.

Therefore, using the expansion:
\begin{equation}
    \log (1+t) =t-t^2/2+o_t(t^3),
\end{equation}
we obtain:
\begin{align}
z \log zn
  &= \frac{d}{4} \log\!\Bigl(\frac{dn}{4}\Bigr)
     + \frac{\sqrt{d}}{2}\,E
     + \frac{E^2}{2}
     + o_d(1),\\[6pt]
z_1 \log(2 z_1)
  &= \frac{dn}{4} \log\!\Bigl(\frac{dn}{4}\Bigr)
     + \sqrt{d}\,E_1
     + E_1^2
     + o_d(1),\\[6pt]
z_2 \log(2 z_2)
  &= \frac{dn}{4} \log\!\Bigl(\frac{dn}{4}\Bigr)
     + \sqrt{d}\,E_2
     + E_2^2
     + o_d(1).
\end{align}

Upon cancellation, (\ref{eq:logz}) simplifies to:
\begin{equation}\label{eq:comb_term}
    \frac{1}{n}\log \pr[ \mathcal{E}(\sigma,z,z_1,z_2)] = -\frac{E^2}{2}- E_1^2 -E_2^2+o_d(1)+o_n(1).
\end{equation}
Recall that the constraint $z_1n+z_2n=\floor{zn}$ implies that $\abs{2E_1+2E_2-E}=o_n(1)$. Furthermore, recall that for any fixed value of $E_1+E_2$, $-E_1^2 - E_2^2$ is maximized at $E_1=E_2=\frac{E_1+E_2}{2}$.  Hence, the constraint $\abs{2E_1+2E_2-E}=o_n(1)$ implies :
\begin{equation}
    \frac{1}{n}\log \pr[ \mathcal{E}(\sigma,z,z_1,z_2)] \leq - E^2 +o_n(1).
\end{equation}

% It is straightforward to check that $z_1\log 2z_1 - z_2\log 2z_2$ subject to $z_1+z_2 = \floor{zn}$ is maximized at $z_1=z_2$.

Substituting in (\ref{eq:firs_mom_rep}) and using the bound in (\ref{eq:combined_bound}),  we obtain: 
\begin{equation}
  \limsup_{d \rightarrow \infty} \limsup_{n \rightarrow \infty}\frac{1}{n} \log \mathbb{E} [X(z,h,r)] \leq \max(\mathcal{K}(2r-1, 2\zeta-1, -\frac{E}{2}, \frac{E}{2},\frac{E}{2},h) - E^2+H(r), -C)+\log 2,
\end{equation} 
where we further used that $\frac{1}{n} \log \abs{D(C,z)} \rightarrow 0$ as $n \rightarrow \infty$. The $-C$ inside the maximum accounts for the contribution outside $D(C,z)$ and the case $E-\frac{h}{2},r=1$, bounded by \eqref{eq:combined_bound}.

Recognizing $\mathcal{K}(2r-1, 2\zeta-1, -\frac{E}{2}, \frac{E}{2},\frac{E}{2},h) - E^2$ as $w(E,h,r)$ in Proposition \ref{prop:neg_first_mom}, we obtain, for any $C > 0$:
\begin{equation}
       \limsup_{d \rightarrow \infty} \limsup_{n \rightarrow \infty}\frac{1}{n} \log \mathbb{E} [X(z,h,r)] \leq \max{(w(E,h,r),-C)}.
\end{equation}

Taking  $C \rightarrow \infty$ then yields:
\begin{equation}
    \limsup_{d \rightarrow \infty} \limsup_{n \rightarrow \infty }\frac{1}{n} \log \mathbb{E} [X(z,h,r)] \leq w(E,h,r).
\end{equation}

 This establishes the upper bound for Proposition~\ref{prop:neg_first_mom}. To obtain a matching lower-bound  when $r=1$ for a fixed value of $E<-\frac{h}{2}$, we set $C > \abs{E}$.  Subsequently, we note that $\exists \tau^\star \in \mathcal{C}(\epsilon,C)$ such that $-E/2 \in [\tau^\star, \tau^\star+\epsilon]$. Furthermore, $\exists \epsilon$ small enough such that $E<-\frac{h}{2}$ implies that $\tau^\star < -h/2$. 

Next, set $z,z_1,z_2 \in \mathbb{Z}$ as follows:
\begin{align*}
     z=2\floor{\frac{d}{8}}-2\floor{E/4\sqrt{d}},\\
     z_1=\ceil{\frac{d}{8}}+\floor{E/4\sqrt{d}},\\
     z_2=\ceil{\frac{d}{8}}+\floor{E/4\sqrt{d}}.
\end{align*}
Since $-E/2 \in [\tau^\star, \tau^\star+\epsilon)$, we obtain:
\begin{equation}\label{eq:lowerbounddisc2}
\begin{split}
d/4+\tau^\star\sqrt{d}+\alpha_{\rm l} &\le z < d/4+\tau^\star\sqrt{d}+\epsilon\sqrt{d}-\alpha_{\rm u}, \\
d/4-\tau^\star\sqrt{d}+\alpha_{\rm l,1} &\le 2z_1 < d/4-\tau^\star\sqrt{d}+\epsilon\sqrt{d}-\alpha_{\rm u,1}, \\
d/4-\tau^\star\sqrt{d}+\alpha_{\rm l,2} &\le 2z_2 < d/4-\tau^\star\sqrt{d}+\epsilon\sqrt{d}-\alpha_{\rm u,2}.
\end{split},
\end{equation} 
where $\alpha_{\rm u},\alpha_{\rm l,j},\alpha_{\rm u,j} \in (0,1]$ ensure that the lower and upper bounds are integers.

We next recall that the inequality in (\ref{eq:firs_mom}) is tight at $r=1$, in the sense that:
\begin{equation}
    \mathbb{E}[X(z,h,1)] = 2^n\sum_{z_1,z_2 \in D(C,z)}  \pr(\cE(\sigma, z_1,z_2))  
     K(n/2, \floor{zn}, 2z_1 n,h,1) K(n/2 , \floor{zn}, 2z_2 n,h,1) + \exp{(-C n)}.
\end{equation}

Therefore, with $\floor{zn}, z_1n, z_2n$ as chosen above, we have:
\begin{align*}
    &\log 2 +\frac{1}{n}\log \pr[ \mathcal{E}(\sigma,z,z_1,z_2)]+ \frac{1}{n}\log K(n/2, \floor{zn}, 2z_1n,h,1)+ \frac{1}{n}\log K(n/2, \floor{zn}, 2z_2n,h,1)\\ &\leq \frac{1}{n} \log \mathbb{E}[X(z,h,1)]. 
\end{align*}
Subsequently, using the monotonicity of $K(\cdot)$ ((\ref{eq:K_in})) and the relations between $\tau^\star, z,z_1,z_2$ specified by Equations \ref{eq:lowerbounddisc2}, we obtain:
\begin{equation}
      \frac{1}{n}\log \pr[ \mathcal{E}(\sigma,z,z_1,z_2)]+\frac{2}{n}\log  K(n/2, \floor{zn}, (d/4 +\tau^\star/2\sqrt{d}+\epsilon\sqrt{d})n,h,1) 
 \leq \frac{1}{n} \log \mathbb{E} [X(z,h,r)].
\end{equation}
By (\ref{eq:comb_term}) and the local-lipschitzness of $x\rightarrow x^2$, the first term satisfies:
\begin{equation}
    \frac{1}{n}\log \pr[ \mathcal{E}(\sigma,z,z_1n,z_2n)] \geq  -2(\tau^\star)^2-C_{1,E} \epsilon,
\end{equation}
for some constant $C_{1,E}>0$. Proposition \ref{prop:K_asymp} and Lemma \ref{lem:L_lips} imply that the second term is bounded for large enough $d,n$ as:
\begin{equation}
   \frac{2}{n}\log K(n/2, (d/4+\tau^\star\sqrt{d})n, (d/4 +\tau^\star/2\sqrt{d}+\epsilon\sqrt{d})n,h,1)  \geq \mathcal{K}(\zeta_1,\zeta_2, 2(\tau^\star),2\tau^\star_1,2\tau^\star_2,b) - C_{2,E}\epsilon, 
\end{equation}
for some constant $C_{2,E} > 0$, derived from the local-Lipschitzness of $L(\cdot)$. Combining, we obtain:
\begin{equation}\label{eq:lowerb}
 -2(\tau^\star)^2+\mathcal{K}(\zeta_1,\zeta_2, 2(\tau^\star),2\tau^\star_1,2\tau^\star_2,b)-(C_{1,E}+C_{2,E})\epsilon \leq \frac{1}{n} \log \mathbb{E} [X(z,h,r)].
\end{equation}

Thus, we obtain the lower-bound in Proposition \ref{prop:neg_first_mom} upon taking $\epsilon \rightarrow 0$.

\subsection{Proof of Proposition \ref{prop:first_moment_sup}}

Recall that $X(z,h,r)$ denotes the number bisections with cut-size $\floor{zn}$ and at least $rn$ $h$-stable vertices. For $a < b \in \mathbb{R}$,
 $X_{a,b}(h,r)$ defined by \eqref{def:xab} can then be expressed as a sum over $X(z,h,r)$ for each possible cut-size $z$ in the corresponding range. 
 
 Define $z_a=\frac{d}{4}-\frac{a}{2}\sqrt{d}, z_a=b=\frac{d}{4}-\frac{b}{2}\sqrt{d}$.
Since the cut-size in $\mathbb{G}(n, d/n)$ ranges between $0$ to $\floor{2dn}$, we have:
\begin{equation}
\Ea{X_{a,b}(h,r)}=\sum_{z: zn \in \mathbb{N}, z_b < z < z_a} \Ea{X(z,h,r)}
\end{equation}

Next, to exclude diverging range of the energies, we introduce the following set (analogous to the set $D(C,z)$ defined in (\ref{eq:def_Dz})) :
\begin{equation}\label{eq:def_Dz}
D(C)=\{z: zn \in \mathbb{N}^+,|zn-dn/4| \le C\sqrt{d}n\}.
\end{equation}

As in Lemma \ref{lemma:range-z}, an application of Hoeffding's inequality yields that the contribution to $\Ea{X(h,r)}$ from $z$ outside $D(C)$ is bounded by $\exp(-Cn)$ for large enough $n$ and $C\geq 2$.

We therefore, obtain:

\begin{equation}\label{eq:X-conditioned-z-restricted_ab}
   \mathbb{E}[X_{a,b}(h,r)] = \sum_{z_i\in D(C) \cap (z_b,z_a)} \mathbb{E}[X(z_i,z_1,z_2, h,r)]+\exp(-Cn),
\end{equation}

resulting in:
\begin{equation}
    \Ea{X(h,r)} \leq \abs{D(C)}\max_{z \in \mathbb{D}(C) \cap (z_b,z_a)} \Ea{X(z,h,r)}+\exp(-Cn).
\end{equation}

As in the proof of Proposition~\ref{prop:neg_first_mom}, we introduce the discretization set $\mathcal{C}(\epsilon, C)=    
\{-C,-C+\epsilon ,-C+2\epsilon,
    -C+3\epsilon,\ldots, C\}$. 
    
For any $z_i \in \mathbb{D}(C)$, $ \exists \tau \in \mathcal{C}(\epsilon)$ such that:
\begin{equation}
  d/4+\tau \sqrt{d}+\epsilon\sqrt{d} +\alpha_l \le z_i \leq d/4+\tau \sqrt{d}+\epsilon\sqrt{d}- \alpha_u,
\end{equation}
for $\alpha_l, \alpha_u \in [0,1)$.

Subsequently, we proceed exactly as in the proof of \eqref{eq:gen_K_bound}. Concretely, we apply the monotonocity of $K(\cdot)$ as in (\ref{eq:K_in}), Proposition \ref{prop:K_asymp}, (\ref{eq:comb_term}), and $\abs{D(C)} \leq Cdn$, to obtain that for large enough $d,n$:

\begin{equation}
    \frac{1}{n} \log \Ea{X_{a,b}(h,r)} \leq \max(\sup_{\tau \in [-C,C] \cap (a,b),\zeta \in [r, r+\epsilon], \abs{\tau+\frac{t}{2}} \leq 3\epsilon} \mathcal{K}(2\zeta-1, 2\zeta-1, 2\sqrt{2}(\tau+\epsilon),\sqrt{2}t, \sqrt{2}t,b)-\frac{\tau^2}{2}-2t^2,-C)
\end{equation}
where we absorbed the $o(n)$ terms due to the error in $\log(\abs{D(C)})$ and the approximation error in Proposition \ref{prop:K_asymp} within the extra $\epsilon$.

As in the analysis following (\ref{eq:gen_K_bound}), to uniformly control the error in maximization over $\tau$ as $\epsilon \rightarrow 0$, we fix $\delta>0$, and split the above maximization into the interval $(-\frac{h}{2} + \delta, -\frac{h}{2})$ and its complement.

By Proposition \ref{prop:Linf}, there exist $\delta, \epsilon$ such that:
\begin{equation}
    \sup_{\tau \in [-C,h/2+\delta],\zeta \in [r, r+\epsilon], \abs{\tau+\frac{t}{2}} \leq 3\epsilon} \mathcal{K}(2\zeta-1, 2\zeta-1, 2\sqrt{2}(\tau+\epsilon),\sqrt{2}t, \sqrt{2}t,b) -\frac{\tau^2}{2}-2t^2\leq -C.
\end{equation}

We therefore obtain:
\begin{align*}
     &\frac{1}{n} \log \Ea{X_{a,b}(h,r)}\\ &\leq \max\left(  \sup_{\tau \in [h/2+\delta,C] \cap (a,b) ,\zeta \in [r, r+\epsilon], \abs{\tau+t} \leq 3\epsilon} \mathcal{K}(2\zeta-1, 2\zeta-1, 2\sqrt{2}(\tau+\epsilon),\sqrt{2}t, \sqrt{2}t,b)-\frac{\tau^2}{2}-2t^2,-C\right).
\end{align*}

By Lemma \ref{lem:L_lips}, the first argument in the above maximum is bounded as:
\begin{align*}
    &\sup_{\tau \in [h/2+\delta,C] \cap (a,b),\zeta \in [r, r+\epsilon], \abs{\tau+\frac{t}{2}} \leq 3\epsilon} \mathcal{K}(2\zeta-1, 2\zeta-1, 2\sqrt{2}(\tau+\epsilon),\sqrt{2}t, \sqrt{2}t,b)-\frac{\tau^2}{2}-2t^2\\ &\leq   \sup_{\tau \in \cap (a,b)}\mathcal{K}(2\zeta-1, 2\zeta-1, 2\sqrt{2}\tau,\sqrt{2}\tau, \sqrt{2}\tau,b)-\tau^2+C_\delta \epsilon,
\end{align*}
for some $C_\delta > 0$. Taking the sequential limit $\epsilon \rightarrow 0$ and $\delta \rightarrow 0$ and recognizing the RHS as $w(E,h)$ completes the proof of the upper bound.

 To establish the lower-bound for $r=1$,
 we  note that by (\ref{eq:E_const}), we may restrict ourselves to values $E\leq -\frac{h}{2}$.
 Proposition \ref{prop:Linf}  further implies, for large enough $C>0$:
 \begin{equation}
     \sup_{E \in (a,b)} w(E,h)=\sup_{E \in [-2C,-h/2] \cap (a,b)} w(E,h) .
 \end{equation}

The lower-bound can be obtained if there exist feasible energy values in $[-2C,-h/2] \cap (a,b)$ achieving the above bound. To this end, we recall Proposition \ref{prop:first_mom_conv} which implies that $w(E,h)$ is strictly concave in the range $(-\infty,-h/2)$ and therefore, attains the supremum at a unique value $E^\star < - \frac{h}{2}$.

 Subsequently,  for any $(a,b)$ with $a<-\frac{E}{2}$, $w(E,h,r)$ is maximized over the interval $[a,b]$ at a unique $E^\star_{a,b} \in [a,b]$. Finally, the proof is completed by setting $z=\frac{d}{4}-\frac{E^\star_{a,b}}{2}\sqrt{d}$ to obtain feasible cut-size values $\floor{zn}, z_1n, z_2n$ participating in the sum given by (\ref{eq:X-conditioned-z-restricted}) (as in the proof of (\ref{eq:lowerb})).

% \subsection{Strict Concavity of $w(E,h)$ }\label{app:concav}

% By Lemma \ref{lem:L_conv}, the function $w(E, h)$ is the pointwise infimum of concave functions and is therefore concave. However, to guarantee the uniqueness of its maximizer, we need to establish that $w(E, h)$ is strictly concave:
% \begin{lemma}\label{lem:w_convex}
%  For any $h \in \mathbb{R}$, $W(E, h)$ is strictly concave in $E$ in the region $E>-\frac{h}{2}$.  
% \end{lemma}

% \dg{EVERYTHING BELOW TILL SECTION B TO BE DELETED LATER}

\newpage

\section{Second Moment Entropy Density}\label{app:sec_mom}
\begin{figure}[h]
    \centering
\includegraphics[width=0.7\textwidth]{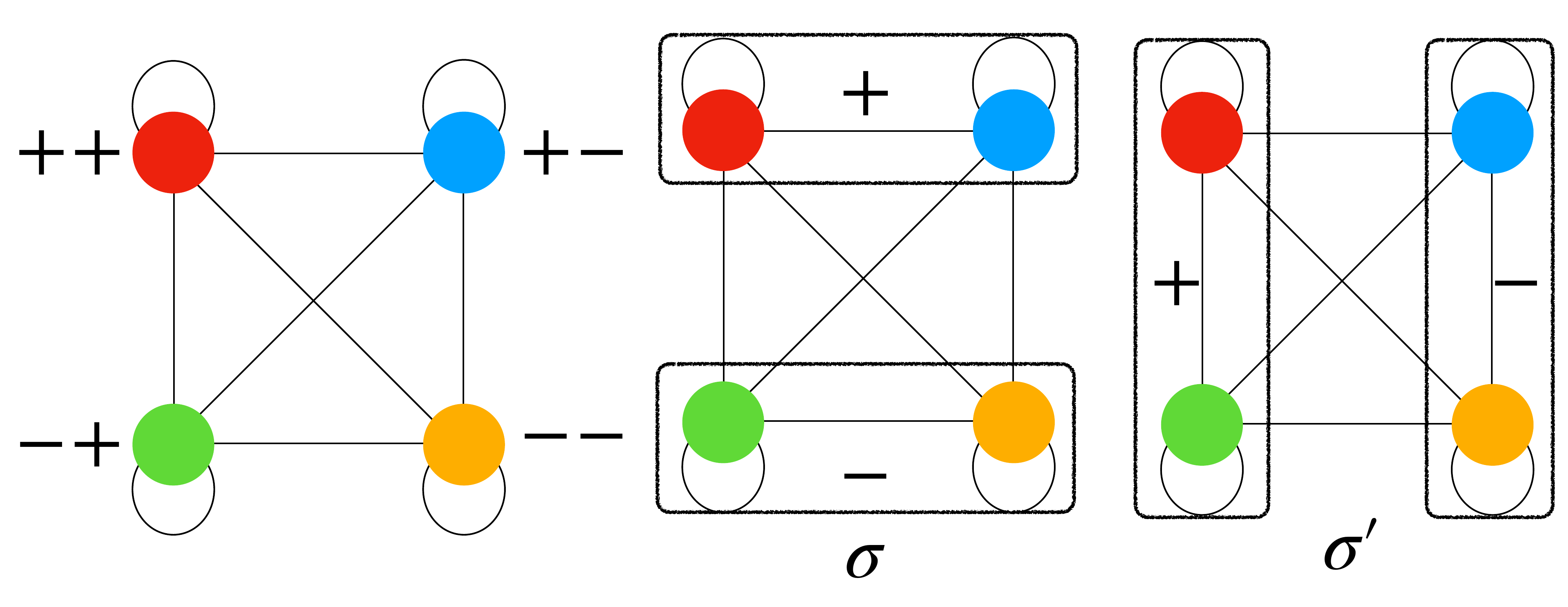}
    \caption{Illustration of the two configurations $\sigma, \sigma'$ and the partitioning of the set of vertices. Left: Each edge in the figure, including loops, denotes the set of edges between vertices belonging to the two sets. Middle: Visualization of the configuration $\sigma$. Right: Visualization of the configuration $\sigma'$. 
    }\label{fig:partition}
\end{figure}

In this section, we provide the proofs for Propositions \ref{prop:second_mom_main} and \ref{prop:second_mom_sup}.
Recall from the statement of  \ref{prop:second_mom_main} that our aim is to obtain the asymptotic limit of the following quantity:
\begin{equation}\label{eq:defx2w}
   X^2_\omega(z,h) = \sum_{\sigma, \sigma' \in M_0, \omega(\sigma,\sigma')=\omega } \mathbf{1}[\cE_{opt}(\sigma,h,z)\cap \cE_{opt}(\sigma',h, z)],
\end{equation}
where we recall that $\cE_{opt}(\sigma,h,z), \cE_{opt}(\sigma',h,z)$ denote the events of configurations $\sigma, \sigma'$ being $h$-stable bisections with cut-sizes $\floor{zn}$.
Following \cite{gamarnikMaxcutSparseRandom2018}, to turn the estimation of the above quantity into a tractable large deviations problem, we condition on the number of edges between different subsets of spins based on the the values these spins take in the two bisections. The different types of spins are illustrated by Figure \ref{fig:partition}. We begin by setting up the required notation.

\subsection{Setup}

Let $\sigma,\sigma' \in \{-1,1\}^{n}$ be two bisections on $n$ vertices such that $\omega(\sigma,\sigma') = \omega$. As in Section \ref{sec:first_moment}, we consider a (multi)-graph $G=(V,E)$ sampled using the configuration model $\tilde{\mathbb{G}}(n,m=\frac{d}{2}n)$. We associate the multi-graph with a weighted graph $G=(v,W)$ with $W \in \mathbb{R}^{n \times n}$ defined as the negative of the adjacency matrix.
We define the following subsets of partitions:
\begin{equation}\label{def:vars}
\begin{split}
  V_1   &= \{j:\sigma_j=1,\sigma'_j=1\} \\
  V_2  &= \{j:\sigma_j=1,\sigma'_j=-1\} \\
  V_3   &= \{j:\sigma_j=-1,\sigma'_j=1\} \\
   V_4  &= \{j:\sigma_j=-1,\sigma'_j=-1\}. \\
\end{split}  
\end{equation}
The partitioning of the vertices is illustrated in Figure \ref{fig:partition}.
Due to the zero magnetization (bisection) constraint, we have:
\begin{equation}  \abs{V_1}+\abs{V_3}=\abs{V_1}+\abs{V_3}=\abs{V_2}+\abs{V_4}=\abs{V_3}+\abs{V_4}=\frac{n}{2}.
\end{equation}

Now, suppose that $\sigma, \sigma'$ have normalized overlap $\frac{\floor{\omega n}}{n} \in (0,1)$. Combined with the bisection constraint, this implies that $\abs{V_1}=\abs{V_4}=(\frac{1+\omega}{4}) n+O(1)$ and $\abs{V_2}=\abs{V_3}=(\frac{1-\omega}{4})n+O(1)$.
Now, consider a vertex $i$ in the set $V_1$, i.e. the set of vertices having positive spins in both configurations. Let $E_i^{(1,1)},E_i^{(1,2)},E_i^{(1,3)},E_i^{(1,4)}$ denote the total number of edges having one end in vertex $i$ and the other in sets $V_1, V_2, V_3, V_4$ respectively. 
By the definition of $s_\sigma$ in (\ref{eq:stab_e}), the stability or unfriendliness of the vertex $i$ is given by:
\begin{align*}
    s_\sigma(i,W) &= (E_i^{(1,3)}+E_i^{(1,4)}-E_i^{(1,1)}-E_i^{(1,2)})/\sqrt{d},\\
     s_{\sigma'}(i,W) &= (E_i^{(1,2)}+E_i^{(1,4)}-E_i^{(1,1)}-E_i^{(1,3)})/\sqrt{d}.
\end{align*}
Therefore, the condition for the simultaneous $h$-stability of $i$ in  both partitions is equivalent to the following set of conditions:
\begin{equation}\label{eq:optimal_1}
   E_i^{(1,4)}-E_i^{(1,1)} \geq \abs{E_i^{(1,3)} - E_i^{(1,2)}}+h\sqrt{d}, \ \forall{i} \in V_1 \, .
\end{equation}
Analogously, we obtain the following conditions for vertices in sets $V_2, V_3, V_4$ respectively:
\begin{equation}\label{eq:optimal_2}
   E_i^{(2,3)}-E_i^{(2,2)} \geq \abs{E_i^{(2,4)} - E_i^{(2,1)}}+h\sqrt{d}, \ \forall{i} \in V_2 \, .
\end{equation}
\begin{equation}\label{eq:optimal_3}
   E_i^{(3,2)}-E_i^{(3,3)} \geq \abs{ E_i^{(3,1)}-E_i^{(3,4)}}+h\sqrt{d}, \ \forall{i} \in V_3 \, .
\end{equation}
\begin{equation}\label{eq:optimal_4}
   E_i^{(4,1)}-E_i^{(4,4)} \geq \abs{E_i^{(4,3)} - E_i^{(4,2)}}+h\sqrt{d}, \ \forall{i} \in V_3 \, .
\end{equation}

Introduce variables $z_{j,k}$ for $j, k \in [4]^2$, such that
$nz_{j,k}$ denote the number edges connecting vertices in $V_i$ to $V_j$, for $j,k \in [4]^2$.  Since both $\sigma, \sigma'$ are constrained to have cut size $\floor{zn}$,  the variables $z_{j,k}$ satisfy:
\begin{equation}\label{eq:cut_constraint}
\begin{split}
    z_{1,3}n+z_{1,4}n + z_{2,3}n+ z_{2,4}n &=\floor{zn},\\
    z_{1,2}n+z_{1,4}n + z_{2,3}n+z_{3,4}n &=\floor{zn}.
\end{split}
\end{equation}

Additionally, since the total number of edges is fixed to $\frac{d}{2}n$, we further have:
\begin{equation}\label{eq:sum_bis}
    \sum_{1 \leq i \leq j \leq 4} z_{i,j} n = dn  
\end{equation}

In what follows, we denote by $\mathcal{S}_{d,n}(z)$ denote the subset of values $\{z_{ij}: ij \in [n]\}$ satisfying the constraints in Equations \ref{eq:cut_constraint} and \ref{eq:sum_bis}.

Similar to the first moment computation, we decouple the optimality constraints for different sets $V_1, V_2, V_3, V_4$ by conditioning on $z_{j,k}$ and adopting the configuration model. Let $\cE_{opt}(\sigma,\sigma',h)$ denote the event of both configurations $\sigma, \sigma'$
satisfying $h$-stability for all the vertices.
Furthermore, let $\cE(\sigma,\sigma', z,(z_{i,j})_{i,j=1}^4)$ denote the event of $\sigma,\sigma'$ defining partitions $V_1, V_2, V_3, V_4$ with cardinalities of the edge sets given by $(nz_{j,k})_{j,k \in [4]^2}$.

Taking expectation over (\ref{eq:defx2w}),  $\mathbb{E}[X^2_\omega(z,h)]$ can be expressed as:
\begin{equation}
\label{eq:EX2m}
\begin{split}
    \mathbb{E}[X^2_\omega(z,h)] 
    &= 2^n\binom{n}{\frac{1+\omega}{2}\frac{1-\omega}{2}}\sum_{(z_{i,j})_{i,j=1}^4 \in \mathcal{S}_{d,n}(z)}  \pr(\cE(\sigma,\sigma', z,(z_{i,j})_{i,j=1}^4))
    \pr[\cE_{opt}(\sigma,\sigma',h)|\cE(\sigma,\sigma', z,(z_{i,j})_{i,j=1}^4)].
\end{split}
\end{equation}

For notational convenience, we additionally define $\beta_i=\frac{1+\omega}{4}$ for $i=1,2$ and $\beta_i=\frac{1-\omega}{4}$ for $i=3,4$. Thus $\abs{V}_i= \beta_in+O(1)$.
Let $\beta_{ij}=\beta_i\beta_j$. To simplify the subsequent analysis, similar to (\ref{eq:def_Dz}), we introduce the following restriction to the cut-sizes:
\begin{equation}\label{eq:def_Dz2}
D_\omega(C,z)=\{(z_{i,j})_{i,j=1}^4\in \mathcal{S}_{d,n}(z): |z_{i,j}n-\beta_{ij} dn| \le C\sqrt{d}n\}.
\end{equation}

By an application of Hoeffding's inequality as in the proof of Lemma \ref{lemma:range-z}, we have:

\begin{lemma}\label{lemma:range-z-2mom}
For every  $C\ge 2$, $d\ge 1$, and $z \in \mathbb{R}$, for large enough $n$, the sum in Eq. \ref{eq:EX2m}
restricted to pairs $(z_{i,j})_{i,j=1}^4$ outside $D_\omega(C,z)$ is bounded by $\exp(-Cn)$.
\end{lemma}

With the above restriction and setup, we now proceed to analyze the asymptotics of the conditional probability:
$$\pr[\cE_{opt}(\sigma,\sigma',h)|\cE(\sigma,\sigma', z,(z_{i,j})_{i,j=1}^4)].$$

Let $O_i$ for $i \in [4]$ denote the event of each vertex in partition $V_i$ being $h$-stable w.r.t both $\sigma, \sigma'$. Under the configuration model, conditioned on $\cE(\sigma,\sigma', z,(z_{i,j})_{i,j=1}^4)$, the events $O_1,O_2,O_3,O_4$ are independent. Therefore:
\begin{equation}\label{eq:optimal}
    \pr[\cE_{opt}(\sigma,\sigma',h)|\cE(\sigma,\sigma', z,(z_{i,j})_{i,j=1}^4)] = \pr[O_1]\pr[O_2]\pr[O_3]\pr[O_4]
\end{equation}

Conditioned on $\cE(\sigma,\sigma', z,(z_{i,j})_{i,j=1}^4)$, each event of the form $O_j$ for $j \in [4]$ can be represented through a balls-in-bins process, which we illustrate for $O_1$: Consider $N$ bins for $N \in \mathbb{N}$ and integer values $\mu_{1,1},\mu_{2,1},\mu_{3,1},\mu_{4,1} > 0$. For each $j \in [4]$, independently assign $\mu_{j,1}$ balls into the $N$ bins, representing the total number of half-edges connecting $V_1$ with $V_1,V_2,V_3,V_4$ respectively.  Let $(A_i^{(j,1)})_{j \in [4]})$ denote the number of balls assigned to the $i_{th}$ bin. When $N=\abs{V_1}$ and $\mu_{1,j}=z_{1,j}$ for $j \in [4]$, the variables $E_i^{(1,1)},E_i^{(2,1)},E_i^{(3,1)},E_i^{(4,1)}$ in (\ref{eq:optimal_1}), denoting the number of half-edges connecting the $i_{th}$ vertex to vertices in the four partitions $V_1, V_2, V_3, V_4$ respectively, possess the same distribution as $(A_i^{(j,1)})_{j \in [4]})$. (To generate the multi-graph, these half-edges (balls) are subsequently matched to half-edges in $V_1,V_2,V_3,V_4$).

Define:
\begin{equation}\label{def:K2def}
K_2(N,\mu_{1,1},\mu_{2,1},\mu_{3,1},\mu_{4,1},h) \coloneqq \mathbb{P}[A_i^{(4,1)}-A_i^{(1,1)} \geq \lvert A_i^{(3,1)}-A_i^{(2,1)} \rvert +h\sqrt{d}, 1 \leq i \leq N],
\end{equation}
where the function $K_2$ plays a role analogous to the role played by $K(\cdot)$ defined in (\ref{eq:Kmu_mubar}) in the first-moment analysis.

By (\ref{eq:optimal}),  under the configuration model, the probability $\pr[\cE_{opt}(\sigma,\sigma',h)|\cE(\sigma,\sigma', z,(z_{i,j})_{i,j=1}^4)]$ can be expressed as:
\begin{equation}\label{eq:opt2ndmom}
\begin{split}
    &\pr[\cE_{opt}(\sigma,\sigma',h)|\cE(\sigma,\sigma', z,(z_{i,j})_{i,j=1}^4)]\\ &= K_2(\abs{V_1},z_{1,1}n,z_{1,2}n,z_{1,3}n,z_{1,4}n,h) K_2(\abs{V_2},z_{2,2}n, z_{2,1}n,z_{2,4}n,z_{2,3}n,h)\\
    &\times K_2(\abs{V_3},z_{3,3}n,z_{3,1}n,z_{3,4}n,z_{3,2}n,h)K_2(\abs{V_4},z_{4,4}n,z_{4,2}n,z_{4,3}n,z_{4,1}n,h).
\end{split}
\end{equation}

\subsection{Large Deviations and Poissonization}

To simplify each of the $K(\cdot)$ terms, we proceed similar to the first moment entropy density, by first utilizing Poissonization to express the probability of each subset of vertices satisfying the corresponding constraints through certain large-deviations rates, and subsequently simplifying the rates through Gaussian approximations and convexity based arguments.

Analogous to Lemma  \ref{lemma_twoPo}, an application of the Poissonization technique (Lemma \ref{lem:poissonization}) yields:
\begin{lemma}
\label{lemmaK4d}
For each $j$, $j=1,\dots,4$, Let $(B_i^{(j,1)})_{i\in [n]}$ be a family of i.i.d. Poisson variables with the same mean $\lambda_{1,j}=\mu_{1,j}/n>0$, then we have
\begin{align}\label{Km_4d}
K_2(n,\mu_1,\mu_2,\mu_3,\mu_4,h)
= \prod_{j=1}^4 \theta_{\mu_j}(\sqrt{\mu_j})\mathbb{P}\left[\sum_{i=1}^{n} B_i^{(j,1)}=\mu_j, 1 \leq j \leq 4   \Biggl| \mathcal{B}_i, 1 \leq i \leq n \right]
                                  (\mathbb{P}[ \mathcal{B}_1 ])^{n},
\end{align}
where $\mathcal{B}_i$ denotes the event $B_i^{(4,1)}-B_i^{(1,1)} \geq \lvert B_i^{(3,1)}-B_i^{(2,1)} \rvert + h\sqrt{d}$.
\end{lemma}

\begin{proof}
For $i \in [n]$, let $(A^{(j,1)})_{j \in [4]}$ be as in \eqref{def:K2def}.
Let $S(\mu_1,\mu_2,\mu_3, \mu_4, h)$ denote the event of all vertices in $V_1$ being $h$-stable w.r.t both $\sigma, \sigma'$, defined as:
    \begin{align*}
&S(\mu_{1,1},\mu_{2,1},\mu_{3,1}, \mu_{4,1}, h)
\\&=\{((A_i^{(1,1)}, A_i^{(2,1)}, A_i^{(3,1)}, A_i^{(4,1)}))_{1\leq i \leq N} \in 
(\mathbb{Z}_{\geq 0})^{4N}:\mathcal{A}_i, 1\leq i \leq N; 
\sum_{i=1}^{N} A_i^{(j,1)},=\mu_{j,1}, j \in [4]\}.
\end{align*}
where $\mathcal{A}_i$ denotes the event $A_i^{(4,1)}-A_i^{(1,1)} \geq 
\abs{A_i^{(3,1)}- A_i^{(2,1)}}+h\sqrt{d}$

By the total law of probability, $K_2(N,\mu_1,\mu_2,\mu_3,\mu_4, h)$ can be expressed as:

\begin{align*}
&K_2(N,\mu_{1,1},\mu_{2,1},\mu_{3,1},\mu_{4,1}, h) \\
&=\sum_{S(\mu_{1,1},\mu_{2,1},\mu_{3,1} \mu_{4,1}, h)} 
 \prod_{j=1}^4\mathbb{P}[A^{(j,1)}_i=t_i, 1\leq i\leq 
 N]
\end{align*}

We then apply Poissonization (Lemma \ref{lem:poissonization}) to the RHS to obtain:
\begin{align*}
&K_2(N,\mu_{1,1},\mu_{2,1},\mu_{3,1},\mu_{4,1}, h)\\ &=\Theta_{\mu_{1,1}}(\sqrt{\mu_{1,1}}) \Theta_{\mu_{2,1}}(\sqrt{\mu_{2,1}})\Theta_{\mu_{3,1}}(\sqrt{\mu_{3,1}})\Theta_{\mu_{4,1}}(\sqrt{\mu_{4,1}})
\sum_{S(\mu_{1,1},\mu_{2,1},\mu_{3,1}, \mu_{4,1}, h)} 
\prod_{j=1}^4 \mathbb{P}[B^{(j,1)}_i=t_i, 1\leq i\leq 
 N]\\
&=
\Theta_{\mu_{1,1}}(\sqrt{\mu_{1,1}}) \Theta_{\mu_{2,1}}(\sqrt{\mu_{2,1}})\Theta_{\mu_{3,1}}(\sqrt{\mu_{3,1}})\Theta_{\mu_{4,1}}(\sqrt{\mu_{4,1}})
 \pr[\sum_{1\le i\le N}B^{(j,1)}_i=\mu_{j,1} \forall j \in [4],
\mathcal{B}_i, 1\le i\le  N] \\
 &=
  \Theta_{\mu_{1,1}}(\sqrt{\mu_{1,1}}) \Theta_{\mu_{2,1}}(\sqrt{\mu_{2,1}})\Theta_{\mu_{3,1}}(\sqrt{\mu_{3,1}})\Theta_{\mu_{4,1}}(\sqrt{\mu_{4,1}})
 \pr[\sum_{1\le i\le N}B^{(j,1)}_i=\mu_{j,1} \forall j \in [4],
\mathcal{B}_i, 1\le i\le  N]\times \\
 &\times \left(\pr[\mathcal{B}_i]\right)^{N},
\end{align*}
where $\mathcal{B}_i$ denotes the event $B_i^{(4,1)}-B_i^{(1,1)} \geq 
\abs{B_i^{(3,1)}- B_i^{(2,1)}}+h\sqrt{d}$
\end{proof}

Next, similar to Corollary \ref{coro:Poisson-local-LD}, we estimate $K_2(\cdot)$ through the following large-deviations result:

\begin{lemma}\label{lem:poiss_2nd_mom}
Suppose $Z_1,Z_2, Z_3,Z_4$ are  Poisson random variables with rates $\lambda_1, \lambda_2,\lambda_3, \lambda_4$. Fix $a>0$ such that:
\begin{equation}\label{eq:lam_cond}
    \lambda_1-\lambda_4 \geq \abs{\lambda_2-\lambda_3} + a+2.
\end{equation}
Let $P$ be the probability distribution of 
$\left({Z_1-\lambda_1\over \sqrt{\lambda_1}},
{Z_2-\lambda_2 \over\sqrt{\lambda_2}}, {Z_3-\lambda_3 \over\sqrt{\lambda_3}}, {Z_4-\lambda_4 \over\sqrt{\lambda_4}}\right)$ conditioned on $Z_4-Z_1\ge \abs{Z_2-Z_3}+a$.
Suppose $X_i\stackrel{d}{=} P$ i.i.d. 
$1\le i\le N$,
Let $S_N=\sum_{i\in [N]}X_i$. Then, there exists a
unique $\theta^*\in \R^4$  satisfying:
\begin{align}\label{eq:thetst2mom}
\nabla \Lambda(\theta^*)=0,
\end{align}
where $\Lambda$ denotes the log-Moment Generating
Function of $P$. 
Furthermore,
\begin{align}\label{eq:ldp2mon}
\lim_{N \rightarrow \infty} {1\over N}\log 
\pr\left(S_N=0\right)
=\Lambda(\theta^*).
\end{align}
\end{lemma}

\begin{proof}
By  Lemma \ref{LCLTLDP}, (\ref{eq:ldp2mon}) follows directly from the existence of $\theta^\star$ satisfying (\ref{eq:thetst2mom}). To show the existence of such $\theta^\star$, we begin by noting that by the convexity of $\Lambda(\theta)$, the condition  $\nabla \Lambda(\theta^*)=0$ is equivalent to the existence of a minimizer of $\Lambda(\theta)$.

Next, by (\ref{eq:lam_cond}), similar to (\ref{eq:support_points}), we note that the following subset of points lie in the support of $P$:
    \begin{equation}
       V= \{(\pm \frac{1}{\sqrt{\lambda_1}}, \pm \frac{1}{\sqrt{\lambda_2}}, \pm \frac{1}{\sqrt{\lambda_3}}, \pm \frac{1}{\sqrt{\lambda_4}})\}
    \end{equation}
Let:
\begin{equation}
    k^\star= \inf_{\theta: \norm{\theta}=1} \max_{v \in V} (\langle v, \theta \rangle)
\end{equation}
Therefore:
\begin{align*}
    \Lambda(\theta)&=\log \left(\sum_{x}\exp\left(\langle \theta, x \rangle \right)P(X=x)\right) \\
&\ge \log[\exp(k^\star \norm{\theta}_2)[\min_{v\in V} P_1(X=v)]].
\end{align*}

Since $k^\star >0$, $\Lambda(\theta) \rightarrow \infty$ as $\theta \rightarrow \infty$, establishing the existence of a minimizer.
\end{proof}

Next, similar to Lemma \ref{lemma:Gaussian-of-Pois}, approximating the Poisson random variables through Gaussians yields the following large-deviations rate:
\begin{lemma}\label{lem:2ndmompoiss}
Let $Z_1,Z_2, Z_3,Z_4$ be Poisson random variables as in Lemma \ref{lem:poiss_2nd_mom} indexed by $t\in\Z_+$, 
with  parameters $\lambda_1(t),\lambda_2(t), \lambda_3(t), \lambda_4(t)$
satisfying the following
\begin{enumerate}
\item $\lambda_j(t)= \zeta_j t+\tau_j\sqrt{t}+o_t(\sqrt{t}), 
j=1,2,3,4$, as $t\to\infty$, where $(\zeta_j)_{j \in [4]} \in \mathbb{R}$ satisfy $\zeta_1=\zeta_4, \zeta_2=\zeta_3$. 

\item $\lambda_j(t), j=1,2$ are integers.
\end{enumerate}

Suppose further that $a_1(t)$ is a sequence such that:
\begin{equation}
    a(t) = b\sqrt{t} + o(\sqrt{t}),
\end{equation}
and:
\begin{equation}
    \tau_4-\tau_1 > \abs{\tau_2-\tau_3}+b.
\end{equation}

Define:
\begin{equation}
    b_1 \coloneqq \frac{\tau_2-\tau_3}{\sqrt{\zeta_2+\zeta_3}}, \ b_2 \coloneqq \sqrt{\frac{\zeta_2+\zeta_3}{\zeta_1+\zeta_4}}, \ b_3 \coloneqq \frac{\tau_4-\tau_1-b}{\sqrt{\zeta_4+\zeta_1}},
\end{equation}
Then:
\begin{equation}\label{eq:limit_poiss}
\begin{split}
    \lim_{t \rightarrow \infty} \lim_{N\rightarrow \infty} {1\over N}\log 
\pr\left(S_N=0\right) &= -\log \Pr\left[Y_1 +b_3\geq b_2 \abs{Y_2+b_1}\right]+\inf_{\theta_1,\theta_2}\log \mathbb{E}\left[\exp(\theta_1 Y_1+\theta_2 Y_2)\mathbf{1}_{[Y_1+b_3 \geq b_2\abs{Y_1+b_1}]}\right],
\end{split}
\end{equation}
where $Y_1,Y_2$ denote i.i.d standard Normal variables.
\end{lemma}

\begin{proof}
Let $\Lambda(\theta)$ for $\theta \in \mathbb{R}^4$ denote again the log-MGF of the normalized and centered Poisson variables  $B_1,B_2,B_3,B_4 \coloneqq \left({Z_1-\lambda_1\over \sqrt{\lambda_1}},
{Z_2-\lambda_2 \over\sqrt{\lambda_2}}, {Z_3-\lambda_3 \over\sqrt{\lambda_3}}, {Z_4-\lambda_4 \over\sqrt{\lambda_4}}\right)$ subject to the constraint $Z_4 - Z_1 \geq \abs{Z_2-Z_3}+a$. Lemma \ref{lem:poiss_2nd_mom} then implies that:
\begin{equation}
  \lim_{N \rightarrow \infty} {1\over N}\log 
\pr\left(S_N=0\right)=  \inf_{\theta \in \mathbb{R}}\Lambda(\theta)
\end{equation}

 Note that the constraint  $Z_4 - Z_1 \geq \abs{Z_2-Z_3}+a$ can be rewritten w.r.t $B_1,B_2,B_3,B_4$ as:
\begin{equation}
    \frac{\sqrt{\lambda_4}B_4-\sqrt{\lambda_1}B_1}{\sqrt{\lambda_1+\lambda_4}}+\frac{\lambda_4-\lambda_1-a}{\sqrt{\lambda_1+\lambda_4}} \geq \sqrt{\frac{\lambda_2+\lambda_3}{\lambda_4+\lambda_1}}\abs{\frac{\sqrt{\lambda_2}B_2-\sqrt{\lambda_3}B_3}{\sqrt{\lambda_2+\lambda_3}}+\frac{\lambda_2-\lambda_3}{\sqrt{\lambda_2+\lambda_3}}}
\end{equation}
Note that $\frac{\lambda_1}{\lambda_1+\lambda_4} = \frac{\zeta_1}{\zeta_1+\zeta_4} +o_t(1)$. Our choice of $\zeta_j$ for $j\in [4]$, then implies that:
\begin{equation}
    \frac{\lambda_1}{\lambda_1+\lambda_4}=\frac{\lambda_4}{\lambda_1+\lambda_4} = \frac{\lambda_2}{\lambda_2+\lambda_3}= \frac{\lambda_3}{\lambda_2+\lambda_3} = \frac{1}{2}+o_t(1)
\end{equation}

Therefore, by Poisson-Gaussian approximation, $\Lambda(\theta)$ 
converges point-wise to the log-MGF for standard Gaussian variables 
$\tilde{Z} \in \mathbb{R}^4$, $\tilde{Z} \sim \mathcal{N}(0,I_4)$  subject to constraint $\frac{1}{\sqrt{2}}(\tilde{Z}_4-\tilde{Z}_1)+b_3 \geq \frac{b_2}{\sqrt{2}} \abs{\tilde{Z}_2-\tilde{Z}_3+b_1}$.

The associated log-MGF is denoted as $\Lambda^\star$, and is given by:
\begin{align}
 \Lambda^\star(\theta_1, \theta_2, \theta_3,\theta_4)&=    -\log \Pr\left[\frac{1}{\sqrt{2}}(\tilde{Z}_4-\tilde{Z}_1)+b_3 \geq b_2 \abs{\frac{1}{\sqrt{2}}(\tilde{Z}_2-\tilde{Z}_3)+b_1}\right]
\\
&+\log
\mathbb{E}\left[\exp(\sum_{i=1}^4 \theta_i \tilde{Z}_i)\mathbf{1}_{\left[\frac{1}{\sqrt{2}}(\tilde{Z}_4-\tilde{Z}_1)+b_3 \geq b_2{}\abs{\frac{1}{\sqrt{2}}(\tilde{Z}_2-\tilde{Z}_3)+b_1}\right]}\right]
\notag.
     \end{align}

Next, through an argument similar to the proof of Lemma \ref{lemma:Gaussian-of-Pois}, we translate the point-wise convergence of $\Lambda(\theta)$ to
$\Lambda^\star(\theta)$ to the convergence of their minimum values.  Let $P^\star$ denote the distribution of $\tilde{Z} \sim \mathcal{N}(0,I_4)$  conditioned on the constraint $\frac{1}{\sqrt{2}}(\tilde{Z}_4-\tilde{Z}_1)+b_3 \geq b_2{}\abs{\frac{1}{\sqrt{2}}(\tilde{Z}_2-\tilde{Z}_3)+b_1}$.
Recall that $\Lambda^\star(\theta)$ is the log-MGF w.r.t $P^\star$ and therefore convex in $\theta$. 
The condition $\tau_4-\tau_1 > \abs{\tau_2-\tau_3}+a$ implies that $b_3 > \abs{b_1}b_2$ and thus $\frac{1}{\sqrt{2}}(\tilde{Z}_4-\tilde{Z}_1)+b_3 \geq b_2 \abs{\frac{1}{\sqrt{2}}(\tilde{Z}_2-\tilde{Z}_3)+b_1}$ is satisfied for $(\tilde{Z}_i)_{i \in [4]}$ restricted to a ball of small enough radius.

Therefore, $\exists \delta > 0$, such that the shell $S_\delta = \{\frac{\delta}{2}< \norm{x}_2 \leq \delta\}$ lies in the support of $P^\star$. We thus obtain the lower-bound
\begin{equation}
    \Lambda^\star(\theta) \geq \exp(\norm{\theta}_2\frac{\delta}{2})P^\star[S_\delta],
\end{equation}
implying that $\Lambda^\star \rightarrow \infty$ as $\theta \rightarrow \infty$ and $\Lambda^\star(\theta)$ admits a unique maximizer $\theta^\star$.

Subsequently, as in the proof of Lemma \ref{lemma:Gaussian-of-Pois}, the uniform convergence of convex functions on compact sets implies:
\begin{equation}
    \inf_{\theta} \Lambda(\theta) \rightarrow \inf_{\theta}\Lambda^\star(\theta),
\end{equation}
as $t \rightarrow \infty$

% Subsequently, as in the proof of Lemma \ref{lemma:Gaussian-of-Pois}, the coercivity and strict convexity of $\Lambda_1(\theta)$, allows us to substitute $\theta^\star$ in (\ref{eq:limit_poiss}) as the maximizer of the limiting Gaussian MGF.

The above rate can be simplified by exploiting symmetries between the roles played by $\tilde{Z}_1,\tilde{Z}_4$ and $\tilde{Z}_2,\tilde{Z}_3$ respectively. Concretely, recall that $\Lambda^\star_{1}(\theta)$ admits a unique maximizer $\theta^\star_1, \theta^\star_2, \theta^\star_3, \theta^\star_4$.

Note that $\Lambda^\star$ is even in $\theta_2-\theta_3$ and odd in $\theta_4-\theta_1$. In particular, we have:
\begin{equation}
    \Lambda^\star(\theta_1, \theta_2, \theta_3,\theta_4) =  \Lambda^\star(\theta_1, \theta_3, \theta_2,\theta_4), 
\end{equation}
and 
\begin{equation}
    \Lambda^\star(\theta_1, \theta_2, \theta_3,\theta_4) =  \Lambda^\star(-\theta_4, \theta_3, \theta_4,-\theta_1). 
\end{equation}
Therefore, any minimizer $\theta^\star_1,\theta^\star_2,,\theta^\star_3,,\theta^\star_4$ of $\Lambda^\star$, maps to another minimizer $-\theta^\star_4,\theta^\star_3,\theta^\star_2,-\theta^\star_1$. Since the minimizer of $\Lambda^\star(\cdot)$ is unique, we obtain:
\[
    \theta^\star_1=-\theta^\star_4,
\]
and \[
\theta^\star_2=\theta^\star_3.\]

Hence, we may set $\Lambda^\star(\theta^\star)$ as the minimum upon the restriction $\theta_4=\theta_1$ and $\theta_3=\theta_2$. Upon such restriction, $\Lambda^\star(\theta^\star)$ simplifies as:
\begin{align*}
 \Lambda^\star(\theta^\star)-&\log \Pr\left[\frac{1}{\sqrt{2}}{\tilde{Z}_4-\tilde{Z}_1}+b_3 \geq b_2 \abs{\frac{1}{\sqrt{2}}(\tilde{Z}_2-\tilde{Z}_3)+b_1}\right]
\\
&+\log
\mathbb{E}\left[\exp(\theta_1 (\tilde{Z}_1+\tilde{Z}_4)+\theta_2 (\tilde{Z}_2-\tilde{Z}_4))\mathbf{1}_{\left[\frac{1}{\sqrt{2}}(\tilde{Z}_4+\tilde{Z}_1)+b_3 \geq b_2\abs{\frac{\tilde{Z}_2-\tilde{Z}_3}{\sqrt{2}} +b_1}\right]}\right]\\
\end{align*}

Finally, note that $Y_1\coloneqq \frac{1}{\sqrt{2}}{(\tilde{Z}_4-\tilde{Z}_1)}$ and $ Y_2 \coloneqq \frac{1}{\sqrt{2}}{(\tilde{Z}_3-\tilde{Z}_2)}$ are distributed as standard independent normal variables.

Therefore, upon replacing $\theta_1 \leftarrow \sqrt{2} \theta_1, \theta_2 \leftarrow \sqrt{2} \theta_2$ $\Lambda^\star(\theta^\star)$ further simplifies to:

\begin{equation}
\Lambda^\star(\theta^\star)= -\log \Pr[Y_1 \geq b_2 \abs{Y_2}+b_3]+\inf_{\theta_1,\theta_2}\log \mathbb{E}\left[\exp(\sum_{i=1}^2 \theta_i Y_i)\mathbf{1}_{[Y_1+b_3 \geq b_2\abs{Y_2+b_1}]}\right].
\end{equation}
\end{proof}

The above Lemma combined with Lemma \ref{lemmaK4d} results in the 
asymptotic expression for $K_2(\cdot)$: 

\begin{corollary}\label{cor:K_asymp_2}
Let $(\mu_{1,i}(d))_{i\in [4]},a(d)$ denote a sequence of integer rates satisfying:
\begin{equation}
    \mu_{1,i}(d) = (\beta_{1,i} \beta_1 d+\tau_{1,i}\sqrt{d} + o_d(\sqrt{d}))N,
    \end{equation}
\begin{equation}
a(d)=h\sqrt{d}+o_d(\sqrt{d}), 
\end{equation}
\begin{equation}\label{eq:Kb2}
\begin{split}
   &\lim_{d \rightarrow \infty}  \lim_{N \rightarrow \infty} \frac{1}{N} \log K_2(\floor{\beta_1 N}, \mu_{1,1}(d) N, \mu_{1,2}(d), \mu_{1,3}(d), \mu_{1,4}(d), a(d))\\ &= L_2\left(\frac{(\tau_{1,2}-\tau_{1,3})}{\beta_1\sqrt{\beta_2+\beta_3}}, \sqrt{\frac{\beta_2+\beta_3}{\beta_4+\beta_1}}, \frac{\tau_{1,4}-\tau_{1,1}}{\beta_1\sqrt{\beta_1+\beta_4}}-\frac{h}{\sqrt{\beta_1+\beta_4}}\right),
\end{split}
\end{equation}
where:
\begin{equation}\label{eq:def_L2}
    L_2(b_1, b_2, b_3)  \triangleq \inf_{\theta_1,\theta_2}\log \mathbb{E}[\exp(\theta_1 Y_1+\theta_2 Y_2)\mathbf{1}_{[Y_1+b_3 \geq b_2\abs{Y_1+b_1}]}].
\end{equation}
\end{corollary}

Combining the above estimate for the four partitions and summing across configurations will leads to a variational problem over the edge partition sizes $z_{ij}$. To analyze the same, we introduce the following function over $16$ scalars $\tau_{i,j}, i, j \in [4]^2$:
\begin{equation}\label{eq:K_def_obj}
\begin{split}
   \mathcal{K}_2(\tau_{i,j}, i, j \in [4]^2)  &\coloneqq  \frac{1+\omega}{4}L_2\left(\frac{4\sqrt{2}(\tau_{1,2}-\tau_{1,3})}{\sqrt{(1-\omega)}(1+\omega)}, \sqrt{\frac{(1-\omega)}{(1+\omega)}}, \frac{4\sqrt{2}(\tau_{1,4}-\tau_{1,1})}{(1+\omega)^{3/2}}-\frac{2h}{\sqrt{1+\omega}}\right)\\&+\frac{1-\omega}{4}L_2\left(\frac{4\sqrt{2}(\tau_{2,1}-\tau_{2,4})}{\sqrt{(1+\omega)}(1-\omega)}, \sqrt{\frac{(1+\omega)}{(1-\omega)}}, \frac{4\sqrt{2}(\tau_{3,2}-\tau_{2,2})}{(1-\omega)^{3/2}}-\frac{2h}{\sqrt{1-\omega}}\right)\\&+\frac{1-\omega}{4}L_2\left(\frac{4\sqrt{2}(\tau_{3,1}-\tau_{3,4})}{\sqrt{(1+\omega)}(1-\omega)}, \sqrt{\frac{(1+\omega)}{(1-\omega)}}, \frac{4\sqrt{2}(\tau_{3,2}-\tau_{3,3})}{(1-\omega)^{3/2}}-\frac{2h}{\sqrt{1-\omega}}\right)\\&+\frac{1+\omega}{4}L_2\left(\frac{4\sqrt{2}(\tau_{4,2}-\tau_{4,3})}{\sqrt{(1-\omega)}(1+\omega)}, \sqrt{\frac{(1-\omega)}{(1+\omega)}}, \frac{4\sqrt{2}(\tau_{4,1}-\tau_{4,4})}{(1+\omega)^{3/2}}-\frac{2h}{\sqrt{1+\omega}}\right).
\end{split}
\end{equation}

Fortunately, the number of parameters in the above objective can be substantially reduced by exploiting certain constraints on $(\tau_{i,j})_{(i,j) \in [4]^2}$ and certain properties of $L_2$. These properties are obtained in the following Lemma, whose proof can be found in Appendix \ref{app:properties}.

\begin{lemma}\label{lem:l2conc}
    $L_2(b_1, b_2, b_3)$ is strictly concave in $b_1,b_3$ and even in $b_1$. 
\end{lemma}

The strict concavity and even symmetry of $L_2(\cdot)$ w.r.t $b_1$ imply that for any $b_2, b_3 \in \mathbb{R}$, $ L_2(b_1,b_2,b_3)$ is uniquely maximized w.r.t $b_1$ at $b_1=0$. Hence, we obtain the following bound on $ L_2(\cdot)$:
   \begin{equation}\label{eq:0min}
       L_2(b_1,b_2,b_3) \leq  L_2(0,b_2,b_3), \, \forall b_1, b_2, b_3 \in \mathbb{R}^3:
   \end{equation}

We now leverage the above properties of $ L_2(\cdot)$, together with specific constraints on $(\tau_{i,j})_{(i,j) \in [4]^2}$, to reduce the optimization of $\mathcal{K}_2(\cdot)$ in ~\eqref{eq:K_def_obj} over $(\tau_{i,j})_{(i,j) \in [4]^2}$ to a problem involving a single parameter. These constraints on \((\tau_{i,j})_{(i,j) \in [4]^2}\) will later be connected to the constraints on \((z_{i,j})_{(i,j) \in [4]^2}\) in~\eqref{eq:cut_constraint} and~\eqref{eq:sum_bis}. The reduction of is formalized in the lemma below, where \(t\) denotes the sole remaining ``free parameter.''

\begin{lemma}\label{lemma:Optimum-LD-rates-2nd-monent}

Consider the objective $\mathcal{K}_2(\cdot)$ defined by (\ref{eq:K_def_obj}). Fix $t \in \mathbb{R}$. Define the constraint set $\mathcal{S}_E(t)$ as the set of $(\tau_{i,j})_{(i,j) \in [4]^2}$ satisfying:
\begin{equation}\label{eq:tau_cut_constraint}
\begin{split}
    \tau_{1,3}+\tau_{1,4} + \tau_{2,3}+ \tau_{2,4} &=-\frac{E}{2},\\
    \tau_{1,2}+\tau_{1,4} + \tau_{2,3}+\tau_{3,4} &=-\frac{E}{2},
\end{split}
\end{equation}
\begin{equation}\label{eq:tot_const}
    \sum_{i<j} \tau_{i,j} + \frac{1}{2} \sum_{i=1}^4 \tau_{i,i}=0,
\end{equation}
\begin{equation}
    \tau_{1,4}-\frac{1}{2}(\tau_{1,1}+\tau_{4,4})=t.
\end{equation}
Then, for any $t \in \mathbb{R}$:
\begin{equation}\label{eq:k2optt}
\begin{split}
    &\sup_{\{\tau_{i,j}:(ij) \in [4]^2\} \in \mathcal{S}_E(t)} \mathcal{K}_2(\tau_{i,j}, i, j \in [4])\\ &=  \frac{1+\omega}{4}L_2\left(0, \sqrt{\frac{(1-\omega)}{(1+\omega)}}, \frac{4\sqrt{2}t}{(1+\omega)^{3/2}}-\frac{2h}{\sqrt{1+\omega}}\right)+ \frac{1-\omega}{4}L_2\left(0, \sqrt{\frac{(1+\omega)}{(1-\omega)}},\frac{4\sqrt{2}(-E-t)}{(1-\omega)^{3/2}}-\frac{2h}{\sqrt{1-\omega}}\right),
\end{split}
\end{equation}
with the supremum achieved at $\tau_{1,4}=\frac{t}{2}$, $\tau_{1,1}=\tau_{4,4}=-\frac{t}{2}$,$\tau_{2,3}=\frac{E-t}{2}$,  $\tau_{2,2}=\tau_{3,3}=-\frac{E-t}{2}$ and $\tau_{1,2}=\tau_{1,3}=\tau_{2,4}=\tau_{3,4}=0$.
\end{lemma}
\begin{proof}

% We begin by defining the following subset of variables:
% \begin{equation}
%     \mathcal{T}_1 \coloneqq \{\tau_{14},\tau_{23},\tau_{11},\tau_{22},\tau_{33},\tau_{44}\}
% \end{equation}
% \begin{equation}  
% \mathcal{T}_2 \coloneqq \{\tau_{12},\tau_{13},\tau_{24},\tau_{34}\}
% \end{equation}

% \begin{equation}\label{eq:supsupk2}
% \sup_{\{\tau_{i,j}:(i,j) \in [4]^2\} \in \mathcal{S}_E(t)} \mathcal{K}_2(\tau_{i,j}, i, j \in [4])=\sup_{\mathcal{T}_1} \sup_{\substack{\mathcal{T}_2 \\ \mathcal{T}_1 \cup \mathcal{T}_2 \in \mathcal{S}_E(t)}}\mathcal{K}_2(\tau_{i,j}, i, j \in [4]),
% \end{equation}

Applying the inequality (\ref{eq:0min}) to each $L_2(\cdot)$ term in the definition of $\mathcal{K}_2(\cdot)$ given by (\ref{eq:K_def_obj}), we obtain the following bound for any $(\tau_{i,j})_{i,j \in [4]^2}$:

\begin{equation}\label{eq:K2bound}
\begin{split}
&\mathcal{K}_2(\tau_{i,j}, i, j \in [4]^2)   \leq \frac{1+\omega}{4}L_2\left(0, \sqrt{\frac{(1-\omega)}{(1+\omega)}}, \frac{4\sqrt{2}(\tau_{1,4}-\tau_{1,1})}{(1+\omega)^{3/2}}-\frac{2h}{\sqrt{1+\omega}}\right)\\&+\frac{1-\omega}{4}L_2\left(0, \sqrt{\frac{(1+\omega)}{(1-\omega)}}, \frac{4\sqrt{2}(\tau_{2,3}-\tau_{2,2})}{(1-\omega)^{3/2}}-\frac{2h}{\sqrt{1-\omega}}\right)+\frac{1-\omega}{4}L_2\left(0, \sqrt{\frac{(1+\omega)}{(1-\omega)}}, \frac{4\sqrt{2}(\tau_{2,3}-\tau_{3,3})}{(1-\omega)^{3/2}}-\frac{2h}{\sqrt{1-\omega}}\right)\\&+\frac{1+\omega}{4}L_2\left(0, \sqrt{\frac{(1-\omega)}{(1+\omega)}}, \frac{4\sqrt{2}(\tau_{4,1}-\tau_{4,4})}{(1+\omega)^{3/2}}-\frac{2h}{\sqrt{1+\omega}}\right).
\end{split}
\end{equation}

Adding the constraints in (\ref{eq:tau_cut_constraint}) and subtracting (\ref{eq:tot_const}), we further have that any $(\tau_{i,j}:i,j \in [4]^2) \in \mathcal{S}_E(t)$ satisfy:
\begin{equation}
    2\tau_{1,4}-\tau_{1,1}-\tau_{4,4}+2\tau_{2,3}-\tau_{2,2} + -\tau_{3,3} =-2E.
\end{equation}

Since $2\tau_{1,4}-\tau_{1,1}-\tau_{4,4}$ is fixed to the value $2t$, we obtain $2\tau_{2,3}-\tau_{2,2} + -\tau_{3,3}=-2E-2t$. Therefore, by defining $t_1=\frac{4\sqrt{2}(\tau_{1,4}-\tau_{1,1})}{(1+\omega)^{3/2}}-\frac{2h}{\sqrt{1+\omega}}, t_2= \frac{4\sqrt{2}(\tau_{2,3}-\tau_{2,2})}{(1-\omega)^{3/2}}-\frac{2h}{\sqrt{1-\omega}}, t_3=\frac{4\sqrt{2}(\tau_{3,2}-\tau_{3,3})}{(1-\omega)^{3/2}}-\frac{2h}{\sqrt{1-\omega}}$, $t_4 = \frac{4\sqrt{2}(\tau_{4,1}-\tau_{4,4})}{(1+\omega)^{3/2}}-\frac{2h}{\sqrt{1+\omega}}$, (\ref{eq:K2bound}) can be rewritten as:
\begin{align*}
\mathcal{K}_2(\tau_{i,j}, i, j \in [4]) &\leq  \frac{1+\omega}{4}\left(L_2\left(0, \sqrt{\frac{(1+\omega)}{(1-\omega)}}, t_1\right)+  L_2\left(0, \sqrt{\frac{(1+\omega)}{(1-\omega)}}, t_4\right)\right)\\
&+\frac{1-\omega}{4}\left(L_2\left(0, \sqrt{\frac{(1-\omega)}{(1+\omega)}},t_2\right)+ L_2\left(0, \sqrt{\frac{(1-\omega)}{(1+\omega)}},t_3\right)\right).
\end{align*}
By the convexity of $L_2(\cdot,0,b)$ (Lemma \ref{lem:l2conc}), we have:
\begin{equation}\label{eq:conv_bound_1}
    L_2\left(0, \sqrt{\frac{(1+\omega)}{(1-\omega)}}, t_1\right) +L_2\left(0, \sqrt{\frac{(1+\omega)}{(1-\omega)}}, t_4\right) \leq 2  L_2\left(0, \sqrt{\frac{(1+\omega)}{(1-\omega)}}, \frac{t_1+t_4}{2}\right),
\end{equation}
and
\begin{equation}\label{eq:conv_bound_2}
 L_2\left(0, \sqrt{\frac{(1-\omega)}{(1+\omega)}},t_2\right)+L_2\left(0, \sqrt{\frac{(1-\omega)}{(1+\omega)}},t_3\right) \leq 2  L_2\left(0, \sqrt{\frac{(1-\omega)}{(1+\omega)}}, \frac{t_2+t_3}{2}\right).
\end{equation}

The constraint $\tau_{1,4}-\frac{1}{2}(\tau_{1,1}+\tau_{4,4})=t$ implies that $\frac{t_1+t_4}{2}=\frac{4\sqrt{2}t}{(1-\omega)^{3/2}}-\frac{2h}{\sqrt{1+\omega}}$ and $\frac{t_2+t_3}{2}=\frac{4\sqrt{2}(-E-t)}{(1-\omega)^{3/2}}-\frac{2h}{\sqrt{1-\omega}}$. Substituting in (\ref{eq:conv_bound_1}), (\ref{eq:conv_bound_2}), we obtain a bound matching the RHS of (\ref{eq:k2optt}):
\begin{equation}
\begin{split}
    \sup_{\{\tau_{i,j}:(ij) \in [4]^2\} \in \mathcal{S}_E(t)} \mathcal{K}_2(\tau_{i,j}, i, j \in [4]) &\leq \frac{1+\omega}{2}L_2\left(0, \sqrt{\frac{(1-\omega)}{(1+\omega)}}, \frac{4\sqrt{2}t}{(1+\omega)^{3/2}}-\frac{2h}{\sqrt{1+\omega}}\right)\\&+ \frac{1-\omega}{2}L_2\left(0, \sqrt{\frac{(1+\omega)}{(1-\omega)}}, \frac{4\sqrt{2}(-E-t)}{(1-\omega)^{3/2}}-\frac{2h}{\sqrt{1-\omega}}\right).
\end{split}
\end{equation}

To show that the above bound is tight, we establish the existence of $(\tau_{i,j},{i,j}\in [4]^2) \in \mathcal{S}_E(t)$ achieving the above bound. Let $t^\star$ denote the maximizer of the above objetive in RHS. Set $\tau^\star_{1,4}=\frac{t^\star}{2}$ and $\tau^\star_{1,1}=\tau^\star_{4,4}=-\frac{t^\star}{2}$,$\tau^\star_{2,3}=\frac{E-t^\star}{2}$,  $\tau^\star_{2,2}=\tau^\star_{3,3}=-\frac{E-t^\star}{2}$ and $\tau^\star_{1,2}=\tau^\star_{1,3}=\tau^\star_{2,4}=\tau^\star_{3,4}=0$. Then, $(\tau^\star_{i,j})_{i,j \in [4]^2}$ satisfy all the constraints in $\mathcal{S}_E(t)$ and ensure that the above bounds are tight, implying: 
\begin{align*}
     &\sup_{\{\tau_{i,j}:(ij) \in [4]^2\} \in \mathcal{S}_E(t)} \mathcal{K}_2(\tau_{i,j}, i, j \in [4]) = \mathcal{K}_2(\tau^\star_{i,j}, i, j \in [4])\\
     &=\frac{1+\omega}{2}L_2\left(0, \sqrt{\frac{(1-\omega)}{(1+\omega)}}, \frac{4\sqrt{2}t}{(1+\omega)^{3/2}}-\frac{2h}{\sqrt{1+\omega}}\right)+ \frac{1-\omega}{2}L_2\left(0, \sqrt{\frac{(1+\omega)}{(1-\omega)}}, \frac{4\sqrt{2}(-E-t)}{(1-\omega)^{3/2}}-\frac{2h}{\sqrt{1-\omega}}\right).
\end{align*}
% Finally expanding $L_2(t_3,0,b)$ recovers $P(\cdot)$
\end{proof}

Additionally, analogous to Lemma \ref{lem:L_lips}, we establish the following regularity property of $L_2(\cdot)$, whose proof can again be found in Appendix \ref{app:properties}:

\begin{lemma}\label{lem:l2cont}
For any $E<-\frac{h}{2}$, $L_2(0,t,b)$ is continuously differentiable in $t,b$ on the (open) half-space $b>0$.
\end{lemma}

\subsection{Putting things together}

We begin by noting the following monotonicity property satisfied by $K_2(\cdot)$, which is a direct consequence of its definition:

\begin{proposition}\label{prop:K_2mon}
Suppose that $\mu_{1,1},\mu_{2,1},\mu_{3,1},\mu_{4,1}$ and $\tilde{\mu}_{1,1},\tilde{\mu}_{2,1},\tilde{\mu}_{3,1},\tilde{\mu}_{4,1}$ satisfy $\mu_{1,4}-\mu_{1,1} \geq \tilde{\mu}_{1,4}-\tilde{\mu}_{1,1}$ and  $\abs{\mu_{3,1}-\mu_{2,1}} \leq \abs{\tilde{\mu}_{3,1}-\tilde{\mu}_{2,1}}$. Then, for any $N \in \mathbb{N}, h \in \mathbb{R}$:
\begin{equation}
K_2(N,\mu_{1,1},\mu_{2,1},\mu_{3,1},\mu_{4,1},h) \geq  K_2(N,\tilde{\mu}_{1,1},\tilde{\mu}_{2,1},\tilde{\mu}_{3,1},\tilde{\mu}_{4,1},h).
\end{equation}
\end{proposition}

\subsection{Proof of Proposition \ref{prop:second_mom_main}}

Recall the decomposition of $\mathbb{E}[X^2_\omega(z,h)]$ in (\ref{eq:EX2m}):
\begin{equation}\label{EX2mrem}
    \mathbb{E}[X^2_\omega(z,h)] 
    = 2^n\binom{n}{\frac{1+\omega}{2}\frac{1-\omega}{2}}\sum_{(z_{i,j})_{i,j=1}^4 \in \mathcal{S}_{d,n}(z)}  \pr(\cE(\sigma,\sigma', z,(z_{i,j})_{i,j=1}^4))
    \pr[\cE_{opt}(\sigma,\sigma',h)|\cE(\sigma,\sigma', z,(z_{i,j})_{i,j=1}^4)]. 
\end{equation}

Similarly as Section \ref{sec:first_mom_fin}, we next introduce discretization partitions. Recall  that $\beta_i = \frac{1+\omega}{4}$ for $i=1,4$ and $\beta_i= \frac{1-\omega}{4}$ for $i=2,3$. Define, for any $\epsilon > 0$:

\begin{align*}
\mathcal{C}(\epsilon, C)=    
\{-C,-C+\epsilon ,-2+2\epsilon,
    -C+3\epsilon,\ldots, C\},
\end{align*}
and set $\tau_{i,j} \in \mathcal{C}(\epsilon,C)$,  $\alpha_{\rm l,ij}, \alpha_{\rm lu,ij} \in [0,1]$ such that:
\begin{equation}\label{eq:beta_ineq}
\begin{split}
\beta_{i}\beta_j d/4+\tau_{i,j}\sqrt{d}+\alpha_{\rm l,i,j} &\le z_{ij} < \beta_{i}\beta_jd+\tau_{i,j}\sqrt{d}+\epsilon\sqrt{d}-\alpha_{\rm u,i,j}, \forall i \neq j \in [4]^2,
\\
\beta^2_{i}+\tau_{i,i}\sqrt{d}+\alpha_{\rm l,i,i} &\le 2z_{ii} < \beta^2_{i}d/4+\tau_{i,i}\sqrt{d}+\epsilon\sqrt{d}-\alpha_{\rm u,i,i}, \ \forall i \in [4].
\end{split}
\end{equation}

Equations \ref{eq:cut_constraint} and \ref{eq:sum_bis} imply the following constraints on $(\tau_{i,j})_{i,j\in [4]}$:

\begin{equation}\label{eq:tau_const_1}
\begin{split}
\abs{\tau_{1,3}+\tau_{1,4} + \tau_{2,3}+ \tau_{2,4} + \frac{E}{2}} &\leq 4\epsilon, \\  \abs{\tau_{1,2}+\tau_{1,4} + \tau_{3,2}+\tau_{3,4} + \frac{E}{2}} &\leq 4\epsilon.
\end{split}
\end{equation}
and
\begin{equation}\label{eq:tau_const_2}
    \abs{\sum_{i<j} \tau_{i,j} + \frac{1}{2} \sum_{i=1}^4 \tau_{i,i}} \leq 8 \epsilon,
\end{equation}
and 
\begin{align*}
    \tau_{14}-\tau_{11} &\geq \abs{\tau_{12}-\tau_{13}}+a,\\
      \tau_{23}-\tau_{22} &\geq \abs{\tau_{21}-\tau_{14}}+a,\\
      \tau_{32}-\tau_{33} &\geq \abs{\tau_{34}-\tau_{31}}+a,\\
      \tau_{41}-\tau_{44} &\geq \abs{\tau_{42}-\tau_{43}}+a.
\end{align*}

Let $\mathcal{S}(\epsilon,C) \in \mathcal{C}(\epsilon, C)$ denote the subset of $(\tau_{i,j})_{i,j\in [4]}$ satisfying the above constraints. Recall (\ref{eq:opt2ndmom}):

\begin{small}
\begin{equation}\label{eq:PE2}
\begin{split}
     \pr[\cE_{opt}(\sigma,\sigma',h)|\cE(\sigma,\sigma', z,(z_{i,j})_{i,j=1}^4)] &= K_2(\frac{1+\omega}{4}n,z_{1,1}n,z_{1,2}n,z_{1,3}n,z_{1,4}n,h) K_2(\frac{1-\omega}{4} n,z_{2,2}n,z_{2,1}n,z_{2,4}n,z_{2,3}n,h)\\
    &\times K_2(\frac{1-\omega}{4} n,z_{3,3}n,z_{3,1}n,z_{3,4}n,z_{3,2}n,h)K_2(\frac{1+\omega}{4}n,z_{4,4}n,z_{4,2}n,z_{4,3}n,z_{4,1}n,h).
\end{split}
\end{equation}
\end{small}

Let $\tilde{z}_{i,j}= \beta_{i}\beta_j d+\tau_{i,j}\sqrt{d}+\epsilon\sqrt{d}-\alpha_{\rm u,i,j}$ for $i,j = (1,4), (2,3)$, $\tilde{z}_{ii}=\beta^2_{i}+\tau_{i,i}\sqrt{d}+\alpha_{\rm l,i,i}$ for $i\in [4]$, and $\tilde{z}_{i,j}= \beta_{i}\beta_j d+\tau_{i,j}\sqrt{d}+\alpha_{\rm l,i,j}$ otherwise. From Equations \ref{eq:beta_ineq}, we then have $\tilde{z}_{i,j}>z_{i,j}$ for $i,j = (1,4), (2,3)$ and $\tilde{z}_{i,j}< z_{i,j}$ otherwise. Proposition \ref{prop:K_2mon} then implies:
 \begin{equation}\label{eq:K_2_bound}
     K_2(\abs{V}_1,z_{1,1}n,z_{1,2}n,z_{1,3}n,z_{1,4}n,h)  \leq K_2(\abs{V}_1,\tilde{z}_{11}n,\tilde{z}_{12}n,\tilde{z}_{13}n,\tilde{z}_{14}n,h).
 \end{equation}
An analogous inequality holds for the remaining terms in (\ref{eq:PE2}).

Applying Corollary \ref{cor:K_asymp_2} to $K_2(\cdot)$ with arguments $(\tilde{z}_{i,j})_{i,j \in [4]^2}$, and recalling that $\abs{V}_1 = \frac{1+\omega}{4}n+O(1)$, we obtain, for large enough $n$ and $d$:

\vspace{-2em}
\begin{align*}
&\frac{1}{n}\log K_2(\abs{V_1},z_{1,1}n,z_{1,2}n,z_{1,3}n,z_{1,4}n,h)\\ &\leq \frac{1+\omega}{4} L_2\left(\frac{4\sqrt{2}(\tau_{1,2}-\tau_{1,3})}{\sqrt{1-\omega}(1+\omega)}, \sqrt{\frac{(1-\omega)}{(1+\omega)}}, \frac{4\sqrt{2}(\tau_{1,4}+\epsilon-\tau_{1,1})}{(1+\omega)^{3/2}}-\frac{2h}{\sqrt{1+\omega}}\right)+\epsilon,\\
&\frac{1}{n}\log  K_2(\abs{V_2},z_{2,2}n,z_{2,1}n,z_{2,4}n,z_{2,3}n,h)\\ &\leq 
\frac{1-\omega}{4}L_2\left(\frac{4\sqrt{2}(\tau_{2,1}-\tau_{2,4})}{\sqrt{1+\omega}(1-\omega)}, \sqrt{\frac{(1+\omega)}{(1-\omega)}}, \frac{4\sqrt{2}(\tau_{2,3}+\epsilon-\tau_{2,2})}{(1-\omega)^{3/2}}-\frac{2h}{\sqrt{1-\omega}}\right)+\epsilon,\\
 &\frac{1}{n} \log K_2(\abs{V_3},z_{3,3}n,z_{3,1}n,z_{3,4}n,z_{3,2}n,h)\\ &\leq \frac{1-\omega}{4} L_2\left(\frac{4\sqrt{2}(\tau_{3,1}-\tau_{3,4})}{\sqrt{1+\omega}(1-\omega)}, \sqrt{\frac{(1+\omega)}{(1-\omega)}}, \frac{4\sqrt{2}(\tau_{3,2}+\epsilon-\tau_{3,3})}{(1-\omega)^{3/2}}-\frac{2h}{\sqrt{1-\omega}}\right)+\epsilon,\\
 &\frac{1}{n} \log K_2(\abs{V_4},z_{4,4}n,z_{4,2}n,z_{4,3}n,z_{4,1}n,h)\\ &\leq \frac{1+\omega}{4} L_2\left(\frac{4\sqrt{2}(\tau_{4,2}-\tau_{4,3})}{\sqrt{1+\omega}(1-\omega)}, \sqrt{\frac{(1-\omega)}{(1+\omega)}}, \frac{4\sqrt{2}(\tau_{4,2}+\epsilon-\tau_{4,4})}{(1+\omega)^{3/2}}-\frac{2h}{\sqrt{1+\omega}}\right)+\epsilon,
\end{align*}
where we absorbed the asymptotic errors within the additional $\epsilon$ terms. Combining with (\ref{eq:PE2}), we obtain, for large enough $n,d$:
\begin{equation}\label{eq:secondmomprob}
\begin{split}
&\frac{1}{n}\log \pr[\cE_{opt}(\sigma,\sigma',h)\mid \cE(\sigma,\sigma', z,(z_{i,j})_{i,j=1}^4)]\\
&\leq \frac{1+\omega}{4}, L_2\left(\frac{4\sqrt{2}(\tau_{1,2}-\tau_{1,3})}{\sqrt{1-\omega}(1+\omega)}, \sqrt{\frac{1-\omega}{1+\omega}}, \frac{4\sqrt{2}(\tau_{1,4}+\epsilon-\tau_{1,1})}{(1+\omega)^{3/2}}-\frac{2h}{\sqrt{1+\omega}}\right)\\
&+\frac{1-\omega}{4}, L_2\left(\frac{4\sqrt{2}(\tau_{2,1}-\tau_{2,4})}{\sqrt{1+\omega}(1-\omega)}, \sqrt{\frac{1-\omega}{1+\omega}}, \frac{4\sqrt{2}(\tau_{2,3}+\epsilon-\tau_{2,2})}{(1+\omega)^{3/2}}-\frac{2h}{\sqrt{1-\omega}}\right)\\
&+\frac{1-\omega}{4}, L_2\left(\frac{4\sqrt{2}(\tau_{3,1}-\tau_{3,4})}{\sqrt{1+\omega}(1-\omega)}, \sqrt{\frac{1-\omega}{1+\omega}}, \frac{4\sqrt{2}(\tau_{3,2}+\epsilon-\tau_{3,3})}{(1+\omega)^{3/2}}-\frac{2h}{\sqrt{1-\omega}}\right)\\
&+\frac{1+\omega}{4}, L_2\left(\frac{4\sqrt{2}(\tau_{4,2}-\tau_{4,3})}{\sqrt{1-\omega}(1+\omega)}, \sqrt{\frac{1-\omega}{1+\omega}}, \frac{4\sqrt{2}(\tau_{4,1}+\epsilon-\tau_{4,4})}{(1+\omega)^{3/2}}-\frac{2h}{\sqrt{1+\omega}}\right)
+4\epsilon.
\end{split}
\end{equation}

% \begin{equation}\begin{split}
%     \tau_{14}-\tau_{11} &\geq \abs{\tau_{12}-\tau_{13}}+\frac{a}{\sqrt{\beta_4-\beta_1}},\\
%     \tau_{14}-\tau_{44} &\geq \abs{\tau_{24}-\tau_{34}}+\frac{a}{\sqrt{\beta_4-\beta_1}}\\
%     \tau_{14}-\tau_{44} &\geq \abs{\tau_{24}-\tau_{34}}+\frac{a}{\sqrt{\beta_4-\beta_1}}
% \end{split}
% \end{equation}

Let $\tilde{\tau}_{i,j} \coloneqq \tau_{i,j}+\epsilon$ for $(i,j)=\{(1,4),(2,3)\}$ and $\tilde{\tau}_{i,j}\coloneqq \tau_{i,j}$ otherwise.
(\ref{eq:secondmomprob}) implies that for large enough $d,n$:
\begin{equation}\label{eq:tauti}
   \frac{1}{n}  \log \pr[\cE_{opt}(\sigma,\sigma',h)|\cE(\sigma,\sigma', z,(z_{i,j})_{i,j=1}^4)] \leq  \sup_{(\tau_{i,j}, i, j \in [4]^2) \in \mathcal{S}(\epsilon, C)}\mathcal{K}_2(\tilde{\tau}_{i,j}, i, j \in [4])+4\epsilon.
\end{equation}

Before proceeding further, we combine the above with the combinatorial term $\pr(\cE(\sigma,\sigma', z,(z_{i,j})_{i,j=1}^4)$ through the following result:

\begin{proposition}\label{prop:2ndmomcount}
Let $z_{i,j} = \beta_{i,j}\frac{d}{4}+\eta_{ij} \sqrt{d}$. Then, for any $(z_{ij})_{i,j \in [4]} \in D(C,z)$:
\begin{equation}\label{eq:2ndmomcountfin}
    \frac{1}{n} \log \pr(\cE(\sigma,\sigma', z,(z_{i,j})_{i,j=1}^4)) = - \sum_{i,j=1}^4\frac{\eta_{i,j}^2}{4 \beta_{i,j}^2} + o_n(1).
\end{equation}
\end{proposition}
\begin{proof}
The proof follows by noting that $E_{i,j}$ follow a multinomial distribution with probabilities $\beta_i\beta_j$. Subsequently, the RHS is obtained through Stirling's approximation along with the expansion of $\log(1+t)$ as in (\ref{eq:comb_term}).
\end{proof}

Note that Equations \ref{eq:tauti} imply that $\abs{\tau_{i,j}-\eta_{i,j}} \leq 2\epsilon$ for all $i,j \in [4]^2$. Since $(\tau_{i,j})_{i,j \in [4]^2}$ are uniformly bounded, the function $(\eta_{i,j})_{i,j \in [4]^2} \rightarrow \sum_{i,j=1}^4\frac{\eta_{i,j}^2}{4 \beta_{i,j}^2}$ is uniformly Lipschitz w.r.t $\eta$.

We obtain, for large enough $n$:
\begin{equation}\label{eq:comb2ndmom}
     \frac{1}{n} \log \pr(\cE(\sigma,\sigma', z,(z_{i,j})_{i,j=1}^4)) = -\sum_{i,j=1}^4\frac{\tau_{i,j}^2}{4 \beta_{i,j}^2}+C\epsilon.
\end{equation}

Combining Equations \ref{eq:tauti} and \ref{eq:comb2ndmom}, we obtain for large enough $n,d$:
\begin{equation}\label{eq:x2bound}
    \frac{1}{n} \log \mathbb{E}[X^2_\omega(z,h)]\leq 2\log 2+H(\frac{1+\omega}{2})+\sup_{(\tau_{i,j}, i, j \in [4]^2) \in \mathcal{S}(\epsilon, C)}\mathcal{K}_2(\tilde{\tau}_{i,j}, i, j \in [4])- \sum_{i,j=1}^4\frac{\tau_{i,j}^2}{4 \beta_{i,j}^2}+(C+4)\epsilon.
\end{equation}

We next leverage Lemma \ref{lemma:Optimum-LD-rates-2nd-monent} to reduce the RHS to a variational problem over a single-parameter $t$. To achieve this, we introduce the constraints $2\tau_{1,4}-\tau_{1,1}-\tau_{4,4}=2t_1$ and $2\tau_{2,3}-\tau_{2,2}-\tau_{3,3}=2t_2$
We then split the optimization as follows:

\begin{equation}\label{eq:2ndmommax}
\begin{split}
  &\sup_{(\tau_{i,j}, i, j \in [4]^2) \in \mathcal{S}(\epsilon, C)}   (\mathcal{K}_2( \tilde{\tau}_{i,j}, i, j \in [4])+\sum_{i,j=1}^4\frac{\tau_{i,j}^2}{4 \beta_{i,j}^2})\\ &= \sup_{t_1, t_2}\sup_{\tau_{i,j}, i, j \in [4], 2\tau_{1,4}-\tau_{1,1}-\tau_{4,4}=2t_1, 2\tau_{2,3}-\tau_{2,2}-\tau_{3,3}=2t_2}   (\mathcal{K}_2(
  \tilde{\tau}_{i,j}, i, j \in [4])- \sum_{i,j=1}^4\frac{\tau_{i,j}^2}{4 \beta_{i,j}^2}).
\end{split}
\end{equation}

Note that since $\tilde{\tau}_{1,4}=\tau_{1,4}+\epsilon$, $\tilde{\tau}_{2,3}=\tau_{2,3}+\epsilon$, the constraints $2\tau_{1,4}-\tau_{1,1}-\tau_{4,4}=2t_1, 2\tau_{2,3}-\tau_{2,2}-\tau_{3,3}=2t_2$ are equivalent to  $2\tilde{\tau}_{1,4}-\tilde{\tau}_{1,1}-\tilde{\tau}_{4,4}=2t_1+2\epsilon, 2\tau_{2,3}-\tau_{2,2}-\tau_{3,3}=2t_2$. By triangle inequality, the constraints in Equations \ref{eq:tau_const_1} and \ref{eq:tau_const_2}
further imply that:
\begin{equation}
\abs{\tau_{1,4}-\frac{1}{2}\tau_{1,1}-\frac{1}{2}\tau_{4,4} + \tau_{2,3}-\frac{1}{2}\tau_{2,2}-\frac{1}{2}\tau_{3,3}+E} \leq 16\epsilon.
\end{equation}

Therefore, $t_1,t_2$ satisfy:
\begin{equation}\label{eq:t1t2const}
    \abs{t_1+t_2+E} \leq 16\epsilon.
\end{equation}

By \eqref{eq:0min} and Lemma \ref{lem:L_conv}, the inner maximization is bounded by the value at $\tilde{\tau}_{1,4}-\tau_{11}=\tilde{\tau}_{1,4}-\tau_{44}=t_1$ and  $\tilde{\tau}_{2,3}-\tau_{22}=\tilde{\tau}_{2,3}-\tau_{33}=t_2$. We obtain:
\begin{equation}\label{eq:2ndmommax}
\begin{split}
&\sup_{(\tau_{i,j}, i, j \in [4]^2) \in \mathcal{S}(\epsilon, C), 2\tau_{1,4}-\tau_{1,1}-\tau_{4,4}=2t_1, 2\tau_{2,3}-\tau_{2,2}-\tau_{3,3}=2t_2}   \mathcal{K}_2(\tilde{\tau}_{i,j}, i, j \in [4])\\ &\leq \frac{1+\omega}{4}L_2\left(0, \sqrt{\frac{(1-\omega)}{(1+\omega)}},  \frac{4\sqrt{2}(t_1+\epsilon)}{(1+\omega)^{3/2}}-\frac{2h}{\sqrt{1+\omega}}\right)+ \frac{1-\omega}{4}L_2\left(0, \sqrt{\frac{(1+\omega)}{(1-\omega)}},\frac{4\sqrt{2}(t_2+\epsilon)}{(1-\omega)^{3/2}})-\frac{2h}{\sqrt{1-\omega}}\right).
\end{split}
\end{equation}

To replace $t_1+\epsilon,t_2+\epsilon$ with $t_1,t_2$ in the RHS of (\ref{eq:2ndmommax}), we require restricting the range of $(\tau_{i,j},i, j \in [4])$ to values where $\mathcal{K}_2(\cdot)$ is regular. To this end, we restrict the range of $t_1,t_2$ in (\ref{eq:2ndmommax}) through the following property of $L_2(\cdot)$ (proof in Appendix \ref{app:properties}):
\begin{proposition}\label{prop:l2inf}
    $L_2(0,b_1,b_2) \rightarrow -\infty$ as $b_2 \uparrow 0$.
\end{proposition}

Proposition \ref{prop:l2inf} allows us to simplifiy (\ref{eq:2ndmommax}) to obtain:
\begin{proposition}\label{prop:omegbound}
For any $\omega \in [-1,1]$ and large enough $n,d$:
\begin{equation}\label{eq:2ndmomupboundfin}
\begin{split}
\frac{1}{n} \log \mathbb{E}[X^2_\omega(z,h)]
&\leq 2\log 2 + \log H\!\left(\frac{1+\omega}{2}\right) 
+ \sup_{t} \Biggl(
    \frac{1+\omega}{2}\, L_2\left(0,\, \sqrt{\frac{1-\omega}{1+\omega}},\, \frac{4\sqrt{2}t}{(1+\omega)^{3/2}}-\frac{2h}{\sqrt{1+\omega}}\right)\\&+ \frac{1-\omega}{2}\, L_2\left(0,\, \sqrt{\frac{1+\omega}{1-\omega}},\, \frac{4\sqrt{2}(-E-t)}{(1-\omega)^{3/2}}-\frac{2h}{\sqrt{1-\omega}}\right) - 4\frac{t^2}{(1+\omega)^2} - 4\frac{(E-t)^2}{(1+\omega)^2}
\Biggr)
+ C_E\epsilon,
\end{split}
\end{equation}
for some constant $C_E>0$.
\end{proposition}
\begin{proof}
Consider the upper bound in (\ref{eq:2ndmommax}):
\begin{equation}
    \frac{1+\omega}{2}L_2\left(0, \sqrt{\frac{(1-\omega)}{(1+\omega)}},  \frac{4\sqrt{2}(t_1+\epsilon)}{(1+\omega)^{3/2}}-\frac{2h}{\sqrt{1+\omega}}\right)+ \frac{1-\omega}{2}L_2\left(0, \sqrt{\frac{(1+\omega)}{(1-\omega)}},\frac{4\sqrt{2}(t_2+\epsilon)}{(1-\omega)^{3/2}}-\frac{2h}{\sqrt{1-\omega}})\right).
\end{equation}

% \begin{equation}
%     \sup_{\tau_{i,j}, i, j \in [4], 2\tau_{1,4}-\tau_{1,1}-\tau_{4,4}=2t}   \mathcal{K}_2(
%   \tilde{\tau}_{i,j}, i, j \in [4])- \sum_{i,j=1}^4\frac{\tau_{i,j}^2}{4 \beta_{i,j}^2}
% \end{equation}

First, suppose that either $t_1< \frac{1}{2\sqrt{2}}(1+\omega) h+\delta$ or $t_2< \frac{1}{2\sqrt{2}}(1-\omega)h+\delta$ for some $\delta>0$. Lemma \ref{prop:l2inf}, then implies that for any $K>0$, there exist small enough $\delta$ such that:
\begin{equation}\label{eq:bound1}
\begin{split}
\frac{1+\omega}{2}L_2\left(0, \sqrt{\frac{(1-\omega)}{(1+\omega)}},  \frac{4\sqrt{2}(t_1+\epsilon)}{(1+\omega)^{3/2}}-\frac{2h}{\sqrt{1+\omega}}\right)+ \frac{1-\omega}{2}L_2\left(0, \sqrt{\frac{(1+\omega)}{(1-\omega)}},\frac{4\sqrt{2}(t_2+\epsilon)}{(1-\omega)^{3/2}}-\frac{2h}{\sqrt{1-\omega}}\right) \leq -K.
\end{split}
\end{equation}

With $\delta>0$ fixed as above, consider the complimentary case i.e. $t_1 > \frac{1}{2\sqrt{2}}(1+\omega) h+\delta$ and $t_2 > \frac{1}{2\sqrt{2}}(1-\omega)h+\delta$.  For small enough $\epsilon$, we obtain that the values $\tilde{\tau}_{1,4}-\tau_{1,1}, \tilde{\tau}_{1,4}-\tau_{4,4},\tilde{\tau}_{2,3}-\tau_{2,2},\tilde{\tau}_{2,3}-\tau_{3,3}$ lie in intervals 
such that the terms $L_2(\cdot)$ appearing in $\mathcal{K}_2( \tilde{\tau}_{i,j}, i, j \in [4])$ are uniformly continuous w.r.t $(\tilde{\tau}_{i,j})_{i,j \in [4]^2}$ by Lemma \ref{lem:l2cont}. Therefore, $\exists$ constant $C_\delta$ such that:
\begin{align*}
&\sup_{\tau_{i,j}, i, j \in [4], 2\tau_{1,4}-\tau_{1,1}-\tau_{4,4}=2t_1, 2\tau_{2,3}-\tau_{2,2}-\tau_{3,3}=t_2}\mathcal{K}_2(\tilde{\tau}_{i,j}, i, j \in [4])\\ &\leq \sup_{\tau_{i,j}, i, j \in [4], 2\tau_{1,4}-\tau_{1,1}-\tau_{4,4}=2t_1, 2\tau_{1,4}-\tau_{1,1}-\tau_{4,4}=2t_2}\mathcal{K}_2(\tau_{i,j}, i, j \in [4])+C_\delta \epsilon. 
\end{align*}

The constraint (\ref{eq:t1t2const}) further allows us to replace $t_1,t_2$ by $t,-E-t$ at the cost of additional $\epsilon$ factors:

\begin{equation}\label{eq:bound3}
\begin{split}
&\sup_{\tau_{i,j}, i, j \in [4], 2\tau_{1,4}-\tau_{1,1}-\tau_{4,4}=2t_1, 2\tau_{2,3}-\tau_{2,2}-\tau_{3,3}=t_2}\mathcal{K}_2(\tilde{\tau}_{i,j}, i, j \in [4])\\ &\leq \sup_{\tau_{i,j}, i, j \in [4], 2\tau_{1,4}-\tau_{1,1}-\tau_{4,4}=2t, 2\tau_{1,4}-\tau_{1,1}-\tau_{4,4}=2(-E-t)}\mathcal{K}_2(\tau_{i,j}, i, j \in [4])+\tilde{C}_\delta \epsilon, 
\end{split}
\end{equation}
for a constant $\tilde{C}_\delta > 0$.

Combining Equations \ref{eq:x2bound}, \ref{eq:bound1}, \ref{eq:bound3}, we obtain, for large enough $n,d$:
\begin{align*}
       &\frac{1}{n} \log \mathbb{E}[X^2_\omega(z,h)]\\  &\leq \sup_{t}\max\left(\sup_{\tau_{i,j}, i, j \in [4], 2\tau_{1,4}-\tau_{1,1}-\tau_{4,4}=2t}\left(   \mathcal{K}_2(\tau_{i,j}, i, j \in [4]) -\sum_{i,j=1}^4\frac{\tau_{i,j}^2}{4 \beta_{i,j}^2} \right),-K \right)+(\tilde{C}_\delta+4)\epsilon,
\end{align*}
Since $K>0$ was arbitrary, taking $\epsilon \rightarrow 0$ and subsequently $\delta \rightarrow 0$, we obtain (\ref{eq:2ndmomupboundfin}).
\end{proof} 

Combining (\ref{eq:2ndmomupboundfin}) with the combinatorial term in (\ref{eq:comb2ndmom}), and substituting in \eqref{EX2mrem}, we obtain the following bound: 
\begin{equation}\label{eq:bound_1K2}
\begin{split}
  &\frac{1}{n} \log \mathbb{E}[X^2_\omega(z,h)]  \leq   2\log 2 +\log H(\frac{1+\omega}{2})+\\&\sup_{t} \Bigr(\frac{1+\omega}{4}L_2\left(0, \sqrt{\frac{(1-\omega)}{(1+\omega)}}, \frac{4\sqrt{2}t}{(1+\omega)^{3/2}}-\frac{2h}{\sqrt{1+\omega}}\right)+ \frac{1-\omega}{4}L_2\left(0, \sqrt{\frac{(1+\omega)}{(1-\omega)}}, \frac{4\sqrt{2}(-E-t)}{(1-\omega)^{3/2}}-\frac{2h}{\sqrt{1-\omega}}\right)\\&- \sup_{\tau_{i,j}, i, j \in [4], 2\tau_{1,4}-\tau_{1,1}-\tau_{4,4}=2t}  \sum_{i,j=1}^4 \frac{\tau_{i,j}^2}{\beta_ij^2}\Bigl)+\epsilon,
\end{split}
\end{equation}
where $H(\cdot)$ denotes the standard binary entropy function and we absorbed the remaining $o_n(1)$ terms into an additional $\epsilon$. By convexity of $f(x)=x^2$, we further obtain:
\begin{equation}
    \sup_{\tau_{i,j}, i, j \in [4], 2\tau_{1,4}-\tau_{1,1}-\tau_{4,4}=2t}  \sum_{i,j=1}^4 \frac{\tau_{i,j}^2}{\beta_{ij}^2} \leq -4\frac{t^2}{(1+\omega)^2}-4\frac{(E-t)^2}{(1-\omega)^2}.
\end{equation}

We next substitute the above bound in (\ref{eq:bound_1K2}), take $\epsilon \rightarrow 0$ and recognize the RHS as $W(E,\omega,h)$. This establishes the upper bound in Proposition \ref{prop:second_mom_main}. To obtain the corresponding lower-bound, let $t^\star$ denote the unique maximizer of the objective defined by (\ref{eq:2ndmomupboundfin}).  By Lemma \ref{lemma:Optimum-LD-rates-2nd-monent}, there exist $(\tau^\star_{i,j})_{i,j \in [4]^2} \in \mathcal{S}(\epsilon, C)$  such that:
\begin{equation}
\begin{split}
   &\mathcal{K}_2(\tau^\star_{i,j}, i, j \in [4])=\\ & \sup_{t} \Biggl(\frac{1+\omega}{2}L_2\left(0, \sqrt{\frac{(1-\omega)}{(1+\omega)}}, \frac{4\sqrt{2}t}{(1+\omega)^{3/2}}-\frac{2h}{\sqrt{1+\omega}}\right)+ \frac{1-\omega}{2}L_2\left(0, \sqrt{\frac{(1+\omega)}{(1-\omega)}}, \frac{4\sqrt{2}(-E-t)}{(1-\omega)^{3/2}}-\frac{2h}{\sqrt{1-\omega}}\right) \Biggr).
\end{split}
\end{equation}
with $\tau^\star_{i,j}$ explicitly given by $\tau^\star_{1,4}=\frac{t^\star}{2},\tau^\star_{1,1}=\tau^\star_{4,4}= - \frac{t^\star}{2}$, $\tau^\star_{2,3}=\frac{E-t^\star}{2},\tau^\star_{2,2}=\tau^\star_{3,3}= - \frac{E-t^\star}{2}$ and $\tau^\star_{1,2}=\tau^\star_{1,3}=\tau^\star_{2,4}=\tau^\star_{3,4}=0$.Furthermore, by convexity of $f(x)=x^2$, the above values simultaneously maximize $\sum_{i,j=1}^4 \frac{\tau_{i,j}^2}{\beta^2_i\beta_j^2}$ subject to $2\tau_{1,4}-\tau_{1,1}-\tau_{4,4}=2t$.

Set $z^\star_{i,j} \in \mathbb{Z}$ as follows:
\begin{align*}
    z^\star_{1,4}=2\floor{\frac{1+\omega}{4}{d}}-2\floor{\frac{t^\star}{4}\sqrt{d}}, \ 
     z^\star_{1,1}=
      z^\star_{2,2}=\ceil{\frac{1+\omega}{4}{d}}-\floor{\frac{t^\star}{4}\sqrt{d}},\\
    z^\star_{2,3}=2\floor{\frac{1+\omega}{4}{d}}-2\floor{\frac{E-t^\star}{4}\sqrt{d}}, \
     z^\star_{1,1}=
      z^\star_{2,2}=\ceil{\frac{1+\omega}{4}{d}}-\floor{\frac{E-t^\star}{4}\sqrt{d}},\\
z^\star_{1,2}=z^\star_{1,3}=2\floor{\frac{1+\omega}{4}{d}}, \ z^\star_{3,4}=z^\star_{2,4}=2\ceil{\frac{1+\omega}{4}{d}}.
\end{align*}
Then, by construction $z^\star_{i,j} \in \mathcal{S}_{d,n}(z)$. Thus:
\begin{align*}
  &H(\frac{1+\omega}{2}) +  \frac{1}{n} \log \pr(\cE(\sigma,\sigma', z,(z_{i,j})_{i,j=1}^4))+\frac{1}{n}\log K_2(\frac{1+\omega}{4}n,z^\star_{1,1}n,z^\star_{1,2}n,z^\star_{1,3}n,z^\star_{1,4}n,h)\\ &+ \frac{1}{n}\log K_2(\frac{1-\omega}{4} n,z^\star_{2,2}n,z^\star_{2,1}n,z^\star_{2,4}n,z^\star_{2,3}n,h)+ \frac{1}{n}\log K_2(\frac{1-\omega}{4} n,z^\star_{3,3}n,z^\star_{3,1}n,z^\star_{3,4}n,z^\star_{3,2}n,h)\\&+\frac{1}{n}\log K_2(\frac{1+\omega}{4}n,z^\star_{4,4}n,z^\star_{4,2}n,z^\star_{4,3}n,z^\star_{4,1}n,h) \leq \frac{1}{n}\log \mathbb{E}[X^2_\omega(z,h)].
\end{align*}

By Corollary \ref{cor:K_asymp_2} and Proposition \ref{prop:2ndmomcount}, the LHS asymptotically converges to $W(E,\omega,h)$, resulting in the lower-bound:
\begin{equation}
  W(E,\omega,h) \leq  \frac{1}{n}\log \mathbb{E}[X^2_\omega(z,h)].
\end{equation}

This completes the proof of Propositions \ref{prop:second_mom_main}. 

\subsection{Proof of Proposition \ref{prop:second_mom_sup}}

Recall that:

\begin{equation}\label{eq:omega_sup}
\mathbb{E}[ X^2(z,h)] = \sum_{\omega \in [-1,1]: \omega n \in \mathbb{N}} \mathbb{E}[X^2_\omega(z,h)].
\end{equation}

It will be convenient to treat the edge-cases $\omega\approx -1,1$ and the remaining range of $\omega$ separately. To this end, fix $\delta >0$ and further split the sum over $\omega$ in the RHS of (\ref{eq:omega_sup}) into the ranges $[-1,-1+\delta) \cup (1-\delta,1]$ and $[-1+\delta, 1-\delta]$. Define:
\begin{align} X^2_{[-1,-1+\delta)\cup (1-\delta,1]}(z,h) &= \sum_{\omega \in [-1,-1+\delta)\cup (1-\delta,1]: \omega n \in \mathbb{N}} X^2_\omega(z,h),\\
X^2_{[-1+\delta,1-\delta]}(z,h)] &= \sum_{\omega \in [-1+\delta,1-\delta]: \omega n \in \mathbb{N}} X^2_\omega(z,h).
\end{align}

To obtain the bound for the first set, recall that:
\begin{equation}\label{eq:sec_mom_bound_2}
\Ea{X^2_\omega(z,h)} = \sum_{\sigma, \sigma' \in M_0, \omega(\sigma,\sigma')=\omega } \Pr[\cE_{opt}(\sigma,h,z)\cap \cE_{opt}(\sigma',h, z)].
\end{equation}

Each term in the RHS can be trivially bounded as follows:
\[
\Pr[\cE_{opt}(\sigma,h,z)\cap \cE_{opt}(\sigma',h, z)] \leq 
\pr[\cE_{opt}(\sigma,h,z)].
\]

The probability $\pr[\cE_{opt}(\sigma,h,z) ]$ was analyzed in the proof of Proposition \ref{prop:neg_first_mom}, wherein we obtained:
\[
\lim_{d \rightarrow \infty}\lim_{n \rightarrow \infty}\frac{1}{n}\log \pr[\cE_{opt}(\sigma,h,z)]
=w(E,h)-\log 2.
\]
Substituting in (\ref{eq:sec_mom_bound_2}), we obtain uniformly over $\omega \in [-1,1]$:
\[
\frac{1}{n} \log \Ea{X^2_\omega(z,h)}
\leq H(\omega)+w(E,h)+o(n).
\]
Hence:
\[\
  \limsup_{d \rightarrow \infty} \limsup_{n \rightarrow \infty}\frac{1}{n} \log \Ea{X^2_{[-1,-1+\delta)\cup (1-\delta,1]}(z,h)} \leq \sup_{\omega \in [-1,-1+\delta)\cup (1-\delta,1]}H(\omega)+w(E,h).
\]

Using the continuity of $H(\cdot)$ on $[-1,1]$ and $H(1)=H(-1)=0$, the above can be expressed as:
\begin{equation}\label{eq:edge_bound_omega}
    \limsup_{d \rightarrow \infty} \limsup_{n \rightarrow \infty}\frac{1}{n} \log \Ea{X^2_{[-1,-1+\delta)\cup (1-\delta,1]}(z,h)} \leq w(E,h)+o(\delta),
\end{equation}
where $o(\delta)$ denotes a function vanishing as $\delta \rightarrow 0$.

The above bound allows us to bound $X^2_\omega(z,h)$ near $\omega=1,-1$  with the limiting first moment $w(E,h)$. To show that the above bound is tight, we next show that $W(E,\omega,h)$ converges to $w(E,h)$ as $\omega \rightarrow 1,-1$

\begin{lemma}\label{lem:W_cont}

For any $E,h \in \mathbb{R}$:
\[\lim_{\omega \downarrow -1} W(E,\omega,h)= w(E,h).\]

\[\lim_{\omega \uparrow 1} W(E,\omega,h)= w(E,h).\]
\end{lemma}

The above lemma implies that for any $\delta>0$:
\[
\sup_{\omega \in (-1,-1+\delta)\cup (1-\delta,1)} W(E,h,\omega) \geq w(E,h).
\]

Combining with (\ref{eq:edge_bound_omega}),
we obtain:
\begin{equation}\label{eq:edge_omeg_bound}
\limsup_{d \rightarrow \infty} \limsup_{n \rightarrow \infty}\frac{1}{n} \log \Ea{X^2_{[-1,-1+\delta)\cup (1-\delta,1]}(z,h)} \leq
\sup_{\omega \in (-1,-1+\delta)\cup (1-\delta,1)} W(E,h,\omega)+o(\delta),
\end{equation}
where $o(\delta)$ denotes a function vanishing as $\delta \rightarrow 0$.

We next move to the term $\Ea{[X^2_{[-1+\delta,1-\delta]}(z,h)]}$.  Recall the definition of the discretization set $\mathcal{C}(\epsilon, C)$ in (\ref{eq:C_partition}), which was utilized to approximate the edge-set sizes $(z_{i,j})_{{i,j}\in [4]^2}$ in (\ref{eq:beta_ineq}). Now, to approximate the  sum over $\omega$, through restriction of $\omega$ to finitely-many values, we introduce an additional  discretization parameter $\tilde{\epsilon}$ and set $\psi \in \mathcal{C}(\tilde{\epsilon},1)$ such that:
\begin{equation}
    \psi \leq \phi \leq \psi+\tilde{\epsilon}.
\end{equation}
Aanalogous to (\ref{eq:K_2_bound}), we aim to upper bound $K_2(\cdot)$ through terms dependent only on $\psi, \tilde{\epsilon}$ instead of $\omega$. 

We first note that by setting $\tilde{n}=\floor{\frac{1+\omega}{1+\psi}n}$, we have:
\[
 K_2(\frac{1+\omega}{4}n,z_{1,1}n,z_{1,2}n,z_{1,3}n,z_{1,4}n,h) =  K_2(\frac{1+\psi}{4}\tilde{n}+O(1),z_{1,1}n,z_{1,2}n,z_{1,3}n,z_{1,4}n,h),
\]
where the $O(1)$ correction ensures that the first argument to $K_2(\cdot)$ is an integer.

Next, as in \eqref{eq:K_2_bound}, we set $\tilde{z}_{i,j}= \beta_{i}\beta_j d+\tau_{i,j}\sqrt{d}+\epsilon\sqrt{d}-\alpha_{\rm u,i,j}$ for $i,j = (1,4), (2,3)$, $\tilde{z}_{i,i}=\beta^2_{i}+\tau_{i,i}\sqrt{d}+\alpha_{\rm l,i,i}$ for $i\in [4]$, and $\tilde{z}_{i,j}= \beta_{i}\beta_j d+\tau_{i,j}\sqrt{d}+\alpha_{\rm l,i,j}$.

To obtain $\tilde{z}_{i,j}\tilde{n}$ in the  bounds, we set $\tilde{\epsilon}$ sufficiently small, such that $\tilde{z}_{1,4} \frac{1+\omega}{1+\psi}> z_{1,4}, \tilde{z}_{2,3} \frac{1+\omega}{1+\psi}> z_{2,3}$.
(Such a choice of $\tilde{\epsilon}$ exists for any $\epsilon> 0$).

Hence, for a such a $\tilde{\epsilon}$ , for sufficiently large $n$, $\tilde{z}_{i,j}\tilde{n} > z_{i,j}n$ for $(i,j)=(1,4),(2,3)$ and $\tilde{z}_{i,j}\tilde{n} < z_{i,j}n$. Proposition \ref{eq:K2bound} then implies:
\[
K_2(\abs{V}_1,z_{1,1}n,z_{1,2}n,z_{1,3}n,z_{1,4}n,h) =  K_2(\frac{1+\psi}{4}\tilde{n}+O(1),\tilde{z}_{1,1}\tilde{n},\tilde{z}_{1,2}\tilde{n},\tilde{z}_{1,3}\tilde{n},\tilde{z}_{1,4}\tilde{n},h).
\]

Similarly, we bound the remaining $K_2(\cdot)$ terms in the expansion of $X^2_\omega(z,h)$. Subsequently, proceeding as in the proof of Proposition \ref{prop:second_mom_main}, we obtain the bound:

\begin{equation}\label{eq:supomega}
\begin{split}
\lim \sup_{d \rightarrow \infty} \lim \sup_{n \rightarrow \infty}\frac{1}{n}\log \mathbb{E}[ X^2_{[-1+\delta, 1-\delta]}(z,h)] 
 &\leq 
2\log 2+H(\frac{1+\omega}{2})+\sup_{(\tau_{i,j}, i, j \in [4]^2) \in \mathcal{S}(\epsilon, C)}\mathcal{K}_2(\tilde{\tau}_{i,j}, i, j \in [4])\\&- \sum_{i,j=1}^4\frac{\tau_{i,j}^2}{4 \beta_{i,j}^2}+(C+4)\epsilon,
\end{split}
\end{equation}
where $\tilde{\tau}_{i,j} \coloneqq \tau_{i,j}+\epsilon$ for $(i,j)=\{(1,4),(2,3)\}$ and $\tilde{\tau}_{i,j}\coloneqq \tau_{i,j}$. 

Next, since $\mathcal{K}_2(\cdot)$ is uniformly continuous in $\tau_{i,j}$ over the compact set $\omega \in [-1+\delta,1-\delta]$, we further have the bound:
\[
\sup_{(\tau_{i,j}, i, j \in [4]^2) \in \mathcal{S}(\epsilon, C)}\mathcal{K}_2(\tilde{\tau}_{i,j}, i, j \in [4]) \leq \sup_{(\tau_{i,j}, i, j \in [4]^2) \in \mathcal{S}(\epsilon, C)}\mathcal{K}_2(\tau_{i,j}, i, j \in [4])+C_\delta\epsilon,
\]
where $C_\delta$ denotes a constant dependent on $\delta$. Substituting in (\ref{eq:supomega}) and combining with (\ref{eq:edge_omeg_bound}), we obtain:
\[
\limsup_{d \rightarrow \infty} \limsup_{n \rightarrow \infty}\frac{1}{n}\log \mathbb{E}[ X^2(z,h)] \leq \sup_{\omega \in (-1,1)} W(E,\omega,h).
\]

An analogous lower-bound then follows similarly by restricting the sums to $\omega=-1,1$ for $X^2_{[-1,-1+\delta)\cup (1-\delta,1]}$ and to a maximizer $\omega^\star \in \operatorname{argmax}_{\omega \in [-1+\delta,1-\delta,1]} W(\omega,E,h)$ for $X^2_{[-1+\delta,1-\delta,1]}$. This completes the proof of Proposition  \ref{prop:second_mom_sup}.

% Proposition  \ref{prop:second_mom_sup} then follows by considering a discretization over $\omega$ as in the proof of Proposition \ref{prop:first_moment_sup} for maximization over $z$. The discretization is justified by the uniform continuity of $W(E,\omega,h)$ w.r.t $\omega$ over the compact domain $[-1,1]$ (Proposition \ref{prop:twice_diff}).
% \begin{equation}
    
% \end{equation}
% % \begin{proposition}
%     For any $\theta_1, \theta_2 \in \mathbb{R}$, $F(E,\omega,h,t,\theta_1, \theta_2)$ is strictly concave in $t$.
% \end{proposition}
% \begin{proof}
%     Prekopa-Leindler
% \end{proof}

% the optimality conditions simplify to $\theta^\star_1=\theta^\star_2$, 
% \begin{proof}
% and $t^\star = -\frac{E}{2}-t^\star$.

% \subsection{Proof of Proposition \ref{prop:convex_2ndmom}}

\section{Proofs of properties of limiting functions}\label{app:properties}

In this section, we collect the proofs of certain properties of the functions such as $w(E,h),W(\cdot), F(\cdot)$ that appear while evaluating the limits of first, second-moment entropy densities. A central tool underlying the results in this section is the  Prékopa–Leindler inequality, which we state here for convenience:

\begin{lemma}[Prékopa–Leindler]\label{lem:prekopa}
   Let $p,q,r:\R^k \rightarrow [0,\infty)$ be non-negative measurable functions, satisfying for some $\lambda \in (0,1)$:
  \[
       r((1-\lambda) x+\lambda y) \geq p(x)^{1-\lambda} q(x)^\lambda, \ \forall x, y \in \R^k.
   \]
Then:
\[
    \int_{\R^k} r(x) \geq \left (\int_{\R^k} p(x) \right)^{1-\lambda}\left (\int_{\R^k} q(x) \right)^\lambda.
\]
\end{lemma}

\subsection{Properties for the first moment entropy density}

For use in the present section, recall that $ L(\zeta,t,b) = \inf_\theta f(\theta, \zeta, t, b)$ where:
    \begin{equation}\label{eq:f_mgf}
         f(\theta, \zeta, t, b) \defeq \theta^2
+\zeta\log\Phi\left(-{t+b-2\theta\over\sqrt{2}}\right).
    \end{equation}

\subsubsection{Strict convexity of $f(\theta, \zeta, t, b)$}\label{sec:logmgf}

Recall that $f(\theta, \zeta, t, b)$ was obtained through the rate function $g(\theta)$ in  Lemma \ref{lemma:Gaussian-of-Pois}. To establish its convexity, it will be convenient to re-express it in terms of a log-MGF. By expanding $\Phi$, $f(\theta, \zeta, t, b)$ may be expressed as:
\begin{align}
f(\theta, \zeta, t, b) &= \theta^2+\zeta \frac{1}{\sqrt{2\pi}}\log \int_{u \geq \frac{t+b-2\theta}{\sqrt{2}}} \exp(-\frac{u^2}{2}) du\\
&=\theta^2+\zeta \frac{1}{\sqrt{2\pi}}\log \int_{v \geq t+b} \exp(-\frac{v^2}{4}+\theta v -\theta^2)dv\\
&= (1-\zeta)\theta^2 +\zeta \log \Eb{z \sim \mathcal{N}(0,2)}{\exp(\theta z)1_{z \geq t+b}},
\end{align}
where in the second line, we used the change of variables $v=\sqrt{2}u+2\theta$.Hence:
\begin{equation}\label{eq:mgf_rel} f(\theta,\zeta,t,b) = \zeta\Lambda(\theta,t+b)+(1-\zeta)\theta^2 +(1-\zeta)\log \Phi(-(t+b)/\sqrt{2}),
\end{equation}
where $\Lambda(\theta,c)$ denotes the log-MGF of  of $Y\sim \mathcal{N}(0,2)$ conditioned on $Y\geq c$. We recognize that we've recovered the term $ \zeta\Lambda(\theta,t+b)+(1-\zeta)\theta^2$ which equals the log-MGF $g(\theta)$ in Lemma \ref{lemma:Gaussian-of-Pois}. Hence, by the properties of log-MGF, we conclude that $f(\theta,\zeta,t,b)$ is strictly-convex in $\theta$. The proof of Lemma \ref{lemma:Gaussian-of-Pois} further implies that $f(\theta,\zeta,t,b)$ admits a unique maximizer when $t+b<0$ or $\zeta < 1$.

\subsubsection{Proof of Lemma \ref{lem:L_conv}}
 Through a change of variables and the definition of $\Phi$, we have that:
     \begin{equation}
         \Phi(-\frac{t+b-2\theta}{\sqrt{2}}) = \frac{1}{\sqrt{2\pi}}\int_{\R} \mathbf{1}_{z \leq 0} e^{-\frac{(z+\frac{t+b-2\theta}{\sqrt{2}})^2}{2}}.
     \end{equation}
Next, the log-concavity of $e^{-\frac{z^2}{2}}$ implies that for any $z,b,\theta, t_1,t_2 \in \R$:
\begin{equation}
    e^{-\frac{(z+\frac{\lambda t_1+(1-\lambda)t_2+b-2\theta}{\sqrt{2}})^2}{2}} \geq (e^{-\frac{(z+\frac{t_1+b-2\theta}{\sqrt{2}})^2}{2}})^{\lambda}(e^{-\frac{(z+\frac{t_2+b-2\theta}{\sqrt{2}})^2}{2}})^{1-\lambda}.
\end{equation}

The Prékopa–Leindler inequality then implies that for any $\lambda \in (0,1)$ and $z,b, t_1,t_2 \in \R$:
\begin{equation}
    \log \Phi(-\frac{(\lambda t_1+ (1-\lambda)t_2)+b-2\theta}{\sqrt{2}}) \geq  \lambda \log \Phi(-\frac{t_1+b-2\theta}{\sqrt{2}})+ (1-\lambda) \log \Phi(-\frac{t_2+b-2\theta}{\sqrt{2}}).
\end{equation}
Therefore, $\zeta\log \Phi(-\frac{t+b-2\theta}{\sqrt{2}})$ is concave in $t$ for any $\zeta > 0, b, \theta \in \mathbb{R}$. Since the point-wise infimum of concave functions in concave, we obtain that $L(\zeta,t,b)$ is concave in $t$. The non-decreasing property in $\zeta$ follows directly by noting that $\Phi \leq 1$.

\subsubsection{Proof of Proposition \ref{prop:Linf}}

Recall by (\ref{eq:mgf_rel}), $L(1, t, b)$ can be expressed as:
\begin{equation}
    L(1, t, b) = \inf_\theta \Lambda(\theta,c),
\end{equation}
where $\Lambda(\theta,c)$ denotes the log-MGF of $Y\sim \mathcal{N}(0,2)$ conditioned on $Y\geq c$.

Expanding $\Lambda(\theta,c)$, we obtain:
    \begin{equation}
       L(1, t, b)  =   \inf_{\theta} \left[\log \E[\exp(
Y\theta)\vert Y \geq (t+b)]+ \log \Phi(-(t+b)/\sqrt{2}) \right],
    \end{equation}
where $Y$ denotes a normal variable with variance $2$. Suppose now that $t+b \geq -\delta$ for some $\delta > 0$.
We have:
\begin{equation}
\begin{split}
     \E[\exp( 
Y\theta)\vert Y \geq t+b] &=\frac{1}{\pr[Y \geq t+b]}\frac{1}{\sqrt{4\pi}} \int_{y=t+b}^\infty \exp(\theta y -y^2/4)
\\& \geq \frac{1}{\pr[Y \geq t+b]}\frac{1}{\sqrt{4\pi}}\int_{y=-\delta}^\infty \exp(\theta y -y^2/4)
\\&=   \frac{1}{\pr[Y \geq t+b]}\frac{1}{\sqrt{4\pi}}\int^{\infty}_{y=\delta}\exp(\theta y -y^2/4)\\&+\frac{1}{\pr[Y \geq t+b]}\frac{1}{\sqrt{4\pi}}\int^\delta_{y=-\delta} \exp(\theta y -y^2/4).
\end{split}
\end{equation}

Setting $\theta = \frac{1}{\delta^{3/2}}$ and using $\int_0^\infty \exp{(-y^2/4)} \leq \sqrt{2\pi} $, the first integral is bounded by $\exp(-\frac{1}{\sqrt{\delta}})$, while using the bound $\exp(\theta y) \leq 1+\theta y+o(\theta y)$, the second integral is bounded by $C\delta$ for some constant $C>0$. Therefore, both the terms vanish as $\delta \rightarrow 0$, implying $L(1,t,b) \rightarrow - \infty$ as $t+b \uparrow 0$. 

% While for $0 < \zeta < 1$, $L(\zeta, t, b)$ is continuously differentiable for any $t, b \in \mathbb{R}^2$.

% \begin{equation}\label{eq:zlips} \abs{L(\zeta, t, b) - L(\zeta, t', b)} \leq C_\zeta(1+\abs{t+b}) \abs{t - t'},
% \end{equation}
% for some constant $C_\zeta > 0$.
% 
\subsubsection{Proof of Lemma \ref{lem:L_lips}}

Recall that in \ref{sec:logmgf}, we established the strict convexity and existence of minimizers for $f(\theta, \zeta, t, b)$. The strict convexity of $f(\theta, \zeta, t, b)$ w.r.t $\theta$ implies that $\nabla^2_\theta f$ is invertible and thus applying the Implicit function theorem to the fixed point condition $\nabla_\theta f = 0$
(\ref{eq:fixed-point-theta}), we obtain that the unique maximizer $\bar{\theta}$ is continuously differentiable in $t,b$ over the specified domain. This in turn implies that $L(\zeta,t,b) = f(\bar{\theta}, \zeta, t, b)$ is continously differentiable in $t,b$.

\subsubsection{Proof of Propositions \ref{prop:first_mom_conv} and \ref{prop:first_mom_conv_r}}

\begin{proof}
Define:
    \begin{equation}
         q(\theta, E, h) \defeq \theta^2
+\log\Phi\left(-{2E+h-2\theta}\right).
    \end{equation}

Then, $w(E, h)$ can be expressed as: 
\begin{equation}
    w(E, h, r)= H(r)+(1-r)\log 2- E^2+(2r-1)\inf_{\theta \in \mathbb{R}} q(\theta, E, h).
\end{equation}
Recall that in the proof of Lemma \ref{lem:L_conv}, we showed that for any fixed $\theta$, $q(\theta, E, h)$ is strictly concave in $E$.

Therefore, for any $ E_1, E_2 < -\frac{h}{2}$ with $E=\lambda E_1+(1-\lambda)E_2$ and $\theta \in \mathbb{R}$:  
\[
q\big(\theta, \lambda E_1 + (1-\lambda)E_2, h\big) > \lambda q(\theta, E_1, h) + (1-\lambda) q(\theta, E_2, h), \quad \text{for } \lambda \in (0,1).
\] \label{eq:strict_conv}

 Lemma \ref{lemma:Gaussian-of-Pois} further implies that for any $E < -\frac{h}{2}$ the infimum of $\big(\theta, E, h)$, with respect to $\theta$ is attained at some $\theta^\star \in \mathbb{R}$. For a fixed $\lambda \in (0,1)$, let $\theta^\star, \theta^*_1$ and $\theta^*_2$ denote the minimizers of $q(\theta, E, h)$ for $E=\lambda E_1+(1-\lambda)E_2, E_1$ and $E=E_2$, respectively. Then, (\ref{eq:strict_conv}) implies that:
\begin{align*}
 q\big(\theta^\star, \lambda E_1 + (1-\lambda)E_2, h\big) &> \lambda q(\theta^\star, E_1, h) + (1-\lambda) q(\theta^\star, E_2, h)\\
 &\geq \lambda q(\theta^\star_1, E_1, h) + (1-\lambda) q(\theta^\star_2, E_2, h)\\
 &= \lambda w(E_1,h)+(1-\lambda) w(E_2,h),
\end{align*}
completing the proof.
\end{proof}

\subsubsection{Proof of Proposition \ref{prop:w_mono}}

Since $\erf(\cdot)$ is strictly increasing while $\log(\cdot)$ is strictly decreasing, we have for any $E \in \mathbb{R}$ and $h<\tilde{h}$:
\begin{equation}
   -E^2+\log(1-\erf(E+h/\sqrt{2})) >  -E^2+\log(1-\erf(E+\tilde{h}/\sqrt{2})).
\end{equation}

Taking the supremum w.r.t $E$ on both sides and recalling that the supremum is achieved by Proposition \ref{prop:first_mom_conv} then implies that $w(h)$ is strictly decreasing.

Next we show that $w(h) \rightarrow -\infty$ as $h \rightarrow \infty$. Combining with the continuity and strict monotonocity of $w(h)$, this will establish the unicity and existence of the root of $w(h)$.

Consider the following two cases:

\begin{enumerate}
    \item $ E\in (-\infty, -\frac{h}{\sqrt{2}}+1)$.
    \item $E\in (-\frac{h}{\sqrt{2}}+1,\infty)$.
\end{enumerate}

In case $1$ i.e when $E\in (-\infty, -\frac{h}{\sqrt{2}}+1)$, we have:
\begin{align}
     -E^2+\log(1-\erf(E+h/\sqrt{2})) &\leq -E^2 + \log(1) \nonumber\\
&\leq \sup_{E \in (-\infty, -\frac{h}{\sqrt{2}}+1)}(-E^2) \nonumber \\
&= \begin{cases}
    0 , & h < \sqrt{2}\\
    -(-\frac{h}{\sqrt{2}}+1)^2,  & \text{otherwise}.
\end{cases}\label{eq:bound_1}
\end{align}

While in the second case i.e $E\in (-\frac{h}{\sqrt{2}}+1,\infty$, using the inequality $1-\erf(u) \leq \frac{e^{-u^2}}{\sqrt{\pi} u}$, we obtain:
\begin{align*}
    -E^2+\log(1-\erf(E+h/\sqrt{2})) &\leq  -E^2+\log( \frac{e^{-(E+h/\sqrt{2})^2}}{\sqrt{\pi} (E+h/\sqrt{2})}) \\
  &=  -E^2-(E+h/\sqrt{2})^2-\log \left(\frac{E+h/\sqrt{2}}{\sqrt{\pi}} \right)\\ 
  & \leq -E^2-(E+h/\sqrt{2})^2 ,
\end{align*}
where in the last step we used that $E+h/\sqrt{2} \geq 1$. Note that the term $-E^2-(E+h/\sqrt{2})^2$ is maximized at $E=-\frac{h}{2\sqrt{2}}$ with the maximum value being $\frac{h^2}{4}$. We thus obtain the following bound for $E\in (-\frac{h}{\sqrt{2}}+1,\infty)$
\begin{equation}\label{eq:bound_2}
    -E^2+\log(1-\erf(E+h/\sqrt{2}))\leq \frac{h^2}{4}.
\end{equation}

Equations \ref{eq:bound_1} and \ref{eq:bound_2} together imply that $w(h) \rightarrow -\infty$ as $h \rightarrow \infty$. Hence, $w(h)$ posseses a unique root.

We move on to proving $h^\star>0$. First note that
$$ w(0,0)=0.$$

Next, a direct computation shows that 

$$\frac{\partial w}{\partial E}(E,0)\vert_{E=0}
 = -\frac{2\,\exp\bigl(-(E + h/\sqrt2)^2\bigr)}          {\sqrt\pi\,\bigl[\,1 - \erf\!\bigl(E + h/\sqrt2\bigr)\bigr]} <0.
$$
Hence, by continuity there exists $\delta>0$ such that
 $$ w(E,0)>w(0,0)=0 \quad\text{whenever}\quad -\delta<E<0.$$ Consequently,
$$ w(0)\;=\;\sup_{E\in\R}w(E,0)\;>\;0, $$
which along with the monotonicity of $w(h)$ immediately implies $h^\star>0$.

\subsubsection{Proof of Corollary \ref{cor:Econthr}}
By (\ref{eq:first_mom_bound}),  $w(E,h,1)= w(E,h)$. Recall that by Proposition \ref{prop:first_mom_conv_r}, $w(E,h,r)$ is differentiable and strictly concave in $E$ for all $r \in (1/2,1]$. Therefore $\nabla^2_E w(E,h,r) < 0, \forall E,h \in \mathbb{R}, r \in (1/2,1]$. Since the optimality of $E^\star(h,r)$ is equivalent to the condition $\nabla_E w(E,h,r) = 0$, the implicit function theorem implies that 
$E^\star(h,r)$
is continuous w.r.t $r$ in $(1/2,1]$.

\subsection{Properties for the second moment entropy density}
\subsubsection{Proof of Lemma \ref{lem:l2conc}}

Recall that from the definition of $L_2$ in (\ref{eq:def_L2}), we have:
     \[
        L_2(b_1, b_2, b_3)=\inf_{\theta_1,\theta_2}\log \mathbb{E}[\exp(\theta_1 Y_1+\theta_2 Y_2)\mathbf{1}_{[Y_1+b_3 \geq b_2\abs{Y_2+b_1}]}].
     \]
By symmetry of $Y_2$ around $0$, we have:
\begin{equation}
    L_2(b_1, b_2, b_3) = L_2(-b_1, b_2, b_3),
\end{equation}
and thus that $L_2(\cdot)$ is even w.r.t $b_1$. Substituing the density of $Y_1,Y_2$ in $L_2(b_1, b_2, b_3)$, gives:
\begin{align*}
    L_2(b_1, b_2, b_3) &= \frac{1}{2\pi}\inf_{\theta_1,\theta_2}\log \int_{\mathbb{R}^2}\exp(\theta_1 y_1+\theta_2 y_2-\frac{(y_1^2+y_2^2)}{2})\mathbf{1}_{[y_1+b_3 \geq b_2\abs{y_2+b_1}]}.
\end{align*}

With a change of variables $z_1 = y_1+b_3, z_2=y_1+b_2$, the above can be expressed as:
\begin{align*}
&=\inf_{\theta_1,\theta_2} (\frac{1}{2}(\theta_1^2+\theta_2^2))+ \log \int_{\mathbb{R}^2}\exp\left(-\frac{(z_1-b_3-\theta_1)^2+(z_2-b_1-\theta_2)^2}{2}\right)\mathbf{1}_{[z_1 \geq b_2\abs{z_2}]}.
\end{align*}
     
Given $\lambda \in [0,1]$, define $b^\lambda_j = \lambda b_j + (1-\lambda)\tilde{b}_j$. The log concavity of $\exp(-\frac{(y_1^2+y_2^2)}{2})$  then implies that for any $\theta_1,\theta_2, (b_1,b_3),(\tilde{b}_1,\tilde{b}_3)  \in \R^2$:
\begin{equation}
\begin{split}
   &\exp\left(\frac{-(z_1-b^\lambda_3-\theta_1)^2-(z_2-b^\lambda_1-\theta_2^2)^2}{2}\right)\\ &\geq (\exp\left(\frac{-(z_1-b_3-\theta_1)^2-(z_2-b_1-\theta_2)^2}{2}\right))^{\lambda}(\exp\left(\frac{-(z_1-\tilde{b}_3-\theta_1)^2-(z_2-\tilde{b}_1-\theta_2)^2}{2}\right))^{1-\lambda}.
\end{split}
\end{equation}

The Prékopa–Leindler inequality then implies that for any $\lambda \in (0,1)$ and $z,b_1, b_2,b_3, \tilde{b}_1, \tilde{b}_3 \in \R$:
\begin{equation}
     \lambda L_2(b_1, b_2, b_3) + (1-\lambda) \lambda L_2(\tilde{b}_1, b_2, \tilde{b}_3)  \leq  L_2( \lambda b_1 + (1-\lambda) \tilde{b}_1,b_2, \lambda b_3 + (1-\lambda) \tilde{b}_3)
\end{equation}

\subsubsection{Proof of Lemma \ref{lem:l2cont}}

The proof is identical to that of Lemma \ref{lem:L_lips}. 
Setting $b_1=0$ in the definition of $L_2(\cdot)$ (\ref{eq:def_L2}), we obtain:
\begin{equation}
    L_2(0, t, b)=\inf_{\theta} \mathcal{Q}(\theta,t,b),
\end{equation}
where $\mathcal{Q}(\cdot)$ is defined as:
\begin{equation}
   \mathcal{Q}(\theta,t,b) \coloneqq \log \mathbb{E}[\exp(\theta  Y_1)\mathbf{1}_{[Y_1+b \geq t\abs{Y_2}]}].
\end{equation}

By adding and subtracting $\log \Pr[Y_1 \geq t\abs{Y_2}+b]]$, $\mathcal{Q}(\theta,t,b)$ can equivalently be expressed as:
\begin{equation}
    \mathcal{Q}(\theta,t,b) = \log \mathbb{E}[\exp(\theta  (Y_1-Y_2))\vert [Y_1+b \geq t\abs{Y_2}]] + \log \Pr[Y_1+b \geq t\abs{Y_2}]].
\end{equation}

Hence, by properties of log-MGF, $\mathcal{Q}(\theta,t,b)$ is strictly-convex in $\theta$. To establish the existence of minimizers, for $0<\delta_1<\delta_2$, let $B(\delta_1,\delta_2)$ denote the spherical shell of points with norm between $\delta_1, \delta_2$ i.e:
\[
B(\delta_1,\delta_2) = \{x\in \mathbb{R}^2\delta_1:\leq\norm{x}\leq \delta_2\}.
\]

The condition $b_2<0$ implies that $\exists \delta>0$ such that $B(\delta)$ lies in the support of $Y_1,Y_2$. Therefore:
\[
\log \mathbb{E}[\exp(\theta  (Y_1-Y_2))\vert [Y_1+b \geq t\abs{Y_2}]] \geq \log \mathbb{E}[\exp(\abs{\theta}\delta_1) 1_{(Y_1,Y_2) \in B(\delta)}\vert [Y_1+b \geq t\abs{Y_2}]].
\]
Hence, when $b>0$, $\mathcal{Q}(\theta,t,b) \rightarrow \infty$ as $\abs{\theta} \rightarrow \infty$, implying that  
$\mathcal{Q}(\theta,t,b)$ possesses a unique maximizer. The continuous differentiability of $L_2$ then follows by the implicit function theorem.

\subsubsection{Regularity of $W(\cdot)$ (Proof of Proposition \ref{prop:twice_diff}).}

Our proof is based on the convex-concavity of $F(\cdot)$ which we establish below. 

\begin{proposition}\label{prop:convex_2ndmom}
For any $t \in \mathbb{R}$, $F(E,\omega,h,t,\theta_1, \theta_2)$ is strictly convex in $\theta_1, \theta_2$ and strictly concave in $t$. Furthermore, the $\sup\inf$ objective in (\ref{eq:Wxbeta}) admits a unique maximizer $t^*$ and unique minimizers $\theta_1^*(E,\omega),\theta_2^*(E,\omega)$ at $t=t^\star$.
\end{proposition}  

\begin{proof}
Both the strict-concavity w.r.t $t$ and the strict convexity w.r.t $\theta_1,\theta_2$ follow directly through the proof of Lemma \ref{lem:l2conc}.
\end{proof}

The above proposition further implies the regularity of $W(\cdot)$ w.r.t its arguments.

\paragraph{Proof of Proposition \ref{prop:twice_diff}}
   Recall that by Proposition \ref{prop:convex_2ndmom}, $F(\cdot)$ defined in \eqref{def:F2ndmom} is strictly convex-concave in $(\theta_1,\theta_2), t$ for any $E \in \mathbb{R}$ and for $E < -h/2$, $F$ admits a unique saddle point $t^\star,(\theta^\star_1,\theta^\star_2)$ achieving the extremum i.e. satisfying:
   \begin{equation}
       F(t^\star,\theta^\star_1,\theta^\star_2) = \sup_t \sup_{\theta_1, \theta_2} F(t,\theta_1, \theta_2).
   \end{equation}
   
   Furthermore, $F$ is continuously differentiable w.r.t $E, \omega, h, t, \theta_1,\theta_2$ for any $\omega \in (-1,1)$.  Consider the Hessian $\nabla^2 F$ at $t^\star,(\theta^\star_1,\theta^\star_2)$:
   \begin{equation}
       \nabla^2 F = \begin{pmatrix}
           \nabla^2_\theta F & \nabla_{\theta} \nabla_t F\\
           \nabla_{t} \nabla_{\theta} F &\nabla^2_t F.
       \end{pmatrix}
   \end{equation}

By the Schur-complement inverse criteria \citep{horn2012matrix},
$\nabla^2 F$ is invertible if and only if the matrices
\begin{equation}
    \nabla^2_\theta F,
\end{equation}
and \begin{equation}
    (\nabla^2_t F - \nabla_{\theta} \nabla_t F  (\nabla^2_t F)^{-1} \nabla_t \nabla_{\theta} F),
\end{equation} 
are invertible.

Now, the strict convexity-concavity of $F$ implies that $\nabla^2_\theta F \succ 0$ while $\nabla^2_t F \prec 0$. Therefore:
\begin{equation}
    \nabla^2_\theta F \succ 0,
\end{equation}
while,
\begin{equation}
    (\nabla^2_t F - \nabla_{\theta} \nabla_t F (\nabla^2_t F)^{-1} \nabla_t \nabla_{\theta}  F^\top) \prec 0,
\end{equation} 
implying the invertibility of both the matrices. Subsequently, the Hessian $\nabla^2 F$ at the unique $t^\star,(\theta^\star_1,\theta^\star_2)$ achieving optimality in the optimization problem defined by (\ref{eq:Wxbeta}) is invertible. 
   
  Since the fixed-point equations for $t^\star,(\theta^\star_1,\theta^\star_2)$ correspond to $\nabla F=0$, 
  the implicit function theorem implies that $t^\star,(\theta^\star_1,\theta^\star_2)$ are  twice-continuously-differentiable w.r.t $\omega, h, E$. This in-turn implies the twice-continuous differentiability of $W(E,\omega,h)$.

\subsubsection{Proof of Proposition \ref{prop:second_mom_max}}

Substituting $\omega=0$ in the definition of  $F(\cdot)$ in \eqref{def:F2ndmom}, we observe that $F(\cdot)$ simplifies as:
\begin{align}
 F(E,0,h,t,\theta_1, \theta_2)   
&= 2 \log 2 \\
&-4{t^2}-4{(-E-t)^2}
+1/2
\log P\left(\theta_1,1,
4\sqrt{2}t-2h\right) \notag\\
&+
1/2\log P\left(\theta_2,1,
{4\sqrt{2}(-E-t)}-2h\right) , \label{def:F2ndmomsim}
\end{align}

By the concavity of $t \rightarrow \inf_{\theta} \log P\left(\theta,1,
{4\sqrt{2}t}-2h\right)$, we obtain that the RHS is maximized at $t^\star=-E-t^\star=-\frac{E}{2}$. For this value of $t^\star$, the optimality conditions for $\theta_1,\theta_2$ further yield $\theta^\star_1=\theta^\star_2$, implying:
\begin{equation}
\begin{split}
    \sup_{t} \inf_{\theta_1,\theta_2} F(E,0,h,t,\theta_1, \theta_2) &=  \sup_{t} \inf_{\theta} \Biggl( 2 \log 2 -4{t^2}-4{(-E-t)^2}
+1/2
\log P\left(\theta,1,
4\sqrt{2}t-2h\right) \notag\\
&+
1/2\log P\left(\theta,1,
{4\sqrt{2}(-E/4}-2h\right) \Biggr).
\end{split}
\end{equation}
Substituting the definition of $P(\cdot)$, we find that the LHS equals $2w(E,h)$.

\subsubsection{Proof of Proposition \ref{prop:l2inf}}

From (\ref{eq:def_L2}), we have:
    \begin{equation}
        L_2(0, b_2, b_3)=\inf_{\theta}\log \mathbb{E}[\exp(\theta_1 Y_1)\mathbf{1}_{[Y_1+b_3 \geq b_2\abs{Y_2}]}],
     \end{equation}
where $Y_1,Y_2 
\sim \mathcal{N}(0,1)$.  Through an identical argument as Proposition \ref{prop:Linf}, setting $\theta=\frac{1}{b_3^{3/2}}$, we obtain that $ L_2(0, b_2, b_3) \rightarrow -\infty$ as $b_3 \rightarrow 0$.

\subsubsection{Proof of Lemma \ref{lem:W_cont}}

\begin{proof}

    Recall the definition of $W(E,\omega,h)$ in (\ref{eq:Wxbeta}):
    \[W(E,\omega,h) = \sup_{t \in \mathbb{R}} \inf_{\theta_1, \theta_2 \in \mathbb{R}^2} F(E,\omega,h,t,\theta_1, \theta_2). \]

By setting $\theta_1,\theta_2=0$, we note that as $\omega \rightarrow -1$, $\inf_{\theta_1, \theta_2 \in \mathbb{R}^2} F(E,\omega,h,t,\theta_1, \theta_2)$ is bounded as:

\[
\inf_{\theta_1, \theta_2 \in \mathbb{R}^2} F(E,\omega,h,t,\theta_1, \theta_2) \leq 
    -\frac{t^2}{1-\omega}.
\]

Hence, as $\omega \rightarrow -1$, the set of maximizers converges to $t=0$. Substituting $t=0$ in the definition of $F(E,\omega,h,t,\theta_1, \theta_2)$ in (\ref{def:F2ndmom}), we recover $w(E,h)$. 
\end{proof}

\subsection{Non-triviality of $E_{\rm cor}(h), h_{\rm cor}$ (Proof of Proposition \ref{prop:hcor})} \label{sec:exist_ecor}

In this section, we prove Proposition \ref{prop:hcor}, i.e. the existence of $E_{\rm cor}(h), h_{\rm cor}$ defined in Equations \ref{def:E_cor},\ref{def:h_cor}, and Lemma \ref{lem:contecor}. To this end, we first establish certain properties of the function $W(E,\omega,h)$ that allow us to prove respectively. These rely on one or more of the Assumptions \ref{ass:unique_max}-\ref{ass:e_max_h}, whose numerical verifications is detailed in Section \ref{sec:num_W}. We recall that the proof of Theorem \ref{thm:sparse_max_E_thres}  only requires the existence of $E_{\rm cor}(h), h_{\rm cor}$ without any assumptions about their numerical values.

We begin by showing the regularity of $E_{min}(h),E_{max}(h)$:
\begin{proposition}\label{prop:e_min_cont}
    $E_{\rm max}(h)-E_{\rm min}(h)$ is continuous w.r.t $h$ for all  $h\leq h^\star$.
\end{proposition}

\begin{proof}
 Recall that Proposition \ref{prop:first_mom_conv}  implies that $w(\cdot,h)$ is continuously differentiable, strictly concave and possesses a unique maximizer  $E^\star(h)$. Since for any  $h<h^\star$, we have that $w(h)= w(E^\star(h),h)>0$, we obtain that $E_{\min}(h)<E^\star(h)<E_{\max}(h)$. The uniqueness of the maximizer $E^\star(h)$ then implies that $\frac{dw}{dE} \neq 0$ for any $E \neq E^\star(h)$.
     
Therefore, when $h<h^\star$, $\frac{dw}{dE} \neq 0$ for $E=E_{\min}(h)$ and $E =E_{\max}(h)$. Subsequently, applying the implicit-function theorem to $w(E,h)=0$ we obtain the continuity of $E_{min}, E_{max}$ w.r.t $h$ for $h<h^\star$.

To obtain continuity at $h=h^\star$ , note that since $E_{min}, E_{max}$ are bounded, the following limits exist:
\begin{equation}
    E^+_{min} = \lim_{h \rightarrow h^\star}E_{min}(h),
\end{equation}
and
\begin{equation}
    E^+_{max} = \lim_{h \rightarrow h^\star}E_{max}(h).
\end{equation}

By continuity of $W(E,h)$ w.r.t $E$, we have that $E^+_{min}= E^+_{max}= E^\star(h^\star)$.

\end{proof}

Combining the above results with Assumptions \ref{ass:unique_max}-\ref{ass:0_unique}, we obtain the non-triviality of $h_{\rm cor}, E_{\rm cor}(h)$.

\subsubsection{Proof of Proposition \ref{prop:hcor}}

The proof follows through a continuity argument based on Propositions \ref{prop:twice_diff}  and \ref{prop:e_min_cont}. 
Recall that by Assumption \ref{ass:0_unique}, $W(E^\star(h^\star),\omega, h^\star)$ is uniquely maximized w.r.t $\omega$ at $\omega=\{-1,0,1\}$. Therefore $\partial_{-}\omega W(E^\star(h^\star),1, h^\star) >0$, $\partial_{+}\omega W(E^\star(h^\star),1, h^\star) <0$ and $\nabla^2_\omega W(E^\star(h^\star),0, h^\star) < 0$.

Subsequently, the twice continuous differentiability of $W(\cdot)$ (Proposition \ref{prop:twice_diff}) implies that $\exists \delta > 0$ and a  neighborhood of $h=h^\star, E=E^\star(h^\star)$ such that throughout the neighborhood, $W(E,\omega,h)$ is strictly decreasing in $[1-\delta,1]$ and strictly concave in $(-\delta,\delta)$. Furthermore, the differentiability of $W$ w.r.t. $h$ for $\omega \in (-1,1)$ implies the uniform Lipschitzness of $W(E,\omega,h)$ w.r.t. $E,h$ over $\omega$ in the closed interval $[\delta, 1-\delta]$. Therefore $\sup_{\omega \in [\delta, 1-\delta]} W(E,\omega,h) < W(E,0,h)$ for $h,E$ sufficiently close to $h^\star, E^\star(h^\star)$.

Combining  with the continuity of $E_{\rm max}(h)-E_{\rm min}(h)$ for $h \in (-\infty, h^\star]$ (Proposition \ref{prop:e_min_cont}), there exists  $\tilde{h}<h^\star$ such that $W(E,\omega,h)$ is uniquely maximized at $\omega=\{0,-1,1\}$ for  $h \in (\tilde{h}, h^\star)$ and $E \in (E_{\rm min}(h), E_{\rm max}(h))$. The feasible set for $h_{\rm cor}$ is therefore bounded by $\tilde{h} < h^\star$.

Next, to prove the existence of $E_{\rm cor}(h)$, we note that for any $h<h^\star$, by definition of $E_{\rm min}(h)$ and the strict-concavity of $w(\cdot,h)$, we have $w(E,h)<0$ for any $E<E_{\rm min}(h)$. Since, by Proposition \ref{prop:second_mom_max} we have that for all $E,h \in \mathbb{R}$, $W(E,0,h)=2w(E,h)$ while $w(E,h)=W(E,-1,h)$, we obtain that  $W(E,0,h)<W(E,-1,h)$ for all $E<E_{\rm min}(h)$. Thus, the feasible set for $E_{\rm cor}(h)$ is non-empty and thus $E_{\rm cor}(h)$ exists. Subsequently, through a similar continuity argument as for $h_{\rm cor}$ and Assumption \ref{ass:e_max_h}, we obtain the analogous result for the non-triviality of $E_{\rm cor}(h)$.
 
\subsubsection{Proof of Lemma \ref{lem:contecor}}

For  $h \in (h_{\rm cor},h^\star)$, Lemma \ref{lem:contecor} follows immediately by noting that  the proof of Proposition \ref{prop:e_min_cont}, implies the continuity of $E_{min}(h),E_{max}(h)$ w.r.t $h$. For $h \in (-0.1, h_{\rm cor})$, by the definition of $E_{cor}(h)$, $\exists E_\ell \leq E_{cor}(h)$ such that:

\begin{equation}
    W(E_\ell,0,h) < \sup_{\omega \in (-1,1)} W(E_\ell,\omega,h).
\end{equation}

Since both the LHS and the RHS are continuous w.r.t $h$, we obtain by Proposition \ref{prop:twice_diff} that for $\tilde{h},\tilde{E}$ sufficiently close to $E_\ell,h$, we have $W(\tilde{E},0,\tilde{h}) < \sup_{\omega \in [-1,1]} W(\tilde{E},\omega,\tilde{h})$.

We further have that for any $E_{\rm cor}(h) < E_u < E_{\rm max}(h)$
\begin{equation}
\operatorname{argmax}_{\omega \in (-1,1)} W(E,\omega,h) = \{0\},
\end{equation}
for all $E \in [E_u,  E_{\rm max}(h))$. An argument identical to the proof of  Proposition \ref{prop:hcor} then implies that the above condition holds for all $\tilde{h} \in (-0.1, h_{\rm cor})$ sufficiently close to $h$.

\section{Verification of Assumptions \ref{ass:unique_max}, \ref{ass:0_unique}, \ref{ass:e_max_h}} \label{sec:num_W}

To verify these assumptions, we leverage the convexity-concavity and the sup-inf structure of the variational objective underlying $W(E, \omega,h)$ to produce certificates bounding $W(E, \omega,h)$ in certain intervals.  We remark that all of the assumptions except Assumption \ref{ass:e_max_h} require checking the properties of $W(E, \omega,h)$ only at the points $h=h^\star$ and $h=0$. 

\subsection{Assumption \ref{ass:unique_max}}

Recall that, by Lemma \ref{lem:W_cont},
$W(E,\omega,h)$ can be continuously extended from  $\omega \in (-1,1)$ to $\omega \in [-1,1]$ by setting $W(E,-1,h)=W(E,1,h)=w(E,h)$.

We verify Assumption \ref{ass:unique_max} through the following series of numerical checks:
\begin{enumerate}
    \item Verify that $\omega=0$ in a maximizer over a sufficiently small neighborhood of the origin.
    \item Verify that $\frac{dW}{d\omega}<0$ ($\frac{dW}{d\omega}>0$) in sufficiently small neighborhood of $\omega=-1$ $(\omega=1)$. By Lemma \ref{lem:W_cont} and the mean value theorem, this verifies that $W(E,0,h)<w(E,h)$ in these neighborhoods.    \item Verify that in the remaining range of $\omega$, $W(E,\omega,h)$ is bounded above by $W(E,0,h)$.
\end{enumerate}

Verifying claims $1,2$ requires only the approximation of $\frac{dW}{d\omega}$ over chosen, (arbitrarily) small neighborhoods of $\omega=-1,1,0$.
To allow tractable verification of the third claim, we shall leverage the sup-inf structure of $W(E,\omega,h)$. By the variational definition of $W$ given by (\ref{eq:Wxbeta}), we may interchange sup-inf to obtain the following upper bound on $W(E,\omega,h)$:
\begin{equation}
    W(E,\omega,h) \leq \sup_{t \in \mathbb{R}} F(E,\omega,h,t,\tilde{\theta}_1, \tilde{\theta}_2),
\end{equation}
for any $\tilde{\theta}_1, \tilde{\theta}_2 \in \mathbb{R}$. This allows us to produce ``certificates" on $W(E,\omega,h)$ by optimizing over $t$ on suitably chosen values of $\theta_1,\theta_2$. These ``candidate" values are obtained using the findroot method of mpmath with warm-start across different values of $\omega$.

Moreover, recall from \eqref{def:F2ndmom}, that $F(\cdot)$ can be expressed as:
\begin{align*}
    F(E,\omega,h,t,\theta_1,\theta_2) &= 2 \log 2 -H(\frac{(1+\omega)}{2})\notag-4{t^2\over (1+\omega)^2}-4{(-E-t)^2\over (1-\omega)^2} + \mathcal{F}(E,\omega,h,t,\theta_1,\theta_2),
\end{align*}
where:
\begin{equation}
\begin{split}
    &\mathcal{F}(E,\omega,h,t,\theta_1,\theta_2)\\ &= (1+\omega)/2
\log \mathcal{Q}\left(\theta_1,\sqrt{1-\omega\over 1+\omega},
{4\sqrt{2}t\over (1+\omega)^{3/2}}-\frac{2h}{\sqrt{(1+\omega)}}\right) \notag\\
&+(1-\omega)/2
\log \mathcal{Q}\left(\theta_2,\sqrt{1+\omega\over 1-\omega},
{4\sqrt{2}(-E-t)\over (1-\omega)^{3/2}}-\frac{2h}{\sqrt{(1-\omega)}}\right). 
\end{split}
\end{equation}

Note that from the definition of $\mathcal{Q}(\theta,a,b)$ in \eqref{eq:Q_def}, we have that increasing $a$ or decreasing $b$ decreases the support involved in the expectation in $\mathcal{Q}(\theta,a,b)$ 
Therefore, $\mathcal{Q}(\theta,a,b)$ is non-increasing in $a$ and non-decreasing in $b$.
Now, suppose that $0 < \omega_1< \omega_2< 1$. Then, for any $\omega \in [\omega_1,\omega_2]$ and $t \in \mathbb{R}$:
\begin{equation}\label{eq:omega_bound_1}
\frac{1+\omega}{2}\mathcal{Q}\left(\theta_1, \sqrt{\frac{(1-\omega)}{(1+\omega)}}, \frac{4\sqrt{2}t}{(1+\omega)^{3/2}}-\frac{2h}{\sqrt{1+\omega}} \right)  \leq \frac{1+\omega_2}{2}\mathcal{Q}\left(\theta_1, \sqrt{\frac{(1-\omega_2)}{(1+\omega_2)}}, \frac{4\sqrt{2}t}{(1+\omega_1)^{3/2}}-\frac{2h}{\sqrt{1+\omega_1}}\right).
\end{equation}

Similarly, we have:

\begin{equation}\label{eq:omega_bound_2}
\frac{1-\omega_2}{2}\mathcal{Q}\left(\theta_2, \sqrt{\frac{(1+\omega)}{(1-\omega)}}, \frac{4\sqrt{2}(-E-t)}{(1-\omega)^{3/2}}-\frac{2h}{\sqrt{1-\omega}}\right)  \leq \frac{1-\omega_1}{2}\mathcal{Q}\left(\theta_2, \sqrt{\frac{(1+\omega_1)}{(1-\omega_1)}}, \frac{4\sqrt{2}(-E-t)}{(1-\omega_2)^{3/2}}-\frac{2h}{\sqrt{1-\omega_2}}\right).
\end{equation}

Furthermore, since $H(\cdot)$ is non-increasing on $[1/2,1]$, we further obtain that for any $\omega \in [\omega_1,\omega_2]$
\begin{equation}
  -(1+\omega)/2\log(1+\omega)-(1-\omega)/2\log(1-\omega) \leq  -(1+\omega_1)/2\log(1+\omega_1)-(1-\omega_1)/2\log(1-\omega_1). 
\end{equation}

Combining the above series of bounds, we obtain that for any $\tilde{\theta}_1,\tilde{\theta}_2$:
\begin{equation}\label{eq:F_2bound}
\begin{split}
 \sup_{\omega \in [\omega_1,\omega_2]}  & W(E,\omega,h)  \leq \sup_{t \in \mathbb{R}} \Biggl( 2 \log 2 -(1+\omega_1)/2\log(1+\omega_1)-(1-\omega_1)/2\log(1-\omega_1)\\
&-4{t^2\over (1+\omega_2)^2}-4{(-E-t)^2\over (1-\omega_2)^2} + \frac{1+\omega_2}{2}\frac{1+\omega_2}{2}\mathcal{Q}\left(\tilde{\theta}_1, \sqrt{\frac{(1-\omega_2)}{(1+\omega_2)}}, \frac{4\sqrt{2}t}{(1+\omega_1)^{3/2}}-\frac{2h}{\sqrt{1+\omega_1}}\right)\\&+\frac{1-\omega_1}{2}\mathcal{Q}\left(\tilde{\theta}_2, \sqrt{\frac{(1+\omega_1)}{(1-\omega_1)}}, \frac{4\sqrt{2}(-E-t)}{(1-\omega_2)^{3/2}}-\frac{2h}{\sqrt{1-\omega_2}}\right).
\end{split}
\end{equation}

Therefore, upper bounding $W(E,\omega,h)$ over a range of $\omega$ reduces to constructing a fine-enough partitioning $\omega_1,\cdots,\omega_m$ and 
producing  for each interval $[\omega_j,\omega_{j+1}]$, parameter values $\tilde{\theta}_1, \tilde{\theta}_2$ such that the RHS in (\ref{eq:F_2bound}) lies below
$W(E,0,h)$.

We note that the above assumption was independently verified in \cite{minzer2023perfectly} through the use of interval arithmetic.

\subsection{Assumption \ref{ass:0_unique}}

We check Assumption \ref{ass:0_unique} numerically by demonstrating the existence of an $\omega_0\neq 0$ such that $W(E,\omega,0) > W(E,0,0)$. By the definition of $W(E,\omega,0)$, we have that for any $t \in \mathbb{R}$:
\begin{equation}
   \inf_{\theta_1,\theta_2}
 F(E,\omega,h,t,\theta_1,\theta_2)  \leq  W(E,\omega,0).
\end{equation}

While Proposition \ref{prop:second_mom_max} implies that $W(E,0,0)=2w(E,0)$. Therefore, to verifying Assumption \ref{ass:0_unique}, it suffices to establish the existence of $t \in \mathbb{R}$ such that:
\begin{equation}
   \inf_{\theta_1,\theta_2}
 F(E,\omega,h,t,\theta_1,\theta_2)  > 2w(E,0).
\end{equation}
where recall that the optimization in the LHS is strictly convex and thus amenable to simple first order algorithms.

\subsection{Assumption \ref{ass:e_max_h}}

Finally, Assumption \ref{ass:e_max_h} is verified through a procedure identical to Assumption  \ref{ass:unique_max}, along with an additional partitioning over $h$. Recall that by the definition of $\mathcal{Q}(\theta,a,b)$, $\mathcal{Q}(\theta,a,b)$ is non-decreasing in $b$. Hence, we obtain, for any $h_1 < h_2$:
\[
\frac{1+\omega}{2}\mathcal{Q}(\theta_1 \sqrt{\frac{(1-\omega)}{(1+\omega)}}, \frac{4\sqrt{2}t}{(1+\omega)^{3/2}}-\frac{2h_1}{\sqrt{1+\omega}}) \geq \frac{1+\omega}{2}\mathcal{Q}(\theta_1, \sqrt{\frac{(1-\omega)}{(1+\omega)}}, \frac{4\sqrt{2}t}{(1+\omega)^{3/2}}-\frac{2h_2}{\sqrt{1+\omega}}),
\]
and:
\[
\frac{1-\omega_2}{2}\mathcal{Q}\left(\theta_2, \sqrt{\frac{(1+\omega)}{(1-\omega)}}, \frac{4\sqrt{2}(-E-t)}{(1-\omega)^{3/2}}-\frac{2h_1}{\sqrt{1-\omega}}\right) \geq \frac{1-\omega_2}{2}\mathcal{Q}\left(\theta_2, \sqrt{\frac{(1+\omega)}{(1-\omega)}}, \frac{4\sqrt{2}(-E-t)}{(1-\omega)^{3/2}}-\frac{2h_2}{\sqrt{1-\omega}}\right)
\]

Combining the above bounds with (\ref{eq:omega_bound_1}) and (\ref{eq:omega_bound_2}) then results in the following bound for $W(E,\omega,h)$ for $\omega,h \in [\omega_1,\omega_2] \times [h_1,h_2]$, analogous to \eqref{eq:F_2bound}:

\begin{equation}
\begin{split}
  \sup_{h \in [h_1,h_2]}\sup_{\omega \in [\omega_1,\omega_2]} \sup_{t \in \mathbb{R}} & W(E,\omega,h)  \leq \sup_{t \in \mathbb{R}} \Biggl( 2 \log 2 -(1+\omega_1)/2\log(1+\omega_1)-(1-\omega_1)/2\log(1-\omega_1)\\
&-4{t^2\over (1+\omega_2)^2}-4{(-E-t)^2\over (1-\omega_2)^2} + \frac{1+\omega_2}{2}\mathcal{Q}\left(\tilde{\theta}_1, \sqrt{\frac{(1-\omega_2)}{(1+\omega_2)}}, \frac{4\sqrt{2}t}{(1+\omega_1)^{3/2}}-\frac{2h_1}{\sqrt{1+\omega_1}}\right)\\&+\frac{1-\omega_1}{2}\mathcal{Q}\left(\tilde{\theta}_2, \sqrt{\frac{(1+\omega_1)}{(1-\omega_1)}}, \frac{4\sqrt{2}(-E-t)}{(1-\omega_2)^{3/2}}-\frac{2h_1}{\sqrt{1-\omega_2}}\right) \Biggr),
\end{split}
\end{equation}
for any $\tilde{\theta}_1, \tilde{\theta}_2$.

The restriction of Assumption \ref{ass:e_max_h} to the bounded interval
$h\in(-0.1,h^\star)$ then allows the applicability of the above bound over finitely many intervals of the form $[h_1,h_2]$.

% Again, by the definition of $L_2$, it is easy to verify that for any $h_1,h_2 \in \mathbb{R}$:
% \begin{equation}
%     \frac{1+\omega}{2}L_2(0, \sqrt{\frac{(1-\omega)}{(1+\omega)}}, \frac{4\sqrt{2}t}{(1+\omega)^{3/2}}-\frac{2h_1}{\sqrt{1-\omega}})-L_2(0, \sqrt{\frac{(1-\omega)}{(1+\omega)}}, \frac{4\sqrt{2}t}{(1+\omega)^{3/2}}-\frac{2h_1}{\sqrt{1-\omega}}) \leq 
% \end{equation}
 
\end{document}

%% file: macros.tex
  \providecommand{\R}{\mathbb{R}} % Reals
  \providecommand{\N}{\mathbb{N}} % Naturals

  \makeatletter
  \def\sign{\@ifnextchar*{\@sgnargscaled}{\@ifnextchar[{\sgnargscaleas}{\@ifnextchar{\bgroup}{\@sgnarg}{\sgn} }}}
  \def\@sgnarg#1{\sgn\rbr{#1}}
  \def\@sgnargscaled#1{\sgn\rbr*{#1}}
  \def\@sgnargscaleas[#1]#2{\sgn\rbr[#1]{#2}}
  \makeatother

  \providecommand{\ddot}{\dot\dot}
  \providecommand{\dddot}{\dot\dot}

  \providecommand{\cE}{\mathcal{E}}

  \providecommand{\cH}{\mathcal{H}}

  \providecommand{\cO}{\mathcal{O}}

  \providecommand{\cX}{\mathcal{X}}


\newcommand{\speedup}[1]{{\color{gray}(\ifdim #1 pt > 0.3pt #1\else $< #1$\fi{}$\times$)}}
\usepackage{amsfonts}
\usepackage{amsmath}
\usepackage{amsthm}
\usepackage{amssymb}
\usepackage{dsfont}

\usepackage{xcolor}
\usepackage{color}
\usepackage{graphicx}

\usepackage{verbatim}
\usepackage{xspace} %

\usepackage{enumerate}
\usepackage{enumitem}

\providecommand{\abs}[1]{\left\lvert#1\right\rvert}
\providecommand{\norm}[1]{\left\lVert#1\right\rVert}

  \providecommand{\R}{\mathbb{R}} %
  \providecommand{\N}{\mathbb{N}} %

  \newcommand{\Ea}[1]{\E\left[#1\right]}
\newcommand{\Eb}[2]{\E_{#1}\left[#2\right]}

  \providecommand{\erf}{\operatorname{erf}}

  \providecommand{\cE}{\mathcal{E}}

  \providecommand{\cH}{\mathcal{H}}

  \providecommand{\cO}{\mathcal{O}}

  \providecommand{\cX}{\mathcal{X}}

  \usepackage{bm}

  \usepackage[textwidth=5cm]{todonotes}
  
\providecommand{\mycomment}[3]{\todo[caption={},size=footnotesize,color=#1!20]{\textbf{#2: }#3}}%
\providecommand{\inlinecomment}[3]{%
  {\color{#1}#2: #3}}%
\newcommand\commenter[2]%
{%
  \expandafter\newcommand\csname i#1\endcsname[1]{\inlinecomment{#2}{#1}{##1}}
  \expandafter\newcommand\csname #1\endcsname[1]{\mycomment{#2}{#1}{##1}}
}

  \definecolor{mydarkblue}{rgb}{0,0.08,0.45}
  \usepackage{hyperref}
  \hypersetup{ %
    pdftitle={},
    pdfauthor={},
    pdfsubject={},
    pdfkeywords={},
    pdfborder=0 0 0,
    pdfpagemode=UseNone,
    colorlinks=true,
    linkcolor=mydarkblue,
    citecolor=mydarkblue,
    filecolor=mydarkblue,
    urlcolor=mydarkblue,
    pdfview=FitH}
\usepackage[capitalize,noabbrev]{cleveref}

\usepackage{url}